\documentclass[11pt]{article}

\usepackage[T1]{fontenc}
\usepackage[utf8]{inputenc}
\usepackage{lmodern}
\usepackage[margin=1in]{geometry}
\usepackage{amsmath,amssymb,amsthm,mathtools}
\usepackage{microtype}
\usepackage{enumitem}
\usepackage{needspace}
\usepackage{xcolor}
\usepackage{booktabs}
\usepackage{graphicx}
\usepackage{tikz}
\usetikzlibrary{arrows.meta,calc,positioning,shapes.misc,shapes.geometric,
  patterns,decorations.pathreplacing,fit,backgrounds}
\usepackage[font=small,labelfont=bf]{caption}
\usepackage{subcaption}
\usepackage{thmtools,thm-restate}
\usepackage[colorlinks=true,linkcolor=blue!50!black,citecolor=green!35!black,
  urlcolor=blue!50!black]{hyperref}
\expandafter\def\expandafter\UrlBreaks\expandafter{\UrlBreaks\do\-}

\theoremstyle{plain}

\newtheorem{proposition}{Proposition}[section]
\newtheorem{lemma}[proposition]{Lemma}
\newtheorem{corollary}[proposition]{Corollary}

\newtheorem*{theorem*}{Theorem}
\theoremstyle{definition}

\newtheorem{remark}[proposition]{Remark}

\newtheorem*{remark*}{Remark}
\counterwithin{equation}{section}

\makeatletter
\@for\PB@env:=proposition,lemma,corollary,claim,definition,remark,example\do{%
  \expandafter\def\csname theH\PB@env\endcsname{\theHsection.\arabic{proposition}}}
\makeatother

\newlist{thmlist}{enumerate}{1}
\setlist[thmlist]{label=\textup{(\alph*)}, ref=(\alph*), leftmargin=*, widest=b,
  topsep=3pt, itemsep=2pt}
\newlist{steplist}{enumerate}{1}
\setlist[steplist]{label=\textup{(\roman*)}, ref=(\roman*), leftmargin=*, widest=iii,
  topsep=3pt, itemsep=2pt}

\newcommand{\E}{\mathbb{E}}
\newcommand{\Prob}{\mathbb{P}}
\newcommand{\R}{\mathbb{R}}
\newcommand{\N}{\mathbb{N}}
\newcommand{\ind}{\mathbf{1}}
\newcommand{\id}{\mathrm{id}}
\DeclareMathOperator{\supp}{supp}
\DeclareMathOperator{\sech}{sech}
\DeclareMathOperator{\Law}{Law}
\DeclareMathOperator{\Shift}{Shift}
\newcommand{\Wone}{W_1}                      %
\newcommand{\Pspace}{\mathcal{P}}             %
\newcommand{\normCfour}[1]{\|#1\|_{C^4[0,1]}}

\newcommand{\xic}{\xi_c}                      %
\newcommand{\azero}{a_0}                      %
\newcommand{\ZN}{Z_N}                         %
\newcommand{\FNb}{F_{N,\beta}}                %

\newcommand{\feq}{f_{\mathrm{eq}}}            %
\newcommand{\Lambar}{\overline{\Lambda}}      %
\newcommand{\Lamhat}{\widehat{\Lambda}}       %
\newcommand{\zetas}{\zeta_*}                  %
\newcommand{\alphas}{\alpha_*}                %
\newcommand{\Obs}{\mathcal{O}}                %
\newcommand{\Ustar}{U_*}                      %
\newcommand{\Wstar}{W_*}                      %
\newcommand{\Dstar}{D_*}                      %
\newcommand{\rhostar}{\rho_*}                 %
\newcommand{\rhob}{\rho_\beta(0)}             %
\newcommand{\cb}{c_\beta}                     %
\newcommand{\cPR}{c_5^{\mathrm{PR}}}          %
\newcommand{\qdot}{\dot q(0)}                 %
\newcommand{\Kcl}{\mathcal{K}}                %
\newcommand{\Kclodd}{\mathcal{K}_3}           %
\newcommand{\Tfrak}{\mathfrak{T}}             %
\newcommand{\Sfrak}{\mathfrak{S}}             %
\newcommand{\qrs}{q^{\mathrm{rs}}}            %
\newcommand{\Ffield}{\mathcal{F}}             %
\newcommand{\mfield}{\mathfrak{m}}            %
\newcommand{\qmin}{q_{\min}}                  %

\newcommand{\condH}{\hyperref[hyp:H]{\textup{(H)}}}

\newcommand{\condFM}{\hyperref[hyp:FM]{\textup{(FM)}}}

\newcommand{\DisplayCondH}{%
\begin{equation*}\tag{H}\label{hyp:H}
 0\ \text{is an accumulation point of}\ \supp\zetas .
\end{equation*}}
\newcommand{\DisplayCondHp}{%
\begin{equation*}\tag{H$'$}\label{hyp:Hp}
 [0,\kappa]\subset\supp\zetas\ \text{for some}\ \kappa>0 ,
\end{equation*}}
\newcommand{\DisplayCondFM}{%
\begin{equation*}\tag{FM}\label{hyp:FM}
 \lim_{N\to\infty}\frac1N\log\E\ZN^{\theta}=\Lambda(\theta)
 \qquad\text{for every}\ \theta\in(0,1),
\end{equation*}}

\newcommand{\StateThmSK}{%
\begin{restatable}[The Sherrington--Kirkpatrick model at every $\beta>1$]{theorem}{thmsk}
\label{thm:sk}
Let $\xi(t)=\beta^2t^2/2$ with $\beta>1$. Recall that zero is an accumulation point of
$\supp\zetas$ \cite[Proposition~5.1]{Lop26a} \textup{(}see also
\cite[Theorem~5.1]{Pei26}\textup{)}, and that \condFM\ holds for this covariance
\cite[Theorem~2.1]{CM19} \textup{(}see Lemma~\ref{lem:A-fm-sk}\textup{)}. The
distribution function $\alphas$ of $\zetas$ is continuously differentiable near zero,
and its slope at zero is
\[
 \rhob:=\alphas'(0)=\frac{\beta^5\Wstar^2}{2}>0,
 \qquad \Wstar=\Phi_{\zetas,xxxx}(0,0).
\]
Put $\cb=\frac{9}{640}\,\beta^3\rhob^{-3}$.
\begin{thmlist}
\item As $\theta\downarrow0$,
 \[
  p(\theta)-\feq=\cb\,\theta^5\,(1+o(1)),\qquad
  q_\theta=\frac{3\theta}{2\rhob}\,(1+o(1)),\qquad
  \frac{\zeta_\theta(\{q_\theta\})}{\theta}\longrightarrow\frac12,
 \]
 and the first positive contact of $\zeta_\theta$ is marginal for every sufficiently
 small $\theta>0$.
\item As $\delta\downarrow0$,
 \[
  \Sigma(\feq+\delta)=-\frac56\,(6\cb)^{-1/5}\,\delta^{6/5}\,(1+o(1)),
 \]
 and there is $\bar\delta>0$ such that the first positive contact at every level in
 $(\feq,\feq+\bar\delta)$ is marginal.
\item For every $f\in(\feq,f_1)$,
 \[
  \lim_{N\to\infty}\frac1N\log\Prob\bigl(N^{-1}\log\ZN\ge f\bigr)=\Sigma(f),
 \]
 and consequently
 \[
  \lim_{\delta\downarrow0}\delta^{-6/5}\lim_{N\to\infty}
  \Bigl(-\frac1N\log\Prob\bigl(N^{-1}\log\ZN\ge\feq+\delta\bigr)\Bigr)
  =\frac56\,(6\cb)^{-1/5}.
 \]
 Both limits also hold for
 $\FNb=\log\sum_{\sigma\in\{-1,1\}^N}\exp\bigl(\beta N^{-1/2}\sum_{i<j}g_{ij}\sigma_i\sigma_j\bigr)$,
 with independent standard Gaussian $g_{ij}$, after $N^{-1}\log\ZN$ is replaced by
 $N^{-1}\FNb-\log2$.
\end{thmlist}
\end{restatable}}

\newcommand{\StateThmLocal}{%
\begin{restatable}[Onset of the constrained measure at zero overlap]{theorem}{thmlocal}
\label{thm:local}
Let $\xi(t)=\sum_{p\ge2}\beta_p^2t^p$ have radius of convergence greater than one, with
$\beta_2>0$ and $\xi'''(0)=0$, and assume \condH. Then $\Dstar>0$, $\zetas(\{0\})=0$ and
$\Ustar=\azero^{-1/2}$. There is $\kappa>0$ such that $[0,\kappa]\subset\supp\zetas$ and
$\alphas$ is continuously differentiable on $[0,\kappa)$, with
\[
 \rhostar:=\alphas'(0)=\frac{\Dstar}{2\sqrt{\azero}} .
\]
Put $c=\dfrac{9\azero^3}{80\Dstar^3}=\dfrac{9}{640}\,\azero^{3/2}\rhostar^{-3}$.
\begin{thmlist}
\item \emph{Marginal onset.} There is $\theta_0\in(0,1)$ such that, for
 $0<\theta<\theta_0$, $\Gamma_{\zeta_\theta}(t)=t$ on an interval $[q_\theta,s_\theta]$
 with $s_\theta>q_\theta$; in particular the first positive contact $q_\theta$ is
 marginal. As $\theta\downarrow0$,
 \[
  \frac{q_\theta}{\theta}\longrightarrow\frac{3\sqrt{\azero}}{\Dstar}=\frac{3}{2\rhostar},
  \qquad
  \frac{\zeta_\theta(\{q_\theta\})}{\theta}\longrightarrow\frac12 .
 \]
\item \emph{Fifth-order law.} As $\theta\downarrow0$,
 \begin{gather*}
  p'(\theta)=5c\,\theta^4(1+o(1)),\qquad
  p(\theta)-\feq=c\,\theta^5(1+o(1)),\\
  \Lambda(\theta)-\theta\feq=c\,\theta^6(1+o(1)),\qquad
  \Lambda'(\theta)-\feq=6c\,\theta^5(1+o(1)).
 \end{gather*}
\item \emph{Six-fifths law.} $\Lambda$ is continuously differentiable and strictly convex
 on $(0,1)$, $\feq<f_1$, and $\theta_f=(\Lambda')^{-1}(f)$ for every $f\in(\feq,f_1)$.
 As $\delta\downarrow0$, with $f=\feq+\delta$,
 \[
  \Sigma(\feq+\delta)=-\frac56\,(6c)^{-1/5}\,\delta^{6/5}(1+o(1)),\qquad
  \theta_f=\Bigl(\frac{\delta}{6c}\Bigr)^{1/5}(1+o(1)),\qquad
  q_f=\frac{3}{2\rhostar}\,\theta_f\,(1+o(1)),
 \]
 and there is $\bar\delta>0$ such that the first positive contact at every level in
 $(\feq,\feq+\bar\delta)$ is marginal.
\item \emph{Free energy.} If \condFM\ holds, then for every $f\in(\feq,f_1)$
 \[
  \lim_{N\to\infty}\frac1N\log\Prob\bigl(N^{-1}\log\ZN\ge f\bigr)=\Sigma(f),
 \]
 and consequently
 \[
  \lim_{\delta\downarrow0}\delta^{-6/5}\lim_{N\to\infty}
  \Bigl(-\frac1N\log\Prob\bigl(N^{-1}\log\ZN\ge\feq+\delta\bigr)\Bigr)
  =\frac56\,(6c)^{-1/5}.
 \]
\end{thmlist}
\end{restatable}}

\newcommand{\StateThmField}{%
\begin{restatable}[A small external field]{theorem}{thmfield}
\label{thm:field}
Assume the hypotheses of Theorem~\ref{thm:local}. For $h\in\R$ let $\zeta_h$ be the
unique minimizer of
$P^h_\zeta=\Phi_\zeta(0,h)-\frac12\int_0^1t\xi''(t)\zeta([0,t])\,dt$, and put
\[
 \Ffield(h)=P^h_{\zeta_h},\qquad \qmin(h)=\min\supp\zeta_h,\qquad
 \mfield(h)=\Phi_{\zeta_h,x}(0,h).
\]
Let $X^{\zeta,h}$ be the diffusion \eqref{eq:intro-diffusion} started from $h$, and
$\Gamma_{\zeta,h}(t)=\E\Phi_{\zeta,x}(t,X^{\zeta,h}_t)^2$. As $h\to0$, the following hold.
\begin{thmlist}
\item \emph{A gap with a marginal edge.} There are $h_1,\kappa_1>0$ such that, for
 $0<|h|<h_1$ and $q=\qmin(h)$,
 \[
  0<q<\kappa_1,\quad \zeta_h([0,q))=0,\quad [q,\kappa_1]\subset\supp\zeta_h,\quad
  \Gamma_{\zeta_h,h}(q)=q,\quad
  \xi''(q)\,\E\Phi_{\zeta_h,xx}(q,X^{\zeta_h,h}_q)^2=1 .
 \]
\item \emph{The gap and its atom.}
 \begin{gather*}
  \qmin(h)^3=\frac{3h^2}{\azero\Dstar}\,(1+o(1)),\\
  \zeta_h(\{\qmin(h)\})=\rhostar\,\qmin(h)\,(1+o(1))=\zetas([0,\qmin(h)])\,(1+o(1)).
 \end{gather*}
\item \emph{Magnetization.}
 $\displaystyle\mfield(h)=\frac{h}{\sqrt{\azero}}
 \Bigl(1-\frac{\Dstar}{4}\,\qmin(h)^2\,(1+o(1))\Bigr)$.
\item \emph{Free energy.} $\Ffield$ is even and continuously differentiable,
 $\Ffield'=\mfield$, $\Ffield''(0)=\azero^{-1/2}$, and
 \[
  \Ffield(h)=\Ffield(0)+\frac{h^2}{2\sqrt{\azero}}
  -\frac{3}{40}\,(9\Dstar)^{1/3}\azero^{-7/6}\,|h|^{10/3}\,(1+o(1)).
 \]
 In particular $\Ffield$ has no fourth-order Taylor expansion at zero.
\end{thmlist}
These conclusions hold for the Sherrington--Kirkpatrick covariance
$\xi(t)=\beta^2t^2/2$ at every $\beta>1$, for which \condH\ holds by
\cite[Proposition~5.1]{Lop26a} \textup{(}see also \cite[Theorem~5.1]{Pei26}\textup{)}.
\end{restatable}}

\newcommand{\StateThmAtom}{%
\begin{restatable}[An atom at zero]{theorem}{thmatom}
\label{thm:atom}
Let $\xi(t)=\sum_{p\ge2}\beta_p^2t^p$ be not identically zero, with radius of
convergence greater than one, and put $m_0=\zetas(\{0\})$.
\begin{thmlist}
\item \emph{Slope at equilibrium.} For every $\delta>0$,
 \[
  \Sigma(\feq+\delta)=-m_0\delta-S(\delta),\qquad
  S(\delta)=\sup_{0\le s\le1-m_0}\bigl\{s\delta-\Lamhat(m_0+s)\bigr\},
 \]
 where $\Lamhat(\theta)=\Lambar(\theta)-\theta\feq$, and $0\le S(\delta)=o(\delta)$ as
 $\delta\downarrow0$. In particular $\Sigma'_+(\feq)=-\zetas(\{0\})$, and
 $|\Sigma(\feq+\delta)|=o(\delta)$ if and only if $\zetas(\{0\})=0$.
\item \emph{The atom.} If zero is an isolated point of $\supp\zetas$, then $m_0>0$;
 this is the case whenever $\xi''(0)<1$. If zero is an accumulation point of
 $\supp\zetas$, then $\xi''(0)\Ustar^2=1$, so $\xi''(0)\ge1$, and
 \[
  \zetas(\{0\})=\frac{\xi'''(0)}{2\xi''(0)^{3/2}}=\frac{3\beta_3^2}{(2\beta_2^2)^{3/2}} .
 \]
\item \emph{Near-critical mixtures with a cubic term.} There is $\varepsilon_3\in(0,\frac14]$
 with the following property. Let $\beta_3>0$, $\xi''(0)>1$ and
 $\normCfour{\xi-\xic}<\varepsilon_3$, where $\xic(t)=t^2/2$. Then
 $\supp\zetas=[0,v_*]$ for some $v_*>0$, $m_0=3\beta_3^2/(2\beta_2^2)^{3/2}\in(0,1)$, and the
 constant
 \[
  \Dstar:=\frac{\xi''''(0)}{\azero}-\frac{3\xi'''(0)^2}{2\azero^2}+\azero^3\Wstar^2 ,
 \]
 which in this part replaces the constant $\Dstar$ of Section~\ref{sec:intro-constrained} and
 equals \eqref{eq:atom-odd-D}, satisfies $\Dstar\ge2$.
 With $c=9\azero^3/(80\Dstar^3)$, as $\delta\downarrow0$,
 \[
  \Sigma(\feq+\delta)=-m_0\delta-\frac45\,(5m_0c)^{-1/4}\,\delta^{5/4}\,(1+o(1)),
  \qquad
  \frac45\,(5m_0c)^{-1/4}=\frac{8}{5\sqrt3}\Bigl(\frac{\Dstar}{\azero}\Bigr)^{3/4}m_0^{-1/4}.
 \]
 At the level $f=\feq+\delta$, for all sufficiently small $\delta>0$, the first positive
 contact $q_f$ of $\zeta_f$ is marginal, and
 \begin{gather*}
  \theta_f-m_0=\Bigl(\frac{\delta}{5m_0c}\Bigr)^{1/4}(1+o(1)),\\
  q_f=\frac{3\sqrt{\azero}}{\Dstar}\,(\theta_f-m_0)(1+o(1)),\qquad
  \zeta_f(\{q_f\})=\frac{\theta_f-m_0}{2}\,(1+o(1)).
 \end{gather*}
\item \emph{Free energy.} If $\zetas\ne\delta_0$ and \condFM\ holds, then
 $\lim_{N}N^{-1}\log\Prob(N^{-1}\log\ZN\ge f)=\Sigma(f)$ for every $f\in(\feq,f_1)$.
\end{thmlist}
\end{restatable}}

\newcommand{\StateCorPureP}{%
\begin{restatable}[Pure $p$-spin models]{maincorollary}{corpurep}
\label{cor:pure-p}
Let $\xi(t)=\beta_p^2t^p$ with $p\ge3$ and $\beta_p>0$. Then $\zetas(\{0\})>0$ and, as
$\delta\downarrow0$,
\[
 \Sigma(\feq+\delta)=-\zetas(\{0\})\,\delta+o(\delta);
 \qquad\text{in particular}\qquad \Sigma'_+(\feq)=-\zetas(\{0\})<0 .
\]
If $\zetas=\delta_0$, then $\Sigma(\feq+\delta)=-\delta$ for every $\delta>0$. If
$\zetas\ne\delta_0$ and \condFM\ holds, then for $0<\delta<f_1-\feq$
\[
 \lim_{N\to\infty}\frac1N\log\Prob\bigl(N^{-1}\log\ZN\ge\feq+\delta\bigr)
 =\Sigma(\feq+\delta).
\]
\end{restatable}}

\newcommand{\StateThmNearCrit}{%
\begin{restatable}[Near criticality]{theorem}{thmnearcrit}
\label{thm:nearcrit}
Put $\xic(t)=t^2/2$. For $\varepsilon>0$ let $\Kcl(\varepsilon)$ be the set of
covariances $\xi(t)=\sum_{p\ge2}\beta_p^2t^p$ with radius of convergence greater than
one, $\beta_3=0$, $\xi''(0)>1$ and $\normCfour{\xi-\xic}<\varepsilon$.
\begin{thmlist}
\item \emph{The near-critical class.} There is $\varepsilon_*\in(0,\frac14]$ such that
 every $\xi\in\Kcl(\varepsilon_*)$ has $\supp\zetas=[0,v_*]$ for some $v_*>0$, so that
 \condH\ holds, and $\Dstar\ge2$. Theorem~\ref{thm:local} therefore applies, and its
 constant satisfies $c=9\azero^3/(80\Dstar^3)\to9/5120$ as $\normCfour{\xi-\xic}\to0$
 within $\Kcl(\varepsilon_*)$. In particular $\cb\to9/5120$ as $\beta\downarrow1$.
\item \emph{A de Almeida--Thouless threshold.} After decreasing $\varepsilon_*$, let
 $\xi\in\Kcl(\varepsilon_*)$. There is a unique $\theta_c\in(0,1)$ with the following
 properties. For $0<\theta<\theta_c$, $\supp\zeta_\theta=\{0\}\cup[q_\theta,v_\theta]$ with
 $0<q_\theta<v_\theta$, and the first positive contact is marginal. For
 $\theta_c\le\theta<1$, the measure $\zeta_\theta$ has exactly two atoms, at zero and at its
 only positive contact, and this contact is strictly Plefka stable when $\theta>\theta_c$.
 Moreover
 \[
  \theta_c=\tfrac23\,(\azero-1)\bigl(1+\tfrac14\xi''''(0)\bigr)+O\bigl((\azero-1)^2\bigr),
 \]
 with an absolute implied constant; for $\xi(t)=\beta^2t^2/2$ with $1<\beta^2<1+\varepsilon_*$
 this reads $\theta_c=\frac23(\beta^2-1)+O((\beta^2-1)^2)$. With $f_c=\Lambar'_+(\theta_c)$,
 one has $\feq<f_c<f_1$, and a level $f\in(\feq,f_1)$ has a marginal first positive contact if
 and only if $f\le f_c$.
\item \emph{Stability above the first contact.} Let $\xi\in\Kcl(\varepsilon_*)$, with
 $\varepsilon_*$ as in \textup{(b)}; let $Q_\theta$ and $Q_*$ be the quantile functions of
 $\zeta_\theta$ and $\zetas$, where $Q(x)=\inf\{t:\zeta([0,t])\ge x\}$, and put
 $x_\theta=\zeta_\theta([0,q_\theta])$. As $\theta\downarrow0$,
 \[
  \sup_{x_\theta<x<1}|Q_\theta(x)-Q_*(x)|=O(\theta^5),\qquad
  \Wone(\zeta_\theta,\zetas)=\frac{5\sqrt{\azero}}{4\Dstar}\,\theta^2\,(1+O(\theta)),
 \]
 and $q_\theta=\frac{3\sqrt{\azero}}{\Dstar}\,\theta\,(1+O(\theta))$,
 $\zeta_\theta(\{q_\theta\})=\frac\theta2\,(1+O(\theta))$. The implied constants depend
 on $\xi$.
\end{thmlist}
\end{restatable}}

\newcommand{\StateThmSKGlobal}{%
\begin{restatable}[The Sherrington--Kirkpatrick model: constrained minimizers at every
$\beta>1$]{theorem}{thmskglobal}
\label{thm:sk-global}
Let $\xi(t)=\beta^2t^2/2$. The proofs of \textup{(a)}--\textup{(e)} use a finite-step
inequality from the unrefereed preprint \cite{OAI26c}
\textup{(}Proposition~\ref{prop:skg-oai}\textup{)}.
\begin{thmlist}
\item \emph{Curvature on plateaus.} Let $\beta>0$ and $\zeta\in\Pspace([0,1])$, and let
 $(u,v)\subset(0,1)$ be an open interval with $\zeta((u,v))=0$, of height $m=\zeta([0,u])$. If
 $m>0$, then
 \[
  \Gamma_\zeta'''(t)\ge\frac1{1000}\,\beta^6m^2\,\E\Phi_{\zeta,xx}(t,X_t)^4>0
  \qquad(u<t<v),
 \]
 so $\Gamma_\zeta'$ is strictly convex on $(u,v)$. If $m=0$, then
 $\Gamma_\zeta''(t)=\beta^4\E\Phi_{\zeta,xxx}(t,X_t)^2>0$ for $u<t<v$.
\item \emph{Gap exclusion.} Let $\beta>0$ and $\zeta\in\Pspace([0,1])$. If $0<u<v<1$ are
 local minimum points of $\Obs_\zeta$, then $\zeta((u,v))>0$. If $\Obs_\zeta(t)\ge\Obs_\zeta(0)$
 for all small $t>0$ and $v\in(0,1)$ is a local minimum point of $\Obs_\zeta$, then
 $\zeta((0,v))>0$.
\item \emph{Point or interval.} Let $\beta>1$ and $0<\theta<1$. Then
 $\zeta_\theta(\{0\})=\theta$, and the positive contacts of $\zeta_\theta$ form the set
 $\supp\zeta_\theta\cap(0,1]=[q_\theta,v_\theta]$, for some $0<q_\theta\le v_\theta<1$. Either
 $q_\theta<v_\theta$, and then $\Gamma_{\zeta_\theta}(t)=t$ and $\Gamma_{\zeta_\theta}'(t)=1$
 for $q_\theta\le t\le v_\theta$, so that every positive contact is marginal; or
 $q_\theta=v_\theta$, and then $\zeta_\theta=\theta\delta_0+(1-\theta)\delta_{q_\theta}$.
\item \emph{The de Almeida--Thouless threshold.} Let $\beta>1$. For each $\theta\in[0,1]$
 the equation $q=\Tfrak(\beta^2q,\theta)$ has exactly one solution $\qrs_\theta\in(0,1)$, and
 $\mathfrak A(\theta)=\beta^2\Sfrak(\beta^2\qrs_\theta,\theta)$ is strictly decreasing on
 $[0,1]$, with $\mathfrak A(0)>1>\mathfrak A(1)$; let $\theta_c(\beta)\in(0,1)$ be the zero of
 $\mathfrak A-1$. For $0<\theta<\theta_c(\beta)$ the first alternative of \textup{(c)}
 holds. For $\theta_c(\beta)\le\theta<1$,
 $\zeta_\theta=\theta\delta_0+(1-\theta)\delta_{\qrs_\theta}$, and at its only positive
 contact $q=\qrs_\theta$
 \[
  1-\xi''(q)\,\E\Phi_{\zeta_\theta,xx}(q,X_q)^2=1-\mathfrak A(\theta),
 \]
 which vanishes at $\theta=\theta_c(\beta)$ and is positive for $\theta>\theta_c(\beta)$. The
 function $\beta\mapsto\theta_c(\beta)$ is continuously differentiable on $(1,\infty)$, and for
 $1<\beta^2<1+\varepsilon_*$ it is the threshold $\theta_c$ of
 Theorem~\ref{thm:nearcrit}\textup{(b)}.
\item \emph{The marginal levels.} Let $\beta>1$ and $f_c(\beta)=\Lambda'(\theta_c(\beta))$.
 Then $\feq<f_c(\beta)<f_1$, and for $f\in(\feq,f_1)$ the following are equivalent: the first
 positive contact at the level $f$ is marginal; every positive contact of $\zeta_f$ is
 marginal; $f\le f_c(\beta)$. For $f\in(f_c(\beta),f_1)$ the measure $\zeta_f$ has exactly two
 atoms and its first positive contact is strictly Plefka stable. Consequently, within
 $(\feq,f_1)$, the strict Plefka hypothesis of \cite[Theorem~1]{Bou26b} holds exactly on
 $(f_c(\beta),f_1)$.
\end{thmlist}
\end{restatable}}

\title{Upper deviations of the Sherrington--Kirkpatrick free energy:\\
the fifth-order law and the onset of constrained Parisi measures}
\author{Yan Ru Pei\\[0.3em]\small Independent researcher}
\date{November 2026}
\hypersetup{pdftitle={Upper deviations of the Sherrington-Kirkpatrick free energy: the
  fifth-order law and the onset of constrained Parisi measures},pdfauthor={Yan Ru Pei}}

\begin{document}

\maketitle

\begin{abstract}
Let $Z_N$ be the partition function of the Sherrington--Kirkpatrick model at inverse
temperature $\beta>1$. Replica theory, through the work of Kondor near the critical
temperature and of Parisi and Rizzo at all temperatures, predicts that the probability that
$N^{-1}\log Z_N$ exceeds its typical value by $\delta$ decays like
$\exp(-Na\,\delta^{6/5})$, with an explicit constant $a$ determined by the slope at zero of
the Parisi order parameter; Aronow and Lopatto recently proved bounds of this order. The
exponent is the Legendre transform of the Parisi variational problem constrained to carry
an atom $\theta$ at zero. We show that the constrained minimum exceeds the equilibrium free
energy by $c_\beta\theta^5(1+o(1))$, where $c_\beta=\frac{9}{640}\beta^3\rho_\beta(0)^{-3}$
and $\rho_\beta(0)>0$ is the density at zero of the Parisi measure. With the
fractional-moment formula, which holds for this model, this gives the exponent
$\frac56(6c_\beta)^{-1/5}\delta^{6/5}(1+o(1))$ in the limit $N\to\infty$ followed by
$\delta\downarrow0$. The constant comes from the shape of the constrained minimizer: its
first positive contact is marginal and, to leading order, sits at $3\theta/(2\rho_\beta(0))$
and carries an atom $\theta/2$, and below that contact a function that equals the identity at the contact
points differs from the identity, to leading order, by a universal cubic. This plateau
calculus applies to every mixture with $\xi'''(0)=0$ whose Parisi measure accumulates at
zero, and it governs two other perturbations of the bottom of the measure: a small external
field $h$ opens a gap of order $|h|^{2/3}$ and adds a term of order $|h|^{10/3}$ to the
free energy, and an atom of the Parisi measure at zero, which every pure $p$-spin model
with $p\ge3$ has, makes the rate function leave zero linearly, with slope equal to that
atom. Near criticality, including the Sherrington--Kirkpatrick model for $\beta$ close to
one, the constant tends to Kondor's value $9/5120$, the marginal regime ends at a de
Almeida--Thouless threshold in the constraint, and above its first contact the constrained
measure agrees with the Parisi measure up to $O(\theta^5)$. For the Sherrington--Kirkpatrick
model at every $\beta>1$, a finite-step curvature inequality from an unrefereed preprint shows
that every constrained minimizer either has an interval of positive contacts, all marginal, or
has exactly two atoms, and that the marginal regime ends at a unique de Almeida--Thouless
threshold $\theta_c(\beta)$. The fifth-order and six-fifths laws at every $\beta>1$ use that the
Parisi measure of this model accumulates at zero, which is a theorem of Lopatto, also proved by
a computer-assisted argument in earlier work of the author; the curvature inequality gives a
third proof.
\end{abstract}

\setcounter{tocdepth}{1}
\tableofcontents

\section{Introduction}\label{sec:intro}

\subsection{Upper deviations of the free energy}\label{sec:intro-deviations}

Let $H_N$ be the centered Gaussian field on $\{-1,1\}^N$ with covariance
$\E H_N(\sigma)H_N(\sigma')=\beta^2(\sigma\cdot\sigma')^2/(2N)$, and put
$\ZN=2^{-N}\sum_{\sigma}e^{H_N(\sigma)}$. This is the Sherrington--Kirkpatrick model \cite{SK75}
at inverse temperature $\beta$, written with the diagonal terms of its quadratic form; in the
usual form $\FNb=\log\sum_\sigma\exp(\beta N^{-1/2}\sum_{i<j}g_{ij}\sigma_i\sigma_j)$, with
independent standard Gaussian couplings $g_{ij}$, they are absent, and $\log\ZN+N\log2$ has the
law of $\FNb+G$ with $G\sim\mathcal N(0,\beta^2/2)$ independent of $\FNb$ \cite{AL26}. By the
Parisi formula \cite{Par79,Par80}, proved by Guerra \cite{Gue03} and Talagrand \cite{Tal06},
$N^{-1}\log\ZN$ converges almost surely to $\feq$, the minimum of the Parisi functional over
probability measures on $[0,1]$, whose minimizer $\zetas$, the Parisi measure, is unique
\cite{AC15b}. For $\beta>1$, how unlikely is it that $N^{-1}\log\ZN$ exceeds $\feq$ by a fixed
$\delta>0$?

By Markov's inequality, $\Prob(N^{-1}\log\ZN\ge f)\le e^{-N\theta f}\,\E\ZN^\theta$ for
$0<\theta<1$, so upper deviations are governed by fractional moments. Talagrand \cite{Tal07}
related the growth rate of $\E\ZN^\theta$ to the Parisi problem constrained to measures with an atom of mass at least
$\theta$ at zero: with $p(\theta)$ the constrained minimum, the fractional-moment formula
\condFM\ reads $\lim_{N}N^{-1}\log\E\ZN^\theta=\Lambda(\theta):=\theta\,p(\theta)$ for
$0<\theta<1$, and for the Sherrington--Kirkpatrick model it is the one-level case of a theorem of
Contucci and Mingione \cite[Theorem~2.1]{CM19}. With an exponential change of measure, \condFM\
identifies the upper-tail rate of $N^{-1}\log\ZN$ as $-\Sigma$, where
$\Sigma(f)=\inf_\theta\{\Lambda(\theta)-\theta f\}$ (that Talagrand's results identify this rate
is noted in \cite[Remark~1.3]{Bou26b}; Appendix~\ref{sec:cons-moments} gives our argument). Near
$\feq$ the rate is decided by how fast $p(\theta)$ decreases to $\feq$: if $p(\theta)-\feq$ is of
order $\theta^5$ as $\theta\downarrow0$, the rate at $\feq+\delta$ is of order $\delta^{6/5}$.

Replica theory predicts a fifth-order increase. Kondor \cite{Kon83} found a term $9n^5/5120$ in
the replica free energy of a truncated model near the critical temperature, as reported by
Parisi and Rizzo \cite{PR08} and \cite[Section~IV]{PR09}, who extended the computation to every temperature below
the critical one. With their replica number $n$ in the role of $\theta$, they predict
$p(\theta)-\feq\approx\cb\theta^5$ with $\cb=\frac{9}{640}\beta^3\rhob^{-3}$, where $\rhob$ is
the density at zero of the Parisi measure, the reciprocal of the slope $\qdot$ at $x=0$ of
Parisi's order parameter $q(x)$, the inverse function of $t\mapsto\zetas([0,t])$; a Legendre
transform turns this into their six-fifths law \cite[(4)]{PR09}, by which $N^{-1}\log\ZN$ exceeds
$\feq$ by $\delta$ with probability $\exp\{-N\frac56(6\cb)^{-1/5}\delta^{6/5}\}$ to leading
order. Appendix~\ref{app:dictionary} translates their notation; locators in \cite{PR09} refer to
its arXiv version~2. Locators in \cite{AC15a}, \cite{JT17} and \cite{Che17} refer to the arXiv
versions listed in the bibliography.

For the Sherrington--Kirkpatrick model at every $\beta>1$, Aronow and Lopatto recently proved that $p(\theta)-\feq$ is bounded above and below by constant multiples of
$\theta^5$ \cite[Lemma~7.2, Proposition~7.3]{AL26}, and that at finite $N$ the probability that
$\FNb$ exceeds $N(\feq+\log2+x)$ lies between $e^{-CNx^{6/5}}$ and $e^{-C'Nx^{6/5}}$ for
$AN^{-2/3}\le x\le x_0$, with constants depending on $\beta$ \cite[Theorem~1.5]{AL26}; after
division by $\beta$ these bounds hold uniformly in $\beta\ge2$, and they extend to the
ground-state energy \cite[Theorem~1.6]{AL26}. Their constants are not specified.

This paper identifies the constant: for the Sherrington--Kirkpatrick model at every $\beta>1$
(Theorem~\ref{thm:sk}),
\[
 \lim_{\delta\downarrow0}\delta^{-6/5}\lim_{N\to\infty}
 \Bigl(-\frac1N\log\Prob\bigl(N^{-1}\log\ZN\ge\feq+\delta\bigr)\Bigr)
 =\frac56\,(6\cb)^{-1/5},
\]
where $\rhob=\beta^5\Wstar^2/2>0$ and $\Wstar$ is the fourth spatial derivative at the origin of
the solution of the Parisi equation for $\zetas$ (Section~\ref{sec:intro-constrained}); the limits
are taken in this order. The same holds for $\FNb$, with $N^{-1}\FNb-\log2$ in place of
$N^{-1}\log\ZN$, and the constant is the leading term of the prediction of Parisi and Rizzo. The
statement at every $\beta>1$ uses that zero is an accumulation point of the support of the
Parisi measure, which is Lopatto's theorem \cite[Proposition~5.1]{Lop26a} and has a
computer-assisted alternative proof \cite[Theorem~5.1]{Pei26}; Theorem~\ref{thm:sk-global}(b)
below gives a third proof, which depends on an unrefereed input
(Corollary~\ref{cor:skg-interval}).

The constant is a statement about the shape of the constrained minimizer $\zeta_\theta$ near the
bottom of its support, at scale $\theta$: for small $\theta$ its first positive contact is
marginal, sits at
$\frac{3\theta}{2\rhob}(1+o(1))$ and carries an atom of mass $\frac\theta2(1+o(1))$, and below
this contact a function $\Gamma_\theta$ that equals the identity at the contacts differs from the
identity by a universal cubic, to leading order (Section~\ref{sec:intro-ideas}). This mechanism is
local at the origin: it gives the constant for every mixed $p$-spin model without a cubic term
whose Parisi measure accumulates at zero (Theorem~\ref{thm:local}), and it governs two other
perturbations of the bottom of the measure, a small external field (Theorem~\ref{thm:field}) and
an atom at zero (Theorem~\ref{thm:atom}). Kondor's coefficient is the limit of $\cb$ as
$\beta\downarrow1$, and near criticality, including the Sherrington--Kirkpatrick model for $\beta$
close to one, the marginal regime ends at a de Almeida--Thouless threshold in the constraint
(Theorem~\ref{thm:nearcrit}). For the Sherrington--Kirkpatrick model at every $\beta>1$, a
finite-step curvature inequality from an unrefereed preprint \cite{OAI26c} shows that every
constrained minimizer either has an interval of positive contacts, all marginal, or has exactly
two atoms, and that the marginal regime ends at a unique de Almeida--Thouless threshold
$\theta_c(\beta)$ (Theorem~\ref{thm:sk-global}).

\subsection{Constrained Parisi measures}\label{sec:intro-constrained}

We state the results for mixed $p$-spin models, since the mechanism is local at the origin. Let
$\xi(t)=\sum_{p\ge2}\beta_p^2t^p$ be not identically zero, with radius of convergence greater than one, let $H_N$ be the
centered Gaussian field on $\{-1,1\}^N$ with covariance \eqref{eq:intro-covariance}, and, for a
probability measure $\zeta$ on $[0,1]$, let $\Phi_\zeta$ solve the Parisi equation
\eqref{eq:intro-pde} and let $P_\zeta$ be the Parisi functional \eqref{eq:intro-functional}:
\begin{gather}
 \E H_N(\sigma)H_N(\sigma')=N\xi\bigl(N^{-1}\sigma\cdot\sigma'\bigr),
 \qquad
 \ZN=2^{-N}\sum_{\sigma\in\{-1,1\}^N}e^{H_N(\sigma)} ,
 \label{eq:intro-covariance}\\
 \partial_t\Phi_\zeta+\frac{\xi''(t)}2\bigl\{\Phi_{\zeta,xx}+\zeta([0,t])\Phi_{\zeta,x}^2\bigr\}=0,
 \qquad \Phi_\zeta(1,x)=\log\cosh x,
 \label{eq:intro-pde}\\
 P_\zeta=\Phi_\zeta(0,0)-\frac12\int_0^1t\xi''(t)\zeta([0,t])\,dt,
 \qquad \feq=\inf_\zeta P_\zeta .
 \label{eq:intro-functional}
\end{gather}
The model of Section~\ref{sec:intro-deviations} is $\xi(t)=\beta^2t^2/2$, that is,
$\beta_2^2=\beta^2/2$ and $\beta_p=0$ for $p\ge3$; an unsubscripted $\beta$ always denotes its
inverse temperature. The Parisi equation is solved by a Cole--Hopf transform between consecutive
atoms when $\zeta$ has finitely many atoms, and its solution extends to every $\zeta$ by
continuity (Section~\ref{sec:cons-standard}); each $\Phi_\zeta(t,\cdot)$ is even. The functional
$\zeta\mapsto P_\zeta$ is strictly convex and has a unique minimizer $\zetas$, the Parisi measure
\cite[Theorem~2 and Corollary~1]{AC15b}; whenever the Parisi formula applies, for instance for every polynomial
$\xi$ \cite{Tal06,Pan14}, $\feq$ is the almost sure limit of $N^{-1}\log\ZN$. We write
$\alphas(t)=\zetas([0,t])$, $\azero=\xi''(0)=2\beta_2^2$, $\Ustar=\Phi_{\zetas,xx}(0,0)$,
$\Wstar=\Phi_{\zetas,xxxx}(0,0)$ and $\Dstar=\xi''''(0)\Ustar^2+\azero^3\Wstar^2$; when $\xi'''(0)>0$, Theorem~\ref{thm:atom}(c) and Section~\ref{sec:atom-odd} use a modified constant, also written $\Dstar$.

\paragraph{The constrained problem.}
For $0\le\theta<1$ the minimizer $\zeta_\theta$ of $P_\zeta$ under the constraint
$\zeta(\{0\})\ge\theta$ is unique, by strict convexity; we put $p(\theta)=P_{\zeta_\theta}$, so
that $\zeta_0=\zetas$ and $p(0)=\feq$, and
$\Lambda(\theta)=\theta\inf_{\zeta(\{0\})\ge\theta}P_\zeta=\theta p(\theta)$. With $\Lambar$ the
extension of $\Lambda$ to $[0,1]$ by its lower limit at $1$, which is convex
\cite[Lemma~3.2]{Bou26b}, let
\begin{equation}\label{eq:intro-dual}
 \Sigma(f)=\inf_{\theta\in[0,1]}\bigl\{\Lambar(\theta)-\theta f\bigr\},
 \qquad f_1=\Lambar'_-(1).
\end{equation}
When $\zetas\ne\delta_0$ we have $\feq<f_1$, and for every level $f\in(\feq,f_1)$ the infimum in
\eqref{eq:intro-dual} is attained at a unique $\theta_f\in(0,1)$
(Lemmas~\ref{lem:cons-variational}(d) and~\ref{lem:cons-regular}(d)); we write $\zeta_f=\zeta_{\theta_f}$.
For $f\ge f_1$ the infimum in \eqref{eq:intro-dual} is attained at $\theta=1$, outside the range of \condFM, and we do not treat these levels.

\paragraph{Contacts.}
With the diffusion of Auffinger and Chen \cite[Theorem~3]{AC15b} associated with $\zeta$, put
\begin{gather}
 dX_t=\xi''(t)\zeta([0,t])\Phi_{\zeta,x}(t,X_t)\,dt+\sqrt{\xi''(t)}\,dB_t,\qquad X_0=0,
 \label{eq:intro-diffusion}\\
 \Gamma_\zeta(t)=\E\Phi_{\zeta,x}(t,X_t)^2,\qquad
 \Obs_\zeta(t)=\frac12\int_t^1\xi''(s)\bigl\{\Gamma_\zeta(s)-s\bigr\}\,ds .
 \label{eq:intro-obstacle}
\end{gather}
We call $\Obs_\zeta$ the \emph{obstacle}. By evenness $\Gamma_\zeta(0)=0$, and $\Gamma_\zeta$ is continuously differentiable with
$\Gamma_\zeta'(t)=\xi''(t)\,\E\Phi_{\zeta,xx}(t,X_t)^2$ \cite[Proposition~3]{AC15a}, read along $X$ by Girsanov's theorem (Section~\ref{sec:plat}). A
\emph{positive contact} of $\zeta$ is a minimizer of $\Obs_\zeta$ over $(0,1]$, and the smallest
one, when it exists, is the \emph{first positive contact}. A positive contact $q$ lies in $(0,1)$ (Lemma~\ref{lem:cons-standard}(f)),
and the first- and second-order conditions for a minimum give $\Gamma_\zeta(q)=q$ and
$\Gamma_\zeta'(q)\le1$, that is, the Plefka inequality
\begin{equation}\label{eq:intro-plefka}
 1-\xi''(q)\,\E\Phi_{\zeta,xx}(q,X_q)^2\ge0 .
\end{equation}
The contact is \emph{marginal} if equality holds and \emph{strictly Plefka stable} otherwise; for
$\zeta=\zeta_f$ and its first positive contact, strict stability is the strict Plefka condition
of Boursier \cite[(1.8)]{Bou26b}. If $\theta>\zetas(\{0\})$, then $\zeta_\theta(\{0\})=\theta$,
the positive contacts of $\zeta_\theta$ form a nonempty compact subset of $(0,1)$ that carries
the restriction of $\zeta_\theta$ to $(0,1]$, and the first positive contact $q_\theta$ exists
(Section~\ref{sec:cons}). A \emph{plateau} of a measure is an open interval on which its
distribution function is constant (in the quantile function $q(x)$ of Parisi and Rizzo, a plateau
of $q(x)$ is an atom of the measure; we use the word only in our sense); the \emph{first plateau}
of $\zeta_\theta$ is $(0,q_\theta)$, where its distribution function equals $\theta$. We write
$q_f$ for the first positive contact of $\zeta_f$, and say that the first positive contact at the
level $f$ is marginal if $q_f$ is. Finally, $\Wone(\mu,\nu)=\int_0^1|\mu([0,t])-\nu([0,t])|\,dt$
is the transport distance between probability measures on $[0,1]$, which metrizes weak
convergence, and $\normCfour{\eta}=\max_{0\le k\le4}\sup_{0\le t\le1}|\eta^{(k)}(t)|$.

\paragraph{Named conditions.}
Zero always lies in $\supp\zetas$ (Lemma~\ref{lem:ons-zero-in-support}). The local results
assume that it is not an isolated point of the support:
\DisplayCondH
When $\xi'''(0)=0$ and $\beta_2>0$, this condition implies the stronger condition
\DisplayCondHp
(Theorem~\ref{thm:local}), and if $\xi'''(0)=0$ it is equivalent to $\zetas(\{0\})=0$
(Remark~\ref{rem:ons-H-no-atom}). Condition \condH\ fails whenever $\xi''(0)<1$ (Theorem~\ref{thm:atom}(b)), for instance for
every pure $p$-spin model with $p\ge3$. It holds for the Sherrington--Kirkpatrick model at every
$\beta>1$ \cite[Proposition~5.1]{Lop26a} (see also \cite[Theorem~5.1]{Pei26}; both are
preprints, and Corollary~\ref{cor:skg-interval} gives a third proof that depends on an
unrefereed input) and for the near-critical covariances of Theorems~\ref{thm:nearcrit}(a) and~\ref{thm:atom}(c); beyond these
cases we do not know a proof of \condH\ for genuinely mixed covariances. The large-deviation statements use the
fractional-moment formula
\DisplayCondFM
which holds for $\xi(t)=\beta^2t^2/2$ by \cite[Theorem~2.1]{CM19} (Lemma~\ref{lem:A-fm-sk});
Remark~\ref{rem:mo-fm} lists its sources for other covariances.

\subsection{The fifth-order law and the six-fifths constant}\label{sec:intro-results}

We state the result for the Sherrington--Kirkpatrick model first; for it $\azero=\beta^2$,
$\xi''''(0)=0$ and $\Dstar=\beta^6\Wstar^2$.

\StateThmSK

Parts (a) and (b) concern the Parisi variational problem only and use \condH; part (c) adds
\condFM\ and an exponential change of measure. All limits are taken at fixed $\beta$, with
$N\to\infty$ before $\delta\downarrow0$: we make no statement about deviations that shrink with
$N$, the $o(1)$ terms carry no rate, and nothing is uniform in $\beta$. Kondor's coefficient
appears only as the limit of $\cb$ as $\beta\downarrow1$ (Theorem~\ref{thm:nearcrit}(a)): at
$\beta=1$ itself $\zetas=\delta_0$, and the theorem does not apply. At every $\beta>1$ the theorem
takes \condH\ from one of two preprints: Lopatto's theorem \cite[Proposition~5.1]{Lop26a} or its
computer-assisted alternative proof \cite[Theorem~5.1]{Pei26} (a third route, through the
unrefereed input of Theorem~\ref{thm:sk-global}, is Corollary~\ref{cor:skg-interval}); for
$1<\beta^2<1+\varepsilon_*$, with
$\varepsilon_*$ as in Theorem~\ref{thm:nearcrit}(a), both inputs, \condH\ and \condFM, rest on refereed results and this
paper (Remark~\ref{rem:fifth-sk-inputs}).

In the dictionary of Appendix~\ref{app:dictionary}, Theorem~\ref{thm:sk} proves the leading terms
of three predictions of Parisi and Rizzo: the fifth-order law \cite[(B3)]{PR09}, the six-fifths
law \cite[(4)]{PR09}, and the lowest nonzero value $q_0=\frac32n\qdot$ of their order
parameter \cite[Section~IV and (B2)]{PR09}, which is our first positive contact $q_\theta$. With
Theorem~\ref{thm:local} it also gives the exact identities \cite[(B9)--(B10)]{PR09}, which Parisi
and Rizzo attribute to Crisanti and Rizzo \cite{CR02}; in our notation they read $\Ustar=1/\beta$
and $\rhob=\beta^5\Wstar^2/2$. For this model both were known: $\Ustar=1/\beta$ is
\cite[Lemma~6.1]{Lop26a}, and the value and positivity of $\rhob$ follow from
\cite[Theorem~B.1 and (B.4)]{Lop26b} and from \cite[Section~7]{AL26}, and are also
\cite[Lemma~5.8]{OAI26a}, an unrefereed preprint (Remark~\ref{rem:fifth-sk-inputs}); Theorem~\ref{thm:local} derives them from \condH\ and
extends the second to every covariance in its scope. The bounds of Aronow and Lopatto
\cite[Theorems~1.5 and~1.6]{AL26} are complementary:
they hold at finite $N$, uniformly in $\beta\ge2$ after rescaling, and at zero temperature (which
we do not treat), and letting first $N\to\infty$ and then $x\downarrow0$ in
\cite[Theorem~1.5]{AL26} shows that its two constants bracket $\frac56(6\cb)^{-1/5}$ (for
$\beta\ge2$, those of \cite[Theorem~1.6]{AL26} bracket $\beta^{6/5}\frac56(6\cb)^{-1/5}$).
Finally, part~(b) bears on TAP states: for strictly convex $\xi$, Boursier \cite[Theorem~1]{Bou26b}
showed that the probability that a TAP state, in the sense of the generalized TAP free energy
\cite{CPS23}, exists at a level $f\in(\feq,f_1)$ has exponential rate $-\Sigma(f)$ under the
strict Plefka condition \cite[(1.8)]{Bou26b} at $f$ (every such level is regular in the sense of
\cite[Definition~1.1]{Bou26b}, by Lemma~\ref{lem:cons-variational}(d)); by part~(b) that condition fails at every
level in $(\feq,\feq+\bar\delta)$, so the levels that carry the six-fifths law lie outside its
scope (Corollary~\ref{cor:fifth-plefka}).

Theorem~\ref{thm:sk} is the case $\xi(t)=\beta^2t^2/2$ of the following local theorem, which uses
only the behaviour of $\xi$ and $\zetas$ at the origin; part~(c) of Theorem~\ref{thm:sk} also
uses \condFM\ for this covariance and the transfer to $\FNb$ (Appendix~\ref{app:guerra}).

\StateThmLocal

The nondegeneracy $\Dstar>0$ is a conclusion. It comes from a universal strict sign,
$\Phi_{\zeta,xxxx}(t,0)\le-\kappa_t<0$ for every probability measure $\zeta$ and every $t<1$,
with an explicit $\kappa_t$ (Lemma~\ref{lem:plat-fourth-sign}), which gives $\Wstar\le-\kappa_0$
and $\Dstar\ge\azero^3\kappa_0^2$. (The non-strict sign, and for the Sherrington--Kirkpatrick
minimizer a strict bound at the origin, were known; Remark~\ref{rem:plat-sign}.) The constant
$c$ depends on $\zetas$ only through its density $\rhostar$ at zero, as in the formula of Parisi
and Rizzo. The hypothesis $\beta_2>0$ is implied by \condH, by Theorem~\ref{thm:atom}(b). By part (a), the constraint places
mass $\theta$ at zero, empties $(0,q_\theta)$ and puts an atom of mass close to $\theta/2$ at
$q_\theta$, where the distribution function of $\zeta_\theta$ thus rejoins that of $\zetas$ to
leading order, because $\theta+\frac\theta2=\rhostar q_\theta(1+o(1))$
(Figure~\ref{fig:bottom}(a), in Section~\ref{sec:ons-onset}). Besides the Sherrington--Kirkpatrick model, the theorem applies to
the near-critical class of Theorem~\ref{thm:nearcrit}(a), which contains genuinely mixed
covariances and covariances with odd degrees $p\ge5$. Part (d) is stated under \condFM; it holds for
$\xi(t)=\beta^2t^2/2$ by \cite[Theorem~2.1]{CM19} (Lemma~\ref{lem:A-fm-sk}), for other even
$\xi$ its sources are \cite[Section~7]{Tal07}, which explains how to modify the proof of the
Parisi formula, and the preprint \cite[Theorem~2.1]{CGKLM}, and for
$\xi$ with odd degrees the only source we know is \cite[Theorem~2.1]{CGKLM}.

\subsection{Two other perturbations of the bottom of the measure}\label{sec:intro-perturbations}

The constraint perturbs the bottom of the Parisi measure by forcing an atom at zero. Two other
perturbations act on the same region and obey the same calculus. The first is a small external
field $h$, which removes the mass of $\zetas$ near zero.

\StateThmField

As in Theorem~\ref{thm:local}, \condH\ is the only condition on $\zetas$, and $\Dstar>0$ is
automatic. The field opens a
gap of size of order $|h|^{2/3}$ at the bottom of the support and gathers the zero-field mass
below $\qmin(h)$ into an atom at $\qmin(h)$ (Figure~\ref{fig:bottom}(b)). For the
Sherrington--Kirkpatrick model, with $h=\beta H$ and $T=1/\beta$, parts (b) and (d) give the
leading terms of the small-field laws \cite[(B4)--(B5)]{PR09} of Parisi and Rizzo, who attribute
the order $h^{10/3}$ to \cite{MPV87,CRT03} (Corollary~\ref{cor:dict-field}). The value
$\Ffield''(0)=1/\beta$ is Toulouse's identity \cite{Tou80}; for the Sherrington--Kirkpatrick
model at every $\beta>1$ it is proved, as a second-order limit, in
\cite[Theorem~3.1(b),(d)]{Pei26}, using \cite[Lemma~6.1]{Lop26a}. Theorem~\ref{thm:field} adds
the value $\azero^{-1/2}$ for every covariance in its scope and the $|h|^{10/3}$ term; that
$\Ffield\in C^1$ with $\Ffield'=\mfield$ is the routine envelope argument of
Lemma~\ref{lem:fld-continuity}. For a fixed field $h>0$, Lopatto \cite[Theorem~1.1]{Lop26b} proved
that the Parisi measure of the Sherrington--Kirkpatrick model is supported on an interval
$[q_-,q_+]$, with an atom at $q_-$ when it is not replica symmetric; Theorem~\ref{thm:field} gives
the small-field asymptotics of $q_-=\qmin(h)$ and of that atom without using that result.

The second perturbation is an atom of the Parisi measure at zero. It occurs when the covariance
has a cubic term or when zero is isolated in the support, and it makes the rate leave the
equilibrium level linearly.

\StateThmAtom

\StateCorPureP

Part (a) of Theorem~\ref{thm:atom} holds for every covariance. The identity
$\Sigma'_+(\feq)=-\zetas(\{0\})$ is stated without proof by Boursier \cite[Remark~1.4]{Bou26a},
who relates it to an extreme-value heuristic; we prove it as a right derivative with the exact
remainder $S$ (Section~\ref{sec:atom-slope}). That $\zetas(\{0\})>0$ whenever $\xi''(0)<1$, as in
part (b) and Corollary~\ref{cor:pure-p}, is \cite[Theorem~1]{AC15a}. Part (c), proved in Section~\ref{sec:atom-odd} in the near-critical setting of Section~\ref{sec:intro-nearcrit}, runs the argument
of Theorem~\ref{thm:local} with $\theta-m_0$ in place of $\theta$, and with the constant $\Dstar$
of part (c), which accounts for the cubic term and the atom. The constraint is inactive up to
$m_0$; beyond it $p(\theta)-\feq$ is of fifth order in $\theta-m_0$, and since
$\Lambda(\theta)=\theta p(\theta)$ with $\theta$ close to $m_0$, this gives a correction of order
$\delta^{5/4}$ to the linear term (Figure~\ref{fig:bottom}(c)). For pure $p$-spin models
the linear rate is the rigorous counterpart of the observation of Nakajima and Hukushima
\cite{NH08} that the replica free energy is constant near $n=0$. For even $p$ the bound
$\limsup_NN^{-1}\log\Prob(N^{-1}\log\ZN\ge\feq+\delta)\le-\zetas(\{0\})\,\delta$ already follows
from Guerra's bound \cite[Section~6]{Tal07}, with $\zetas$ as trial measure, and Markov's
inequality. Theorem~\ref{thm:atom}(a) and Corollary~\ref{cor:pure-p} add the exact remainder, the
characterization of a superlinear rate and the matching lower bound under \condFM. For covariances
with odd degrees, in particular in Theorem~\ref{thm:atom}(d) for mixtures with a cubic term and in
Corollary~\ref{cor:pure-p} for odd $p$, the only source of \condFM\ we know is
\cite[Theorem~2.1]{CGKLM}, a preprint whose Guerra-type upper bound is written out for convex $\xi$
and sketched otherwise.

\subsection{Near criticality}\label{sec:intro-nearcrit}

Close to the critical covariance $\xic(t)=t^2/2$, the pair $(\xic,\delta_0)$, whose Parisi
solution is $\log\cosh x+\frac12(1-t)$, serves as a reference. Closeness to it gives global
information that the local calculus does not provide: interval support of $\zetas$, and the
behaviour of $\zeta_\theta$ for constraints that are not small.

\StateThmNearCrit

Part (a) contains the Sherrington--Kirkpatrick model for $1<\beta^2<1+\varepsilon_*$, and its last
assertion is Kondor's value as a limit: as $\xi\to\xic$, one has $\azero\to1$,
$\xi''''(0)\to0$ and $\Wstar\to(\log\cosh)''''(0)=-2$, so that $\Dstar\to4$ and
$c\to9/(80\cdot4^3)=9/5120$. Parts (b) and (c) hold on the whole class, in particular for the
Sherrington--Kirkpatrick model with $1<\beta^2<1+\varepsilon_*$. For that model they are the
constrained-Parisi form, near the critical temperature, of two further predictions of Parisi and
Rizzo: their line of de Almeida--Thouless instability \cite{dAT78} in the replica-number
direction, $n_{\rm dAT}=\frac43(1-T)+O((1-T)^2)$ \cite[(22)--(23)]{PR09}, whose leading term is
that of $\theta_c$ in (b), since $\frac23(\beta^2-1)=\frac43(1-T)+O((1-T)^2)$; and their
stability statement $q(x)=q_{\rm free}(x)+O(n^5)$ for $x_c<x<1$, that is, above the first
positive contact \cite[Section~IV]{PR09}, which is (c) with $Q_\theta$ and $Q_*$ in the roles of
$q$ and $q_{\rm free}$. Parisi and Rizzo state both at every temperature below the critical one;
we prove them near it, and Theorem~\ref{thm:sk-global}(d) below extends the first to every
$\beta>1$, with an unrefereed input. A positive quartic coefficient multiplies the leading term of $\theta_c$
by $1+\frac14\xi''''(0)$. Part (b) also shows that marginality is not automatic: above
$\theta_c$ the first positive contact is strictly Plefka stable, so that within $(\feq,f_1)$ the
strict Plefka condition \cite[(1.8)]{Bou26b} holds exactly on $(f_c,f_1)$
(Corollary~\ref{cor:nc-dat-levels}).

\subsection{The Sherrington--Kirkpatrick model away from criticality}\label{sec:intro-skglobal}

Theorem~\ref{thm:nearcrit}(b) locates the end of the marginal regime only near criticality,
because its proof takes two global facts from closeness to $\xic$: no plateau separates two
contacts below a fixed level, and no contact lies above that level. For the
Sherrington--Kirkpatrick covariance a strict curvature inequality on plateaus supplies both at
every temperature. It comes from a finite-step coercivity inequality in an unrefereed preprint on
the zero-temperature Sherrington--Kirkpatrick order parameter \cite[Proposition~4.2]{OAI26c}.
For fixed step data that inequality concerns the Parisi equation of this model at positive
temperature (Proposition~\ref{prop:skg-oai}), and we transfer it to every order parameter
(Proposition~\ref{prop:skg-plateau}).

To state the result, let $\Tfrak(w,\theta)$ and $\Sfrak(w,\theta)$ be the means of $\tanh^2$ and
$\sech^4$ under the $\cosh^\theta$-tilted $\mathcal N(0,w)$ law, for $w\ge0$ and
$\theta\in[0,1]$:
\begin{equation}\label{eq:intro-TS}
 \Tfrak(w,\theta)=\frac{\E\bigl[\tanh^2(\sqrt wZ)\cosh^\theta(\sqrt wZ)\bigr]}
 {\E\cosh^\theta(\sqrt wZ)},\qquad
 \Sfrak(w,\theta)=\frac{\E\bigl[\sech^4(\sqrt wZ)\cosh^\theta(\sqrt wZ)\bigr]}
 {\E\cosh^\theta(\sqrt wZ)},
\end{equation}
with $Z$ standard Gaussian. For $\xi(t)=\beta^2t^2/2$ and the two-atom measure
$\theta\delta_0+(1-\theta)\delta_q$, the contact conditions at $q$ read $q=\Tfrak(\beta^2q,\theta)$
and $\beta^2\Sfrak(\beta^2q,\theta)\le1$: the replica-symmetric fixed-point equation and the de
Almeida--Thouless condition at replica number $\theta$ (Section~\ref{sec:nc-dat}).

\StateThmSKGlobal

Parts (a) and (b) hold for every probability measure, not only for minimizers. Part (b) is
Lemma~\ref{lem:plat-local-gap} without its two restrictions, closeness to $\zetas$ and contacts
near zero. For $\beta>1$ it shows that $\supp\zetas=[0,q_P]$ with $q_P>0$, so that
$\zetas(\{0\})=0$ and \condH\ holds (Corollary~\ref{cor:skg-interval}); this conclusion is
Lopatto's theorem \cite[Theorem~1.1]{Lop26a}, obtained here by the trapezoid argument with the
curvature input of part (a), and the proof of Theorem~\ref{thm:sk-global} given here uses
neither \cite{Lop26a} nor \cite{Pei26}. Part (a) is the positive-temperature counterpart of the
zero-temperature curvature estimate \cite[Proposition~5.1]{OAI26c}, and it strengthens a result
of Lopatto, who proved $\Gamma_\zeta'''>0$ at the points of a plateau of positive height where
$\Gamma_\zeta''=0$ (\cite[Proposition~5.1]{Lop26a} in version~4). That result already shows that
$\Gamma_\zeta'$ has no interior local maximum on such a plateau, which is all that the proofs of
(b)--(e) use; so (b)--(e) also follow with Lopatto's result in place of the
unrefereed input, and only the quantitative bound in (a) needs \cite{OAI26c}
(Remark~\ref{rem:skg-curvature}).

Parts (c)--(e) settle, for this model at every $\beta>1$, the first question that the local
calculus leaves open away from zero (Section~\ref{sec:disc-open}): for
$0<\theta<\theta_c(\beta)$ every positive contact of $\zeta_\theta$ is marginal, and the
marginal onset of Theorem~\ref{thm:local}(a) holds with $\theta_0=\theta_c(\beta)$. With
$n=\theta$, the equations $q=\Tfrak(\beta^2q,\theta)$ and $\mathfrak A(\theta)=1$ are
\cite[(10) and (22)]{PR09} (Remark~\ref{rem:nc-dat}), so $\theta_c(\beta)$ is the replica
number of the de Almeida--Thouless instability of Parisi and Rizzo, and (d) is the
constrained-Parisi form of their prediction at every temperature below the critical one. In (d),
the uniqueness of $\qrs_\theta$ and the monotonicity of $\mathfrak A$ concern Gaussian integrals
only and do not use \cite{OAI26c}; the decisive inequality is that $\Tfrak(w,\theta)/w$ strictly
decreases in $w$ for every $\theta\in[0,1]$ (Lemma~\ref{lem:skg-branch}). Part (e) extends
Corollary~\ref{cor:nc-dat-levels} to this model at every $\beta>1$. The stability of
Theorem~\ref{thm:nearcrit}(c) above the first contact remains a near-critical statement.

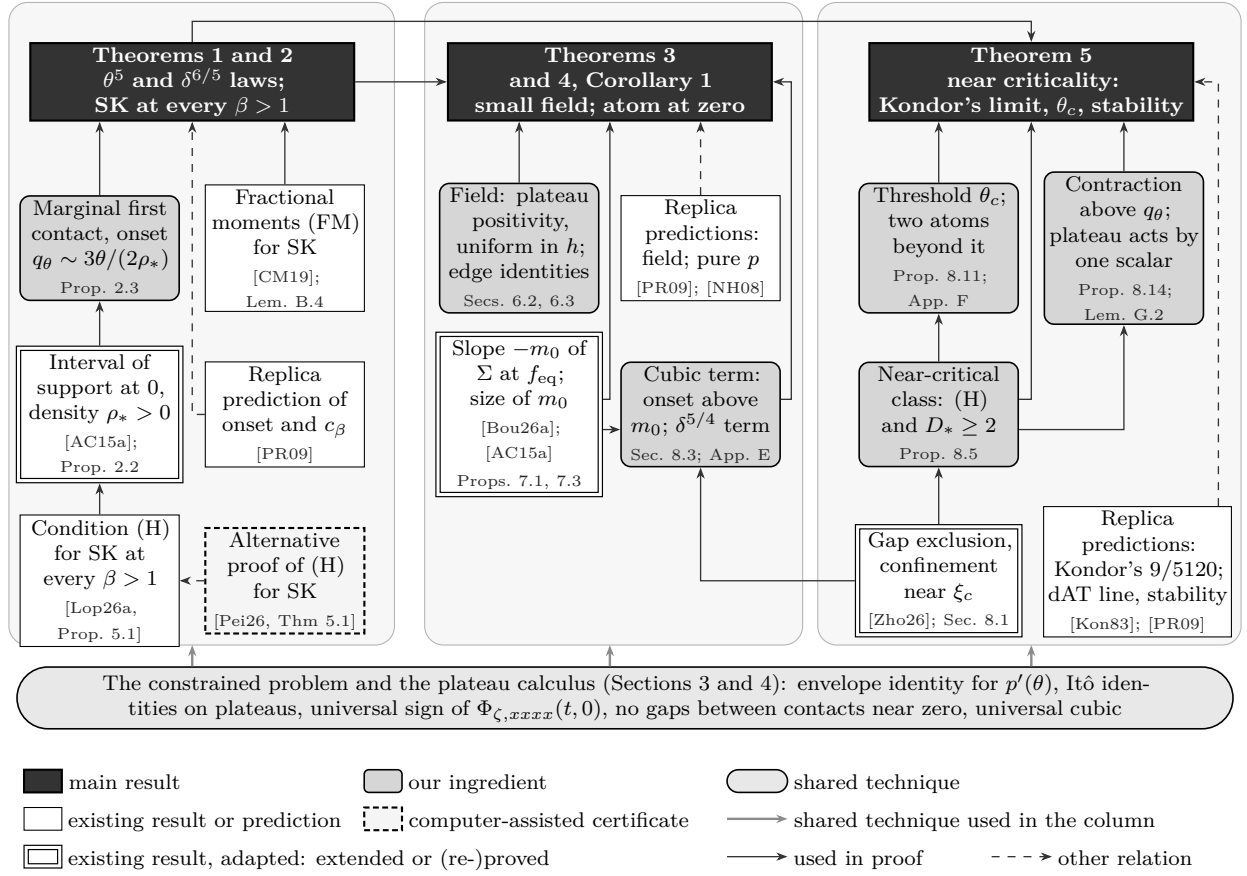
\begin{figure}[tp]
\centering
\newcommand*{\rmloc}[1]{\\[1pt]{\tiny\color{black!85}#1}}%
\begin{NoHyper}%
\begin{tikzpicture}[
  x=1cm, y=1cm, font=\scriptsize,
  execute at begin node={\thickmuskip=5mu\medmuskip=4mu\relax},
  >={Stealth[length=1.5mm,width=1.1mm]},
  thmlook/.style={fill=black!80, draw=black, line width=0.8pt},
  ourslook/.style={rounded corners=1.3mm, fill=black!16, draw=black, line width=0.6pt},
  knownlook/.style={fill=white, draw=black, line width=0.4pt},
  redolook/.style={fill=white, draw=black, line width=0.45pt,
    path picture={\draw[black, line width=0.45pt]
      ([shift={(1.6pt,1.6pt)}]path picture bounding box.south west) rectangle
      ([shift={(-1.6pt,-1.6pt)}]path picture bounding box.north east);}},
  certlook/.style={fill=black!4, draw=black, line width=0.8pt, dash pattern=on 2.2pt off 1.3pt},
  toollook/.style={rounded rectangle, fill=black!9, draw=black, line width=0.6pt},
  box/.style={rectangle, align=center, text width=2.02cm, inner xsep=1.2pt, inner ysep=1.5pt,
    minimum height=0.9cm},
  wide/.style={text width=4.2cm},
  thm/.style={box, wide, thmlook, text=white, font=\scriptsize\bfseries},
  ours/.style={box, ourslook},
  known/.style={box, knownlook},
  redo/.style={box, redolook, inner xsep=2.6pt, inner ysep=2.8pt},
  cert/.style={box, certlook},
  tool/.style={toollook, align=center, inner sep=2.5pt, minimum height=0.5cm},
  sw/.style={rectangle, minimum width=0.5cm, minimum height=0.3cm, inner sep=0pt},
  lg/.style={anchor=west, inner sep=2pt, text height=1.6ex, text depth=0.3ex},
  use/.style={->, line width=0.5pt, draw=black!80},
  rel/.style={use, dashed},
  band/.style={->, line width=0.9pt, draw=black!45},
  col/.style={rounded corners=2.5mm, fill=black!3, draw=black!30, line width=0.4pt},
]
\draw[col] (0.00,1.05) rectangle (5.10,-7.45);
\draw[col] (5.50,1.05) rectangle (10.50,-7.45);
\draw[col] (10.70,1.05) rectangle (16.30,-7.45);
\node[thm] (T1) at (2.425,0) {Theorems~\ref{thm:sk} and~\ref{thm:local}\\
  $\theta^5$ and $\delta^{6/5}$ laws;\\ SK at every $\beta>1$};
\node[thm] (T2) at (7.950,0) {Theorems~\ref{thm:field} and~\ref{thm:atom},
  Corollary~\ref{cor:pure-p}\\ small field; atom at zero};
\node[thm] (T3) at (13.525,0) {Theorem~\ref{thm:nearcrit}\\ near criticality:\\
  Kondor's limit, $\theta_c$, stability};
\node[ours]  (A1) at (1.20,-2.2) {Marginal first\\ contact, onset\\
  $q_\theta\sim3\theta/(2\rho_*)$\rmloc{Prop.~\ref{prop:ons-onset}}};
\node[known] (A4) at (3.65,-2.2) {Fractional\\ moments (FM)\\ for SK%
  \rmloc{\cite{CM19}; Lem.~\ref{lem:A-fm-sk}}};
\node[redo]  (A3) at (1.20,-4.4) {Interval of\\ support at 0,\\ density $\rho_*>0$%
  \rmloc{\cite{AC15a}; Prop.~\ref{prop:ons-equilibrium}}};
\node[known] (A6) at (3.65,-4.4) {Replica\\ prediction of\\ onset and $c_\beta$%
  \rmloc{\cite{PR09}}};
\node[known] (A5) at (1.20,-6.6) {Condition (H)\\ for SK at\\ every $\beta>1$%
  \rmloc{\cite[Prop.~5.1]{Lop26a}}};
\node[cert]  (A7) at (3.65,-6.6) {Alternative\\ proof of (H)\\ for SK%
  \rmloc{\cite[Thm~5.1]{Pei26}}};
\node[ours]  (C1) at (6.75,-2.2) {Field: plateau\\ positivity,\\ uniform in $h$;\\ edge identities%
  \rmloc{Secs.~\ref{sec:fld-proof}, \ref{sec:fld-identities}}};
\node[known] (C6) at (9.15,-2.2) {Replica\\ predictions:\\ field; pure $p$%
  \rmloc{\cite{PR09}; \cite{NH08}}};
\node[redo]  (C2) at (6.75,-4.4) {Slope $-m_0$ of\\ $\Sigma$ at $f_{\rm eq}$;\\ size of $m_0$%
  \rmloc{\cite{Bou26a}; \cite{AC15a}}%
  \rmloc{Props.~\ref{prop:atom-slope}, \ref{prop:atom-atom}}};
\node[ours]  (C5) at (9.15,-4.4) {Cubic term:\\ onset above\\ $m_0$; $\delta^{5/4}$ term%
  \rmloc{Sec.~\ref{sec:atom-odd}; App.~\ref{app:cubic}}};
\node[ours]  (B2) at (12.30,-2.2) {Threshold $\theta_c$;\\ two atoms\\ beyond it%
  \rmloc{Prop.~\ref{prop:nc-dat-threshold}; App.~\ref{app:dat}}};
\node[ours]  (B4) at (14.75,-2.2) {Contraction\\ above $q_\theta$;\\ plateau acts by\\ one scalar%
  \rmloc{Prop.~\ref{prop:nc-above-contact}; Lem.~\ref{lem:D-plateau}}};
\node[ours]  (B1) at (12.30,-4.4) {Near-critical\\ class: (H)\\ and $D_*\ge2$%
  \rmloc{Prop.~\ref{prop:nc-class}}};
\node[redo]  (B3) at (12.30,-6.6) {Gap exclusion,\\ confinement\\ near $\xi_c$%
  \rmloc{\cite{Zho26}; Sec.~\ref{sec:nc-reference}}};
\node[known, text width=2.40cm, anchor=south] (B5) at (14.93,0 |- B3.south) {Replica\\
  predictions:\\ Kondor's $9/5120$;\\ dAT line, stability\rmloc{\cite{Kon83}; \cite{PR09}}};
\draw[use] (T1.east) -- (T2.west);
\draw[use] (T1.north) -- (2.425,0.80) -- (13.525,0.80) -- (T3.north);
\draw[use] (A1) -- (A1 |- T1.south);
\draw[use] (A4) -- (A4 |- T1.south);
\draw[use] (A3) -- (A1);
\draw[use] (A5) -- (A3);
\draw[use] (C1) -- (C1 |- T2.south);
\coordinate (C2u) at ([yshift=2mm]C2.east);
\coordinate (C2d) at ([yshift=-2mm]C2.east);
\coordinate (C5u) at ([yshift=2mm]C5.east);
\draw[use] (C2u) -- (7.950,0 |- C2u) -- (7.950,0 |- T2.south);
\draw[use] (C2d) -- (C5.west |- C2d);
\draw[use] (C5u) -- (10.35,0 |- C5u) -- (10.35,0) -- (T2.east);
\draw[use] (B3.west) -- (9.15,-6.6) -- (C5.south);
\draw[use] (B2) -- (B2 |- T3.south);
\draw[use] (B4) -- (B4 |- T3.south);
\draw[use] (B1) -- (B2);
\coordinate (B1u) at ([yshift=2mm]B1.east);
\coordinate (B1d) at ([yshift=-2mm]B1.east);
\draw[use] (B1u) -- (13.525,0 |- B1u) -- (13.525,0 |- T3.south);
\draw[use] (B1d) -- (14.75,0 |- B1d) -- (B4.south);
\draw[use] (B3) -- (B1);
\draw[rel] (A7) -- (A5);
\draw[rel] (A6.west) -- (2.425,-4.4) -- (2.425,-4.4 |- T1.south);
\draw[rel] (C6) -- (C6 |- T2.south);
\draw[rel] (16.00,0 |- B5.north) -- (16.00,0) -- (T3.east);
\node[tool, text width=15.6cm] (S1) at (8.15,-8.15) {The constrained problem and the plateau
  calculus (Sections~\ref{sec:cons} and~\ref{sec:plat}): envelope identity for $p'(\theta)$,
  It\^o identities on plateaus, universal sign of $\Phi_{\zeta,xxxx}(t,0)$, no gaps between
  contacts near zero, universal cubic};
\draw[band] (2.425,0 |- S1.north) -- (2.425,-7.45);
\draw[band] (7.950,0 |- S1.north) -- (7.950,-7.45);
\draw[band] (13.525,0 |- S1.north) -- (13.525,-7.45);
\node[sw, thmlook]  (g1) at (0.45,-9.25) {}; \node[lg] at (g1.east) {main result};
\node[sw, ourslook, rounded corners=0.6mm] (g2) at (4.95,-9.25) {};
  \node[lg] at (g2.east) {our ingredient};
\node[sw, toollook, minimum width=0.8cm] (g3) at (9.90,-9.25) {};
  \node[lg] at (g3.east) {shared technique};
\node[sw, knownlook] (g4) at (0.45,-9.75) {};
  \node[lg] at (g4.east) {existing result or prediction};
\node[sw, certlook]  (g6) at (4.95,-9.75) {};
  \node[lg] at (g6.east) {computer-assisted certificate};
\draw[band] (9.50,-9.75) -- (10.30,-9.75) node[lg] {shared technique used in the column};
\node[sw, redolook]  (g5) at (0.45,-10.25) {};
  \node[lg] at (g5.east) {existing result, adapted: extended or (re-)proved};
\draw[use]  (9.50,-10.25) -- (10.30,-10.25) node[lg] {used in proof};
\draw[rel]  (13.00,-10.25) -- (13.80,-10.25) node[lg] {other relation};
\end{tikzpicture}%
\end{NoHyper}
\caption[How the main results are proved]{How the main results are proved. Each column leads to
one group of results; SK and dAT stand for Sherrington--Kirkpatrick and de Almeida--Thouless. A
solid arrow means that its source is used in the proof of its target. The two arrows that leave
the first theorem box say that Theorem~\ref{thm:local} is used in the proofs of
Theorems~\ref{thm:field} and~\ref{thm:nearcrit}, and the arrow from the third column into the
second says that Theorem~\ref{thm:atom}(c) is proved with the estimates near $\xic$ of
Sections~\ref{sec:nc-reference} and~\ref{sec:nc-class} (Proposition~\ref{prop:nc-support}).
The grey arrows from the bottom band mark the calculus of Sections~\ref{sec:cons}
and~\ref{sec:plat}, which every column uses. Only the main dependencies are drawn. The arrow
from condition~\condH\ for SK means that Lopatto's theorem supplies the hypothesis~\condH\ of
Proposition~\ref{prop:ons-equilibrium} for SK; it enters the proof of Theorem~\ref{thm:sk}.
Dashed arrows mark two other relations: an alternative, computer-assisted proof of
condition~\condH\ for SK, and predictions of replica theory that the target proves in the sense
stated next. Theorems~\ref{thm:sk} and~\ref{thm:field} prove leading terms of predictions of
Parisi and Rizzo~\cite{PR09}, and Corollary~\ref{cor:pure-p} is the rigorous counterpart of an
observation of Nakajima and Hukushima~\cite{NH08}. Theorem~\ref{thm:nearcrit}(b),(c) hold on the
whole near-critical class of part~(a), which contains SK for $\beta$ close to one; the dashed
arrow into Theorem~\ref{thm:nearcrit} marks Kondor's coefficient~\cite{Kon83}, obtained as a
limit, and the dAT line and stability prediction of Parisi and Rizzo, whose near-critical
constrained-Parisi form it proves. For SK the fractional-moment formula enters only
Theorem~\ref{thm:sk}(c), through Lemma~\ref{lem:A-fm-sk}; in Theorems~\ref{thm:local}(d)
and~\ref{thm:atom}(d) and in Corollary~\ref{cor:pure-p} it is a hypothesis. Box styles give provenance; grey labels give
sources and locate each item in the paper. Section~\ref{sec:intro-ideas} explains each column,
and Section~\ref{sec:disc-prior} lists the existing results that we re-prove or extend. The
bottom band combines existing tools, the first-variation formula~\cite{Che17}, the It\^o
identities of~\cite{AC15a} and the duality facts of~\cite{Bou26b}, with the ingredients of
Section~\ref{sec:plat}. For SK, the density $\rho_*>0$ in the box on the interval of support is
also in~\cite{Lop26b}, \cite[Section~7]{AL26} and~\cite[Lemma~5.8]{OAI26a}. Theorem~\ref{thm:sk-global},
which rests on an unrefereed finite-step inequality~\cite{OAI26c}, is not drawn; its proof
architecture is described in Section~\ref{sec:intro-ideas} and carried out in
Section~\ref{sec:skg}.}
\label{fig:roadmap}
\end{figure}

\subsection{How the proofs work}\label{sec:intro-ideas}

Figure~\ref{fig:roadmap} shows the architecture: each column leads to one group of theorems, and
every column uses the calculus of Sections~\ref{sec:cons} and~\ref{sec:plat}, drawn as the bottom
band. We describe the first column, which contains the decisive steps, and then what changes in
the other two.

\paragraph{The envelope identity and what it requires.}
Write $\Gamma_\theta=\Gamma_{\zeta_\theta}$ and $\Obs_\theta=\Obs_{\zeta_\theta}$. The starting
point is an envelope identity for the constrained minimum (Proposition~\ref{prop:cons-envelope}):
at every $\theta\in(\zetas(\{0\}),1)$ the function $p$ is differentiable, and
\begin{equation}\label{eq:intro-envelope}
 p'(\theta)=\Obs_\theta(0)-\Obs_\theta(q_\theta)
 =\frac12\int_0^{q_\theta}\xi''(s)\bigl\{\Gamma_\theta(s)-s\bigr\}\,ds .
\end{equation}
It follows from the first-variation formula \cite[Theorem~2]{Che17} by a squeeze in the
parametrization $\zeta=\theta\delta_0+(1-\theta)\nu$. In words, the marginal cost of the
constraint is the area between $\Gamma_\theta$ and the identity on the first plateau, weighted by
$\xi''/2$. Computing it is a singular perturbation problem: the cost $p(\theta)-\feq$ turns out to
be of order $\theta^5$, while $\zeta_\theta$ differs from $\zetas$ by at least order $\theta^2$ in
$\Wone$, since the constraint moves a mass of order $\theta$ over a distance of order $\theta$. A
trial measure gives an upper bound of the right order \cite[Lemma~7.2]{AL26}, and a
second-variation argument the matching lower order \cite[Proposition~7.3]{AL26}; neither
determines the shape of $\zeta_\theta$ at scale $\theta$, and that shape determines the constant.
Through \eqref{eq:intro-envelope} it enters only via the profile of $\Gamma_\theta$ on the first
plateau, which is pinned down by four facts: (i) the first positive contact is marginal,
$\Gamma_\theta'(q_\theta)=1$, which is not automatic, since the Plefka inequality can be strict
(Theorem~\ref{thm:nearcrit}(b)); (ii) $\Gamma_\theta'''$ is controlled uniformly on the first
plateau; (iii) this control requires a sign, $\Dstar>0$; (iv) zero is not isolated in the support
of $\zetas$, condition \condH, which brings the contacts of $\zeta_\theta$ down to zero as
$\theta\downarrow0$. Parisi and Rizzo assume that the order parameter is unperturbed above the
first positive contact \cite[Appendix~B]{PR09}; this ansatz presupposes (i).

\paragraph{The deduction.}
Granting (i)--(iv), the constant follows in a few lines. On the first plateau the distribution
function of $\zeta_\theta$ is constant, so $\Phi_{\zeta_\theta}$ is a Cole--Hopf transform there,
and It\^o's formula expresses $\Gamma_\theta''$ and $\Gamma_\theta'''$ through expectations of
polynomials in the second to fourth spatial derivatives of $\Phi_{\zeta_\theta}$ along the
diffusion (Section~\ref{sec:plat-ito}). Since $\xi'''(0)=0$, these identities give
$\Gamma_\theta'''\to\Dstar$ uniformly on $(0,q_\theta)$ as $\theta\downarrow0$
(Lemma~\ref{lem:plat-local-plateau}). The function $h_\theta=\Gamma_\theta-\id$, which satisfies
$h_\theta(0)=0$ by evenness, $h_\theta(q_\theta)=0$ by the contact condition and
$h_\theta'(q_\theta)=0$ by marginality, is therefore to leading order the universal cubic
\begin{equation}\label{eq:intro-cubic}
 h_\theta(t)=\frac{\Dstar}{6}\,t\,(q_\theta-t)^2+o(q_\theta^3)\qquad(0\le t\le q_\theta),
\end{equation}
and the same holds for its second derivative at both ends, up to $o(q_\theta)$. Two curvatures of
this cubic fix the onset (Figure~\ref{fig:plateau}(b)). At zero, the It\^o identity gives
$\Gamma_\theta''(0+)=-2\azero^2\theta\,\bar U_\theta^3$, where $\bar U_\theta=\Phi_{\zeta_\theta,xx}(0,0)$
tends to $\Ustar=\azero^{-1/2}$; this is where the value $\theta$ of the distribution function on
the plateau enters, and matching with the curvature $-\frac23\Dstar q_\theta$ of the cubic gives
$q_\theta/\theta\to3\sqrt{\azero}/\Dstar=3/(2\rhostar)$. At $q_\theta$, the cubic has curvature
$\frac13\Dstar q_\theta$ from the left, while $\Gamma_\theta''=0$ from the right, because
$\Gamma_\theta$ is the identity on an interval of contacts. The drop $\Gamma_\theta''(q_\theta-)-\Gamma_\theta''(q_\theta+)$ is produced by the atom at
$q_\theta$: it equals
$2\xi''(q_\theta)^2\zeta_\theta(\{q_\theta\})\E\Phi_{\zeta_\theta,xx}(q_\theta,X_{q_\theta})^3$, and
matching gives $\zeta_\theta(\{q_\theta\})/\theta\to\frac12$. Finally, the area under the cubic is
$\Dstar q_\theta^4/72$ and $\xi''(s)=\azero+O(s^2)$, so \eqref{eq:intro-envelope} gives
(Figure~\ref{fig:plateau}(c))
\[
 p'(\theta)=\frac{\azero\Dstar}{144}\,q_\theta^4\,(1+o(1))=5c\,\theta^4\,(1+o(1)),\qquad
 c=\frac{9\azero^3}{80\Dstar^3}.
\]
Integrating from $p(0+)=\feq$ gives $p(\theta)-\feq=c\theta^5(1+o(1))$, hence
$\Lambda(\theta)-\theta\feq=c\theta^6(1+o(1))$ and $\Lambda'(\theta)-\feq=6c\theta^5(1+o(1))$. At
the level $f=\feq+\delta$ the dual parameter solves $\Lambda'(\theta_f)=f$, so that
$\theta_f\sim(\delta/6c)^{1/5}$ and
\[
 \Sigma(\feq+\delta)=\Lambda(\theta_f)-\theta_f(\feq+\delta)=-5c\,\theta_f^6\,(1+o(1))
 =-\frac56\,(6c)^{-1/5}\,\delta^{6/5}\,(1+o(1)).
\]
Under \condFM, an exponential change of measure turns $\Sigma$ into the upper-tail rate
(Lemma~\ref{lem:cons-gartner-ellis}); this needs $\Lambda$ to be differentiable, which is again
the envelope identity.

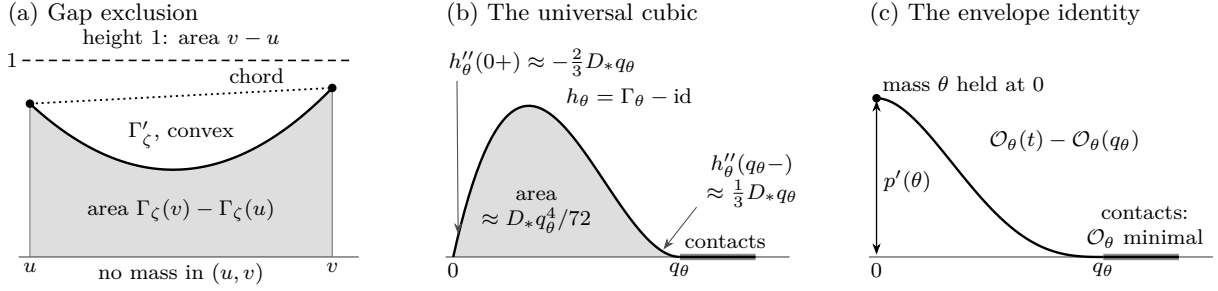
\begin{figure}[tp]
\centering
\begin{tikzpicture}[x=1cm, y=1cm, font=\scriptsize,
  >={Stealth[length=1.5mm,width=1.1mm]},
  axis/.style={line width=0.4pt, draw=black!75},
  curve/.style={line width=0.9pt, draw=black},
  level/.style={line width=0.6pt, draw=black, dash pattern=on 3pt off 2pt},
  chord/.style={line width=0.7pt, draw=black, dash pattern=on 0.8pt off 1.4pt},
  side/.style={line width=0.4pt, draw=black!55},
  contacts/.style={line width=2.4pt, draw=black!45},
  area/.style={fill=black!13},
  pointer/.style={->, line width=0.4pt, draw=black!70},
  ttl/.style={font=\footnotesize, anchor=south west, inner sep=0pt},
  dot/.style={circle, fill=black, inner sep=0pt, minimum size=3.2pt},
]
\begin{scope}
 \node[ttl] at (-0.3,3.05) {(a) Gap exclusion};
 \fill[area] (0,0) -- plot[domain=0:1, samples=60, variable=\s]
   ({4*\s},{2.6*(0.78-1.42*\s+1.5*\s*\s)}) -- (4,0) -- cycle;
 \draw[axis] (-0.15,0) -- (4.35,0);
 \draw[side] (0,0) -- (0,2.028);
 \draw[side] (4,0) -- (4,2.236);
 \draw[level] (-0.1,2.6) -- (4.25,2.6);
 \node[anchor=east, inner sep=1.5pt] at (-0.1,2.6) {$1$};
 \node[anchor=south] at (2,2.62) {height $1$: area $v-u$};
 \draw[chord] (0,2.028) -- (4,2.236);
 \node at (3.0,2.40) {chord};
 \draw[curve] plot[domain=0:1, samples=60, variable=\s]
   ({4*\s},{2.6*(0.78-1.42*\s+1.5*\s*\s)});
 \node at (2.0,1.62) {$\Gamma_\zeta'$, convex};
 \node at (2.0,0.62) {area $\Gamma_\zeta(v)-\Gamma_\zeta(u)$};
 \node[dot] at (0,2.028) {};
 \node[dot] at (4,2.236) {};
 \node[anchor=north, inner sep=2pt] at (0,0) {$u$};
 \node[anchor=north, inner sep=2pt] at (4,0) {$v$};
 \node[anchor=north, inner sep=2pt] at (2,0) {no mass in $(u,v)$};
\end{scope}
\begin{scope}[xshift=5.6cm]
 \node[ttl] at (-0.1,3.05) {(b) The universal cubic};
 \fill[area] (0,0) -- plot[domain=0:1, samples=60, variable=\s]
   ({3*\s},{13.5*\s*(1-\s)^2}) -- cycle;
 \draw[axis] (-0.1,0) -- (4.45,0);
 \draw[contacts] (3,0) -- (4,0);
 \draw[curve] plot[domain=0:1, samples=60, variable=\s] ({3*\s},{13.5*\s*(1-\s)^2});
 \draw[curve] (3,0) -- (4,0);
 \node[anchor=south west, inner sep=1pt] at (1.45,1.95) {$h_\theta=\Gamma_\theta-\id$};
 \node[align=center] at (1.1,0.6) {area\\ $\approx\Dstar q_\theta^4/72$};
 \node[anchor=south west, inner sep=1pt] at (-0.1,2.35)
   {$h_\theta''(0+)\approx-\frac23\Dstar q_\theta$};
 \draw[pointer] (0.05,2.33) -- (0.065,0.30);
 \node[anchor=south east, align=center, inner sep=1pt] at (4.6,0.62)
   {$h_\theta''(q_\theta-)$\\ $\approx\frac13\Dstar q_\theta$};
 \draw[pointer] (3.22,0.66) -- (2.78,0.10);
 \node[anchor=south, inner sep=1.5pt] at (3.6,0.06) {contacts};
 \node[anchor=north, inner sep=2pt] at (0,0) {$0$};
 \node[anchor=north, inner sep=2pt] at (3,0) {$q_\theta$};
\end{scope}
\begin{scope}[xshift=11.2cm]
 \node[ttl] at (-0.1,3.05) {(c) The envelope identity};
 \draw[axis] (-0.1,0) -- (4.45,0);
 \draw[contacts] (3,0) -- (4,0);
 \draw[curve] plot[domain=0:1, samples=60, variable=\s] ({3*\s},{2.1*(1-\s)^3*(1+3*\s)});
 \draw[curve] (3,0) -- (4,0);
 \draw[<->, line width=0.5pt] (0,0.03) -- (0,2.07);
 \node[anchor=west, inner sep=2pt] at (0.02,1.0) {$p'(\theta)$};
 \node[dot] at (0,2.1) {};
 \node[anchor=south west, inner sep=2pt] at (0,2.12) {mass $\theta$ held at $0$};
 \node[anchor=south west, inner sep=1pt] at (1.45,1.35) {$\Obs_\theta(t)-\Obs_\theta(q_\theta)$};
 \node[anchor=south, align=center, inner sep=1.5pt] at (3.55,0.06)
   {contacts:\\ $\Obs_\theta$ minimal};
 \node[anchor=north, inner sep=2pt] at (0,0) {$0$};
 \node[anchor=north, inner sep=2pt] at (3,0) {$q_\theta$};
\end{scope}
\end{tikzpicture}
\caption{The plateau calculus, schematically (constant $\Gamma'''$ and $\xi''$; vertical scales
differ by panel).
(a)~Gap exclusion (Lemma~\ref{lem:plat-local-gap}). If two contacts $u<v$ near zero of a
measure $\zeta$ close to $\zetas$ enclosed a plateau, then $\Gamma_\zeta'''\ge\Dstar/2>0$ there
(Lemma~\ref{lem:plat-local-plateau}); $\Gamma_\zeta'$, at most $1$ at the contacts, would be
convex and below its chord, so the grey area $\Gamma_\zeta(v)-\Gamma_\zeta(u)$ would be less
than $v-u$, contradicting $\Gamma_\zeta(u)=u$, $\Gamma_\zeta(v)=v$. Hence for small $\theta$ the
contacts of $\zeta_\theta$ near zero fill an interval $[q_\theta,s_\theta]$ where
$\Gamma_\theta=\id$, and the first positive contact is marginal, $\Gamma_\theta'(q_\theta)=1$
(Proposition~\ref{prop:ons-onset}(b)).
(b)~The universal cubic (Lemma~\ref{lem:plat-cubic}, Corollary~\ref{cor:plat-first-plateau}).
On the first plateau $(0,q_\theta)$, $h_\theta=\Gamma_\theta-\id$ vanishes at $0$ (evenness)
and with $h_\theta'$ at $q_\theta$ (contact, marginality), and $h_\theta'''\to\Dstar$: to
leading order it is the cubic $(\Dstar/6)\,t\,(q_\theta-t)^2$, with curvature
$-\frac23\Dstar q_\theta$ at zero and $\frac13\Dstar q_\theta$ at $q_\theta-$; matching these
with the curvatures that the atoms at $0$ and at $q_\theta$ produce fixes $q_\theta$ and the
atom at $q_\theta$ (Section~\ref{sec:intro-ideas}).
(c)~The envelope identity (Proposition~\ref{prop:cons-envelope}). The obstacle $\Obs_\theta$ is
minimal on the contacts; the constraint holds the mass $\theta$ at zero, where $\Obs_\theta$ is
larger by $p'(\theta)=\Obs_\theta(0)-\Obs_\theta(q_\theta)
=\frac12\int_0^{q_\theta}\xi''(s)\,h_\theta(s)\,ds$. By the area in (b) this is
$(\azero\Dstar/144)\,q_\theta^4(1+o(1))=5c\,\theta^4(1+o(1))$, the fifth-order law of
Theorem~\ref{thm:local}(b).}
\label{fig:plateau}
\end{figure}

\paragraph{Where the four facts come from.}
The sign (iii) is the universal sign of Section~\ref{sec:plat-sign}: a Feynman--Kac representation
of $\Phi_{\zeta,xxx}$ along the diffusion killed at zero, with the reflection principle, gives
$\Phi_{\zeta,xxx}(t,x)\le-\kappa_tx$ for $0\le x\le\frac12$ and every $\zeta$. With
$\xi'''(0)=0$ it makes $\Gamma_\zeta'''\ge\Dstar/2$ on every plateau, in a fixed interval near
zero, of every $\zeta$ close to $\zetas$, which is (ii); a trapezoid principle then excludes
plateaus between two contacts in that interval (Section~\ref{sec:plat-gap},
Figure~\ref{fig:plateau}(a)). Section~\ref{sec:ons} combines this gap exclusion with (iv) to
obtain an interval of support at zero, with $\Ustar=\azero^{-1/2}$, $\zetas(\{0\})=0$ and
$\alphas'(0)=\rhostar$, and then the marginal onset (i); Section~\ref{sec:fifth} assembles the
deduction.

\paragraph{The other two columns.}
For the field, $\zeta_h$ has no mass in $[0,\qmin(h))$. On that gap $\Phi_{\zeta_h}$ solves a heat
equation, and the diffusion started at $h$ is Gaussian at time $q=\qmin(h)$, with mean $h$ and
variance $\xi'(q)$. The contact condition, marginality and the definition of $\mfield$ become
three exact identities for the odd profile $v=\Phi_{\zeta_h,x}(q,\cdot)$: $\E v(X)^2=q$,
$\E v'(X)^2=1/\xi''(q)$ and $\E v(X)=\mfield(h)$ for $X\sim\mathcal N(h,\xi'(q))$; expanding the
first two balances $h^2/\azero$ against $\frac13\Dstar q^3$, which gives the gap, and the third
gives the magnetization (Section~\ref{sec:fld}). For the atom at zero, Theorem~\ref{thm:atom}(a)
is convex duality and part (b) evaluates the It\^o identity at the origin
(Section~\ref{sec:atom}); part (c) runs the plateau calculus at $\theta-m_0$ with global inputs
from the third column, the interval support of $\zetas$ and the confinement of the contacts near
zero, hence the arrow into the second column (Section~\ref{sec:atom-odd}). Near $\xic$,
closeness to the reference pair $(\xic,\delta_0)$ replaces the local inputs: it gives
$\Gamma'''\ge2$ on plateaus in a fixed interval near zero, an exclusion of gaps between contacts
that adapts Zhou's mechanism \cite{Zho26} to the constrained problem for mixed covariances, and
the confinement of every positive contact to that interval (Section~\ref{sec:nc-reference}),
hence Theorem~\ref{thm:nearcrit}(a) (Section~\ref{sec:nc-class}). The de Almeida--Thouless threshold comes from an explicit computation along the two-atom branch
$\theta\delta_0+(1-\theta)\delta_q$ under a $\cosh^\theta$-tilted Gaussian law, which runs for
every covariance of the class, whose nonnegative coefficients control the branch uniformly, and
from a convexity argument that identifies the minimizer (Section~\ref{sec:nc-dat},
Appendix~\ref{app:dat}). Stability above the first contact is a contraction argument above
$q_\theta$, where $\zeta_\theta$ and $\zetas$ satisfy the same relation on their supports, so
that their difference is driven by the laws of the two diffusions at $q_\theta$. Tested against
even functions, these laws differ, to the order that matters, only through one scalar, the
first moment of the difference of the two distribution functions on the first plateau, and the
common contact $\Gamma_\theta(q_\theta)=q_\theta=\Gamma_{\zetas}(q_\theta)$ makes that scalar
small (Section~\ref{sec:nc-stability}, Appendix~\ref{app:stability}).

\paragraph{The Sherrington--Kirkpatrick model at every temperature.}
Theorem~\ref{thm:sk-global} is not drawn in Figure~\ref{fig:roadmap}; its decisive step is
part (a). On a plateau of height $m$, the It\^o identities express $\Gamma_\zeta'''$ as
$\beta^6$ times the mean of a polynomial in the second to fourth spatial derivatives of
$\Phi_\zeta$ along the diffusion (Section~\ref{sec:plat-ito}). For step-function order parameters
on a uniform grid and at grid times, the finite-step inequality of \cite{OAI26c} bounds that mean
below by $m^2/1000$ times the fourth moment of $\Phi_{\zeta,xx}$. The constant does not depend on
the steps and the temperature is fixed, so the bound passes to every measure by weak
approximation (Appendix~\ref{sec:H-transfer}). The trapezoid principle then excludes a plateau
between any two contacts, which gives (b) and the dichotomy (c). In the two-atom case,
$\Gamma'-1$ is convex on both plateaus, so a sign-pattern lemma fixes the sign of
$\Gamma-\id$ on each side of the atom; the obstacle then has its strict minimum at the atom, and
the first-variation formula identifies the minimizer. Hence the two-atom measure at the
replica-symmetric root is the constrained minimizer if $\mathfrak A(\theta)<1$ and is not if
$\mathfrak A(\theta)>1$ (Proposition~\ref{prop:skg-two-atom}). That $\mathfrak A$ decreases
reduces, after an integration by parts against the tilted Gaussian law, to a pointwise inequality
in $x$ (Lemma~\ref{lem:skg-branch}), and the threshold itself is reached by the weak continuity of
$\theta\mapsto\zeta_\theta$ (Section~\ref{sec:skg}).

\subsection{Organization}\label{sec:intro-organization}

Section~\ref{sec:fifth} states four inputs about the bottom of $\zetas$ and $\zeta_\theta$: the envelope identity
(Proposition~\ref{prop:cons-envelope}), the equilibrium measure near zero
(Proposition~\ref{prop:ons-equilibrium}), the marginal onset (Proposition~\ref{prop:ons-onset})
and the area of the first plateau (Lemma~\ref{lem:fifth-area}), and deduces
Theorems~\ref{thm:local} and~\ref{thm:sk} from them; Sections~\ref{sec:cons}, \ref{sec:plat}
and~\ref{sec:ons} prove these inputs. Sections~\ref{sec:fld}, \ref{sec:atom} and~\ref{sec:nc}
prove Theorem~\ref{thm:field}, Theorem~\ref{thm:atom}(a),(b),(d) with Corollary~\ref{cor:pure-p},
and Theorem~\ref{thm:nearcrit} with Theorem~\ref{thm:atom}(c), and Section~\ref{sec:skg} proves
Theorem~\ref{thm:sk-global}. Section~\ref{sec:disc} discusses
earlier work, including the existing results that we re-prove or extend, and open
questions. The appendices contain the routine proofs (Appendix~\ref{app:analytic}), the
spin-system side, fractional moments and Guerra's bound (Appendix~\ref{app:guerra}), the
dictionary to the replica predictions (Appendix~\ref{app:dictionary}), and the longer proofs for
the field, the cubic term, the de Almeida--Thouless threshold and the stability above the first
contact (Appendices~\mbox{\ref{app:field}--\ref{app:stability}}), and the transfer of the
finite-step inequality and the Gaussian lemmas of Section~\ref{sec:skg}
(Appendix~\ref{app:skg}).

\section{Proof of Theorems~\ref*{thm:sk} and~\ref*{thm:local}}\label{sec:fifth}

Theorem~\ref{thm:sk} is the case $\xi(t)=\beta^2t^2/2$ of Theorem~\ref{thm:local}, once~\condH\
is known; its part~(c) also uses \condFM, which holds for this covariance (Lemma~\ref{lem:A-fm-sk}),
and, for $\FNb$, Lemma~\ref{lem:cons-original-sk} (Section~\ref{sec:fifth-sk}). We deduce
Theorem~\ref{thm:local} from four statements
about the bottom of $\zetas$ and $\zeta_\theta$ (Section~\ref{sec:fifth-inputs}), proved in
Sections~\ref{sec:cons}--\ref{sec:ons}: the envelope identity writes $p'(\theta)$ as a weighted
area of $\Gamma_\theta-\id$ over the first plateau $(0,q_\theta)$, the area of $\Gamma_\theta-\id$
there is that of the universal cubic, $\Dstar q_\theta^4/72$ to leading order, and the marginal
onset gives $q_\theta\sim3\theta/(2\rhostar)$. Hence $p'(\theta)\sim5c\,\theta^4$; integration
gives the fifth-order law, and the Legendre transform of a sixth power is a power $6/5$.

Throughout, $\xi(t)=\sum_{p\ge2}\beta_p^2t^p$ is not identically zero and has radius of
convergence greater than one, and $m_0=\zetas(\{0\})$; from
Proposition~\ref{prop:ons-equilibrium} on we assume the hypotheses $\beta_2>0$, $\xi'''(0)=0$
and~\condH\ of Theorem~\ref{thm:local}. We write $\Gamma_\theta=\Gamma_{\zeta_\theta}$,
$\Obs_\theta=\Obs_{\zeta_\theta}$, and $U_t=\Phi_{\zeta,xx}(t,X_t)$, $V_t=\Phi_{\zeta,xxx}(t,X_t)$
along the diffusion~\eqref{eq:intro-diffusion} of a measure $\zeta$. All asymptotics are at a
fixed covariance, with no uniformity in $\xi$.

\subsection{Four inputs}\label{sec:fifth-inputs}

The first input is exact and holds for every covariance.

\begin{proposition}[Envelope identity]\label{prop:cons-envelope}
The function $p$ is continuously differentiable on $(0,1)$. For every $\theta\in(0,1)$, with
$\zeta_\theta=\theta\delta_0+(1-\theta)\mu_\theta$,
\begin{equation}\label{eq:cons-envelope-general}
 p'(\theta)=\Obs_\theta(0)-\int\Obs_\theta\,d\mu_\theta .
\end{equation}
In particular $p'=0$ on $(0,1)\cap(0,m_0]$, and for $m_0<\theta<1$
\begin{equation}\label{eq:cons-envelope}
 p'(\theta)=\Obs_\theta(0)-\Obs_\theta(q_\theta)
 =\frac12\int_0^{q_\theta}\xi''(s)\bigl\{\Gamma_\theta(s)-s\bigr\}\,ds>0 .
\end{equation}
Consequently $\Lambda$ is continuously differentiable on $(0,1)$, with
$\Lambda'(\theta)=p(\theta)+\theta p'(\theta)$.
\end{proposition}

Its proof squeezes $p(\theta+h)-p(\theta)$ between first variations at the old and the new
minimizer (Section~\ref{sec:cons-envelope}); the right side of~\eqref{eq:cons-envelope} sees
$\Gamma_\theta$ only on $[0,q_\theta]$, at whose ends $\Gamma_\theta-\id$ vanishes
(Figure~\ref{fig:plateau}(b),(c)).

Let $K_{\rm loc}\in(0,1)$ be the constant of Lemma~\ref{lem:plat-local-plateau}, below which two
contacts of a measure close to $\zetas$ always enclose some mass
(Lemma~\ref{lem:plat-local-gap}).

\begin{proposition}[The equilibrium measure near zero]\label{prop:ons-equilibrium}
Assume $\beta_2>0$, $\xi'''(0)=0$ and~\condH. Then the following hold.
\begin{thmlist}
\item There is $\kappa\in(0,K_{\rm loc})$ with $[0,\kappa]\subset\supp\zetas$.
\item $\zetas(\{0\})=0$ and $\Ustar=\azero^{-1/2}$.
\item The distribution function $\alphas$ is continuously differentiable on $[0,\kappa)$.
 For $t\in[0,\kappa)$, with $U_t,V_t$ taken for $\zeta=\zetas$,
 \begin{equation}\label{eq:ons-quotient}
  \alphas(t)=\frac{\xi'''(t)\,\E U_t^2+\xi''(t)^2\,\E V_t^2}{2\,\xi''(t)^2\,\E U_t^3},
 \end{equation}
 and the slope at zero is
 \begin{equation}\label{eq:ons-slope}
  \rhostar:=\alphas'(0)=\frac{\Dstar}{2\sqrt{\azero}}>0 .
 \end{equation}
\end{thmlist}
\end{proposition}

In~(a), gap exclusion turns the accumulation point of~\condH\ into an interval of support. On
such an interval, endpoint zero included, the smoothness of $\alphas$ and~\eqref{eq:ons-quotient}
are \cite[Theorem~2(ii) and its proof]{AC15a}; \eqref{eq:ons-slope} is the derivative of that
formula at the endpoint zero (Section~\ref{sec:ons-equilibrium}).

\begin{proposition}[Marginal onset]\label{prop:ons-onset}
Assume $\beta_2>0$, $\xi'''(0)=0$ and~\condH. Then the following hold.
\begin{thmlist}
\item As $\theta\downarrow0$: $p(\theta)\to\feq$, $\Wone(\zeta_\theta,\zetas)\to0$ and
 $q_\theta\to0$.
\item There is $\theta_0\in(0,1)$ such that, for $0<\theta<\theta_0$,
 $\Gamma_\theta(t)=t$ on an interval $[q_\theta,s_\theta]$ with $s_\theta>q_\theta$. In
 particular the first positive contact is marginal:
 \[
  \Gamma_\theta'(q_\theta)
  =\xi''(q_\theta)\,\E\Phi_{\zeta_\theta,xx}(q_\theta,X_{q_\theta})^2=1 .
 \]
\item As $\theta\downarrow0$,
 \[
  \frac{q_\theta}{\theta}\longrightarrow\frac{3\sqrt{\azero}}{\Dstar}=\frac{3}{2\rhostar},
  \qquad
  \frac{\zeta_\theta(\{q_\theta\})}{\theta}\longrightarrow\frac12 .
 \]
\end{thmlist}
\end{proposition}

Part~(b) is gap exclusion for $\zeta_\theta$; part~(c) matches the universal cubic with the
curvatures that the atoms at $0$ and at $q_\theta$ produce (Figures~\ref{fig:plateau}(b)
and~\ref{fig:bottom}(a); Section~\ref{sec:ons-onset}).

By Proposition~\ref{prop:ons-equilibrium}(b), $m_0=0$, so~\eqref{eq:cons-envelope} holds at
every $\theta\in(0,1)$. The last input evaluates its right side.

\begin{lemma}[The area of the first plateau]\label{lem:fifth-area}
Under the hypotheses of Theorem~\ref{thm:local}, as $\theta\downarrow0$,
\[
 p'(\theta)=\frac{\azero\Dstar}{144}\,q_\theta^4\,(1+o(1)).
\]
\end{lemma}

To leading order $\Gamma_\theta-\id$ is the universal cubic $\frac{\Dstar}6\,t(q_\theta-t)^2$
of Lemma~\ref{lem:plat-cubic} on the first plateau, with integral $\Dstar q_\theta^4/72$, and
$\xi''\approx\azero$ there (Section~\ref{sec:ons-area}).

\paragraph{Other inputs.}
The deduction also uses the universal sign $\Dstar>0$ (Lemma~\ref{lem:plat-fourth-sign}), the
convex duality of Lemmas~\ref{lem:cons-variational}(b),(d) and~\ref{lem:cons-regular} and
Corollary~\ref{cor:cons-level}, the
Lipschitz bound~\eqref{eq:cons-lipschitz}, the exponential tilting of
Lemma~\ref{lem:cons-gartner-ellis}, and, for the Sherrington--Kirkpatrick model,
Lemmas~\ref{lem:A-fm-sk} and~\ref{lem:cons-original-sk} (Appendix~\ref{app:guerra}). The
proofs in Sections~\ref{sec:cons}--\ref{sec:ons} do not use Theorems~\ref{thm:sk}
and~\ref{thm:local}.

\subsection{Proof of Theorem~\ref*{thm:local}}\label{sec:fifth-proof-local}

The inequality $\Dstar>0$ is Lemma~\ref{lem:plat-fourth-sign}. The identities
$\zetas(\{0\})=0$ and $\Ustar=\azero^{-1/2}$, the interval $[0,\kappa]\subset\supp\zetas$, and
the regularity and slope of $\alphas$ are Proposition~\ref{prop:ons-equilibrium}. Item~(a) is
Proposition~\ref{prop:ons-onset}(b),(c); throughout, $\theta_0\in(0,1)$ is its threshold.

\begin{proof}[Proof of Theorem~\ref{thm:local}\textup{(b)}]
By Lemma~\ref{lem:fifth-area} and the onset limit $q_\theta/\theta\to3\sqrt{\azero}/\Dstar$
of item~(a),
\begin{equation}\label{eq:fifth-p-prime}
\begin{split}
 p'(\theta)&=\frac{\azero\Dstar}{144}\,q_\theta^4\,(1+o(1))
 =\frac{\azero\Dstar}{144}\Bigl(\frac{3\sqrt{\azero}}{\Dstar}\Bigr)^4\theta^4(1+o(1))\\
 &=\frac{9\azero^3}{16\Dstar^3}\,\theta^4(1+o(1))
 =5c\,\theta^4(1+o(1)).
\end{split}
\end{equation}
Write $p'(s)=5cs^4(1+r(s))$ with $r(s)\to0$ as $s\downarrow0$, and choose
$\theta_1\in(0,\theta_0)$ with $|r|\le1$ on $(0,\theta_1]$. Then $p$ is differentiable
at every point of $(0,\theta_1]$, with $|p'|\le10c\,\theta_1^4$ there. Hence $p$ is
Lipschitz on $(0,\theta_1]$, and $p(\theta)-p(\eta)=\int_\eta^\theta p'(s)\,ds$ for
$0<\eta<\theta\le\theta_1$. As $\eta\downarrow0$ we have $p(\eta)\to\feq$: the measure
$\eta\delta_0+(1-\eta)\zetas$ satisfies the constraint at $\eta$ and lies within
$\Wone$-distance $\eta$ of $\zetas$, so \eqref{eq:cons-lipschitz} gives
$\feq\le p(\eta)\le\feq+\frac12\xi''(1)\eta$. Letting $\eta\downarrow0$,
\[
 \bigl|p(\theta)-\feq-c\,\theta^5\bigr|\le c\,\theta^5\sup_{0<s\le\theta}|r(s)|,
\]
which is $p(\theta)-\feq=c\,\theta^5(1+o(1))$. Since $\Lambda(\theta)=\theta p(\theta)$,
\[
 \Lambda(\theta)-\theta\feq=\theta\bigl\{p(\theta)-\feq\bigr\}=c\,\theta^6(1+o(1)),\qquad
 \Lambda'(\theta)-\feq=\bigl\{p(\theta)-\feq\bigr\}+\theta p'(\theta)=6c\,\theta^5(1+o(1)).
 \qedhere
\]
\end{proof}

\begin{proof}[Proof of Theorem~\ref{thm:local}\textup{(c)}]
Since $\zetas(\{0\})=0$, Corollary~\ref{cor:cons-level} with $m_0=0$ and
Lemma~\ref{lem:cons-variational}(d) give that $\Lambda$ is continuously differentiable and
strictly convex on $(0,1)$, that $\Lambda'$ maps $(0,1)$ homeomorphically onto $(\feq,f_1)$,
that $\feq<f_1$ (Lemma~\ref{lem:cons-regular}(d)), and that $\theta_f=(\Lambda')^{-1}(f)$.

Now let $\delta\downarrow0$, and write $\theta_\delta=\theta_{\feq+\delta}$ for
$0<\delta<f_1-\feq$. Continuity of $(\Lambda')^{-1}$ gives $\theta_\delta\to0$, and by
item~(b) $\delta=\Lambda'(\theta_\delta)-\feq=6c\,\theta_\delta^5(1+o(1))$, so
$\theta_\delta=(\delta/(6c))^{1/5}(1+o(1))$. The infimum that defines $\Sigma(\feq+\delta)$ is
attained at $\theta_\delta<1$, so
\[
 \Sigma(\feq+\delta)=\bigl\{\Lambda(\theta_\delta)-\theta_\delta\feq\bigr\}-\delta\,\theta_\delta
 =c\,\theta_\delta^6(1+o(1))-6c\,\theta_\delta^6(1+o(1))
 =-5c\Bigl(\frac{\delta}{6c}\Bigr)^{6/5}(1+o(1)),
\]
and $5c(6c)^{-6/5}=\frac56(6c)^{-1/5}$. The leading coefficients $c$ and $6c$ do not
cancel, so the relative errors combine as written. Since $\theta_f\to0$ as
$f\downarrow\feq$, the asymptotics of $q_f=q_{\theta_f}$ follow from item~(a). Finally,
put
\begin{equation}\label{eq:fifth-delta0}
 \bar\delta=\Lambda'(\theta_0)-\feq .
\end{equation}
Then $0<\bar\delta<f_1-\feq$, because $\Lambda'$ is strictly increasing and $\theta_0<1$.
If $0<\delta<\bar\delta$, then $\theta_\delta<\theta_0$, and item~(a) applies to
$\zeta_{\theta_\delta}$: the first positive contact at the level $\feq+\delta$ is
marginal.
\end{proof}

\begin{proof}[Proof of Theorem~\ref{thm:local}\textup{(d)}]
Under the hypotheses of Theorem~\ref{thm:local}, $\zetas(\{0\})=0$, so
$\zetas\ne\delta_0$. By item~(c), $\Lambda$ is continuously differentiable and strictly
convex on $(0,1)$, and every level $f\in(\feq,f_1)$ has the dual minimizer
$\theta_f=(\Lambda')^{-1}(f)\in(0,1)$. With \condFM, Lemma~\ref{lem:cons-gartner-ellis}
therefore gives
\[
 \lim_{N\to\infty}\frac1N\log\Prob\bigl(N^{-1}\log\ZN\ge f\bigr)=\Sigma(f)
 \qquad(\feq<f<f_1).
\]
Since $\feq<f_1$, the level $\feq+\delta$ lies in $(\feq,f_1)$ for
$0<\delta<f_1-\feq$, and then the inner limit in the second display of item~(d) equals
$-\Sigma(\feq+\delta)$. By item~(c), $\delta^{-6/5}\{-\Sigma(\feq+\delta)\}\to
\frac56(6c)^{-1/5}$ as $\delta\downarrow0$.
\end{proof}

The limit $N\to\infty$ is taken first, at fixed $\delta$: item~(d) says nothing at finite $N$
or about deviations $\delta_N$ that shrink with $N$, and its $o(1)$ terms carry no rate.
Under~\condFM, the identification of $-\Sigma$ as the upper-tail rate function is not new: it
follows from Talagrand's large-deviation principle \cite[Remark~1.3]{Bou26b}. Boursier's direct
lower bound for $\log\ZN$, proved for $\xi$ strictly convex on $[-1,1]$
\cite[Assumption~1.1]{Bou26b}, at regular levels under a strict Plefka condition
\cite[Remark~1.4 and Proposition~3.12]{Bou26b} does not apply in $(\feq,\feq+\bar\delta)$
(Corollary~\ref{cor:fifth-plefka}), so near $\feq$ the lower bound in item~(d) rests on~\condFM.
What Theorem~\ref{thm:local} adds is the behaviour of $\Sigma$ near $\feq$: the power $6/5$ and
its constant. For $\xi$ convex on $[-1,1]$, the upper bounds on the probabilities in item~(d)
hold without~\condFM\ (Remark~\ref{rem:mo-fm}).

\subsection{The Sherrington--Kirkpatrick model: proof of Theorem~\ref*{thm:sk}}\label{sec:fifth-sk}

\begin{proof}[Proof of Theorem~\ref{thm:sk}]
The covariance $\xi(t)=\beta^2t^2/2$ is a polynomial with $\beta_2^2=\beta^2/2>0$ and
$\xi'''=\xi''''=0$. Thus $\azero=\beta^2$ and $\Dstar=\beta^6\Wstar^2$. Hypothesis~\condH\
holds for every $\beta>1$ by Lopatto's theorem \cite[Proposition~5.1]{Lop26a}: his Parisi
equation and functional \cite[(1.2)--(1.3)]{Lop26a} are those of this covariance, with a
functional larger than ours by the constant $\log2$, so his Parisi measure $\mu_\beta$ is
$\zetas$. The same statement, for the same equation and functional, has an alternative
computer-assisted proof \cite[Theorem~5.1]{Pei26} (Remark~\ref{rem:fifth-sk-inputs}).

Theorem~\ref{thm:local} therefore applies. Its first assertions show that $\alphas$ is
continuously differentiable near zero, with
$\rhob=\alphas'(0)=\rhostar=\Dstar/(2\beta)=\beta^5\Wstar^2/2$, which is positive because
$\Wstar<0$ (Lemma~\ref{lem:plat-fourth-sign}). The constant of Theorem~\ref{thm:local} is
\[
 c=\frac9{640}\,\azero^{3/2}\rhostar^{-3}=\frac9{640}\,\beta^3\rhob^{-3}=\cb
 =\frac{9}{80\,\beta^{12}\Wstar^6}.
\]
Parts (a) and~(b) of Theorem~\ref{thm:sk} are items (a)--(c) of
Theorem~\ref{thm:local} with $c=\cb$ and $3/(2\rhostar)=3/(2\rhob)$.

For part~(c), \condFM\ holds for this covariance by Lemma~\ref{lem:A-fm-sk}, and
Theorem~\ref{thm:local}(d) gives both limits for $\ZN$. Here $\ZN$ is the partition
function of the field with covariance $N\xi(N^{-1}\sigma\cdot\sigma')$, whose
interaction includes the diagonal terms. Lemma~\ref{lem:cons-original-sk} transfers the
first limit, at every $f\in(\feq,f_1)$, to $N^{-1}\FNb-\log2$. The second limit follows,
because its inner limits are again $-\Sigma(\feq+\delta)$.
\end{proof}

\begin{remark}[What Theorem~\ref*{thm:sk} rests on]\label{rem:fifth-sk-inputs}
Besides the present paper, the proof uses refereed results on the Parisi functional and its
equation \cite{AC15a,AC15b,Che17} and the following inputs.
\begin{steplist}
\item \emph{Condition~\condH.} Lopatto's theorem \cite[Proposition~5.1]{Lop26a} is proved in a
 preprint, and so is the alternative proof \cite[Theorem~5.1]{Pei26}, which uses nothing from
 \cite[Sections~3--5]{Lop26a} and whose single computer-assisted step is one polynomial
 inequality, verified in exact rational arithmetic \cite[Section~5 and Appendix~A.5]{Pei26}.
 No other property of $\zetas$ is taken from these papers: the interval of support near
 zero, $\Ustar=1/\beta$ and the absence of an atom at zero are derived from~\condH\ in
 Proposition~\ref{prop:ons-equilibrium}. Thus, like the results of \cite{AL26} at every
 $\beta>1$, which use Lopatto's interval-support theorem \cite[Theorem~1.1]{Lop26a}, the
 statement at every $\beta>1$ rests on a preprint. A third proof of~\condH\ for this
 covariance is Corollary~\ref{cor:skg-interval}; it rests on the finite-step inequality of the
 unrefereed preprint \cite{OAI26c} (Proposition~\ref{prop:skg-oai}).
\item \emph{Condition~\condFM.} It enters only part~(c). For this covariance it is the one-level
 case of the multiscale formula of Contucci and Mingione \cite[Theorem~2.1]{CM19}, a refereed
 result (Lemma~\ref{lem:A-fm-sk}); other sources are Talagrand's argument for even covariances
 \cite[Section~7]{Tal07} and the preprints \cite[Theorem~2.1]{CGKLM} and, quantitatively for
 $\theta\le\frac12$, \cite[Theorem~1.1]{AL26}. Its upper half is Guerra's bound for fractional
 moments \cite[Section~6]{Tal07} (Proposition~\ref{prop:A-guerra}).
\item \emph{Temperatures close to critical.} If $1<\beta^2<1+\varepsilon_*$, with
 $\varepsilon_*$ as in Theorem~\ref{thm:nearcrit}(a), then $\xi\in\Kcl(\varepsilon_*)$, because
 $\normCfour{\xi-\xic}=\beta^2-1$, and Theorem~\ref{thm:nearcrit}(a) proves~\condH\ within this
 paper, from refereed results only (Proposition~\ref{prop:nc-class}). For $\beta$ close to $1$,
 \condH\ also follows from Zhou's theorem \cite[Theorem~1]{Zho26}, which is stated for the
 covariance $\beta_Z^2t^2$ with $\frac1{\sqrt2}<\beta_Z\le\frac1{\sqrt2}+\eta$ for some
 $\eta>0$, that is, for $\beta=\sqrt2\beta_Z\in(1,1+\sqrt2\eta]$; we use only its inclusion
 $[0,\upsilon_\beta]\subset\supp\zetas$. For these $\beta$, Theorem~\ref{thm:sk} rests on
 refereed results and the present paper alone.
\end{steplist}
For this model, Proposition~\ref{prop:ons-equilibrium}(b) is
\cite[Lemmas~6.1 and~6.2]{Lop26a}, and by \cite[Theorem~1.1]{Lop26a} $\zetas$ has a smooth
density on $[0,q_\beta)$, where $q_\beta$ is the top of its support, whose value at zero is
therefore $\rhob$; this is not used here. The value and the positivity of $\rhob$, in physics
form \cite[(B9)--(B10)]{PR09} of Crisanti and Rizzo~\cite{CR02}, also follow from
\cite[Theorem~B.1 and (B.4)]{Lop26b}, which uses the interval support of $\zetas$, and from
\cite[(7.8), (7.10) and (7.12)]{AL26} on differentiation at zero and
\cite[Lemma~7.1, (7.7), and Lemma~7.2]{AL26} (Section~\ref{sec:intro-results}). The formula
$\alphas(t)=\rhob\,t+o(t)$ with $\rhob=\beta^5\Wstar^2/2>0$ is also
\cite[Lemma~5.8 and its proof]{OAI26a}, an unrefereed preprint, which takes the interval support from
\cite[Theorem~1.1]{Lop26a} and writes the slope as $TV_{zzzz}(0,0)^2/(2V_{zz}(0,0)^3)$ with
$T=\beta^2$, in a notation where $V$ solves our Parisi equation; with
$V_{zz}(0,0)=\Ustar=1/\beta$ and $V_{zzzz}(0,0)=\Wstar$ this is $\beta^5\Wstar^2/2$.
\end{remark}

\begin{remark}[Comparison with Aronow and Lopatto]\label{rem:fifth-al}
For every fixed $\beta>1$, Aronow and Lopatto prove
$e^{-C_1Nx^{6/5}}\le\Prob(\FNb\ge N(P_\beta+x))\le e^{-C_2Nx^{6/5}}$, where
$P_\beta=\feq+\log2$, at finite $N\ge N_0(\beta)$ and for $A_\beta N^{-2/3}\le x\le x_\beta$,
with positive constants depending on $\beta$ \cite[Theorem~1.5]{AL26} (uniformly in $\beta\ge2$
for $\FNb/\beta$, with a ground-state counterpart \cite[Theorem~1.6 and (1.14)]{AL26}). By
Theorem~\ref{thm:sk}(c), $x^{-6/5}\lim_{N}N^{-1}\log\Prob(\FNb\ge N(P_\beta+x))\to
-\frac56(6\cb)^{-1/5}$ as $x\downarrow0$. They also prove that the constrained minimum increases
by an amount of order $\lambda^5$ \cite[Lemma~7.2 and Proposition~7.3]{AL26}; in their rescaled
variables \cite[Section~7]{AL26}, $E_\beta(\lambda)=\beta^{-1}P_\beta(\lambda/\beta)$, where
$P_\beta(s)=p(s)+\log2$, and $u_\beta(q,x)=\beta^{-1}\Phi_{\zetas}(q,\beta x)$ up to an additive
constant, so that $u_{\beta,xxxx}(0,0)=\beta^3\Wstar$; Theorem~\ref{thm:sk}(a) reads
\[
 \lim_{\lambda\downarrow0}\lambda^{-5}\bigl(E_\beta(\lambda)-E_\beta(0)\bigr)
 =\cb\,\beta^{-6}=\frac9{80}\,u_{\beta,xxxx}(0,0)^{-6}.
\]
Section~\ref{sec:intro-results} compares the scope of the two results.
\end{remark}

\subsection{Strict Plefka stability fails near equilibrium}\label{sec:fifth-plefka}

Boursier's large-deviation theorem for the existence of TAP states \cite[Theorem~1]{Bou26b} is
proved, for $\xi$ strictly convex on $[-1,1]$ \cite[Assumption~1.1]{Bou26b}, at regular levels
$f\in(\feq,f_1)$, those with a unique dual minimizer
\cite[Definition~1.1]{Bou26b}, under the strict Plefka condition \cite[(1.8)]{Bou26b}
\begin{equation}\label{eq:fifth-strict-plefka}
 1-\xi''(q_f)\,\E\Phi_{\zeta_f,xx}(q_f,X_{q_f})^2>0 ,
\end{equation}
at the first positive contact $q_f$ of $\zeta_f$, with $X$ its
diffusion~\eqref{eq:intro-diffusion}. The left side equals $1-\Gamma_{\zeta_f}'(q_f)\ge0$
(see~\eqref{eq:intro-plefka}), so the condition asks that $q_f$ not be marginal
\cite[Remark~1.2]{Bou26b}.

\begin{corollary}\label{cor:fifth-plefka}
Assume the hypotheses of Theorem~\ref{thm:local}; by the proof of
Theorem~\ref{thm:sk}, they hold for $\xi(t)=\beta^2t^2/2$ with $\beta>1$. Let $\bar\delta$
be as in \eqref{eq:fifth-delta0}.
Every level $f\in(\feq,\feq+\bar\delta)$ is regular, and $\Gamma_{\zeta_f}(t)=t$ on an
interval $[q_f,s_f]$ with $s_f>q_f$. In particular
\[
 1-\xi''(q_f)\,\E\Phi_{\zeta_f,xx}(q_f,X_{q_f})^2=0,
\]
so the strict Plefka condition \cite[(1.8)]{Bou26b} fails at every level in
$(\feq,\feq+\bar\delta)$.
\end{corollary}

\begin{proof}
By Theorem~\ref{thm:local}(c), every $f\in(\feq,f_1)$ has the unique dual minimizer
$\theta_f=(\Lambda')^{-1}(f)$, so it is regular. If $f<\feq+\bar\delta=\Lambda'(\theta_0)$,
then $\theta_f<\theta_0$, because $\Lambda'$ is increasing. Theorem~\ref{thm:local}(a)
at $\theta=\theta_f$ gives $\Gamma_{\zeta_f}=\id$ on $[q_f,s_f]$ with $s_f>q_f$. Hence
$\Gamma'_{\zeta_f}(q_f)=1$, and the It\^o identity
$\Gamma_\zeta'(t)=\xi''(t)\E\Phi_{\zeta,xx}(t,X_t)^2$ (Lemma~\ref{lem:plat-ito}) gives the
displayed equality.
\end{proof}

The Plefka factor $1-\Gamma'_{\zeta_f}$ in fact vanishes on the whole interval $[q_f,s_f]$. So
the levels that carry the six-fifths law, all levels in $(\feq,\feq+\bar\delta)$, lie outside
the strict-Plefka hypothesis of \cite[Theorem~1]{Bou26b}; by Theorem~\ref{thm:sk}(b), for the
Sherrington--Kirkpatrick model this holds at every $\beta>1$. We state no result on TAP states
here. On the near-critical class $\Kcl(\varepsilon_*)$ of Theorem~\ref{thm:nearcrit}(b), which
contains the Sherrington--Kirkpatrick covariance for $1<\beta^2<1+\varepsilon_*$, the picture is
complete: a level has a marginal first positive contact exactly when $f\le f_c$, so the strict
Plefka condition fails exactly on $(\feq,f_c]$ and holds on $(f_c,f_1)$
(Corollary~\ref{cor:nc-dat-levels}). For the Sherrington--Kirkpatrick model the same holds at
every $\beta>1$, with the level $f_c(\beta)$ of Theorem~\ref{thm:sk-global}(e), which rests on an
unrefereed input.

\section{The constrained problem and the envelope identity}\label{sec:cons}

This section proves the envelope identity, Proposition~\ref{prop:cons-envelope}, and collects
the convex-duality facts used in Section~\ref{sec:fifth}: for $\theta$ above the equilibrium atom
$m_0=\zetas(\{0\})$, the minimizer $\zeta_\theta$ carries exactly the mass $\theta$ at zero and
no mass in $(0,q_\theta)$, below its first positive contact $q_\theta$, where
$\Gamma_\theta(q_\theta)=q_\theta$ and $\Gamma_\theta'(q_\theta)\le1$
(Lemma~\ref{lem:cons-variational}), and each level $f\in(\feq,f_1)$ has a dual parameter
$\theta_f$ (Lemmas~\ref{lem:cons-variational}(d) and~\ref{lem:cons-regular}(d),
Corollary~\ref{cor:cons-level}).
Section~\ref{sec:cons-standard} recalls standard facts, and most of
Section~\ref{sec:cons-variational} follows Boursier~\cite[Lemma~3.2]{Bou26b}. We write the
arguments out, in Appendix~\ref{sec:an-cons}, because \cite{Bou26b} assumes $\xi$ strictly convex
on $[-1,1]$, while later sections use them for every covariance of the class below.

\paragraph{Standing assumption and notation.}
Throughout this section $\xi(t)=\sum_{p\ge2}\beta_p^2t^p$ is not identically zero and has
radius of convergence greater than one; then $\xi''$ is positive on $(0,1]$ and nondecreasing on
$[0,1]$, so that $\xi''\le\xi''(1)$ there. No convexity of $\xi$ on $[-1,0]$, no condition on
$\xi'''(0)$ and no hypothesis on $\supp\zetas$ is assumed. We use the objects
\eqref{eq:intro-pde}--\eqref{eq:intro-obstacle}, write $\alpha_\zeta(t)=\zeta([0,t])$, and write
$\zeta_n\Rightarrow\zeta$ for weak convergence on $\Pspace([0,1])$, which
$\Wone(\mu,\nu)=\int_0^1|\alpha_\mu-\alpha_\nu|\,dt$ metrizes. We abbreviate
$\Gamma_\theta=\Gamma_{\zeta_\theta}$, $\Obs_\theta=\Obs_{\zeta_\theta}$, and similarly
$\Gamma_f,\Obs_f$ for $\zeta_f$ and $\Gamma_*,\Obs_*$ for $\zetas$.

\subsection{Standard properties of the Parisi functional}\label{sec:cons-standard}

The facts below concern only the scalar problem; no spin system appears in this section.

\begin{lemma}[Standard properties]\label{lem:cons-standard}
Let $\zeta\in\Pspace([0,1])$.
\begin{thmlist}
\item \emph{Regularity.} For every $j\ge0$ the derivative $\partial_x^j\Phi_\zeta$ exists and
 is continuous on $[0,1]\times\R$. The function $\Phi_\zeta(t,\cdot)$ is even, and
 $|\Phi_{\zeta,x}|\le1$, $0<\Phi_{\zeta,xx}\le1$. If $\zeta_n\Rightarrow\zeta$, then
 $\partial_x^j\Phi_{\zeta_n}\to\partial_x^j\Phi_\zeta$ uniformly on $[0,1]\times\R$ for every
 $j\ge0$.
\item \emph{The function $\Gamma_\zeta$.} $\Gamma_\zeta$ is continuously differentiable on
 $[0,1]$, with
 \[
  \Gamma_\zeta'(t)=\xi''(t)\,\E\Phi_{\zeta,xx}(t,X_t)^2,\qquad
  \Gamma_\zeta(0)=0,\qquad 0\le\Gamma_\zeta\le1,\qquad \Gamma_\zeta(1)<1 .
 \]
 If $\zeta_n\Rightarrow\zeta$, then $\Gamma_{\zeta_n}\to\Gamma_\zeta$,
 $\Gamma_{\zeta_n}'\to\Gamma_\zeta'$ and $\Obs_{\zeta_n}\to\Obs_\zeta$ uniformly on $[0,1]$.
\item \emph{Convexity.} The map $\zeta\mapsto P_\zeta$ is strictly convex on $\Pspace([0,1])$
 and has a unique minimizer $\zetas$.
\item \emph{First variation.} For $\nu\in\Pspace([0,1])$ and $s\in[0,1]$ put
 $\zeta_s=(1-s)\zeta+s\nu$. The function $s\mapsto P_{\zeta_s}$ is convex and continuously
 differentiable on $[0,1]$, with one-sided derivatives at the endpoints, and
 \begin{equation}\label{eq:cons-first-variation}
  \frac{d}{ds}P_{\zeta_s}=\int\Obs_{\zeta_s}\,d(\nu-\zeta)
  =\frac12\int_0^1\xi''(t)\bigl\{\alpha_\nu(t)-\alpha_\zeta(t)\bigr\}
   \bigl\{\Gamma_{\zeta_s}(t)-t\bigr\}\,dt .
 \end{equation}
\item \emph{Lipschitz bound.} For $\mu,\nu\in\Pspace([0,1])$,
 \begin{equation}\label{eq:cons-lipschitz}
  |P_\mu-P_\nu|\le\tfrac12\,\xi''(1)\,\Wone(\mu,\nu).
 \end{equation}
 In particular $\zeta\mapsto P_\zeta$ is weakly continuous.
\item \emph{The obstacle.} $\Obs_\zeta\in C^2[0,1]$, $\Obs_\zeta(1)=0$, and
 \begin{equation}\label{eq:cons-obstacle}
  \Obs_\zeta'=-\tfrac12\,\xi''\,(\Gamma_\zeta-\id),\qquad
  \Obs_\zeta''=-\tfrac12\,\xi'''\,(\Gamma_\zeta-\id)-\tfrac12\,\xi''\,(\Gamma_\zeta'-1),
  \qquad \Obs_\zeta'(1)>0 .
 \end{equation}
 Consequently $1$ is not a minimum point of $\Obs_\zeta$ on $(0,1]$, and every minimum point
 $t\in(0,1)$ of $\Obs_\zeta$ on $(0,1]$ satisfies $\Gamma_\zeta(t)=t$ and
 $\Gamma_\zeta'(t)\le1$.
\item \emph{Optimality of $\zetas$.} Put $o_*=\int\Obs_*\,d\zetas$. Then
 \begin{equation}\label{eq:cons-optimality}
  \Obs_*\ge o_*\ \text{ on }[0,1],\qquad \Obs_*=o_*\ \text{ on }\supp\zetas .
 \end{equation}
 Conversely, if $\int\Obs_\zeta\,d(\nu-\zeta)\ge0$ for every $\nu\in\Pspace([0,1])$, then
 $\zeta=\zetas$.
\end{thmlist}
\end{lemma}

The proof, which collects results of~\cite{AC15a,AC15b,Che17,JT17}, is in
Appendix~\ref{sec:an-cons}. The papers~\cite{AC15a,AC15b,Che17} assume a summability condition
such as $\sum_p2^p\beta_p^2<\infty$, which our radius assumption need not imply; it serves there
only to define the Hamiltonian, and the cited statements involve only finitely many derivatives
of $\xi$ on $[0,1]$, which are bounded here. Auffinger and Chen's statements along a time-changed
Brownian motion hold along the diffusion \eqref{eq:intro-diffusion} by Girsanov's theorem (see
the analytic facts at the beginning of Section~\ref{sec:plat}).

\subsection{The constrained problem and its first contact}\label{sec:cons-variational}

For $0\le\theta<1$ let
$\mathcal A_\theta=\{\zeta\in\Pspace([0,1]):\zeta(\{0\})\ge\theta\}$. Every
$\zeta\in\mathcal A_\theta$ is uniquely of the form $\zeta=\theta\delta_0+(1-\theta)\mu$ with
$\mu\in\Pspace([0,1])$, namely $\mu=(\zeta-\theta\delta_0)/(1-\theta)$, and every measure of
this form lies in $\mathcal A_\theta$. Thus
\[
 \Lambda(\theta)=\theta p(\theta),\qquad
 p(\theta)=\inf_{\zeta\in\mathcal A_\theta}P_\zeta
 =\inf_{\mu\in\Pspace([0,1])}P_{\theta\delta_0+(1-\theta)\mu}.
\]
As in~\cite[Section~1.3]{Bou26b}, the closed extension $\Lambar$ of $\Lambda$ equals $\Lambda$ on
$[0,1)$, takes the value $\liminf_{\theta\uparrow1}\Lambda(\theta)$ at $1$, and is $+\infty$
outside $[0,1]$; then $\Sigma$ and $f_1=\Lambar'_-(1)$ are given by \eqref{eq:intro-dual}. Put
$m_0=\zetas(\{0\})$ and $\Lamhat(\theta)=\Lambar(\theta)-\theta\feq$ for $0\le\theta\le1$.

\begin{lemma}[The constrained problem]\label{lem:cons-variational}
The following hold.
\begin{thmlist}
\item For every $\theta\in[0,1)$, $P$ has a unique minimizer $\zeta_\theta$ on
 $\mathcal A_\theta$, and $\zeta_\theta=\zetas$ if $\theta\le m_0$. The function
 $p(\theta)=P_{\zeta_\theta}$ is continuous and nondecreasing on $[0,1)$, with
 $\feq\le p(\theta)\le\feq+\tfrac12\,\xi''(1)\,\theta$, and $\theta\mapsto\zeta_\theta$ is weakly
 continuous on $[0,1)$.
\item $\Lambar$ is closed and convex, $\Lambar'_+(0)=\feq$,
 $\Lambar(1)=\lim_{\theta\uparrow1}\Lambda(\theta)=P_{\delta_0}$, and $f_1<\infty$.
\item \emph{Saturation.} If $m_0<\theta<1$, then $\zeta_\theta(\{0\})=\theta$. The function $p$ is
 strictly increasing on $[m_0,1)$, and $\Lambda$ is strictly convex on $[m_0,1)$.
\item \emph{Dual minimizer.} For every $f\in(\feq,f_1)$, the function
 $\theta\mapsto\Lambar(\theta)-\theta f$ has a unique minimizer $\theta_f$ on $[0,1]$, and
 $\theta_f\in(0,1)\cap[m_0,1)$.
\item \emph{First contact.} Let $m_0<\theta<1$, write
 $\zeta_\theta=\theta\delta_0+(1-\theta)\mu_\theta$, and put
 $o_\theta=\int\Obs_\theta\,d\mu_\theta$. Then $\Obs_\theta\ge o_\theta$ on $[0,1]$,
 $\Obs_\theta=o_\theta$ on $\supp\mu_\theta$, and $\Obs_\theta(0)>o_\theta$. The positive
 contact set
 \[
  C_\theta=\Bigl\{t\in(0,1]:\Obs_\theta(t)=\min_{(0,1]}\Obs_\theta\Bigr\}
 \]
 is a nonempty compact subset of $(0,1)$ that contains $\supp\zeta_\theta\cap(0,1]$. Its
 smallest point $q_\theta$ satisfies
 \begin{equation}\label{eq:cons-contact}
  \zeta_\theta((0,q_\theta))=0,\qquad \Gamma_\theta(q_\theta)=q_\theta,\qquad
  \Gamma_\theta'(q_\theta)=\xi''(q_\theta)\,
  \E\Phi_{\zeta_\theta,xx}(q_\theta,X_{q_\theta})^2\le1,
 \end{equation}
 where $X$ is the diffusion of $\zeta_\theta$, and $\Obs_\theta(r)>\Obs_\theta(q_\theta)$ for
 $0\le r<q_\theta$.
\end{thmlist}
\end{lemma}

We call $q_\theta$ the \emph{first positive contact} of $\zeta_\theta$. The inequality in
\eqref{eq:cons-contact} is the Plefka inequality \eqref{eq:intro-plefka}; the contact is
\emph{marginal} if $\Gamma_\theta'(q_\theta)=1$ and \emph{strictly Plefka stable} if
$\Gamma_\theta'(q_\theta)<1$, which is the strict Plefka condition~\cite[(1.8)]{Bou26b}. By
part~(d), every level $f\in(\feq,f_1)$ is regular in the sense of~\cite[Definition~1.1]{Bou26b}.
We write $\zeta_f=\zeta_{\theta_f}$; since $\theta_f>m_0$ (Lemma~\ref{lem:cons-regular}(d)
below), its first positive contact $q_f=q_{\theta_f}$ is defined, and the first positive contact
at the level~$f$ is called marginal if $q_f$ is.

For $\xi$ strictly convex on $[-1,1]$ \cite[Assumption~1.1]{Bou26b}, part (b), the saturation in
(c), the interiority of $\theta_f$ in (d), and part (e) without the inequality
$\Gamma_\theta'(q_\theta)\le1$ are contained in~\cite[Lemma~3.2]{Bou26b} and its proof, which we
follow; its part (3) covers every $\theta\in(m_0,1)$, since by Corollary~\ref{cor:cons-level}
every such $\theta$ is a dual minimizer. That inequality at dual minimizers is noted
in~\cite[Remark~1.2]{Bou26b}, and the uniqueness of $\zeta_\theta$ in~\cite[Section~1.3]{Bou26b}.
The continuity in (a), the strict monotonicity and strict convexity in (c) and the uniqueness of
$\theta_f$ in (d) are needed below; Section~\ref{sec:ons} uses (e) as $\theta\downarrow0$. The
proof, in Appendix~\ref{sec:an-cons}, uses compactness of $\mathcal A_\theta$, strict convexity
of $P$ (Lemma~\ref{lem:cons-standard}(c),(g)), Lemma~\ref{lem:cons-standard}(f),
\eqref{eq:cons-lipschitz} and \eqref{eq:cons-first-variation}, and not the strict convexity of
$\xi$ on $[-1,1]$ assumed in~\cite{Bou26b}.

\subsection{Proof of the envelope identity}\label{sec:cons-envelope}

We prove Proposition~\ref{prop:cons-envelope}. Since $\mu_\theta$ is the unique minimizer of
$\mu\mapsto P_{\theta\delta_0+(1-\theta)\mu}$, the envelope principle suggests that $p'(\theta)$
is the partial derivative in $\theta$ at the fixed measure $\mu_\theta$, which the
first-variation formula \eqref{eq:cons-first-variation} evaluates through the obstacle; a
squeeze between the old and the new minimizer makes this precise at every $\theta$.

\begin{proof}[Proof of Proposition~\ref{prop:cons-envelope}]
For $u\in[0,1)$ and $\mu\in\Pspace([0,1])$ put $\pi(u,\mu)=P_{u\delta_0+(1-u)\mu}$. Then
$p(u)=\min_\mu\pi(u,\mu)$, with the unique minimizer $\mu_u=(\zeta_u-u\delta_0)/(1-u)$, and
$u\mapsto\mu_u$ is weakly continuous by Lemma~\ref{lem:cons-variational}(a). For fixed $\mu$,
$u\mapsto u\delta_0+(1-u)\mu$ is a segment, and \eqref{eq:cons-first-variation} gives
\[
 \partial_u\pi(u,\mu)=\int\Obs_{u\delta_0+(1-u)\mu}\,d(\delta_0-\mu)\qquad(0<u<1).
\]
This is jointly continuous in $(u,\mu)$. Indeed, if $u_n\to u$ and $\mu_n\Rightarrow\mu$, the
corresponding obstacles converge uniformly by Lemma~\ref{lem:cons-standard}(b), and
$\int\Obs\,d\mu_n\to\int\Obs\,d\mu$ for every continuous $\Obs$.

Fix $\theta\in(0,1)$ and a small $h\ne0$. Since $p(\theta)=\pi(\theta,\mu_\theta)
\le\pi(\theta,\mu_{\theta+h})$ and $p(\theta+h)=\pi(\theta+h,\mu_{\theta+h})
\le\pi(\theta+h,\mu_\theta)$,
\[
 \pi(\theta+h,\mu_{\theta+h})-\pi(\theta,\mu_{\theta+h})
 \le p(\theta+h)-p(\theta)
 \le\pi(\theta+h,\mu_\theta)-\pi(\theta,\mu_\theta).
\]
By the mean value theorem the outer terms equal $h\,\partial_u\pi(u',\mu_{\theta+h})$ and
$h\,\partial_u\pi(u'',\mu_\theta)$ for some $u',u''$ between $\theta$ and $\theta+h$. Divide by
$h$; for $h<0$ this reverses both inequalities. As $h\to0$, joint continuity and
$\mu_{\theta+h}\Rightarrow\mu_\theta$ show that both bounds converge to
$\partial_u\pi(\theta,\mu_\theta)$. Hence $p$ is differentiable at $\theta$, and
$p'(\theta)=\partial_u\pi(\theta,\mu_\theta)$, which is \eqref{eq:cons-envelope-general}
because $\theta\delta_0+(1-\theta)\mu_\theta=\zeta_\theta$. The same joint continuity makes
$p'$ continuous.

By Lemma~\ref{lem:cons-variational}(a), $p$ is constant on $[0,m_0]$; since $p$ is
differentiable, $p'=0$ on $(0,1)\cap(0,m_0]$. Let $m_0<\theta<1$. By
Lemma~\ref{lem:cons-variational}(e), $\Obs_\theta$ equals $o_\theta=\Obs_\theta(q_\theta)$ on
$\supp\mu_\theta$, so $\int\Obs_\theta\,d\mu_\theta=\Obs_\theta(q_\theta)$, and
$\Obs_\theta(0)>\Obs_\theta(q_\theta)$. The integral form follows from
\eqref{eq:intro-obstacle}. The last assertion follows from $\Lambda(\theta)=\theta p(\theta)$.
\end{proof}

Three features of \eqref{eq:cons-envelope} are used later. First, it holds at every
$\theta\in(m_0,1)$ (by Corollary~\ref{cor:cons-level}, exactly the dual minimizers), and $p'$ is
continuous and nonnegative; as $p$
is continuous on $[0,1)$ with $p(m_0)=\feq$, this gives
$p(\theta)-\feq=\int_{m_0}^\theta p'(u)\,du$ for $m_0<\theta<1$, which is how
Section~\ref{sec:fifth} turns asymptotics of $p'$ into the fifth-order law. Second, it uses
$\Gamma_\theta$ only on $[0,q_\theta]$, the closure of the first plateau $(0,q_\theta)$: on
$[0,q_\theta)$ the distribution
function of $\zeta_\theta$ equals $\theta$ by \eqref{eq:cons-contact}, and $\Gamma_\theta-\id$
vanishes at both ends, since $\Gamma_\theta(0)=0$ by Lemma~\ref{lem:cons-standard}(b) and
$\Gamma_\theta(q_\theta)=q_\theta$. Section~\ref{sec:plat} computes its profile in between
(Figure~\ref{fig:plateau}(b),(c)). Third, at a dual minimizer it yields the level identity of
Corollary~\ref{cor:cons-level}.

\subsection{Regular levels and the level identity}\label{sec:cons-regular}

The next lemma records how the atom $m_0$ enters the dual problem; it is the bookkeeping needed
in Section~\ref{sec:atom}, where $m_0$ may be positive.

\begin{lemma}[Regular levels]\label{lem:cons-regular}
The following hold.
\begin{thmlist}
\item $\Lamhat\ge0$ on $[0,1]$, $\Lamhat=0$ on $[0,m_0]$, and $\Lamhat>0$ on $(m_0,1]$.
\item If $m_0<\theta<1$, then $\zeta_\theta(\{0\})=\theta$.
\item $\Lambar$ is affine on no nondegenerate subinterval of $[m_0,1]$.
\item If $m_0<1$, equivalently $\zetas\ne\delta_0$, then $\feq<f_1<\infty$, and for every
 $f\in(\feq,f_1)$ the dual minimizer $\theta_f$ of Lemma~\ref{lem:cons-variational}(d) lies
 in $(m_0,1)$. If $m_0=1$, then $f_1=\feq$.
\end{thmlist}
\end{lemma}

When $m_0=0$, parts (a)--(d) reduce to $\Lamhat>0$ on $(0,1]$, $\theta_f\in(0,1)$ and
$\feq<f_1$. The next corollary inverts the dual relation.

\begin{corollary}[Levels and the level identity]\label{cor:cons-level}
Assume $\zetas\ne\delta_0$. Then $\Lambda'$ is continuous and strictly increasing on $(m_0,1)$,
with $\Lambda'(\theta)\to\feq$ as $\theta\downarrow m_0$ and $\Lambda'(\theta)\to f_1$ as
$\theta\uparrow1$. Hence $\Lambda'$ maps $(m_0,1)$ homeomorphically onto $(\feq,f_1)$, and
every $f\in(\feq,f_1)$ has $\theta_f=(\Lambda')^{-1}(f)$ and
\begin{equation}\label{eq:cons-level}
 f=\Lambda'(\theta_f)=P_{\zeta_f}+\theta_f\bigl\{\Obs_f(0)-\Obs_f(q_f)\bigr\}.
\end{equation}
\end{corollary}

Both are proved in Appendix~\ref{sec:an-cons} from Lemma~\ref{lem:cons-variational}, the
uniqueness of $\zetas$ and Proposition~\ref{prop:cons-envelope}. The latter gives
$\Lambda'=p+\theta p'$ on $(0,1)$, hence $\Lambda'(m_0)=\feq$ when $0<m_0<1$, which excludes
$\theta_f=m_0$ in Lemma~\ref{lem:cons-regular}(d).

\section{The plateau calculus}\label{sec:plat}

Every main result concerns the bottom of a Parisi-type measure: the constrained minimizers
$\zeta_\theta$ as $\theta\downarrow0$, the minimizers $\zeta_h$ in a small field and, in
Theorem~\ref{thm:atom}, $\zetas$ itself when it has an atom at zero. In each case one must
locate the contacts near zero and describe how $\Gamma_\zeta(t)=\E\Phi_{\zeta,x}(t,X_t)^2$
approaches the identity there. This section supplies the calculus for both tasks (the bottom
band of Figure~\ref{fig:roadmap}); Section~\ref{sec:ons} uses it to prove
Propositions~\ref{prop:ons-equilibrium} and~\ref{prop:ons-onset} and
Lemma~\ref{lem:fifth-area}, and Sections~\ref{sec:fld}--\ref{sec:nc} use it again. Three
statements do the work: a \emph{universal sign} $\Phi_{\zeta,xxx}(t,x)\le-\kappa_tx$ for
$0\le x\le\frac12$, with $\kappa_t>0$ independent of $\zeta$, which gives $\Wstar<0$ and hence,
when $\beta_2>0$, $\Dstar>0$ (Section~\ref{sec:plat-sign}); \emph{gap exclusion}, which forbids a plateau between two
contacts near zero of a measure close to $\zetas$ and so makes the first positive contact of
$\zeta_\theta$ marginal (Section~\ref{sec:plat-gap}, Figure~\ref{fig:plateau}(a)); and the
\emph{universal cubic}: below a marginal contact $q$, $\Gamma_\zeta-\id$ is to leading order
$\frac{\Dstar}6\,t(q-t)^2$, whatever the covariance, and its end curvatures and area give the
onset $q_\theta$, the atom at $q_\theta$ and the fifth-order law (Section~\ref{sec:plat-cubic},
Figure~\ref{fig:plateau}(b),(c)). Sections~\ref{sec:plat-ito} and~\ref{sec:plat-obstacle}
record It\^o identities on plateaus, several of them due to Auffinger and
Chen~\cite{AC15a}, and the conditions at a contact; a reader willing to accept them can
start at Section~\ref{sec:plat-sign}.

\paragraph{Standing assumptions and notation.}
In this section $\xi(t)=\sum_{p\ge2}\beta_p^2t^p$ is not identically zero and has radius of
convergence greater than one, so that $\xi''>0$ on $(0,1]$. The bounds
\eqref{eq:plat-bounds}, Sections~\ref{sec:plat-ito} and~\ref{sec:plat-obstacle}, and
Lemma~\ref{lem:plat-fourth-sign} use nothing more. Lemmas~\ref{lem:plat-jets}
and~\ref{lem:plat-coupling} (Appendix~\ref{sec:plat-jets}) and
Sections~\ref{sec:plat-gap}--\ref{sec:plat-cubic} assume in addition $\beta_2>0$; the
condition $\xi'''(0)=0$ is assumed only where stated, in those two sections, and \condH\ is
not used in this section. We write $a(t)=\xi''(t)$ and $\alpha_\zeta(t)=\zeta([0,t])$, so
that $a$ is nondecreasing and $\azero\le a\le\xi''(1)$ on $[0,1]$, with $\azero>0$ when
$\beta_2>0$. Along the diffusion $X=X^\zeta$ of \eqref{eq:intro-diffusion} we abbreviate
$U_t=\Phi_{\zeta,xx}(t,X_t)$, $V_t=\Phi_{\zeta,xxx}(t,X_t)$ and
$W_t=\Phi_{\zeta,xxxx}(t,X_t)$, and $\Gamma_\zeta$, $\Obs_\zeta$ are as in
\eqref{eq:intro-obstacle}. Since $X_0=0$ and $\Phi_\zeta(t,\cdot)$ is even,
$U_0=\Phi_{\zeta,xx}(0,0)$ and $V_0=0$. The symbol $\Wone$ is always the transport distance
and $B$ the driving Brownian motion. A \emph{plateau} of $\zeta$ is an interval on which
$\alpha_\zeta$ is constant; the constant is its \emph{height}. An open interval $I$ is a
plateau if and only if $\zeta(I)=0$.

\paragraph{Analytic facts.}
For every $\zeta\in\Pspace([0,1])$ the function $\Phi_\zeta(t,\cdot)$ is even, every spatial
derivative $\partial_x^k\Phi_\zeta$ exists and is continuous on $[0,1]\times\R$, and
\begin{equation}\label{eq:plat-bounds}
 |\Phi_{\zeta,x}|\le1,\qquad 0<\Phi_{\zeta,xx}\le1,\qquad
 \sup_{[0,1]\times\R}|\partial_x^k\Phi_\zeta|\le B_k\quad(k\ge1),
\end{equation}
with $B_k$ depending only on $\xi$ and $k$: the first two bounds
are~\cite[Proposition~2, (15)--(16)]{AC15a}, and the third is~\cite[Proposition~5,
(66)]{AC15a} for measures with finitely many atoms and passes to every $\zeta$
by~\cite[Proposition~1(i)]{AC15a}, which also gives
$\partial_x^k\Phi_{\zeta_n}\to\partial_x^k\Phi_\zeta$ uniformly on $[0,1]\times\R$
($k\ge0$) whenever $\zeta_n\to\zeta$ weakly. None of this uses $\beta_2>0$, and the results
of~\cite{AC15a}, stated under a summability condition on the $\beta_p$, apply to our
covariances (Section~\ref{sec:cons-standard}). By Girsanov's theorem, the exponential
weight~\cite[(13)]{AC15a} with which Auffinger and Chen work along $B(\xi'(t))$ is the
density of the law of $(X_s)_{s\le t}$ with respect to that of $(B(\xi'(s)))_{s\le t}$. So
their statements hold along $X$: $\Gamma_\zeta$ is their $\Gamma_\mu$~\cite[(18)]{AC15a},
and expectations along $X$ of polynomials in finitely many spatial derivatives of
$\Phi_\zeta$ are continuous in $t$ and converge uniformly in $t$ along weakly convergent
sequences of measures~\cite[Proposition~1(ii)]{AC15a}. The quantitative forms we use (jet
stability~\eqref{eq:plat-jet-stability} in $\Wone$ and its covariance-varying
form~\eqref{eq:plat-jet-general}, the bounds~\eqref{eq:plat-time-lip},
\eqref{eq:plat-small-time}, \eqref{eq:plat-origin} in time, at small time and at the origin,
and the coupling Lemma~\ref{lem:plat-coupling}) are routine maximum-principle and Gronwall
estimates, proved in Appendix~\ref{sec:plat-jets}.

\subsection{It\^o identities on plateaus}\label{sec:plat-ito}

On a plateau the Parisi equation linearizes: if $\alpha_\zeta=m$ on an interval, then
$e^{m\Phi_\zeta}$ solves the backward heat equation $w_t+\frac a2w_{xx}=0$ there. Two
consequences are used throughout. The first is the law of the diffusion across a plateau.

\begin{lemma}[Plateau law]\label{lem:plat-law}
Let $0\le s<t\le1$ with $\zeta((s,t))=0$, and put $m=\zeta([0,s])$. Conditionally on
$X_s=x$, the law of $X_t$ has the density
$y\mapsto\exp\{m[\Phi_\zeta(t,y)-\Phi_\zeta(s,x)]\}\,\varphi_{\xi'(t)-\xi'(s)}(y-x)$,
where $\varphi_T$ is the $\mathcal N(0,T)$ density. In particular, if $0<q\le1$,
$\zeta((0,q))=0$ and $m=\zeta(\{0\})$, then
\begin{equation}\label{eq:plat-law}
 \frac{d\Law(X_q)}{d\mathcal N(0,\xi'(q))}(y)
 =\exp\bigl\{m\bigl[\Phi_\zeta(q,y)-\Phi_\zeta(0,0)\bigr]\bigr\} .
\end{equation}
\end{lemma}

This is the plateau Cole--Hopf identity of Boursier~\cite[Lemma~A.3, (A.11)]{Bou26b}, first stated
in~\cite[Lemma~B.3]{Bou26a};
\eqref{eq:plat-law} is its case with lower endpoint zero, and that law has a positive
density with Gaussian tails. The proof (Appendix~\ref{sec:plat-jets}) is a Doob transform:
on the plateau $e^{m\Phi_\zeta}$ is space-time harmonic for the martingale part of $X$.

The second consequence concerns the derivatives of $\Gamma_\zeta$. Write
$u=\Phi_{\zeta,x}$, $U=\Phi_{\zeta,xx}$ and $V=\Phi_{\zeta,xxx}$. On a plateau of height
$m$, and for a measure with finitely many atoms on each closed strip between consecutive
atoms, differentiating \eqref{eq:intro-pde} gives
\begin{equation}\label{eq:plat-third-pde}
 \begin{gathered}
  u_t+\tfrac a2u_{xx}+amuu_x=0,\qquad
  U_t+\tfrac a2U_{xx}+am(uU_x+U^2)=0,\\
  V_t+\tfrac a2V_{xx}+amuV_x+3amUV=0 .
 \end{gathered}
\end{equation}
Along $X$, It\^o's formula turns the last two equations into
$dU_t=-ma(t)U_t^2\,dt+\sqrt{a(t)}\,V_t\,dB_t$ and
$dV_t=-3ma(t)U_tV_t\,dt+\sqrt{a(t)}\,W_t\,dB_t$.

\begin{lemma}[It\^o identities]\label{lem:plat-ito}
Let $\zeta\in\Pspace([0,1])$ and $\Gamma=\Gamma_\zeta$.
\begin{thmlist}
\item $\Gamma\in C^1[0,1]$, $\Gamma(0)=0$ and
 $\Gamma'(t)=a(t)\E U_t^2$.
\item For $t\in[0,1)$ the right derivative of $\Gamma'$ at $t$ exists and equals
 $a'(t)\E U_t^2+a(t)^2\E\bigl(V_t^2-2\zeta([0,t])\,U_t^3\bigr)$; for $t\in(0,1]$ the left derivative exists and is given by the same expression with
 $\zeta([0,t))$ in place of $\zeta([0,t])$. Writing $\Gamma''(t+)$ and $\Gamma''(t-)$ for
 these one-sided derivatives,
 \begin{gather}
  \Gamma''(0+)=\xi'''(0)\,\Phi_{\zeta,xx}(0,0)^2
  -2\azero^2\,\zeta(\{0\})\,\Phi_{\zeta,xx}(0,0)^3,\label{eq:plat-zero}\\
  \Gamma''(t-)-\Gamma''(t+)=2a(t)^2\,\zeta(\{t\})\,\E U_t^3\qquad(0<t<1).\label{eq:plat-jump}
 \end{gather}
\item On an open plateau of height $m$, $\Gamma$ is three
 times continuously differentiable, and
 \begin{align}
  \Gamma''&=a'\E U_t^2+a^2\E\bigl(V_t^2-2mU_t^3\bigr),\label{eq:plat-second}\\
  \Gamma'''&=a''\E U_t^2+3aa'\E\bigl(V_t^2-2mU_t^3\bigr)
  +a^3\E\bigl(W_t^2-12mU_tV_t^2+6m^2U_t^4\bigr),\label{eq:plat-third}
 \end{align}
 where $a$, $a'$, $a''$ are evaluated at $t$.
\end{thmlist}
\end{lemma}

Parts (a) and (b) are~\cite[Proposition~3, (20)--(22)]{AC15a}, read along $X$, and part
(c) follows by taking expectations in the stochastic differentials above
(Appendix~\ref{sec:plat-jets}). Part (b) holds at atoms without any assumption of a density,
and it converts curvature into mass. Put
\begin{equation}\label{eq:plat-NS}
 N_\zeta(t)=a'(t)\,\E U_t^2+a(t)^2\,\E V_t^2,\qquad
 \sigma_\zeta(t)=2a(t)^2\,\E U_t^3
\end{equation}
(the functions $\gamma_{1,\mu}$, $\gamma_{2,\mu}$ of~\cite[Proposition~3]{AC15a}). By (b),
the right and left derivatives of $\Gamma_\zeta'$ at $t$ are
$N_\zeta(t)-\zeta([0,t])\,\sigma_\zeta(t)$ and $N_\zeta(t)-\zeta([0,t))\,\sigma_\zeta(t)$; at
the end of a plateau, \eqref{eq:plat-jump} reads off the atom $\zeta(\{t\})$ from the jump
$\zeta(\{t\})\,\sigma_\zeta(t)$ of $\Gamma_\zeta''$.

\subsection{Contact conditions and the trapezoid principle}\label{sec:plat-obstacle}

Integration against the obstacle $\Obs_\zeta$ is the derivative of the Parisi functional
in directions of total mass zero \eqref{eq:cons-first-variation}. The positive support
points of an unconstrained or constrained minimizer are minimum points of its obstacle
(Section~\ref{sec:cons-variational}), and a minimum point of $\Obs_\zeta$ on $(0,1]$ is a
positive contact: there $\Gamma_\zeta$ meets the identity, and the second-order condition is
the Plefka inequality \eqref{eq:intro-plefka}.

\needspace{3\baselineskip}
\begin{lemma}[Contact conditions]\label{lem:plat-obstacle}
Let $\zeta\in\Pspace([0,1])$, $\Gamma=\Gamma_\zeta$ and $\Obs=\Obs_\zeta$.
\begin{thmlist}
\item $\Obs\in C^2[0,1]$, with $\Obs'(t)=-\tfrac12a(t)\bigl(\Gamma(t)-t\bigr)$ and
 $\Obs''(t)=-\tfrac12a'(t)\bigl(\Gamma(t)-t\bigr)-\tfrac12a(t)\bigl(\Gamma'(t)-1\bigr)$.
 Moreover $\Obs'(1)>0$, so $1$ is not a minimum point of $\Obs$ on $[0,1]$.
\item At every local minimum point $t\in(0,1)$ of
 $\Obs$,
 \begin{equation}\label{eq:plat-contact}
  \Gamma(t)=t,\qquad \Gamma'(t)=\xi''(t)\,\E\Phi_{\zeta,xx}(t,X_t)^2\le1 .
 \end{equation}
\item \emph{Endpoint inequality.} If $\Obs(t)\ge\Obs(0)$ for all small $t>0$, then
 $\azero\Phi_{\zeta,xx}(0,0)^2=\Gamma'(0)\le1$.
\item \emph{Trapezoid principle.} Let $0\le u<v<1$
 with $\Gamma(u)=u$, $\Gamma(v)=v$, $\Gamma'(u)\le1$ and $\Gamma'(v)\le1$. Then $\Gamma'$
 is not strictly convex on $(u,v)$.
\end{thmlist}
\end{lemma}

\begin{proof}
(a) and (b) are Lemma~\ref{lem:cons-standard}(f), whose proof uses only
$\Obs'(t)=0\le\Obs''(t)$ at the minimum point and therefore applies at local minima.
(c) By Lemma~\ref{lem:plat-ito}(a) and $X_0=0$, $\Gamma'(0)=\azero\Phi_{\zeta,xx}(0,0)^2$, which vanishes if $\azero=0$. If $\azero>0$: since
$\Gamma(0)=0$ and $\Gamma\in C^1$,
$\Obs(t)-\Obs(0)=-\frac12\int_0^ta(s)(\Gamma(s)-s)\,ds
=-\frac{\azero}4(\Gamma'(0)-1)t^2+o(t^2)$,
so the quadratic coefficient is nonnegative. (d) If $\Gamma'$ were strictly convex on
$(u,v)$, then, being continuous on $[u,v]$, it would lie strictly below its chord inside
the interval, and
\[
 1=\frac{\Gamma(v)-\Gamma(u)}{v-u}=\frac1{v-u}\int_u^v\Gamma'(t)\,dt
 <\frac{\Gamma'(u)+\Gamma'(v)}2\le1,
\]
a contradiction.
\end{proof}

For support points of the Parisi measure, part (b) is also~\cite[Corollary~3.10]{JT17}.
Part (c) applies to $\zetas$, whose obstacle is minimal on $\supp\zetas$, because
$0\in\supp\zetas$ (Lemma~\ref{lem:ons-zero-in-support}). Part (d) is how marginality
enters the paper: on a plateau between two contacts, $\Gamma(v)-\Gamma(u)=v-u$ forces the
mean of $\Gamma'$ to be one, while strict convexity of $\Gamma'$ and part (b) force it below one.
Section~\ref{sec:plat-gap} shows that near zero $\Gamma'$ is indeed strictly convex on every
plateau of a measure close to $\zetas$.

\subsection{A universal sign at the origin}\label{sec:plat-sign}

The constant $\Dstar=\xi''''(0)\Ustar^2+\azero^3\Wstar^2$ enters all asymptotic formulas
of Theorems~\ref{thm:sk}--\ref{thm:field}, through $\rhostar=\Dstar/(2\sqrt{\azero})$ and
$c=9\azero^3/(80\Dstar^3)$, and their proofs need $\Dstar>0$; for the
Sherrington--Kirkpatrick model $\xi''''(0)=0$, so positivity rests entirely on $\Wstar\ne0$.
The next lemma gives a strict sign for the third spatial derivative of every solution of the
Parisi equation, with a constant that does not depend on the measure.

\begin{lemma}[Universal sign]\label{lem:plat-fourth-sign}
Let $\xi(t)=\sum_{p\ge2}\beta_p^2t^p$ be not identically zero, with radius of convergence
greater than one. Put $T_t=\xi'(1)-\xi'(t)$, let $\varphi_T$ be the $\mathcal N(0,T)$
density, and set
\[
 c_1=\min\bigl\{2\tanh y\,\sech^2y:\ 1\le y\le2+\xi'(1)\bigr\},\qquad
 \kappa_t=\frac{c_1\,\varphi_{T_t}(5/2)}{T_t}>0\qquad(0\le t<1).
\]
For every $\zeta\in\Pspace([0,1])$, $0\le t<1$ and $0\le x\le\frac12$,
\[
 \Phi_{\zeta,xxx}(t,x)\le-\kappa_tx,\qquad\text{hence}\qquad
 \Phi_{\zeta,xxxx}(t,0)\le-\kappa_t<0 .
\]
In particular $\Wstar\le-\kappa_0$. If $\beta_2>0$, then the constant
$\Dstar=\xi''''(0)\Ustar^2+\azero^3\Wstar^2$ of Section~\ref{sec:intro-constrained} satisfies
$\Dstar\ge\azero^3\kappa_0^2>0$.
\end{lemma}

The idea is that $V=\Phi_{\zeta,xxx}$ vanishes on $\{x=0\}$ by oddness, equals
$-2\tanh x\sech^2x<0$ for $x>0$ at the terminal time, and solves the linear equation
\eqref{eq:plat-third-pde}, whose potential term $3a\alpha_\zeta UV$ has a nonnegative
coefficient. A Feynman--Kac representation killed at the origin carries the terminal sign
backward in time, and a reflection estimate makes it quantitative.

\begin{proof}
First let $\zeta$ have finitely many atoms, and write $u=\Phi_{\zeta,x}$,
$U=\Phi_{\zeta,xx}$, $V=\Phi_{\zeta,xxx}$ and $\alpha=\alpha_\zeta$. On each closed strip
between consecutive atoms, $V$ solves the last equation of \eqref{eq:plat-third-pde}, with
$m$ the height of the strip. Fix $t<1$ and $0<x\le\frac12$. Let $Y$ solve
\[
 dY_s=a(s)\alpha(s)u(s,Y_s)\,ds+\sqrt{a(s)}\,dB_s\qquad(t\le s\le1),\qquad Y_t=x ;
\]
the drift is bounded and Lipschitz in space, since $|u|\le1$ and $0<u_x\le1$. Put
$c(s,y)=3a(s)\alpha(s)U(s,y)\ge0$. By It\^o's formula on each strip,
\[
 M_s=V(s,Y_s)\exp\Bigl\{\int_t^sc(r,Y_r)\,dr\Bigr\}
\]
is a continuous martingale on $[t,1]$, bounded by $B_3e^{3\xi'(1)}$; it is continuous
across atoms because $V$ is continuous in time. Let $\tau=\inf\{s\ge t:Y_s=0\}$. Since $V$
is odd in $x$, $V(s,0)=0$, and optional stopping gives
\[
 V(t,x)=\E M_{\tau\wedge1}
 =\E\Bigl[V(1,Y_1)\,e^{\int_t^1c(r,Y_r)\,dr};\ \tau>1\Bigr].
\]
On $\{\tau>1\}$ we have $Y_1>0$, so $V(1,Y_1)=-2\tanh Y_1\sech^2Y_1<0$, while the
exponential factor is at least one. Therefore $V(t,x)\le\E[V(1,Y_1);A]$ for every event
$A\subset\{\tau>1\}$.

Let $\widetilde Y_s=x+\int_t^s\sqrt{a(r)}\,dB_r$. The function $u(s,\cdot)$ is odd and
increasing, so $u(s,y)\ge0$ for $y\ge0$. Hence, for $t\le s\le\tau\wedge1$,
$0\le Y_s-\widetilde Y_s=\int_t^sa(r)\alpha(r)u(r,Y_r)\,dr\le\xi'(1)$.
If $\widetilde Y$ stays positive on $[t,1]$, then $\tau>1$, for otherwise
$\widetilde Y_\tau\le Y_\tau=0$. Take
$A=\{\widetilde Y_s>0\ \text{for}\ t\le s\le1,\ \widetilde Y_1\in[1,2]\}$.
Then $A\subset\{\tau>1\}$ and $Y_1\in[1,2+\xi'(1)]$ on $A$, so $V(t,x)\le-c_1\Prob(A)$.
The process $\widetilde Y$ is a Brownian motion started at $x$ and run with the clock
$\xi'(s)-\xi'(t)$, which reaches $T=T_t>0$ at $s=1$ because $\xi''>0$ on $(0,1]$. The
reflection principle gives
$\Prob(A)=\int_1^2\{\varphi_T(y-x)-\varphi_T(y+x)\}\,dy$. For $y\in[1,2]$ and $0<x\le\frac12$,
\[
 \varphi_T(y-x)-\varphi_T(y+x)=\int_{-x}^x\frac{y+r}T\,\varphi_T(y+r)\,dr
 \ge\frac xT\,\varphi_T(5/2),
\]
because $y+r\in[\frac12,\frac52]$. Hence $V(t,x)\le-\kappa_tx$.

For a general $\zeta$, choose measures $\zeta_n$ with finitely many atoms and
$\zeta_n\to\zeta$ weakly. The constant $\kappa_t$ does not depend on $n$, and
$\Phi_{\zeta_n,xxx}\to\Phi_{\zeta,xxx}$ uniformly~\cite[Proposition~1(i)]{AC15a}, so the
inequality passes to the limit. Finally $V(t,0)=0$ and $V(t,\cdot)$ is continuously
differentiable, so $\Phi_{\zeta,xxxx}(t,0)=\lim_{x\downarrow0}V(t,x)/x\le-\kappa_t$. At
$t=0$ and $\zeta=\zetas$ this is $\Wstar\le-\kappa_0$; since $\xi''''(0)\ge0$, the last
assertion follows.
\end{proof}

\begin{remark}\label{rem:plat-sign}
The non-strict sign $\Phi_{\zeta,xxx}(t,x)\le0$ for $t,x>0$ and every $\zeta$
is~\cite[(7.6)]{Lop26a} for the Sherrington--Kirkpatrick equation; the time change
$t\mapsto\xi'(t)/\xi'(1)$, which turns the Parisi equation of $\xi$ into that of the
Sherrington--Kirkpatrick covariance with $\beta^2=\xi'(1)$, carries it to every $\xi$. A
Feynman--Kac proof for the Sherrington--Kirkpatrick equation is in~\cite[Lemma~5.4]{Pei26},
a maximum-principle proof for a regularized zero-temperature equation
in~\cite[Lemma~3.4 and Appendix~B.3]{Che26}, and for the Sherrington--Kirkpatrick minimizer
Aronow and Lopatto prove $-\beta^3\Wstar\ge c_0>0$, uniformly in $\beta\ge\beta_0>1$, by
parabolic comparison~\cite[Lemma~7.1, (7.7)]{AL26}. For the Sherrington--Kirkpatrick equation and
every order parameter, the unrefereed preprint \cite[Lemma~5.3]{OAI26a} proves
$c\le-\Phi_{\zeta,xxxx}(0,0)\le C$, with constants depending only on $\beta$, by the same
mechanism: a Feynman--Kac representation of $\Phi_{\zeta,xxx}$ killed at the origin and a
reflected Gaussian kernel. We use the strict bound at the origin with
an explicit constant, for every covariance: it makes $\Dstar>0$ automatic in
Theorems~\ref{thm:local} and~\ref{thm:field} (Remark~\ref{rem:fld-nondegeneracy}). Thus
$\Dstar>0$ comes from $\Wstar<0$; the term $\xi''''(0)\Ustar^2$ only adds to it.
\end{remark}

\subsection{Local plateau positivity and gap exclusion}\label{sec:plat-gap}

In this subsection we assume $\xi'''(0)=0$; Section~\ref{sec:plat-cubic} uses it only in the
last sentence of Corollary~\ref{cor:plat-first-plateau}. Recall $\Ustar=\Phi_{\zetas,xx}(0,0)$,
$\Wstar=\Phi_{\zetas,xxxx}(0,0)$ and $\Dstar=\xi''''(0)\Ustar^2+\azero^3\Wstar^2$, which is
positive by Lemma~\ref{lem:plat-fourth-sign}. For $\zeta\in\Pspace([0,1])$ and
$t,m\in[0,1]$, let $\Psi_\zeta(t,m)$ be the right side of \eqref{eq:plat-third} with the
height replaced by a free parameter:
\begin{equation}\label{eq:plat-psi}
 \Psi_\zeta(t,m)=a''\E U_t^2+3aa'\E\bigl(V_t^2-2mU_t^3\bigr)
 +a^3\E\bigl(W_t^2-12mU_tV_t^2+6m^2U_t^4\bigr),
\end{equation}
with $a$, $a'$, $a''$ evaluated at $t$. On an open plateau of $\zeta$ of height $m$,
$\Gamma_\zeta'''=\Psi_\zeta(\cdot,m)$. At $t=0$ and $\zeta=\zetas$ the expression is
explicit: $X_0=0$, $V_0=0$ and $a'(0)=\xi'''(0)=0$, so
$\Psi_{\zetas}(0,m)=\Dstar+6m^2\azero^3\Ustar^4\ge\Dstar$. The next lemma says that this
value controls $\Gamma_\zeta'''$ on every plateau near the origin, for every measure close
to $\zetas$.

\begin{lemma}[Local plateau positivity]\label{lem:plat-local-plateau}
Assume $\xi'''(0)=0$. There is $C_\Psi$, depending only on $\xi$, such that
\[
 \bigl|\Psi_\zeta(t,m)-\Dstar-6m^2\azero^3\Ustar^4\bigr|
 \le C_\Psi\bigl(t^{1/2}+\Wone(\zeta,\zetas)\bigr)
\]
for all $\zeta\in\Pspace([0,1])$, $t\in[0,1]$ and $m\in[0,1]$. Consequently there are
$K_{\rm loc}\in(0,1)$ and $r_{\rm loc}>0$ with the following property: if
$\Wone(\zeta,\zetas)<r_{\rm loc}$ and $(u,v)\subset(0,K_{\rm loc})$ is an open interval with
$\zeta((u,v))=0$, then $\Gamma_\zeta'''\ge\Dstar/2$ on $(u,v)$.
\end{lemma}

\begin{proof}
For $k=2,3,4$ put $J_k(t)=\partial_x^k\Phi_\zeta(t,X^\zeta_t)$ and
$j_k=\partial_x^k\Phi_{\zetas}(0,0)$, so that $j_3=0$ by parity. By \eqref{eq:plat-origin}
with $\nu=\zetas$, $\E|J_k(t)-j_k|\le C(t^{1/2}+\Wone(\zeta,\zetas))$ for $t\le1$, with $C$
depending only on $\xi$. The quantity $\Psi_\zeta(t,m)$ is a polynomial in $a(t)$, $a'(t)$,
$a''(t)$, $m$ and expectations of monomials of degree at most four in $J_2,J_3,J_4$, all
bounded by the constants $B_k$; the functions $a$, $a'$, $a''$ are Lipschitz on $[0,1]$.
Replacing each $J_k(t)$ by $j_k$ and each $a^{(i)}(t)$ by $a^{(i)}(0)$ therefore changes
$\Psi_\zeta(t,m)$ by at most $C_\Psi(t^{1/2}+\Wone(\zeta,\zetas))$, uniformly in
$m\in[0,1]$, and the result is
\[
 \xi''''(0)\Ustar^2+3\azero\xi'''(0)\bigl(0-2m\Ustar^3\bigr)
 +\azero^3\bigl(\Wstar^2+6m^2\Ustar^4\bigr)=\Dstar+6m^2\azero^3\Ustar^4,
\]
since $\xi'''(0)=0$. Choose $K_{\rm loc}\in(0,1)$ and $r_{\rm loc}>0$ with
$C_\Psi(K_{\rm loc}^{1/2}+r_{\rm loc})\le\Dstar/2$. If $\zeta((u,v))=0$, then $(u,v)$ is a
plateau of some height $m\in[0,1]$, and
$\Gamma_\zeta'''=\Psi_\zeta(\cdot,m)\ge\Dstar+6m^2\azero^3\Ustar^4-\Dstar/2\ge\Dstar/2$
there.
\end{proof}

\begin{lemma}[Gap exclusion]\label{lem:plat-local-gap}
Assume $\xi'''(0)=0$, and let $K_{\rm loc},r_{\rm loc}$ be as in
Lemma~\ref{lem:plat-local-plateau}. Let $\Wone(\zeta,\zetas)<r_{\rm loc}$, and let
$0<u<v<K_{\rm loc}$ be local minimum points of $\Obs_\zeta$. Then $\zeta((u,v))>0$.
\end{lemma}

\begin{proof}
By Lemma~\ref{lem:plat-obstacle}(b), $\Gamma_\zeta(u)=u$, $\Gamma_\zeta(v)=v$ and
$\Gamma_\zeta'(u),\Gamma_\zeta'(v)\le1$. If $\zeta((u,v))=0$, then
Lemma~\ref{lem:plat-local-plateau} gives $\Gamma_\zeta'''\ge\Dstar/2>0$ on $(u,v)$, so
$\Gamma_\zeta'$ is strictly convex there. This contradicts
Lemma~\ref{lem:plat-obstacle}(d).
\end{proof}

Figure~\ref{fig:plateau}(a) shows the contradiction. Under \condH\ the lemma fills an
interval $[0,\kappa]$ with support of $\zetas$ (Proposition~\ref{prop:ons-equilibrium}); for
$\zeta_\theta$ with $\theta$ small it makes every point between $q_\theta$ and a support point
$s_\theta$ near zero a contact, so that $\Gamma_{\zeta_\theta}=\id$ on $[q_\theta,s_\theta]$
and $\Gamma_{\zeta_\theta}'(q_\theta)=1$ (Proposition~\ref{prop:ons-onset}). Excluding a gap by the incompatibility of $\Gamma(u)=u$,
$\Gamma(v)=v$, $\Gamma'(u)\le1$, $\Gamma'(v)\le1$ is the mechanism by which
Zhou~\cite{Zho26} showed that the Sherrington--Kirkpatrick Parisi measure is supported on an
interval starting at the origin slightly below the critical temperature. Here the
incompatibility comes from the sign of $\Gamma'''$ near the origin, for every covariance with
$\xi'''(0)=0$ and every measure close to $\zetas$; Section~\ref{sec:nc-reference} gives a
version uniform near the critical covariance. For the Sherrington--Kirkpatrick covariance,
Theorem~\ref{thm:sk-global}(b) removes both restrictions: no plateau separates two local minimum
points of $\Obs_\zeta$ anywhere in $(0,1)$, for every $\zeta$, with a proof that rests on the
unrefereed finite-step inequality of Proposition~\ref{prop:skg-oai}.

\subsection{The universal cubic}\label{sec:plat-cubic}

Below the first positive contact $q$ of a constrained minimizer the measure has no mass
except its atom at zero, so $(0,q)$ is a plateau. If the contact is marginal and followed
by an interval of contacts, as Lemma~\ref{lem:plat-local-gap} guarantees for small
$\theta$ (Proposition~\ref{prop:ons-onset}), then $h=\Gamma_\zeta-\id$ satisfies $h(0)=0$ by
evenness, $h(q)=0$ because $q$ is a contact, and $h'(q)=0$ by marginality. These three
boundary conditions and an almost constant third derivative determine a cubic; two exact
identities make this quantitative.

\begin{lemma}[Two identities on a plateau]\label{lem:plat-kernel}
Let $q>0$, and let $g\in C^1[0,q]$ be three times continuously differentiable on $(0,q)$
with $g'''$ bounded and $g(0)=g(q)=g'(q)=0$. Then $g''$ extends continuously to $[0,q]$; write $g''(0+)$ and $g''(q-)$ for its end
values. Moreover
\begin{equation}\label{eq:plat-onset-identity}
 0=\int_0^qt\,g''(t)\,dt=\frac{q^2}2\,g''(0+)+\frac12\int_0^q(q^2-s^2)\,g'''(s)\,ds,
\end{equation}
and
\begin{equation}\label{eq:plat-kernel}
 \int_0^qg(s)\,ds=\int_0^qJ(u)\,g'''(u)\,du,\qquad J(u)=\frac{u^2(q-u)}6\ge0,\qquad
 \int_0^qJ(u)\,du=\frac{q^4}{72}.
\end{equation}
\end{lemma}

For $g=\Gamma_\zeta-\id$ the first identity reads
$0=q\Gamma_\zeta'(q)-\Gamma_\zeta(q)=\int_0^qt\,\Gamma_\zeta''(t)\,dt$.

\begin{proof}
Since $g'''$ is bounded, $g''$ is Lipschitz on $(0,q)$ and extends continuously to the
endpoints. Integration by parts gives $\int_0^qtg''(t)\,dt=qg'(q)-g(q)+g(0)=0$, and
inserting $g''(t)=g''(0+)+\int_0^tg'''(s)\,ds$ gives the second form of
\eqref{eq:plat-onset-identity}. For \eqref{eq:plat-kernel}, the conditions $g(q)=g'(q)=0$
give $g(s)=\int_s^q(u-s)g''(u)\,du$, and Fubini's theorem gives
$\int_0^qg=\frac12\int_0^qu^2g''(u)\,du$. Since $\int_0^qug''(u)\,du=0$, we may subtract
$\frac q3$ times it:
\[
 \int_0^qg=\int_0^q\Bigl(\frac{u^2}2-\frac{qu}3\Bigr)g''(u)\,du
 =-\int_0^qJ'(u)\,g''(u)\,du=\int_0^qJ(u)\,g'''(u)\,du,
\]
because $J(0)=J(q)=0$ and $g''$ is bounded. Finally
$\int_0^qJ=\frac16(\frac{q^4}3-\frac{q^4}4)=\frac{q^4}{72}$.
\end{proof}

The model case is $g'''\equiv D$ for a constant $D$. The unique solution with the three
boundary conditions is $h_D(t)=\frac D6\,t(q-t)^2$, with $h_D''(t)=Dt-\frac23Dq$,
so $h_D''(0)=-\frac23Dq$, $h_D''(q)=\frac13Dq$ and $\int_0^qh_D=\frac{Dq^4}{72}$. The
next lemma says that $g$ is close to $h_D$ when $g'''$ is close to $D$, with errors that
are explicit multiples of the deviation.

\begin{lemma}[The universal cubic]\label{lem:plat-cubic}
Let $g$ be as in Lemma~\ref{lem:plat-kernel}, let $D\in\R$, and let $\varepsilon$ satisfy
$\varepsilon\ge\sup_{0<t<q}|g'''(t)-D|$. Then, for $0<t<q$,
\[
 \bigl|g''(t)-h_D''(t)\bigr|\le\frac53\,\varepsilon q,\qquad
 \bigl|g(t)-h_D(t)\bigr|\le\frac56\,\varepsilon q\,(q-t)^2 .
\]
In particular
\[
 \Bigl|g''(0+)+\frac23Dq\Bigr|\le\frac23\,\varepsilon q,\qquad
 \Bigl|g''(q-)-\frac13Dq\Bigr|\le\frac53\,\varepsilon q,\qquad
 \Bigl|\int_0^qg-\frac{Dq^4}{72}\Bigr|\le\frac{\varepsilon q^4}{72}.
\]
\end{lemma}

\begin{proof}
The function $e=g-h_D$ satisfies the hypotheses of Lemma~\ref{lem:plat-kernel} with
$|e'''|\le\varepsilon$. By \eqref{eq:plat-onset-identity} applied to $e$,
$|e''(0+)|\le q^{-2}\varepsilon\int_0^q(q^2-s^2)\,ds=\frac23\varepsilon q$. Then
$e''(t)=e''(0+)+\int_0^te'''$ gives $|e''(t)|\le\frac23\varepsilon q+\varepsilon t
\le\frac53\varepsilon q$, and $e(t)=\int_t^q(u-t)e''(u)\,du$ gives
$|e(t)|\le\frac56\varepsilon q(q-t)^2$. The area bound is \eqref{eq:plat-kernel} applied
to $e$, with $J\ge0$.
\end{proof}

The cubic is universal in that only $D$ and $q$ enter, not the covariance or the measure.
The next corollary applies it to the first plateau, where the boundary curvatures encode two
physical quantities: the curvature at zero is fixed by the atom at zero through
\eqref{eq:plat-zero}, and the curvature at $q-$ is the jump \eqref{eq:plat-jump} produced by
the atom at $q$ (Figure~\ref{fig:plateau}(b)).

\begin{corollary}[The first plateau]\label{cor:plat-first-plateau}
Let $\zeta\in\Pspace([0,1])$ and $0<q<s<1$ be such that $\zeta((0,q))=0$ and
$\Gamma_\zeta(t)=t$ for $q\le t\le s$. Put $m=\zeta(\{0\})$ and, for $D\in\R$,
$\varepsilon=\sup_{0<t<q}|\Gamma_\zeta'''(t)-D|$. Then $\varepsilon<\infty$, and:
\begin{thmlist}
\item $\bigl|\xi'''(0)\,\Phi_{\zeta,xx}(0,0)^2-2\azero^2m\,\Phi_{\zeta,xx}(0,0)^3
 +\frac23Dq\bigr|\le\frac23\,\varepsilon q$;
\item $\bigl|2a(q)^2\,\zeta(\{q\})\,\E U_q^3-\frac13Dq\bigr|\le\frac53\,\varepsilon q$;
\item $\bigl|\Gamma_\zeta(t)-t-\frac D6t(q-t)^2\bigr|\le\frac56\,\varepsilon q(q-t)^2$ for
 $0\le t\le q$, and
 $\bigl|\int_0^q(\Gamma_\zeta(t)-t)\,dt-\frac{Dq^4}{72}\bigr|\le\frac{\varepsilon q^4}{72}$.
\end{thmlist}
If $\xi'''(0)=0$ and $D=\Dstar$, then moreover
$\varepsilon\le C_\Psi(q^{1/2}+\Wone(\zeta,\zetas))+6m^2\azero^3\Ustar^4$, with $C_\Psi$ as
in Lemma~\ref{lem:plat-local-plateau}.
\end{corollary}

\begin{proof}
On $[0,q)$ we have $\alpha_\zeta=m$, so $(0,q)$ is a plateau of height $m$. By
Lemma~\ref{lem:plat-ito}, $g=\Gamma_\zeta-\id$ is in $C^1[0,q]$ and three times
continuously differentiable on $(0,q)$, with $g'''=\Gamma_\zeta'''=\Psi_\zeta(\cdot,m)$
bounded by \eqref{eq:plat-bounds}; thus $\varepsilon<\infty$. Also $g(0)=0$ by
Lemma~\ref{lem:plat-ito}(a), and $g(q)=g'(q)=0$ because $\Gamma_\zeta=\id$ on $[q,s]$.
The end values of $g''$ are one-sided derivatives of $\Gamma_\zeta'$, by the mean value
theorem. At zero, $g''(0+)$ is given by \eqref{eq:plat-zero} with $\zeta(\{0\})=m$. At
$q$, the right derivative of $\Gamma_\zeta'$ vanishes because $\Gamma_\zeta'=1$ on
$[q,s]$, so \eqref{eq:plat-jump} gives $g''(q-)=2a(q)^2\zeta(\{q\})\E U_q^3$.
Lemma~\ref{lem:plat-cubic} now gives (a)--(c). The bound on $\varepsilon$ is
Lemma~\ref{lem:plat-local-plateau}, since $t<q$ on the plateau.
\end{proof}

Under the hypotheses of Theorem~\ref{thm:local}, Section~\ref{sec:ons} applies the corollary
to $\zeta=\zeta_\theta$, with $m=\theta$, $q=q_\theta$ and $D=\Dstar$, once
Lemma~\ref{lem:plat-local-gap} has supplied the interval of contacts above $q_\theta$; as
$\theta\downarrow0$, $\varepsilon\to0$, and parts (a)--(c) give the onset, the atom at
$q_\theta$ and the area that enters the envelope identity (Proposition~\ref{prop:ons-onset}(c),
Lemma~\ref{lem:fifth-area}). When the first contact is
strictly Plefka stable, as beyond the de Almeida--Thouless threshold of
Theorem~\ref{thm:nearcrit}(b), $g'(q)=0$ fails and Corollary~\ref{cor:plat-first-plateau}
does not apply.

\section{The bottom of the Parisi measure and the marginal onset}\label{sec:ons}

With the plateau calculus of Section~\ref{sec:plat} in hand, this section proves three of
the four inputs from which Section~\ref{sec:fifth} deduced Theorems~\ref{thm:sk}
and~\ref{thm:local}: the shape of $\zetas$ near zero
(Proposition~\ref{prop:ons-equilibrium}), the marginal onset of $\zeta_\theta$
(Proposition~\ref{prop:ons-onset}) and the area of its first plateau
(Lemma~\ref{lem:fifth-area}). All three read a measure off the function $\Gamma_\zeta$
of~\eqref{eq:intro-obstacle}. Positive support points are contacts, where $\Gamma_\zeta(t)=t$ and
$\Gamma_\zeta'(t)\le1$ (for $\zetas$ these conditions also hold at zero,
Lemma~\ref{lem:ons-zero-in-support}), and near zero two contacts cannot enclose an interval
without mass (Lemma~\ref{lem:plat-local-gap}). So the contacts near zero fill an interval on which
$\Gamma_\zeta'\equiv1$, and the one-sided It\^o formula of Lemma~\ref{lem:plat-ito} turns
this into equations for the distribution function: a quotient formula for $\alphas$ and, from the
atom $\theta$ that the constraint places at zero, the first positive contact
$q_\theta\approx3\theta/(2\rhostar)$ of $\zeta_\theta$ and its atom of about $\theta/2$
(Figure~\ref{fig:bottom}(a)). The universal
cubic then evaluates the area that enters the envelope identity.

\paragraph{Setting.}
Lemma~\ref{lem:ons-zero-in-support} holds for every covariance
$\xi(t)=\sum_{p\ge2}\beta_p^2t^p$ that is not identically zero and has radius of convergence
greater than one. From Section~\ref{sec:ons-equilibrium} on we assume in addition the
hypotheses of Theorem~\ref{thm:local}: $\beta_2>0$, $\xi'''(0)=0$ and~\condH. As in
Section~\ref{sec:plat}, $a=\xi''$, and for the measure $\zeta$ under discussion
$U_t=\Phi_{\zeta,xx}(t,X_t)$, $V_t=\Phi_{\zeta,xxx}(t,X_t)$ and
$W_t=\Phi_{\zeta,xxxx}(t,X_t)$ along its diffusion $X=X^\zeta$. We write
$\Gamma_*=\Gamma_{\zetas}$ and $\Obs_*=\Obs_{\zetas}$; $C_\Psi$, $K_{\rm loc}\in(0,1)$ and
$r_{\rm loc}>0$ are the constants of Lemma~\ref{lem:plat-local-plateau}. All constants depend
on $\xi$, and every limit is taken at fixed $\xi$.

\subsection{Zero lies in the support}\label{sec:ons-zero-support}

\begin{lemma}[Zero lies in the support]\label{lem:ons-zero-in-support}
Let $\xi(t)=\sum_{p\ge2}\beta_p^2t^p$ be not identically zero, with radius of convergence
greater than one. Then $0\in\supp\zetas$ and $1\notin\supp\zetas$. The obstacle $\Obs_*$
attains its minimum over $[0,1]$ at every point of $\supp\zetas$, and every
$u\in\supp\zetas$ satisfies $\Gamma_*(u)=u$ and $\Gamma_*'(u)\le1$.
\end{lemma}

The inclusion $0\in\supp\zetas$ is \cite[Theorem~1(i)]{AC15a}, and the contact conditions on
the support are stated in \cite[Theorem~5]{AC15a} and \cite[Lemma~3.8 and Corollary~3.10]{JT17}.
The proof of \cite[Theorem~1(i)]{AC15a} uses \cite[Theorem~5]{AC15a} at the smallest support
point; when that point is isolated, the inequality $\Gamma_*'\le1$ there can be taken from
\cite[Corollary~3.10]{JT17}, as \cite[Proposition~2.1~(P6) and Appendix~A.1]{Pei26} also notes;
the derivative formulas behind that corollary are stated without proof \cite[Lemma~8.4]{JT17}.
Our direct proof is in
Appendix~\ref{sec:an-zero}: if the smallest support point $t_0$ were positive, $(0,t_0)$
would be a plateau of height zero on which~\eqref{eq:plat-second} makes $\Gamma_*-\id$
strictly convex, and then $\Obs_*(0)<\Obs_*(t_0)=\min_{[0,1]}\Obs_*$. It uses the covariance
only through $\xi''>0$ on $(0,1]$ and $\xi'''\ge0$, so it applies to every covariance in
this paper.

\begin{remark}[The condition \textup{(H)}]\label{rem:ons-H-no-atom}
By Lemma~\ref{lem:ons-zero-in-support}, every neighbourhood of zero has positive
$\zetas$-mass. Condition~\condH\ asks for more: zero is not an isolated point of
$\supp\zetas$, or equivalently $\zetas((0,\varepsilon))>0$ for every $\varepsilon>0$. If
$\zetas(\{0\})=0$, then~\condH\ follows from the lemma, because
$\zetas((0,\varepsilon))=\zetas([0,\varepsilon))>0$. If $\xi'''(0)=0$, the converse also
holds:~\condH\ implies $\zetas(\{0\})=0$ (Proposition~\ref{prop:ons-equilibrium}(b) under
the hypotheses of Theorem~\ref{thm:local}, and Theorem~\ref{thm:atom}(b) in general). So
when $\xi'''(0)=0$, condition~\condH\ says exactly that $\zetas$ has no atom at zero. If
$\xi'''(0)>0$, condition~\condH\ forces an atom at zero, of the size given in
Theorem~\ref{thm:atom}(b), and the two conditions separate: the near-critical mixtures with
a cubic term ($\beta_3>0$) of Theorem~\ref{thm:atom}(c) satisfy~\condH\ and have
$\zetas(\{0\})>0$. Where~\condH\ is known to hold or fail, with its sources, is recorded in
Section~\ref{sec:intro-constrained}.
\end{remark}

\subsection{Interval support and the density at zero}\label{sec:ons-equilibrium}

Proposition~\ref{prop:ons-equilibrium} turns~\condH, which asks only that zero not be
isolated in $\supp\zetas$, into an interval $[0,\kappa]$ of support, by the gap exclusion of
Lemma~\ref{lem:plat-local-gap}, and reads $\zetas$ off the identity $\Gamma_*=\id$ there.
The value $\Gamma_*'(0)=1$ gives $\Ustar$. At zero the It\^o formula for $\Gamma_*''$
contains the atom $\zetas(\{0\})$, multiplied by $-2\azero^2\Ustar^3$, and no other nonzero
term, so the atom vanishes. Inside the interval, $\Gamma_*''=0$ is the linear
equation~\eqref{eq:ons-quotient} for $\alphas(t)$, whose numerator vanishes at zero with
derivative $\Dstar$; so $\alphas(t)=\rhostar t+o(t)$, and the quantile function of $\zetas$
has slope $1/\rhostar$ at zero. On an interval $[t_1,t_2)$ of the support, $t_1=0$ included,
smoothness of $\alphas$ and the quotient formula are \cite[Theorem~2(ii) and its proof]{AC15a},
and~\eqref{eq:ons-slope} is the derivative of that formula at zero when an interval
$[0,\varepsilon)$ lies in the support. Our proof needs only~\condH, from which part~(a) produces
such an interval. We include the short derivation of~\eqref{eq:ons-quotient} along $X$ because
its moment identities~\eqref{eq:ons-moments} also give the derivative $\Dstar$ of the numerator
at zero, hence~\eqref{eq:ons-slope}. For the Sherrington--Kirkpatrick model, the same computation
of the density at zero from $\Gamma_*''=0$ near zero, with its positivity from a strict bound on
$\Wstar$, is \cite[Lemma~5.8 and its proof]{OAI26a}, an unrefereed preprint that takes the
interval support from \cite[Theorem~1.1]{Lop26a} and the strict bound from
\cite[Lemma~5.3]{OAI26a} (Remarks~\ref{rem:plat-sign} and~\ref{rem:fifth-sk-inputs}).

\begin{proof}[Proof of Proposition~\ref{prop:ons-equilibrium}]
By Lemma~\ref{lem:ons-zero-in-support}, every point of $\supp\zetas\cap(0,1)$ is an
interior minimum of $\Obs_*$, at which $\Gamma_*=\id$ and $\Gamma_*'\le1$.

(a) By~\condH\ there is $s_1\in\supp\zetas\cap(0,K_{\rm loc})$. Suppose that some
$t\in(0,s_1)$ is not in $\supp\zetas$. By~\condH, the support has points in $(0,t)$. Put
$u=\max(\supp\zetas\cap[0,t])$ and $v=\min(\supp\zetas\cap[t,s_1])$. Then $0<u<t<v\le s_1$
and $\zetas((u,v))=0$. The points $u,v\in(0,K_{\rm loc})$ are minima of $\Obs_*$ with
$\Gamma_*(u)=u$, $\Gamma_*(v)=v$, $\Gamma_*'(u)\le1$ and $\Gamma_*'(v)\le1$. Since
$\Wone(\zetas,\zetas)=0<r_{\rm loc}$, this contradicts Lemma~\ref{lem:plat-local-gap}. Hence
$(0,s_1]\subset\supp\zetas$. The support is closed, so $[0,s_1]\subset\supp\zetas$; put
$\kappa=s_1$.

(b) By~(a), $\Gamma_*(t)=t$ for $0<t\le\kappa$, and $\Gamma_*(0)=\Phi_{\zetas,x}(0,0)^2=0$ by
oddness. Since $\Gamma_*$ is continuously differentiable, $\Gamma_*'\equiv1$ on
$[0,\kappa]$. At $t=0$ we have $X_0=0$, so $1=\Gamma_*'(0)=\azero\Ustar^2$; since
$\Ustar>0$, $\Ustar=\azero^{-1/2}$. The right derivative of $\Gamma_*'$ at zero is
therefore zero. By the right-derivative formula of Lemma~\ref{lem:plat-ito}(b) at $t=0$ it
equals
\[
 \xi'''(0)\Ustar^2+\azero^2\bigl(\Phi_{\zetas,xxx}(0,0)^2-2\zetas(\{0\})\Ustar^3\bigr)
 =-2\azero^2\zetas(\{0\})\Ustar^3,
\]
because $\xi'''(0)=0$ and $\Phi_{\zetas,xxx}(0,0)=0$ by oddness. Since $\Ustar>0$, we get
$\zetas(\{0\})=0$.

(c) Put $A_2(t)=\E U_t^2$, $A_3(t)=\E U_t^3$ and $B(t)=\E V_t^2$, so that, in the
notation~\eqref{eq:plat-NS}, $N_{\zetas}=a'A_2+a^2B$ and $\sigma_{\zetas}=2a^2A_3$. Since
$0<U_t\le1$, Jensen's inequality gives $A_3\ge A_2^{3/2}>0$. For $0\le s\le t\le1$,
\begin{equation}\label{eq:ons-moments}
\begin{aligned}
 A_2(t)-A_2(s)&=\int_s^ta(\tau)\,\E\bigl(V_\tau^2-2\alphas(\tau)U_\tau^3\bigr)\,d\tau,\\
 A_3(t)-A_3(s)&=3\int_s^ta(\tau)\,\E\bigl(U_\tau V_\tau^2-\alphas(\tau)U_\tau^4\bigr)\,d\tau,\\
 B(t)-B(s)&=\int_s^ta(\tau)\,\E\bigl(W_\tau^2-6\alphas(\tau)U_\tau V_\tau^2\bigr)\,d\tau .
\end{aligned}
\end{equation}
For a measure with finitely many atoms these identities follow from It\^o's formula: on an
interval where the distribution function equals a constant $m$, the processes $U$ and $V$
solve the stochastic differential equations stated after~\eqref{eq:plat-third-pde}, and
$U,V,W$ are continuous across atoms. For $\zetas$ the identities follow by approximation with
finitely atomic measures, using the jet stability~\eqref{eq:plat-jet-stability} of
Lemma~\ref{lem:plat-jets}(a) and the coupling continuity of Lemma~\ref{lem:plat-coupling}.

By~(a) and~(b), $aA_2=\Gamma_*'\equiv1$ on $[0,\kappa]$. The function $A_2$ is
absolutely continuous, and at almost every $t\in(0,\kappa)$ its derivative is the
integrand in~\eqref{eq:ons-moments}. Differentiating $aA_2\equiv1$ at such $t$ gives
$a'A_2+a^2(B-2\alphas A_3)=0$, that is, $\alphas(t)=\psi(t)$ with
$\psi=N_{\zetas}/\sigma_{\zetas}=(a'A_2+a^2B)/(2a^2A_3)$. The functions $A_2,A_3,B$ are
continuous by~\eqref{eq:ons-moments}, and $a>0$ and $A_3>0$, so $\psi$ is continuous on
$[0,\kappa]$. Since $\alphas$ is right continuous, approximating
$t\in[0,\kappa)$ from the right through such points gives $\alphas(t)=\psi(t)$ for every
$t\in[0,\kappa)$. This is~\eqref{eq:ons-quotient}. Now $\alphas$ is continuous on
$[0,\kappa)$, so the integrands in~\eqref{eq:ons-moments} are continuous there. Hence
$A_2,A_3,B$ are continuously differentiable on $[0,\kappa)$, and so is $\alphas=\psi$.

At $t=0$, the numerator $N_{\zetas}$ vanishes, since $a'(0)=\xi'''(0)=0$ and
$B(0)=\Phi_{\zetas,xxx}(0,0)^2=0$. The terms of $N_{\zetas}'(0)$ that contain $a'(0)$ vanish,
so
\[
 N_{\zetas}'(0)=a''(0)A_2(0)+\azero^2B'(0)=\xi''''(0)\Ustar^2+\azero^3\Wstar^2=\Dstar ,
\]
by~\eqref{eq:ons-moments} with $\alphas(0)=0$ and $V_0=0$. The denominator at zero is
$\sigma_{\zetas}(0)=2\azero^2A_3(0)=2\azero^2\Ustar^3=2\sqrt{\azero}$. Since
$N_{\zetas}(0)=0$, $\alphas'(0)=N_{\zetas}'(0)/(2\sqrt{\azero})=\Dstar/(2\sqrt{\azero})$. This
is positive by Lemma~\ref{lem:plat-fourth-sign}.
\end{proof}

\subsection{Marginal onset}\label{sec:ons-onset}

Fix $\theta\in(0,1)$ and write $\Gamma_\theta=\Gamma_{\zeta_\theta}$ and
$\Obs_\theta=\Obs_{\zeta_\theta}$. Since $\zetas(\{0\})=0$
(Proposition~\ref{prop:ons-equilibrium}(b)), the constraint is active,
$\zeta_\theta(\{0\})=\theta$ (Lemma~\ref{lem:cons-regular}(b)), and
Lemma~\ref{lem:cons-variational}(e) applies: the positive contact set $C_\theta$ is a
nonempty compact subset of $(0,1)$ that carries $\zeta_\theta|_{(0,1]}$, and its smallest
point $q_\theta$, the first positive contact, satisfies~\eqref{eq:cons-contact} and
$\Obs_\theta(0)>\Obs_\theta(q_\theta)$. At every positive contact $t$ we have
$\Gamma_\theta(t)=t$ and $\Gamma_\theta'(t)\le1$ by~\eqref{eq:plat-contact}; at $q_\theta$
this is the Plefka inequality~\eqref{eq:intro-plefka}.

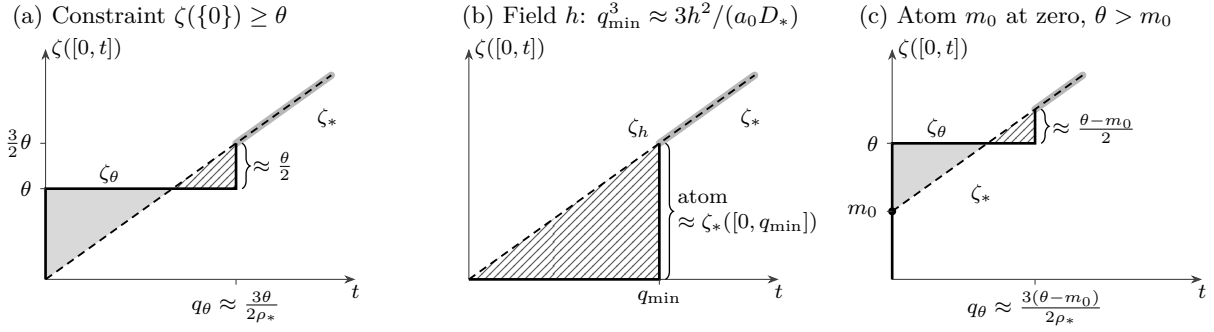
\begin{figure}[tp]
\centering
\begin{tikzpicture}[x=1cm, y=1cm, font=\scriptsize,
  >={Stealth[length=1.5mm,width=1.1mm]},
  axis/.style={->, line width=0.4pt, draw=black!75},
  eq/.style={line width=0.7pt, draw=black, dash pattern=on 3pt off 2pt},
  pert/.style={line width=1.0pt, draw=black, line join=miter},
  sketch/.style={line width=2.8pt, draw=black!28, line cap=round},
  excess/.style={fill=black!15},
  deficit/.style={pattern=north east lines, pattern color=black!65},
  brace/.style={decorate, decoration={brace, amplitude=3pt}, line width=0.5pt},
  tick/.style={line width=0.4pt, draw=black!75},
  ttl/.style={font=\footnotesize, anchor=south west, inner sep=0pt},
  dot/.style={circle, fill=black, inner sep=0pt, minimum size=3pt},
]
\begin{scope}
 \node[ttl] at (-0.45,3.3) {(a) Constraint $\zeta(\{0\})\ge\theta$};
 \fill[excess] (0,0) -- (0,1.2) -- (1.68,1.2) -- cycle;
 \fill[deficit] (1.68,1.2) -- (2.52,1.2) -- (2.52,1.8) -- cycle;
 \draw[axis] (0,0) -- (4.1,0) node[anchor=north, inner sep=2pt] {$t$};
 \draw[axis] (0,0) -- (0,3.05) node[anchor=west, inner sep=2pt] {$\zeta([0,t])$};
 \draw[sketch] (2.52,1.8) -- (3.78,2.7);
 \draw[eq] (0,0) -- (3.78,2.7);
 \draw[pert] (0,0) -- (0,1.2) -- (2.52,1.2) -- (2.52,1.8);
 \draw[tick] (0,1.2) -- (-0.07,1.2) node[anchor=east, inner sep=3pt] {$\theta$};
 \draw[tick] (0,1.8) -- (-0.07,1.8) node[anchor=east, inner sep=3pt] {$\frac32\theta$};
 \draw[tick] (2.52,0) -- (2.52,-0.07)
   node[anchor=north, inner sep=2pt] {$q_\theta\approx\frac{3\theta}{2\rhostar}$};
 \draw[brace] (2.6,1.8) -- (2.6,1.2);
 \node[anchor=west, inner sep=1pt] at (2.7,1.5) {$\approx\frac\theta2$};
 \node[anchor=south, inner sep=2pt] at (0.85,1.2) {$\zeta_\theta$};
 \node[anchor=north west, inner sep=1pt] at (3.55,2.3) {$\zetas$};
\end{scope}
\begin{scope}[xshift=5.6cm]
 \node[ttl] at (-0.1,3.3) {(b) Field $h$: $\qmin^3\approx3h^2/(\azero\Dstar)$};
 \fill[deficit] (0,0) -- (2.52,0) -- (2.52,1.8) -- cycle;
 \draw[axis] (0,0) -- (4.1,0) node[anchor=north, inner sep=2pt] {$t$};
 \draw[axis] (0,0) -- (0,3.05) node[anchor=west, inner sep=2pt] {$\zeta([0,t])$};
 \draw[sketch] (2.52,1.8) -- (3.78,2.7);
 \draw[eq] (0,0) -- (3.78,2.7);
 \draw[pert] (0,0) -- (2.52,0) -- (2.52,1.8);
 \draw[tick] (2.52,0) -- (2.52,-0.07) node[anchor=north, inner sep=2pt] {$\qmin$};
 \draw[brace] (2.6,1.8) -- (2.6,0);
 \node[anchor=west, align=left, inner sep=1pt] at (2.72,0.9)
   {atom\\ $\approx\zetas([0,\qmin])$};
 \node[anchor=south east, inner sep=2pt] at (2.48,1.8) {$\zeta_h$};
 \node[anchor=north west, inner sep=1pt] at (3.55,2.3) {$\zetas$};
\end{scope}
\begin{scope}[xshift=11.2cm]
 \node[ttl] at (-0.45,3.3) {(c) Atom $m_0$ at zero, $\theta>m_0$};
 \fill[excess] (0,0.9) -- (0,1.8) -- (1.26,1.8) -- cycle;
 \fill[deficit] (1.26,1.8) -- (1.89,1.8) -- (1.89,2.25) -- cycle;
 \draw[axis] (0,0) -- (4.1,0) node[anchor=north, inner sep=2pt] {$t$};
 \draw[axis] (0,0) -- (0,3.05) node[anchor=west, inner sep=2pt] {$\zeta([0,t])$};
 \draw[sketch] (1.89,2.25) -- (2.52,2.7);
 \draw[eq] (0,0.9) -- (2.52,2.7);
 \node[dot] at (0,0.9) {};
 \draw[pert] (0,0) -- (0,1.8) -- (1.89,1.8) -- (1.89,2.25);
 \draw[tick] (0,0.9) -- (-0.07,0.9) node[anchor=east, inner sep=3pt] {$m_0$};
 \draw[tick] (0,1.8) -- (-0.07,1.8) node[anchor=east, inner sep=3pt] {$\theta$};
 \draw[tick] (1.89,0) -- (1.89,-0.07)
   node[anchor=north, inner sep=2pt] {$q_\theta\approx\frac{3(\theta-m_0)}{2\rhostar}$};
 \draw[brace] (1.97,2.25) -- (1.97,1.8);
 \node[anchor=west, inner sep=1pt] at (2.08,2.0) {$\approx\frac{\theta-m_0}2$};
 \node[anchor=south, inner sep=2pt] at (0.6,1.8) {$\zeta_\theta$};
 \node[anchor=north west, inner sep=1pt] at (1.0,1.3) {$\zetas$};
\end{scope}
\end{tikzpicture}
\caption{Three perturbations of the bottom of the Parisi measure, schematically:
distribution functions $t\mapsto\zeta([0,t])$ near $t=0$, dashed for the Parisi measure
$\zetas$, of slope $\rhostar$ at zero; solid for the perturbed measure, thick grey
(schematic) above its first positive contact. The perturbed function lies above the
dashed one where shaded and below it where hatched.
(a)~The constrained minimizer $\zeta_\theta$ for small $\theta$ under the hypotheses of
Theorem~\ref{thm:local}: an atom $\theta$ at zero, no mass in $(0,q_\theta)$ and, to
leading order, an atom $\frac\theta2$ at $q_\theta\approx\frac{3\theta}{2\rhostar}$
(Proposition~\ref{prop:ons-onset}).
(b)~The minimizer $\zeta_h$ in a small field $h$ (Theorem~\ref{thm:field}): no mass in
$[0,\qmin(h))$, then an atom carrying, to leading order, the mass $\zetas([0,\qmin(h)])$;
the edge $\qmin(h)$ is a marginal contact.
(c)~A near-critical mixture with a cubic term and $m_0=\zetas(\{0\})>0$, constrained at
$\theta$ slightly above $m_0$ (Theorem~\ref{thm:atom}(c),
Proposition~\ref{prop:atom-odd-onset}): the first contact and its atom are those of (a)
with $\theta-m_0$ in place of $\theta$.
To leading order the curves rejoin at the first positive contact; in (a) and (c) this is
the balance $\theta+\frac\theta2=\rhostar q_\theta$, respectively
$\theta+\frac{\theta-m_0}2=m_0+\rhostar q_\theta$, in (b) it is
Theorem~\ref{thm:field}(b). In (a), if $\zetas([0,t])=\rhostar t$, the onset
$q_\theta=\frac{3\theta}{2\rhostar}$ is exactly the condition
$\int_0^{q_\theta}t\,\{\zeta_\theta([0,t])-\zetas([0,t])\}\,dt=0$: the two triangles have
opposite first moments. Their total area is
$\frac{5\theta^2}{8\rhostar}=\frac{5\sqrt{\azero}}{4\Dstar}\,\theta^2$, so
$\Wone(\zeta_\theta,\zetas)$ is at least of order $\theta^2$, whereas $p(\theta)-\feq$ is
of order $\theta^5$. For $\xi\in\Kcl(\varepsilon_*)$, Theorem~\ref{thm:nearcrit}(c) shows
that this area is $\Wone(\zeta_\theta,\zetas)$ up to a factor $1+O(\theta)$ and that the
quantile functions agree above the first contact up to $O(\theta^5)$
(Section~\ref{sec:nc-stability}).}
\label{fig:bottom}
\end{figure}

Proposition~\ref{prop:ons-onset}(b) fixes the shape of $\zeta_\theta$ at its first contact:
as for $\zetas$, Lemma~\ref{lem:plat-local-gap} rules out two contacts near zero with no mass
between them, and since $\zeta_\theta$ is close to $\zetas$, which charges every subinterval
of $[0,\kappa]$, the contacts of $\zeta_\theta$ fill an interval $[q_\theta,s_\theta]$, with
$\Gamma_\theta'=1$ at its left end. At a general contact the Plefka inequality may be strict
(Theorem~\ref{thm:nearcrit}(b)); near zero it is an equality. Part~(c) is a computation on
the plateau $(0,q_\theta)$. There $h_\theta=\Gamma_\theta-\id$ vanishes at zero and, with its
derivative, at $q_\theta$, and $h_\theta'''\approx\Dstar$; so $h_\theta$ is close to the
universal cubic $\frac{\Dstar}6\,t\,(q_\theta-t)^2$ (Figure~\ref{fig:plateau}(b)), whose
second derivative is $-\frac23\Dstar q_\theta$ at zero and $\frac13\Dstar q_\theta$ at
$q_\theta$. Matching the first with the exact value
$\Gamma_\theta''(0+)=-2\azero^2\theta\,\Phi_{\zeta_\theta,xx}(0,0)^3$ produced by the atom at
zero fixes $q_\theta$; the drop of $\Gamma_\theta''$ from $\frac13\Dstar q_\theta$ to zero at
$q_\theta$ is carried, in the jump formula~\eqref{eq:plat-jump}, by the atom
$\zeta_\theta(\{q_\theta\})$, which fixes its size.

\begin{proof}[Proof of Proposition~\ref{prop:ons-onset}]
(a) By Lemma~\ref{lem:cons-variational}(a), $\feq\le p(\theta)\le\feq+\frac12\xi''(1)\,\theta$,
and $\theta\mapsto\zeta_\theta$ is weakly continuous at $\theta=0$, where $\zeta_0=\zetas$. Since
$\Wone$ metrizes the weak topology, $\Wone(\zeta_\theta,\zetas)\to0$. Let $\kappa$ be as in
Proposition~\ref{prop:ons-equilibrium}(a), and let $0<\varepsilon<\kappa$. The open sets
$(0,\varepsilon)$ and $(\kappa/2,\kappa)$ meet $\supp\zetas\supset[0,\kappa]$, so they have
positive $\zetas$-mass. By the portmanteau theorem, they have positive $\zeta_\theta$-mass
for all small $\theta$. Hence $\zeta_\theta$ has support points in both intervals. These
points lie in $C_\theta$, which is closed and carries $\zeta_\theta|_{(0,1]}$; so
$q_\theta<\varepsilon$ for all small $\theta$.

(b) By~(a), we may choose $\theta_0\in(0,1)$ such that, for $0<\theta<\theta_0$,
$\Wone(\zeta_\theta,\zetas)<r_{\rm loc}$, $q_\theta<\kappa/2$, and $\zeta_\theta$ has a
support point $s_\theta\in(\kappa/2,\kappa)$. Then $s_\theta\in C_\theta$, and
$s_\theta<\kappa<K_{\rm loc}$. Suppose some $t\in(q_\theta,s_\theta)$ is not in $C_\theta$.
Put $u=\max(C_\theta\cap[q_\theta,t])$ and $v=\min(C_\theta\cap[t,s_\theta])$. Then $u<t<v$
and $(u,v)\cap C_\theta=\emptyset$, so $\zeta_\theta((u,v))=0$. The points
$u,v\in(0,K_{\rm loc})$ are minima of $\Obs_\theta$, with $\Gamma_\theta=\id$ and
$\Gamma_\theta'\le1$ there. This contradicts Lemma~\ref{lem:plat-local-gap}. Thus
$[q_\theta,s_\theta]\subset C_\theta$, and $\Gamma_\theta=\id$ on $[q_\theta,s_\theta]$.
Since $\Gamma_\theta$ is continuously differentiable, $\Gamma_\theta'(q_\theta)=1$, and
$\Gamma_\theta'=a\,\E U^2$ (Lemma~\ref{lem:plat-ito}) gives the displayed form of
marginality.

(c) Fix $0<\theta<\theta_0$, and write $q=q_\theta$ and
$\bar U_\theta=\Phi_{\zeta_\theta,xx}(0,0)$. By~(b) and~\eqref{eq:cons-contact},
$\zeta_\theta((0,q))=0$ and $\Gamma_\theta=\id$ on $[q,s_\theta]$, so
Corollary~\ref{cor:plat-first-plateau} applies to $\zeta_\theta$ with $m=\theta$ and
$D=\Dstar$. Since $\xi'''(0)=0$, \eqref{eq:plat-zero} reads
\begin{equation}\label{eq:ons-onset-zero}
 \Gamma_\theta''(0+)=-2\azero^2\theta\,\bar U_\theta^3,\qquad
 \bar U_\theta\longrightarrow\Ustar=\azero^{-1/2},
\end{equation}
where the limit follows from~\eqref{eq:plat-jet-stability} and~(a). The last bound of the
corollary and~(a) give
\[
 \eta_\theta:=\sup_{0<t<q_\theta}\bigl|\Gamma_\theta'''(t)-\Dstar\bigr|
 \le C_\Psi\bigl(q_\theta^{1/2}+\Wone(\zeta_\theta,\zetas)\bigr)+6\theta^2\azero^3\Ustar^4
 \longrightarrow0 .
\]

\emph{The onset.} Part~(a) of the corollary, which rests on the exact
identity~\eqref{eq:plat-onset-identity}, gives
\[
 \Bigl|2\azero^2\theta\,\bar U_\theta^3-\frac23\Dstar q\Bigr|\le\frac23\,\eta_\theta\,q .
\]
Since $\Dstar>0$ (Lemma~\ref{lem:plat-fourth-sign}),
\[
 \frac{q_\theta}{\theta}=\frac{3\azero^2\bar U_\theta^3}{\Dstar+O(\eta_\theta)}
 \longrightarrow\frac{3\azero^2\Ustar^3}{\Dstar}=\frac{3\sqrt{\azero}}{\Dstar}
 =\frac{3}{2\rhostar}.
\]

\emph{The atom at the contact.} Part~(b) of the corollary, which rests on the jump
formula~\eqref{eq:plat-jump} and on $\Gamma_\theta'\equiv1$ on $[q,s_\theta]$, gives
$|2a(q)^2\zeta_\theta(\{q\})\E U_q^3-\frac13\Dstar q|\le\frac53\eta_\theta q$.
By~\eqref{eq:plat-origin} with $k=2$, $\nu=\zetas$ and $0<U\le1$,
$a(q_\theta)^2\E U_{q_\theta}^3\to\azero^2\Ustar^3=\sqrt{\azero}$. Hence
\[
 \frac{\zeta_\theta(\{q_\theta\})}{\theta}
 =\frac{\{\frac13\Dstar+O(\eta_\theta)\}\,q_\theta/\theta}{2\sqrt{\azero}+o(1)}
 \longrightarrow\frac{\Dstar}{3}\cdot\frac{3\sqrt{\azero}}{\Dstar}\cdot
 \frac{1}{2\sqrt{\azero}}=\frac12 .
\qedhere
\]
\end{proof}

Propositions~\ref{prop:ons-onset}(c) and~\ref{prop:ons-equilibrium}(c) give
$\zeta_\theta([0,q_\theta])=\frac32\theta+o(\theta)$ and
$\alphas(q_\theta)=\rhostar q_\theta+o(q_\theta)=\frac32\theta+o(\theta)$: to leading order
the distribution function of $\zeta_\theta$ rejoins that of $\zetas$ at the first positive
contact, and the constraint gathers the mass $\frac32\theta$ that $\zetas$ spreads over
$[0,q_\theta]$ into atoms $\theta$ at zero and $\frac12\theta$ at $q_\theta$
(Figure~\ref{fig:bottom}(a)). Above $q_\theta$, Proposition~\ref{prop:ons-onset} says only
that $[q_\theta,s_\theta]$ consists of contacts; for $\xi\in\Kcl(\varepsilon_*)$,
Theorem~\ref{thm:nearcrit}(c) shows that the quantile functions of $\zeta_\theta$ and
$\zetas$ agree above $x_\theta=\zeta_\theta([0,q_\theta])$ up to $O(\theta^5)$. The same
calculus describes two other perturbations of the bottom of the measure. In a small field
$h$ the mass of $\zetas$ on $[0,\qmin(h)]$ is gathered, to leading order, into one atom at
$\qmin(h)$ (Theorem~\ref{thm:field}(b), Figure~\ref{fig:bottom}(b)). The near-critical
mixtures with a cubic term of Theorem~\ref{thm:atom}(c) have an atom $m_0=\zetas(\{0\})$ at
zero; at levels $f$ slightly above $\feq$, $\theta_f>m_0$ and $\zeta_f$ has a marginal first
contact $q_f$ of order $\theta_f-m_0$ carrying an atom $\frac12(\theta_f-m_0)$ to leading
order (Section~\ref{sec:atom-odd}, Figure~\ref{fig:bottom}(c)).

\subsection{The area of the first plateau}\label{sec:ons-area}

Since $\zetas(\{0\})=0$, the envelope identity~\eqref{eq:cons-envelope} holds at every
$\theta\in(0,1)$, and Lemma~\ref{lem:fifth-area} evaluates its right side
$\frac12\int_0^{q_\theta}\xi''(s)\{\Gamma_\theta(s)-s\}\,ds$. To leading order,
$\Gamma_\theta-\id$ is the universal cubic $\frac{\Dstar}6\,t(q_\theta-t)^2$ of
Lemma~\ref{lem:plat-cubic} on the first plateau, whose integral is $\Dstar q_\theta^4/72$,
and $\xi''\approx\azero$ there (Figure~\ref{fig:plateau}(b),(c));
Corollary~\ref{cor:plat-first-plateau} makes this quantitative from the three boundary values
and the uniform limit of the third derivative.
With the lemma, all four inputs of Section~\ref{sec:fifth-inputs} are proved.

\begin{proof}[Proof of Lemma~\ref{lem:fifth-area}]
In this proof $C$ denotes constants that depend only on $\xi$. Fix $\theta\in(0,\theta_0)$,
with $\theta_0$ as in Proposition~\ref{prop:ons-onset}(b), and write $q=q_\theta$ and
$g=\Gamma_\theta-\id$ on $[0,q]$. As in the proof of Proposition~\ref{prop:ons-onset}(c),
Corollary~\ref{cor:plat-first-plateau} applies to $\zeta_\theta$ with $m=\theta$ and
$D=\Dstar$: the measure has no mass in $(0,q)$ by~\eqref{eq:cons-contact}, and
$\Gamma_\theta=\id$ on an interval $[q,s_\theta]$ by Proposition~\ref{prop:ons-onset}(b). Its
error parameter is the quantity $\eta_\theta$ of that proof, which tends to zero as
$\theta\downarrow0$, as shown there, because $q_\theta\to0$, $\zeta_\theta\Rightarrow\zetas$
(Proposition~\ref{prop:ons-onset}(a)) and $\theta\to0$. Part~(c) of the corollary, which rests on the kernel
identity~\eqref{eq:plat-kernel}, gives
\begin{equation}\label{eq:fifth-area}
 \Bigl|\int_0^qg(s)\,ds-\frac{\Dstar}{72}\,q^4\Bigr|\le\frac{\eta_\theta}{72}\,q^4,
 \qquad |g(s)|\le\Bigl(\frac{\Dstar}6+\frac56\,\eta_\theta\Bigr)q^3\quad(0\le s\le q).
\end{equation}
Since $\xi'''(0)=0$, Taylor's formula gives $|\xi''(s)-\azero|\le Cs^2$ on $[0,1]$. The
envelope identity and the pointwise bound in~\eqref{eq:fifth-area} give
\begin{equation}\label{eq:fifth-envelope-error}
 \Bigl|p'(\theta)-\frac{\azero}2\int_0^qg(s)\,ds\Bigr|
 \le\frac12\int_0^q|\xi''(s)-\azero|\,|g(s)|\,ds\le Cq^6
\end{equation}
for small $\theta$. Together with the area bound in~\eqref{eq:fifth-area},
$p'(\theta)=\frac{\azero\Dstar}{144}\,q^4\{1+O(\eta_\theta)\}+O(q^6)$. Since $\Dstar>0$,
$\eta_\theta\to0$ and $q_\theta\to0$, this is the claim.
\end{proof}

\section{A small external field: proof of Theorem~\ref*{thm:field}}\label{sec:fld}

The constraint forces mass at zero; an external field $h\ne0$ does the opposite and removes zero
from the support (Lemma~\ref{lem:fld-diffusion}(c)). For small $|h|$ the support of $\zeta_h$
therefore starts after a gap $[0,\qmin(h))$, with an atom at $\qmin(h)$ that carries, to leading
order, the whole zero-field mass of $[0,\qmin(h)]$, followed by an interval of support. The edge
$\qmin(h)$ is a marginal contact of order $|h|^{2/3}$, and the free energy acquires a term of order
$|h|^{10/3}$ (Figure~\ref{fig:bottom}(b)). We prove Theorem~\ref{thm:field} by running the plateau
calculus of Section~\ref{sec:plat} from the starting point $h$. The hypotheses are those of
Theorem~\ref{thm:local}: radius of convergence greater than one, $\beta_2>0$, $\xi'''(0)=0$ and
\condH; nothing else is assumed.

\begin{remark}[No nondegeneracy hypothesis]\label{rem:fld-nondegeneracy}
The proof needs $\Dstar>0$, for the uniform convexity of Lemma~\ref{lem:fld-plateau}, to
divide by $\Dstar$ in part~(b), and for $\lambda>0$ in part~(d). This is automatic: since
$\beta_2>0$, Lemma~\ref{lem:plat-fourth-sign}, which holds for every probability measure on
$[0,1]$ and uses neither \condH\ nor $\xi'''(0)=0$, gives $\Dstar\ge\azero^3\kappa_0^2>0$.
Likewise an interval of support
at zero and a positive density there are not assumed: under \condH\ both are conclusions of
Theorem~\ref{thm:local} (Proposition~\ref{prop:ons-equilibrium}).
\end{remark}

\paragraph{Idea of the proof.}
Let $q=\qmin(h)$. The gap $[0,q)$ is a plateau of height zero. On it the diffusion
\eqref{eq:fld-sde} has no drift, so its value at the edge is exactly Gaussian, with law
$\mathcal N(h,\xi'(q))$. The contact condition $\Gamma_{\zeta_h,h}(q)=q$, the marginality
condition of part~(a) and the definition of the magnetization then become three exact moment
identities for the odd profile $v=\Phi_{\zeta_h,x}(q,\cdot)$ under this law
(see \eqref{eq:fld-exact}):
\[
 \E v(X)^2=q,\qquad \E v'(X)^2=\frac1{\xi''(q)},\qquad \E v(X)=\mfield(h),
 \qquad X\sim\mathcal N\bigl(h,\xi'(q)\bigr).
\]
In the combination $\E v(X)^2-\xi'(q)\,\E v'(X)^2$ the terms of order $\xi'(q)$ and $\xi'(q)^2$
of the Gaussian expansion of $v$ cancel, and what remains balances $h^2/\azero$ against
$\Dstar q^3/3$ (Lemma~\ref{lem:fld-taylor}). This is the exponent $2/3$. The constant is the
$\Dstar$ of the constrained onset because the same jets $\Ustar$ and $\Wstar$ of $\Phi_{\zetas}$ at
the origin enter; Remark~\ref{rem:fld-plateau-picture} reads the exponent directly off the
plateau calculus. The identities presuppose part~(a): the bottom of the support of $\zeta_h$ is a
single interval whose lower edge is a marginal contact. Its proof is the gap exclusion of
Section~\ref{sec:plat-gap}, made uniform in a small starting point $h$ of the diffusion.

Section~\ref{sec:fld-setting} states three routine lemmas on the diffusion started at $h$, proved
in Appendix~\ref{app:field}; Section~\ref{sec:fld-proof} proves part~(a),
Section~\ref{sec:fld-identities} parts (b)--(d), and Section~\ref{sec:fld-sk} the last sentence of
Theorem~\ref{thm:field}.

\subsection{Setting and inputs}\label{sec:fld-setting}

Throughout this section $\xi(t)=\sum_{p\ge2}\beta_p^2t^p$ has radius of convergence greater than
one, $\azero=\xi''(0)=2\beta_2^2>0$ and $\xi'''(0)=6\beta_3^2=0$; condition \condH\ is assumed only
where stated. Put $a(t)=\xi''(t)$. Since all coefficients are nonnegative, $\azero\le a\le\xi''(1)$
on $[0,1]$, and $\xi'$ is strictly increasing there. For a probability measure $\zeta$ on $[0,1]$
we write $\alpha_\zeta(t)=\zeta([0,t])$, and $\Phi_\zeta$ is the solution of
\eqref{eq:intro-pde}; it does not depend on the field.

\paragraph{The free energy in a field.}
For every $h\in\R$ the functional $P^h_\zeta$ of Theorem~\ref{thm:field} has a unique minimizer
$\zeta_h$ \cite[Theorem~2 and Corollary~1]{AC15b}. Moreover $P^0_\zeta=P_\zeta$, $\zeta_0=\zetas$
and $\Ffield(0)=\feq$. Let $H_N$ be as in \eqref{eq:intro-covariance}. Whenever the Parisi formula
with an external field \cite{Tal06,Pan14} (see also \cite[Theorem~1]{AC15b}) applies, for instance for every polynomial $\xi$,
\begin{equation}\label{eq:fld-free-energy}
 \lim_{N\to\infty}\frac1N\log\sum_{\sigma\in\{-1,1\}^N}
 \exp\Bigl\{H_N(\sigma)+h\sum_{i=1}^N\sigma_i\Bigr\}=\log2+\Ffield(h)
 \qquad\text{almost surely,}
\end{equation}
so that $\Ffield$ is the free energy in the field $h$ and $\mfield$ is the magnetization.
Theorem~\ref{thm:field} concerns only the variational quantities $\Ffield$ and $\zeta_h$; the
formula \eqref{eq:fld-free-energy} serves to interpret them.

The diffusion of Theorem~\ref{thm:field} is \eqref{eq:intro-diffusion} started at $h$,
\begin{equation}\label{eq:fld-sde}
 dX_t=\xi''(t)\alpha_\zeta(t)\Phi_{\zeta,x}(t,X_t)\,dt+\sqrt{\xi''(t)}\,dB_t,
 \qquad X_0=h .
\end{equation}
We write $X^{\zeta,h}$ for its solution and
\begin{equation}\label{eq:fld-obstacle}
 \Gamma_{\zeta,h}(t)=\E\Phi_{\zeta,x}(t,X^{\zeta,h}_t)^2,\qquad
 \Obs_{\zeta,h}(t)=\frac12\int_t^1\xi''(s)\bigl\{\Gamma_{\zeta,h}(s)-s\bigr\}\,ds ,
\end{equation}
so that $\Gamma_{\zeta,0}=\Gamma_\zeta$ and $\Obs_{\zeta,0}=\Obs_\zeta$ are the objects of
\eqref{eq:intro-obstacle}. As in Section~\ref{sec:plat}, $U_t$, $V_t$ and $W_t$ denote
$\Phi_{\zeta,xx}$, $\Phi_{\zeta,xxx}$ and $\Phi_{\zeta,xxxx}$ evaluated at $(t,X^{\zeta,h}_t)$.

\paragraph{Facts about the Parisi equation.}
Let $\zeta$ be a probability measure on $[0,1]$; the results of \cite{AC15a,AC15b} apply under our
radius assumption (Section~\ref{sec:cons-standard}). If $\zeta$ has finitely many atoms and
$\alpha_\zeta=m$ on $[t_l,t_{l+1})$, where $t_{l+1}$ is the next atom or $t_{l+1}=1$, then for
$t_l\le s<t_{l+1}$ the Cole--Hopf form \cite[(11)--(12)]{AC15a} reads
\[
 \Phi_\zeta(s,x)=\frac1m\log\E\exp\Bigl\{m\,\Phi_\zeta\bigl(t_{l+1},
 x+\sqrt{\xi'(t_{l+1})-\xi'(s)}\,Z\bigr)\Bigr\},
\]
with $Z$ standard Gaussian, read as $\E\Phi_\zeta(t_{l+1},x+\sqrt{\xi'(t_{l+1})-\xi'(s)}\,Z)$ when
$m=0$. For general $\zeta$, $\Phi_\zeta$ is the uniform limit of $\Phi_{\zeta_n}$ along finitely
atomic $\zeta_n\to\zeta$. By Lemma~\ref{lem:cons-standard}(a) and \eqref{eq:plat-bounds},
$\Phi_\zeta(t,\cdot)$ is even, its spatial derivatives are continuous and converge uniformly along
weakly convergent sequences of measures, $|\Phi_{\zeta,x}|\le1$, $0<\Phi_{\zeta,xx}\le1$ and
$|\partial_x^j\Phi_\zeta|\le B_j$ on $[0,1]\times\R$ for $j\ge1$; and jet stability
\eqref{eq:plat-jet-stability} gives
$\|\partial_x^k(\Phi_\zeta-\Phi_\nu)\|_\infty\le C_k\|a\|_\infty\Wone(\zeta,\nu)$ for $k\ge1$ and
every probability measure $\nu$ on $[0,1]$.

\paragraph{The diffusion started at $h$.}
The next two lemmas extend to \eqref{eq:fld-sde} facts that are standard for $h=0$, and the third
records the dependence on $h$, by the routine envelope argument. They are proved in
Appendix~\ref{app:field}; in the first two the starting point enters only through the symmetry
$h\mapsto-h$, the initial gap and the sign of the obstacle at zero.

\needspace{4\baselineskip}
\begin{lemma}[Symmetry, the initial gap and optimality]\label{lem:fld-diffusion}
Let $\zeta$ be a probability measure on $[0,1]$ and $h\in\R$.
\begin{thmlist}
\item Equation~\eqref{eq:fld-sde} has a pathwise unique continuous solution, and $X^{\zeta,-h}$
 has the law of $-X^{\zeta,h}$. Consequently $\Gamma_{\zeta,-h}=\Gamma_{\zeta,h}$,
 $\E\Phi_{\zeta,xx}(t,X^{\zeta,-h}_t)^2=\E\Phi_{\zeta,xx}(t,X^{\zeta,h}_t)^2$,
 $P^{-h}_\zeta=P^h_\zeta$ and $\zeta_{-h}=\zeta_h$.
\item If $0<q\le1$ and $\zeta([0,q))=0$, then $X^{\zeta,h}_q\sim\mathcal N(h,\xi'(q))$ and
 \begin{equation}\label{eq:fld-heat}
  \Phi_\zeta(t,x)=\E\Phi_\zeta\bigl(q,x+\sqrt{\xi'(q)-\xi'(t)}\,Z\bigr)\qquad(0\le t\le q).
 \end{equation}
\item The measure $\zeta_h$ gives full mass to the set of minimizers of $\Obs_{\zeta_h,h}$ over
 $[0,1]$, and every $q\in\supp\zeta_h\cap(0,1)$ satisfies
 \begin{equation}\label{eq:fld-sc}
  \Gamma_{\zeta_h,h}(q)=q,\qquad \xi''(q)\,\E\Phi_{\zeta_h,xx}(q,X^{\zeta_h,h}_q)^2\le1 .
 \end{equation}
 If $h\ne0$, then $0\notin\supp\zeta_h$.
\end{thmlist}
\end{lemma}

For $h\ge0$ the first statement of (c) is \cite[Proposition~1.1]{JT17}, once the weak solution
of the Parisi equation used there is identified with $\Phi_\zeta$, and the last statement is
\cite[Lemma~3.9]{JT17}. We use \eqref{eq:fld-sc} only at points of $(0,1)$, where it follows
from the first statement by the contact calculus of Section~\ref{sec:plat-obstacle}.

\begin{lemma}[It\^o identities from a nonzero starting point]\label{lem:fld-ito}
Let $\zeta$ be a probability measure on $[0,1]$, $h\in\R$, $X=X^{\zeta,h}$, $\Gamma=\Gamma_{\zeta,h}$
and $\alpha=\alpha_\zeta$.
\begin{thmlist}
\item For every polynomial $\mathsf P$ and every $j\ge1$, the function
 $t\mapsto\E\,\mathsf P(\partial_x\Phi_\zeta,\dots,\partial_x^j\Phi_\zeta)(t,X_t)$ is continuous
 on $[0,1]$.
\item $\Gamma\in C^1[0,1]$ and $\Gamma'(t)=a(t)\E U_t^2$.
\item For every $t\in[0,1)$, $\Gamma'$ has the right derivative
 $a'(t)\E U_t^2+a(t)^2\E\bigl(V_t^2-2\alpha(t)U_t^3\bigr)$.
\item If $\alpha=m$ on an open interval $I\subset(0,1)$, then $\Gamma\in C^3(I)$ and, on $I$,
 \begin{align}
  \Gamma''&=a'\E U^2+a^2\E\bigl(V^2-2mU^3\bigr),\label{eq:fld-second}\\
  \Gamma'''&=a''\E U^2+3aa'\E\bigl(V^2-2mU^3\bigr)
   +a^3\E\bigl(W^2-12mUV^2+6m^2U^4\bigr).\label{eq:fld-third}
 \end{align}
\end{thmlist}
\end{lemma}

For $h=0$ these are the identities of Section~\ref{sec:plat-ito}; see also
\cite[Proposition~3]{AC15a}. For $h\ge0$, parts (b) and (c) are also stated in
\cite[Lemma~8.4]{JT17}, whose proofs are omitted there; part~(d) is not stated there. For the Sherrington--Kirkpatrick covariance and $h\ge0$, parts (b) and (d), and part (c) in almost-everywhere form, are also in \cite[Lemma~2.1]{Lop26b}, in the time variable $\beta^2t$.

\begin{lemma}[Continuity in the field]\label{lem:fld-continuity}
The map $h\mapsto\zeta_h$ is weakly continuous, and $\zeta_{-h}=\zeta_h$. The function $\Ffield$
is even and continuously differentiable with $\Ffield'=\mfield$, and $\mfield$ is odd and
continuous.
\end{lemma}

\subsection{The bottom of the support}\label{sec:fld-proof}

We prove part~(a) of Theorem~\ref{thm:field}. The gap exclusion of Section~\ref{sec:plat-gap}
rests on the positivity of the third derivative of $\Gamma$ on short plateaus near the origin
(Lemma~\ref{lem:plat-local-plateau}). We need this positivity for the diffusion started at a small
$h$, uniformly in $h$. For a probability measure $\zeta$ on $[0,1]$, $h\in\R$, $t\in[0,1]$ and
$m\in[0,1]$, let $\Psi_{\zeta,h}(t,m)$ be the right side of \eqref{eq:plat-psi} with $U,V,W$
taken along $X^{\zeta,h}$, that is, the right side of \eqref{eq:fld-third} at time $t$ with the
height $m$ as a free parameter. Thus $\Gamma_{\zeta,h}'''=\Psi_{\zeta,h}(\cdot,m)$ on every open
interval on which $\alpha_\zeta=m$, and $\Psi_{\zeta,0}=\Psi_\zeta$ is the quantity controlled in
Lemma~\ref{lem:plat-local-plateau}.

\begin{lemma}[Plateau positivity, uniformly in a small field]\label{lem:fld-plateau}
As $t\downarrow0$, $h\to0$ and $\Wone(\zeta,\zetas)\to0$,
\[
 \sup_{0\le m\le1}\bigl|\Psi_{\zeta,h}(t,m)-\Dstar-6m^2\azero^3\Ustar^4\bigr|\longrightarrow0 .
\]
Consequently there are $K_{\rm fld}\in(0,1)$, $h_0>0$ and $\varepsilon_{\rm fld}>0$ with the following property.
If $|h|\le h_0$, $\Wone(\zeta,\zetas)\le\varepsilon_{\rm fld}$ and $\alpha_\zeta$ is constant on an open
interval $I\subset(0,K_{\rm fld})$, then $\Gamma_{\zeta,h}'''\ge\Dstar/2$ on $I$.
\end{lemma}

It is proved in Appendix~\ref{app:field} by a continuity argument: for small $t$ the diffusion
$X^{\zeta,h}_t$ stays near $h$ with high probability, so the jets along it are close to those of
$\Phi_{\zetas}$ at the origin.

\begin{lemma}[No gaps near the origin]\label{lem:fld-gap}
Let $K_{\rm fld},h_0,\varepsilon_{\rm fld}$ be as in Lemma~\ref{lem:fld-plateau}, and let $|h|\le h_0$ with
$\Wone(\zeta_h,\zetas)\le\varepsilon_{\rm fld}$. There are no two points $u<w$ of $\supp\zeta_h$ in
$(0,K_{\rm fld})$ with $(u,w)\cap\supp\zeta_h=\emptyset$.
\end{lemma}

\begin{proof}
Suppose there are, and write $\Gamma=\Gamma_{\zeta_h,h}$. On $(u,w)$ the distribution function of
$\zeta_h$ is constant, so $\Gamma'''\ge\Dstar/2>0$ there by Lemma~\ref{lem:fld-plateau}. Thus
$\Gamma'$ is strictly convex on $(u,w)$, and it is continuous on $[u,w]$ by
Lemma~\ref{lem:fld-ito}(b). Both endpoints are support points in $(0,1)$, so by
Lemma~\ref{lem:fld-diffusion}(c) they satisfy \eqref{eq:fld-sc}: $\Gamma(u)=u$, $\Gamma(w)=w$ and
$\Gamma'(u),\Gamma'(w)\le1$. The trapezoid argument of Lemma~\ref{lem:plat-obstacle}(d), which
uses only these four relations and the strict convexity of $\Gamma'$, gives a contradiction.
\end{proof}

\begin{lemma}[The bottom of the support for small fields]\label{lem:fld-structure}
Assume \condH. Then part~(a) of Theorem~\ref{thm:field} holds with $h_1\le1$ and
$\xi'(\kappa_1)\le1$, and $\qmin(h)\to0$ as $h\to0$.
\end{lemma}

\begin{proof}
By Theorem~\ref{thm:local} (Proposition~\ref{prop:ons-equilibrium}), there is $\kappa>0$ with
$[0,\kappa]\subset\supp\zetas$. Let $K_{\rm fld},h_0,\varepsilon_{\rm fld}$ be as in Lemma~\ref{lem:fld-plateau}.
Fix $\bar\kappa\in(0,\min\{\kappa,K_{\rm fld}\})$ with $\xi'(\bar\kappa)\le1$, and put
$\kappa_1=\bar\kappa/2$. For $\varepsilon\in(0,\kappa_1]$ the open intervals $(0,\varepsilon)$
and $(\kappa_1,\bar\kappa)$ lie in $\supp\zetas$, so they have positive $\zetas$-mass. Since
$\Wone$ metrizes weak convergence, Lemma~\ref{lem:fld-continuity} and the portmanteau theorem
give $h_1(\varepsilon)\in(0,1]$ such that, for $0<|h|<h_1(\varepsilon)$, both intervals have
positive $\zeta_h$-mass, $|h|\le h_0$ and $\Wone(\zeta_h,\zetas)\le\varepsilon_{\rm fld}$. For such $h$,
$\qmin(h)<\varepsilon$, and some $v\in\supp\zeta_h$ lies in $(\kappa_1,\bar\kappa)$. Letting
$\varepsilon\downarrow0$ proves $\qmin(h)\to0$.

Take $h_1=h_1(\kappa_1)$, let $0<|h|<h_1$, and put $q=\qmin(h)$, so that $q<\kappa_1$. By
Lemma~\ref{lem:fld-diffusion}(c), $q>0$, and $\zeta_h([0,q))=0$ by the definition of $q$. Suppose
that $t\in(q,v)\setminus\supp\zeta_h$, and let $u=\max(\supp\zeta_h\cap[q,t])$ and
$w=\min(\supp\zeta_h\cap[t,v])$. Then $0<q\le u<t<w\le v<K_{\rm fld}$ and $(u,w)$ misses the support, which
contradicts Lemma~\ref{lem:fld-gap}. Hence $[q,\kappa_1]\subset[q,v]\subset\supp\zeta_h$. By
\eqref{eq:fld-sc}, $\Gamma_{\zeta_h,h}(t)=t$ on $[q,v]$. Since $\Gamma_{\zeta_h,h}$ is $C^1$ with
$\Gamma_{\zeta_h,h}'(t)=\xi''(t)\E\Phi_{\zeta_h,xx}(t,X^{\zeta_h,h}_t)^2$
(Lemma~\ref{lem:fld-ito}(b)), this gives $\Gamma_{\zeta_h,h}'(q)=1$, the last identity of
part~(a).
\end{proof}

Part~(a) of Theorem~\ref{thm:field} is Lemma~\ref{lem:fld-structure}.

\subsection{Three identities at the edge of the gap}\label{sec:fld-identities}

We now derive parts (b)--(d) of Theorem~\ref{thm:field} from part~(a).
By Lemma~\ref{lem:fld-diffusion}(a) and Lemma~\ref{lem:fld-continuity}, $\zeta_h$ and $\qmin(h)$
are even in $h$ and $\mfield$ is odd. It therefore suffices to consider $0<h<h_1\le1$. Put
\[
 q=\qmin(h),\qquad \sigma^2=\xi'(q),\qquad X=h+\sigma Z,\qquad v=\Phi_{\zeta_h,x}(q,\cdot),
\]
with $Z$ standard Gaussian. Then $\sigma^2\le\xi'(\kappa_1)\le1$. The function $v$ is odd and
smooth, with $|v|\le1$ and $\|v^{(j)}\|_\infty\le B_{j+1}$ for $j\ge1$; these bounds do not depend
on $h$. Its Taylor coefficients at the origin satisfy
\begin{equation}\label{eq:fld-jets}
 c_1:=v'(0)=\Phi_{\zeta_h,xx}(q,0)\longrightarrow\Ustar,\qquad
 c_3:=v'''(0)=\Phi_{\zeta_h,xxxx}(q,0)\longrightarrow\Wstar\qquad(h\to0),
\end{equation}
because, for $k=2,4$,
\[
 \bigl|\partial_x^k\Phi_{\zeta_h}(q,0)-\partial_x^k\Phi_{\zetas}(0,0)\bigr|
 \le\|\partial_x^k(\Phi_{\zeta_h}-\Phi_{\zetas})\|_\infty
 +\bigl|\partial_x^k\Phi_{\zetas}(q,0)-\partial_x^k\Phi_{\zetas}(0,0)\bigr| ,
\]
and both terms tend to zero: the first by jet stability and Lemma~\ref{lem:fld-continuity}, the
second by continuity of $\partial_x^k\Phi_{\zetas}$, since $q\to0$. Since $\zeta_h([0,q))=0$,
Lemma~\ref{lem:fld-diffusion}(b) shows that $X^{\zeta_h,h}_q$ has the law of $X$, and
differentiating \eqref{eq:fld-heat} in $x$ at $(t,x)=(0,h)$ gives $\Phi_{\zeta_h,x}(0,h)=\E v(X)$.
Together with part~(a) this yields three exact identities:
\begin{equation}\label{eq:fld-exact}
 \E v(X)^2=\Gamma_{\zeta_h,h}(q)=q,\qquad
 \E v'(X)^2=\frac1{\xi''(q)},\qquad
 \E v(X)=\mfield(h).
\end{equation}
The covariance enters through the expansions, valid as $q\to0$ because $\xi'''(0)=0$ and $\xi$
is analytic near zero,
\begin{equation}\label{eq:fld-xi-exp}
 \begin{gathered}
 \xi'(q)=\azero q+O(q^3),\qquad
 \frac1{\xi''(q)}=\frac1\azero-\frac{\xi''''(0)}{2\azero^2}\,q^2+O(q^3),\\
 q-\frac{\xi'(q)}{\xi''(q)}=\frac{\xi''''(0)}{3\azero}\,q^3+O(q^4).
 \end{gathered}
\end{equation}
To expand the left sides of \eqref{eq:fld-exact} we use the following Gaussian moments of an odd
profile; the proof, a Taylor expansion, is in Appendix~\ref{app:field}.

\begin{lemma}[Gaussian moments of an odd profile]\label{lem:fld-taylor}
Let $v\in C^\infty(\R)$ be odd with $|v|\le1$ and $\|v^{(j)}\|_\infty\le C$ for $1\le j\le8$. Put
$c_1=v'(0)$, $c_3=v'''(0)$ and $c_5=v^{(5)}(0)$. Let $0<h\le1$, $0<\sigma\le1$ and $X=h+\sigma Z$
with $Z$ standard Gaussian. Then, with implied constants depending only on $C$,
\begin{align}
 \E v(X)&=h\Bigl(c_1+\frac{c_3}2\sigma^2+\frac{c_5}8\sigma^4+O(h^2+\sigma^6)\Bigr),
 \label{eq:fld-t1}\\
 \E v'(X)^2&=c_1^2+c_1c_3\sigma^2+\Bigl(\frac34c_3^2+\frac14c_1c_5\Bigr)\sigma^4
  +O(h^2+\sigma^6),\label{eq:fld-t2}\\
 \E v(X)^2-\sigma^2\E v'(X)^2&=c_1^2h^2-\frac{c_3^2}3\,\sigma^6
  +O(h^2\sigma^2+h^4+\sigma^8).\label{eq:fld-t3}
\end{align}
\end{lemma}

The point of \eqref{eq:fld-t3} is a cancellation: in $\E v(X)^2-\sigma^2\E v'(X)^2$ the terms of
order $\sigma^2$ and $\sigma^4$ cancel, so the first term without $h$ is of order $\sigma^6$, with
coefficient $-c_3^2/3$.

\begin{proof}[Proof of Theorem~\ref{thm:field}\textup{(b)}, the gap]
Insert \eqref{eq:fld-exact} into \eqref{eq:fld-t3}. The left side becomes
$q-\xi'(q)/\xi''(q)$, and $\sigma^6=\xi'(q)^3=\azero^3q^3+O(q^5)$. By \eqref{eq:fld-xi-exp},
\[
 h^2\bigl(c_1^2+O(q+h^2)\bigr)
 =\frac{q^3}3\Bigl(\frac{\xi''''(0)}{\azero}+\azero^3c_3^2\Bigr)+O(q^4).
\]
By \eqref{eq:fld-jets} and $\Ustar^2=\azero^{-1}$ (Theorem~\ref{thm:local}), the left side is
$h^2(\azero^{-1}+o(1))$, and the bracket on the right tends to
$\xi''''(0)\Ustar^2+\azero^3\Wstar^2=\Dstar$. Since $\Dstar>0$ and $q\to0$, this gives
$q^3=3h^2(\azero\Dstar)^{-1}(1+o(1))$.
\end{proof}

\begin{proof}[Proof of Theorem~\ref{thm:field}\textup{(b)}, the atom]
By part~(a), $\Gamma_{\zeta_h,h}'\equiv1$ on $[q,\kappa_1]$, so the right derivative of
$\Gamma_{\zeta_h,h}'$ at $q$ vanishes. Since $\zeta_h([0,q))=0$, we have
$\alpha_{\zeta_h}(q)=\zeta_h(\{q\})$. Lemma~\ref{lem:fld-ito}(c) at $t=q$, with
$U_q=v'(X^{\zeta_h,h}_q)$, $V_q=v''(X^{\zeta_h,h}_q)$ and the law of $X^{\zeta_h,h}_q$, gives
\[
 \zeta_h(\{q\})=\frac{a'(q)\E v'(X)^2+a(q)^2\E v''(X)^2}{2a(q)^2\E v'(X)^3}.
\]
Here $\E v'(X)^k\to\Ustar^k>0$ for $k=2,3$, because $|v'(x)-c_1|\le Cx^2$,
$\E X^2=h^2+\sigma^2\to0$ and $c_1\to\Ustar$ by \eqref{eq:fld-jets}. Since $v''$ is odd with $v'''(0)=c_3$, $|v''(x)-c_3x|\le C|x|^3$, and
therefore
\[
 \E v''(X)^2=c_3^2(h^2+\sigma^2)+O(h^4+\sigma^4)=\azero\Wstar^2q+o(q),
\]
by the first claim of part~(b), which gives $h^2=O(q^3)$, and by $\sigma^2=\azero q+O(q^3)$.
Also $a'(q)=\xi''''(0)q+O(q^2)$ and $a(q)=\azero+O(q^2)$, because $\xi'''(0)=0$. Thus
\[
 \zeta_h(\{q\})=\frac{\xi''''(0)\Ustar^2+\azero^3\Wstar^2+o(1)}{2\azero^2\Ustar^3+o(1)}\,q
 =\frac{\Dstar}{2\sqrt\azero}\,q\,(1+o(1))=\rhostar q\,(1+o(1)),
\]
using $\Ustar=\azero^{-1/2}$. Finally $\zetas([0,q])=\rhostar q(1+o(1))$, because, by
Theorem~\ref{thm:local}, $\alphas(0)=\zetas(\{0\})=0$ and $\alphas'(0)=\rhostar>0$.
\end{proof}

\begin{proof}[Proof of Theorem~\ref{thm:field}\textup{(c)}]
Square \eqref{eq:fld-t1} and subtract \eqref{eq:fld-t2}; the terms in $c_1c_5$ cancel, and
\[
 \Bigl(\frac{\mfield(h)}h\Bigr)^2-\E v'(X)^2=-\frac{c_3^2}2\,\sigma^4+O(h^2+\sigma^6).
\]
By part~(b), $h^2+\sigma^6=O(q^3)$. With \eqref{eq:fld-exact} and \eqref{eq:fld-xi-exp},
\[
 \Bigl(\frac{\mfield(h)}h\Bigr)^2
 =\frac1\azero-\frac{q^2}{2\azero}\Bigl(\frac{\xi''''(0)}{\azero}+\azero^3c_3^2\Bigr)+O(q^3)
 =\frac1\azero\Bigl(1-\frac\Dstar2\,q^2\,(1+o(1))\Bigr).
\]
Since $\Phi_{\zeta_h,x}(0,\cdot)$ is odd and strictly increasing, $\mfield(h)/h>0$;
square roots give part~(c).
\end{proof}

\begin{proof}[Proof of Theorem~\ref{thm:field}\textup{(d)}]
By Lemma~\ref{lem:fld-continuity}, $\Ffield$ is even and continuously differentiable with
$\Ffield'=\mfield$. By parts (b) and (c), for $0<s<h_1$,
\[
 \mfield(s)=\frac s{\sqrt\azero}-\lambda s^{7/3}\bigl(1+\epsilon(s)\bigr),\qquad
 \lambda=\frac\Dstar{4\sqrt\azero}\Bigl(\frac3{\azero\Dstar}\Bigr)^{2/3},
\]
where $\epsilon$ is continuous on $(0,h_1)$, because $\mfield$ is, and $\epsilon(s)\to0$ as
$s\downarrow0$. Hence $\mfield(h)/h\to\azero^{-1/2}$; since $\Ffield'(0)=\mfield(0)=0$ and
$\mfield$ is odd, this limit is $\Ffield''(0)$. Integrating,
\[
 \Ffield(h)-\Ffield(0)=\int_0^h\mfield(s)\,ds
 =\frac{h^2}{2\sqrt\azero}-\frac3{10}\,\lambda\,h^{10/3}(1+o(1))\qquad(h\downarrow0),
\]
where
\[
 \frac3{10}\,\lambda=\frac3{40}\,3^{2/3}\Dstar^{1/3}\azero^{-7/6}
 =\frac3{40}\,(9\Dstar)^{1/3}\azero^{-7/6} .
\]
Evenness gives the same expansion with $|h|$ for $h<0$.
If $\Ffield$ had an expansion $\sum_{j\le4}c_jh^j+o(h^4)$ at zero, comparison with this display
for $h>0$ would force $c_1=c_3=0$ and $c_2=(2\sqrt\azero)^{-1}$, and then
$c_4h^4+o(h^4)=-\frac3{10}\lambda h^{10/3}(1+o(1))$, which is impossible since $\lambda>0$.
\end{proof}

This completes the proof of parts (a)--(d) of Theorem~\ref{thm:field}. Its last sentence is
Corollary~\ref{cor:fld-examples}(a) below.

\begin{remark}[The exponent from the plateau calculus]\label{rem:fld-plateau-picture}
The first claim of part~(b) can also be read off the plateau calculus, and this explains why the
constant $\Dstar$ of the constrained onset appears. Let $0<h<h_1$, $q=\qmin(h)$ and
$\Gamma=\Gamma_{\zeta_h,h}$. The distribution function of $\zeta_h$ vanishes on $[0,q)$, so
$\Gamma'''=\Psi_{\zeta_h,h}(\cdot,0)$ on $(0,q)$, and this tends to $\Dstar$ uniformly on $(0,q)$
as $h\to0$, by Lemma~\ref{lem:fld-plateau}, since $q\to0$ and $\Wone(\zeta_h,\zetas)\to0$. At the
ends of the gap, $\Gamma(q)=q$ and $\Gamma'(q)=1$ by part~(a), $\Gamma(0)=\mfield(h)^2$ because
$X_0=h$, and the right derivative of $\Gamma'$ at zero is
$\azero^2\Phi_{\zeta_h,xxx}(0,h)^2=O(h^2)$ by Lemma~\ref{lem:fld-ito}(c), since
$\Phi_{\zeta_h,xxx}(0,\cdot)$ is odd with bounded derivative; since $\Gamma'''$ is bounded on $(0,q)$, $\Gamma''$ extends continuously to $[0,q]$, with this right derivative as its value at zero (as in the first step of the proof of Lemma~\ref{lem:plat-kernel}, and by the mean value theorem). Integrating by parts,
\[
 \mfield(h)^2=\int_0^qt\,\Gamma''(t)\,dt=\frac{\Dstar}3\,q^3(1+o(1))+O(h^2q^2).
\]
Since $\mfield(h)=h\,\Phi_{\zeta_h,xx}(0,y)$ for some $y\in(0,h)$, and
$\Phi_{\zeta_h,xx}(0,y)\to\Ustar$, this gives again $q^3=3h^2(\azero\Dstar)^{-1}(1+o(1))$.
\end{remark}

\begin{remark}[Where $\xi'''(0)=0$ enters]\label{rem:fld-scope}
This condition enters through Theorem~\ref{thm:local}, through $a'(0)=0$ in Lemma~\ref{lem:fld-plateau} and in Remark~\ref{rem:fld-plateau-picture},
through $a'(q)=\xi''''(0)q+O(q^2)$ in the proof of the atom in part~(b), and through
\eqref{eq:fld-xi-exp}. Without it, $q-\xi'(q)/\xi''(q)$ is of order $q^2$, and the exponent $2/3$
is not expected; we do not treat that case. All asymptotics in Theorem~\ref{thm:field} are as
$h\to0$ at fixed $\xi$; no uniformity in $\xi$ or in the temperature is claimed.
\end{remark}

\subsection{Covariances covered}\label{sec:fld-sk}

We prove the last sentence of Theorem~\ref{thm:field} and record a class of near-critical
covariances for which the input \condH\ comes from Theorem~\ref{thm:nearcrit}(a).

\begin{corollary}[Covariances covered by Theorem~\ref*{thm:field}]\label{cor:fld-examples}
\begin{thmlist}
\item For $\xi(t)=\beta^2t^2/2$ with $\beta>1$, the hypotheses of Theorem~\ref{thm:local}, and
 hence those of Theorem~\ref{thm:field}, hold.
\item Let $\varepsilon_*$ be as in Theorem~\ref{thm:nearcrit}(a). The hypotheses of
 Theorem~\ref{thm:field} hold for every $\xi\in\Kcl(\varepsilon_*)$. This class contains
 $\xi(t)=bt^2+et^4$ whenever $b>\frac12$, $e\ge0$ and $24(b-\frac12+e)<\varepsilon_*$, and
 $\xi(t)=\beta^2t^2/2$ whenever $1<\beta^2<1+\varepsilon_*$.
\end{thmlist}
\end{corollary}

\begin{proof}
(a) Here $\beta_2^2=\beta^2/2>0$ and $\beta_p=0$ for $p\ge3$, so $\xi'''(0)=0$ and the radius of
convergence is infinite. Condition \condH\ holds for every $\beta>1$ by Lopatto's theorem
\cite[Proposition~5.1]{Lop26a}; an alternative proof, with one computer-assisted step, is
\cite[Theorem~5.1]{Pei26}. Both are preprints (Remark~\ref{rem:fifth-sk-inputs}); for $1<\beta^2<1+\varepsilon_*$, part~(b) supplies \condH\ from refereed results and this paper. This is the input recalled in Theorem~\ref{thm:sk}.

(b) By Theorem~\ref{thm:nearcrit}(a), every $\xi\in\Kcl(\varepsilon_*)$ has
$\supp\zetas=[0,v_*]$ for some $v_*>0$, so \condH\ holds; moreover $\beta_3=0$ gives
$\xi'''(0)=0$, and $\xi''(0)>1$ gives $\beta_2>0$. For $\xi(t)=bt^2+et^4$ the derivatives of
order at most four of $\xi-\xic$ are bounded on $[0,1]$ by $24(b-\frac12+e)$, and
$\xi''(0)=2b>1$. For $\xi(t)=\beta^2t^2/2$ they are bounded by $\beta^2-1$, and
$\xi''(0)=\beta^2>1$.
\end{proof}

For the Sherrington--Kirkpatrick covariance, Corollary~\ref{cor:dict-field} writes parts (b)
and (d) of Theorem~\ref{thm:field} in the variables of Parisi and Rizzo, where they are the
leading terms of (B4) and (B5) of \cite{PR09}.

\begin{remark}[A fixed field]\label{rem:fld-earlier}
For fixed $h>0$, Lopatto proves that the Parisi measure of the Sherrington--Kirkpatrick model is
supported on an interval $[q_-,q_+]$, nondegenerate exactly when the measure is not replica
symmetric, and that it then has an atom at $q_-$ \cite[Theorem~1.1]{Lop26b}. For $0<|h|<h_1$,
Theorem~\ref{thm:field}(a) shows that $\zeta_h$ is not replica symmetric, and (b) gives the
small-field asymptotics of $q_-=\qmin(h)$ and of the atom $\zeta_h(\{q_-\})$. The proof uses neither
that theorem (near the origin the gap exclusion of Lemma~\ref{lem:fld-gap} suffices) nor
\cite[Theorem~B.1]{Lop26b} (interval support and a smooth positive density for $\zetas$ at $h=0$
and $\beta>1$).
\end{remark}

\begin{remark}[The atom carries the zero-field mass]\label{rem:fld-ansatz}
By parts (a) and (b) of Theorem~\ref{thm:field}, the masses of $[0,\qmin(h)]$ under $\zeta_h$ and $\zetas$
satisfy $\zeta_h([0,\qmin(h)])=\zetas([0,\qmin(h)])\,(1+o(1))$: to
leading order the field moves the whole zero-field mass of $[0,\qmin(h)]$ into the atom at
$\qmin(h)$, and the distribution function of $\zeta_h$ rejoins that of $\zetas$ at the edge
(Figure~\ref{fig:bottom}(b)). So the plateau of the order parameter $q(x)$ at height $\qmin(h)$
ends, asymptotically, at the zero-field value $x=\zetas([0,\qmin(h)])$, consistently with the
ansatz of \cite[Appendix~B]{PR09} that $q(x)$ is unperturbed above the plateau, which
Theorem~\ref{thm:field} does not prove. Under the constraint of Theorem~\ref{thm:local}(a) the
distribution function also rejoins the equilibrium one at the first positive contact to leading
order, but the atom at zero carries $\theta$ and the one at $q_\theta$ carries $\theta/2$
(Figure~\ref{fig:bottom}(a)).
\end{remark}

\section[An atom at zero: proof of Theorem~\ref{thm:atom}(a),(b),(d) and
Corollary~\ref{cor:pure-p}]{An atom at zero: proof of Theorem~\ref*{thm:atom}(a),(b),(d)\\
and Corollary~\ref*{cor:pure-p}}\label{sec:atom}

An atom of $\zetas$ at zero, the third perturbation of the bottom of the Parisi measure,
changes the rate at first order. If $m_0=\zetas(\{0\})>0$, then $\zetas$ itself satisfies
$\zeta(\{0\})\ge\theta$ for every $\theta\le m_0$, so the constraint costs nothing up to $m_0$
and $\Sigma$ leaves the equilibrium level with slope $-m_0$. Section~\ref{sec:atom-slope}
proves this by convex duality (Theorem~\ref{thm:atom}(a)) and deduces (d), the only part that
uses \condFM; Section~\ref{sec:atom-atom} reads the atom off from $\Gamma_*$ at the origin
(Theorem~\ref{thm:atom}(b)); Section~\ref{sec:atom-pure} treats pure $p$-spin models. Part
(c), where the argument of Theorem~\ref{thm:local}, built on the plateau calculus of
Section~\ref{sec:plat}, runs with $\theta-m_0$ in place of
$\theta$ (Figure~\ref{fig:bottom}(c)), uses the estimates of Section~\ref{sec:nc-reference}
and is proved in Section~\ref{sec:atom-odd}.

\paragraph{Standing assumption.}
In this section $\xi(t)=\sum_{p\ge2}\beta_p^2t^p$ is not identically zero and has radius of
convergence greater than one ($\beta_2=0$ is allowed), so $\xi',\xi''>0$ on $(0,1]$ and
$\xi'''\ge0$ on $[0,1]$. The results of Section~\ref{sec:cons}, including
Proposition~\ref{prop:cons-envelope} proved there, Lemma~\ref{lem:ons-zero-in-support}, the
bounds \eqref{eq:plat-bounds} and Sections~\ref{sec:plat-ito}--\ref{sec:plat-obstacle} hold in
this generality: their proofs use the covariance only through $\xi',\dots,\xi''''$ on $[0,1]$ and
the positivity of $\xi''$ on $(0,1]$. The notation is that of Sections~\ref{sec:cons}
and~\ref{sec:plat}; $\zetas$ is the unique minimizer of $P$ \cite[Theorem~2 and
Corollary~1]{AC15b}. From Lemmas~\ref{lem:cons-variational} and~\ref{lem:cons-regular} we
use: $\Lambar$ is convex; $\Lamhat=0$ on $[0,m_0]$ and $\Lamhat>0$ on $(m_0,1]$; the
constraint is active for $m_0<\theta<1$; if $m_0<1$, then $\feq<f_1<\infty$ and each
$f\in(\feq,f_1)$ has a unique dual minimizer $\theta_f\in(m_0,1)$.

\subsection{The slope at equilibrium}\label{sec:atom-slope}

\begin{proposition}[Slope at equilibrium]\label{prop:atom-slope}
The function $\Sigma$ is finite and concave on $\R$, and $\Sigma=0$ on $(-\infty,\feq]$. For
every $\delta>0$,
\begin{equation}\label{eq:atom-exact}
 \Sigma(\feq+\delta)=-m_0\delta-S(\delta),\qquad
 S(\delta):=\sup_{0\le s\le1-m_0}\bigl\{s\delta-\Lamhat(m_0+s)\bigr\},
\end{equation}
and $0\le S(\delta)=o(\delta)$ as $\delta\downarrow0$. Consequently
\begin{equation}\label{eq:atom-slope}
 \Sigma(\feq+\delta)=-\zetas(\{0\})\,\delta+o(\delta),\qquad
 \Sigma'_+(\feq)=-\zetas(\{0\}),
\end{equation}
and $|\Sigma(\feq+\delta)|=o(\delta)$ if and only if $\zetas(\{0\})=0$. If $m_0<1$, then
$\theta_f\to m_0$ and $\zeta_f\Rightarrow\zetas$ weakly as $f\downarrow\feq$.
\end{proposition}

The proposition uses only the uniqueness of $\zetas$, these duality facts and the Lipschitz
bound \eqref{eq:cons-lipschitz}. In words,
$-\Sigma(\feq+\delta)$ is the largest vertical gap between the line $\theta\delta$ and the
graph of $\Lamhat$: already $m_0\delta$ at $\theta=m_0$, plus a remainder $S$, the convex
conjugate of $s\mapsto\Lamhat(m_0+s)$ on $[0,1-m_0]$, decided by the growth of $\Lamhat$ to
the right of $m_0$. When $m_0=0$, $S$ is the whole rate, and under the hypotheses of
Theorem~\ref{thm:local}, $\Lamhat(\theta)=\theta(p(\theta)-\feq)\sim c\theta^6$ gives the law
$\delta^{6/5}$. When $m_0>0$, $\Lamhat(m_0+u)=(m_0+u)(p(m_0+u)-\feq)$ carries the factor
$m_0+u\approx m_0$ rather than $u$, and for the mixtures of Section~\ref{sec:atom-odd} this
gives $\Lamhat(m_0+u)\sim m_0cu^5$ and a correction of order $\delta^{5/4}$.

\begin{proof}
The function $\Sigma$ is an infimum of affine functions of $f$, hence concave, and it is
finite because $\Lambar$ is finite and lower semicontinuous on $[0,1]$. For $f\le\feq$ we
have $\Lambar(\theta)-\theta f=\Lamhat(\theta)+\theta(\feq-f)\ge0$, with equality at
$\theta=0$; so $\Sigma=0$ on $(-\infty,\feq]$.

Fix $\delta>0$. Then $\Sigma(\feq+\delta)=\inf_{0\le\theta\le1}\{\Lamhat(\theta)-\theta\delta\}$.
On $[0,m_0]$ the bracket equals $-\theta\delta$, and its infimum there is $-m_0\delta$,
attained at $\theta=m_0$. On $[m_0,1]$ write $\theta=m_0+s$; the infimum there is
$-m_0\delta-S(\delta)$, and the choice $s=0$ shows $S(\delta)\ge-\Lamhat(m_0)=0$. This
proves \eqref{eq:atom-exact}.

If $m_0=1$, then $S\equiv0$. Let $m_0<1$ and $0<\eta\le1-m_0$, and put
$c_\eta:=\Lamhat(m_0+\eta)>0$. For $s\ge\eta$, writing $m_0+\eta$ as a convex combination of
$m_0$ and $m_0+s$, convexity and $\Lamhat(m_0)=0$ give $\Lamhat(m_0+s)\ge c_\eta$. Let
$0<\delta<c_\eta$. If $s\ge\eta$, then $s\delta-\Lamhat(m_0+s)\le\delta-c_\eta<0$; if
$s<\eta$, then $s\delta-\Lamhat(m_0+s)\le\eta\delta$. Hence $0\le S(\delta)\le\eta\delta$ for
$\delta<c_\eta$. Since $\eta$ is arbitrary, $S(\delta)=o(\delta)$, which gives
\eqref{eq:atom-slope} and the dichotomy.

For the same $\eta$ and $\delta<c_\eta$, every $\theta<m_0$ gives
$\Lamhat(\theta)-\theta\delta=-\theta\delta>-m_0\delta\ge\Sigma(\feq+\delta)$, and every
$\theta=m_0+s$ with $s\ge\eta$ gives
$\Lamhat(\theta)-\theta\delta\ge c_\eta-\delta-m_0\delta>-m_0\delta$. So the dual minimizer
at the level $\feq+\delta$ lies in $[m_0,m_0+\eta)$, and $\theta_f\to m_0$. For $\theta>m_0$
the measure $\nu=s\delta_0+(1-s)\zetas$ with $s=(\theta-m_0)/(1-m_0)$ satisfies
$\nu(\{0\})=\theta$ and $\Wone(\nu,\zetas)\le s$, so the Lipschitz bound
\eqref{eq:cons-lipschitz} gives $\feq\le p(\theta)\le\feq+\frac12\xi''(1)s$. Hence
$p(\theta_f)\to\feq$. If $\zeta'$ is a weak subsequential limit of $\zeta_f$, the same bound
gives $P_{\zeta'}=\feq$, and uniqueness gives $\zeta'=\zetas$. By compactness the whole
family converges.
\end{proof}

\begin{proof}[Proof of Theorem~\ref{thm:atom}\textup{(a)} and \textup{(d)}]
Part (a) is contained in Proposition~\ref{prop:atom-slope}. For (d), let $\zetas\ne\delta_0$,
that is $m_0<1$, and assume \condFM. Then $\feq<f_1$, and every $f\in(\feq,f_1)$ has its
dual minimizer $\theta_f$ in $(m_0,1)$ (Lemma~\ref{lem:cons-regular}), where $\Lambda$ is
differentiable (Proposition~\ref{prop:cons-envelope}). These are the inputs of
Lemma~\ref{lem:cons-gartner-ellis}: Markov's inequality at $\theta_f$ gives the upper bound
$\Sigma(f)$, and tilting by $\ZN^{\theta}$ at $\theta=\theta_{f'}$ with $f'$ slightly above
$f$ gives the matching lower bound. Hence
$\lim_NN^{-1}\log\Prob(N^{-1}\log\ZN\ge f)=\Sigma(f)$ for every $f\in(\feq,f_1)$.
\end{proof}

Part (d), the identification of $-\Sigma$ with the upper-tail rate, is noted in
\cite[Remark~1.3]{Bou26b} as a consequence of Talagrand's large-deviation principle;
Theorem~\ref{thm:atom} adds the behaviour of $\Sigma$ near $\feq$, which
Proposition~\ref{prop:atom-slope} gives for every upper-deviation exponent known to equal
$\Sigma$. Remark~\ref{rem:mo-fm} records that, for $\xi$ convex on $[-1,1]$, the upper bound
in (d) holds without \condFM.

\begin{remark}[An earlier statement of the slope identity]\label{rem:atom-earlier}
Boursier \cite[Remark~1.4]{Bou26a} states, for the Parisi measure $\zetas$, the identity
$(-\Lambda^*)'(\feq)=-\zetas(\{0\})$ without proof, and relates it to an extreme-value
heuristic. Her
dual function uses the equality constraint $\zeta(\{0\})=\theta$ \cite[(1.12)]{Bou26a}, and her
terminal value $\log(2\cosh x)$ shifts $P_\zeta$, $\feq$ and the level by $\log2$, which does
not affect slopes; after this shift her $-\Lambda^*$ agrees with $\Sigma$ above $\feq$. Indeed,
for $\theta\in[m_0,1]$ the two constraints give the same infimum, since the inequality
constraint is active for $m_0<\theta<1$ and $\zetas$, $\delta_0$ are admissible for both at
$\theta=m_0$, $\theta=1$. For $\theta<m_0$ and $f>\feq$,
\[
 \theta\inf_{\zeta(\{0\})=\theta}P_\zeta-\theta f\ge-\theta(f-\feq)>-m_0(f-\feq)
 =\Lambar(m_0)-m_0f,
\]
and also $\Lambar(\theta)-\theta f=-\theta(f-\feq)>\Lambar(m_0)-m_0f$; so both dual functions are the infimum over $[m_0,1]$ of a common bracket. As $\Sigma=0$ on
$(-\infty,\feq]$, the identity concerns the right derivative, and
Proposition~\ref{prop:atom-slope} proves it in this form, with the exact remainder $S$.
\end{remark}

\subsection{The atom}\label{sec:atom-atom}

For every probability measure $\zeta$, Lemma~\ref{lem:plat-ito}(a),(b) at $t=0$, where
$X_0=0$ and $\Phi_{\zeta,x}(0,0)=\Phi_{\zeta,xxx}(0,0)=0$ because $\Phi_\zeta(0,\cdot)$ is even,
give
\begin{equation}\label{eq:atom-at-zero}
\begin{gathered}
 \Gamma_\zeta(0)=0,\qquad
 \Gamma_\zeta'(0)=\xi''(0)\,\Phi_{\zeta,xx}(0,0)^2,\\
 \Gamma_\zeta''(0+)=\xi'''(0)\,\Phi_{\zeta,xx}(0,0)^2
   -2\xi''(0)^2\zeta(\{0\})\,\Phi_{\zeta,xx}(0,0)^3 ,
\end{gathered}
\end{equation}
where the right derivative $\Gamma_\zeta''(0+)$ of $\Gamma_\zeta'$ at zero involves
$\alpha_\zeta(0)=\zeta([0,0])=\zeta(\{0\})$, hence the atom (this is \eqref{eq:plat-zero}). By
Lemma~\ref{lem:ons-zero-in-support}, the support of $\zetas$ contains zero and not one, and
every $u\in\supp\zetas$ satisfies $\Gamma_*(u)=u$ and $\Gamma_*'(u)\le1$. Put
$\Ustar=\Phi_{\zetas,xx}(0,0)\in(0,1]$. If zero is isolated in $\supp\zetas$, the atom is
positive. If it is an accumulation point, $\Gamma_*(u)=u$ along support points $u\downarrow0$
pins $\Gamma_*'(0)$ to one and $\Gamma_*''(0+)$ to zero, and \eqref{eq:atom-at-zero} turns
these two equations into the values of $\Ustar$ and $m_0$.

\begin{proposition}[The atom]\label{prop:atom-atom}
\leavevmode
\begin{thmlist}
\item If zero is an isolated point of $\supp\zetas$, then $m_0>0$. This is the case whenever
 $\xi''(0)<1$, in particular whenever $\beta_2=0$.
\item If zero is an accumulation point of $\supp\zetas$, then $\xi''(0)\Ustar^2=1$, so
 $\xi''(0)\ge1$, and
 \begin{equation}\label{eq:atom-formula}
  \zetas(\{0\})=\frac{\xi'''(0)}{2\xi''(0)^{3/2}}=\frac{3\beta_3^2}{(2\beta_2^2)^{3/2}} .
 \end{equation}
\item $\zetas(\{0\})=0$ if and only if $\beta_3=0$ and zero is an accumulation point of
 $\supp\zetas$.
\end{thmlist}
\end{proposition}

Parts (a) and (b) are Theorem~\ref{thm:atom}(b).

\begin{proof}
(a) If $(0,r)\cap\supp\zetas=\emptyset$ for some $r>0$, then
$\zetas(\{0\})=\zetas([0,r))>0$, because $0\in\supp\zetas$ and $[0,r)$ is a relatively open
neighbourhood of $0$ in $[0,1]$. Suppose $\xi''(0)<1$, and choose $r\in(0,1)$ with $\xi''<1$
on $[0,r]$. If some $u\in(0,r]$ belonged to $\supp\zetas$, Lemma~\ref{lem:ons-zero-in-support}
and $\Gamma_*(0)=0$ would give $\Gamma_*(u)-\Gamma_*(0)=u$, and the mean value theorem would
give $v\in(0,u)$ with $1=\Gamma_*'(v)=\xi''(v)\E U_v^2\le\xi''(v)<1$, since $0<U_v\le1$. Hence
zero is isolated. This is the argument of \cite[Theorem~1(ii)]{AC15a}.

(b) Choose $u_n\in\supp\zetas$ with $u_n\downarrow0$ strictly. By
Lemma~\ref{lem:ons-zero-in-support} and \eqref{eq:atom-at-zero},
$\xi''(0)\Ustar^2=\Gamma_*'(0)=\lim_n\Gamma_*(u_n)/u_n=1$. Since $\Ustar\le1$, this gives
$\xi''(0)\ge1$ and $\Ustar=\xi''(0)^{-1/2}$. The mean value theorem on $[u_{n+1},u_n]$ gives
$v_n\in(u_{n+1},u_n)$ with
$\Gamma_*'(v_n)=(\Gamma_*(u_n)-\Gamma_*(u_{n+1}))/(u_n-u_{n+1})=1$, so $v_n\downarrow0$.
Evaluating the right derivative of $\Gamma_*'$ at zero along $v_n$, and using
\eqref{eq:atom-at-zero},
\[
 0=\lim_n\frac{\Gamma_*'(v_n)-\Gamma_*'(0)}{v_n}=\Gamma_*''(0+)
  =\xi'''(0)\Ustar^2-2\xi''(0)^2m_0\Ustar^3 .
\]
Since $\xi''(0)\Ustar>0$, it follows that
$m_0=\xi'''(0)/(2\xi''(0)^2\Ustar)=\xi'''(0)/(2\xi''(0)^{3/2})$. Finally
$\xi''(0)=2\beta_2^2$ and $\xi'''(0)=6\beta_3^2$.

(c) If $m_0=0$, zero is not isolated by (a), and (b) gives $\beta_3=0$. The converse is (b).
\end{proof}

\begin{remark}\label{rem:atom-atom-comparison}
When $\xi'''(0)=0$, part (c) is the equivalence of \condH\ with $\zetas(\{0\})=0$
(Remark~\ref{rem:ons-H-no-atom}). When an interval $(0,b)$ lies in $\supp\zetas$, part (b)
combines $\Gamma_*'(0)=1$ with the value at $u=0$ of the identity
$\zetas([0,u])\,\sigma_{\zetas}(u)=N_{\zetas}(u)$ obtained in the proof of
\cite[Theorem~2(ii)]{AC15a} ($N_\zeta,\sigma_\zeta$ as in \eqref{eq:plat-NS}, their
$\gamma_1,\gamma_2$ \cite[Proposition~3]{AC15a}); part (b) needs only the accumulation. The
right side of \eqref{eq:atom-formula} is the value at $q=0$ of $\xi'''(q)/(2\xi''(q)^{3/2})$,
the Parisi distribution function on an interval of the support in the spherical model
\cite[Theorem~6(ii)]{AC15a}. For $\xi(t)=\beta^2t^2/2$ with $\beta>1$, part (b) and the
accumulation of the support at zero \cite[Proposition~5.1]{Lop26a} (a preprint; see also
\cite[Theorem~5.1]{Pei26}) give $\Ustar=1/\beta$ and
$\zetas(\{0\})=0$; these are \cite[Lemmas~6.1 and~6.2]{Lop26a}, proved there by essentially the
same use of the right derivative of $\Gamma_*'$ at zero.
\end{remark}

\subsection{Pure \texorpdfstring{$p$}{p}-spin models}\label{sec:atom-pure}

By Theorem~\ref{thm:atom}(a), the rate leaves the equilibrium level linearly exactly when
$\zetas$ charges zero. A pure $p$-spin covariance with $p\ge3$ has no quadratic term, so
$\xi''(0)=0<1$, and Proposition~\ref{prop:atom-atom}(a) puts it on the linear side at every
temperature. That the atom is positive when $\xi''(0)<1$ is \cite[Theorem~1]{AC15a};
Proposition~\ref{prop:atom-atom}(a) repeats its argument with
Lemma~\ref{lem:ons-zero-in-support} in place of \cite[Theorem~5]{AC15a}
(Section~\ref{sec:intro-perturbations} gives the context).

\begin{proof}[Proof of Corollary~\ref{cor:pure-p}]
The covariance is a polynomial and is not identically zero, so the standing assumption
holds. Since $\xi''(0)=0<1$, Proposition~\ref{prop:atom-atom}(a) gives $m_0=\zetas(\{0\})>0$,
and Proposition~\ref{prop:atom-slope} gives the expansion and the right derivative. If
$\zetas=\delta_0$, then $m_0=1$, $\Lamhat=0$ on $[0,1]$, and \eqref{eq:atom-exact} reads
$\Sigma(\feq+\delta)=-\delta$. If $\zetas\ne\delta_0$ and \condFM\ holds, the last
assertion is Theorem~\ref{thm:atom}(d) at $f=\feq+\delta$.
\end{proof}

\section{Near criticality: proof of Theorem~\ref*{thm:nearcrit} and of
Theorem~\ref*{thm:atom}(c)}\label{sec:nc}

At the critical covariance $\xic(t)=t^2/2$ the Parisi measure is $\delta_0$, and the Parisi
equation \eqref{eq:intro-pde} is solved by $\log\cosh x+\frac12(1-t)$. Closeness to this
reference controls the constrained minimizers $\zeta_\theta$ for every $\theta\in[0,1]$,
whereas the local analysis of Sections~\ref{sec:plat} and~\ref{sec:ons} describes them only
for small $\theta$. Section~\ref{sec:nc-reference} proves three uniform facts near $\xic$: the
plateau polynomial behind $\Gamma'''$ is positive below a fixed level $K$, so no support-free
interval separates two contacts below $K$, and every constrained minimizer is close to
$\delta_0$, so no contact lies in $[K,1]$. Section~\ref{sec:nc-class} deduces
Theorem~\ref{thm:nearcrit}(a), and Section~\ref{sec:atom-odd} Theorem~\ref{thm:atom}(c), the
same argument with a cubic term. For every covariance of the class, including the
Sherrington--Kirkpatrick covariance for $1<\beta^2<1+\varepsilon_*$, the global control also
shows that the marginal regime ends exactly at a de Almeida--Thouless threshold $\theta_c$, at and beyond which
$\zeta_\theta$ has two atoms (Section~\ref{sec:nc-dat}, Theorem~\ref{thm:nearcrit}(b)), and
that above its first contact $\zeta_\theta$ differs from $\zetas$ by $O(\theta^5)$
(Section~\ref{sec:nc-stability}, Theorem~\ref{thm:nearcrit}(c)); these two long computations
are outlined here and carried out in Appendices~\ref{app:dat} and~\ref{app:stability}.

Throughout this section a \emph{covariance} is a function $\xi(t)=\sum_{p\ge2}\beta_p^2t^p$
with radius of convergence greater than one. When the covariance varies we record it as a
superscript, writing $\Phi^\xi_\zeta$, $P^\xi_\zeta$, $\Gamma^\xi_\zeta$ and $\Obs^\xi_\zeta$,
and we write $\zeta_{\xi,\theta}$ for the minimizer of $P^\xi_\zeta$ subject to
$\zeta(\{0\})\ge\theta$, which is unique by Lemma~\ref{lem:cons-variational} (at $\theta=1$ the only feasible measure is $\delta_0$); thus
$\zeta_{\xi,0}=\zetas$. The superscript is dropped when $\xi$ is fixed. As in
Section~\ref{sec:plat}, $a=\xi''$, $\alpha_\zeta(t)=\zeta([0,t])$, and $(U_t,V_t,W_t)$ are the
second to fourth spatial derivatives of $\Phi_\zeta$ along the diffusion
\eqref{eq:intro-diffusion}. A \emph{contact} of $\zeta$ is a minimum point of $\Obs_\zeta$ on
$[0,1]$, and a \emph{plateau} of $\zeta$ is an interval on which $\alpha_\zeta$ is constant.

\subsection{Closeness to the critical reference}\label{sec:nc-reference}

Here closeness to $(\xic,\delta_0)$ replaces the closeness to $\zetas$ of
Section~\ref{sec:plat-gap}, so that gap exclusion applies to every constrained minimizer,
uniformly in the covariance. This adapts the near-critical mechanism of Zhou~\cite{Zho26},
proved there for the equilibrium measure of the Sherrington--Kirkpatrick model, to
constrained minimizers and mixed covariances.

\paragraph{Varying the covariance.}
If $\normCfour{\xi-\xic}<\frac14$, then $|a-1|$, $|a'|$ and $|a''|$ are at most
$\normCfour{\xi-\xic}$ on $[0,1]$; in particular $\frac34<a<\frac54$ there. These covariances,
together with $\xic$, form a family of the kind considered in Lemma~\ref{lem:plat-jets}(b): the
bounds \eqref{eq:plat-bounds} hold with constants uniform over the family, and
\eqref{eq:plat-jet-general} compares two pairs $(\xi,\zeta)$ and $(\tilde\xi,\tilde\zeta)$ in
it. Since
$\|a\alpha-\tilde a\tilde\alpha\|_{L^1}\le\|a-\tilde a\|_\infty+\frac54\Wone(\zeta,\tilde\zeta)$,
the right side of \eqref{eq:plat-jet-general} tends to zero when $\xi''$ converges uniformly and
the measures converge weakly. Under this convergence $\Gamma^\xi_\zeta$ and
$(\Gamma^\xi_\zeta)'$ converge uniformly on $[0,1]$ (Lemma~\ref{lem:plat-coupling}), and
$P^\xi_\zeta$ converges. For the last point, write $u=\Phi^\xi_{\zeta,x}$ and integrate
$\Phi_t=-\frac a2(u_x+\alpha u^2)$ in time at $x=0$: by \eqref{eq:plat-jet-general} with
$k=0,1$ the change of $\Phi^\xi_\zeta(0,0)$ is at most a constant times
$\|a-\tilde a\|_{L^1}+\|a\alpha-\tilde a\tilde\alpha\|_{L^1}$, and the linear term
$\frac12\int_0^1ta\alpha\,dt$ changes by at most
$\frac12\|a\alpha-\tilde a\tilde\alpha\|_{L^1}$.

\paragraph{The reference pair.}
For $\xic$ and $\delta_0$ the distribution function is identically one, and
$\Phi_c(t,x)=\log\cosh x+\frac12(1-t)$. Its derivatives $U,V,W$ at $x=0$ are $(1,0,-2)$ for
every $t$. The diffusion solves $dX_t=\tanh X_t\,dt+dB_t$, and since $e^{-t/2}\cosh x$ is
space-time harmonic, its law is the $\cosh$ transform of $\mathcal N(0,t)$:
\begin{equation}\label{eq:nc-reference-law}
 \E F(X_t)=e^{-t/2}\,\E\bigl[F(\sqrt tZ)\cosh(\sqrt tZ)\bigr],\qquad Z\sim\mathcal N(0,1).
\end{equation}
Hence $\Gamma_c(0)=0$ and $\Gamma_c'(t)=\E\sech^4X_t=e^{-t/2}\E\sech^3(\sqrt tZ)<1$ for $t>0$,
so that $\Gamma_c(t)<t$ for every $t>0$. The critical obstacle therefore satisfies
$\Obs_c'(t)=\frac12(t-\Gamma_c(t))>0$ for $t>0$, and its only minimum point is zero. Thus
$\int\Obs_c\,d(\nu-\delta_0)\ge0$ for every $\nu\in\Pspace([0,1])$, and the optimality
criterion of Lemma~\ref{lem:cons-standard}(g) makes $\delta_0$ the unique minimizer of
$P^{\xic}$.

\paragraph{The plateau polynomial.}
For a covariance $\xi$ write $\Psi^\xi_\zeta(t,m)$ for the function \eqref{eq:plat-psi}
computed with $\xi$; its definition does not use $\xi'''(0)=0$. By
Lemma~\ref{lem:plat-ito}(c), $\Gamma'''_\zeta=\Psi^\xi_\zeta(\cdot,m)$ on every open plateau of
$\zeta$ of height $m$. At the reference pair and $t=0$, where $X_0=0$, it equals $4+6m^2$.

\begin{lemma}[Positivity of the plateau polynomial]\label{lem:nc-positive-third}
There are $K_P\in(0,1)$, $\varepsilon_P\in(0,\frac14]$ and a weak neighbourhood
$\mathcal W_P$ of $\delta_0$ in $\Pspace([0,1])$ such that
\[
 \Psi^\xi_\zeta(t,m)\ge2\qquad(0\le t\le K_P,\ 0\le m\le1)
\]
for every covariance $\xi$ with $\normCfour{\xi-\xic}<\varepsilon_P$ and every
$\zeta\in\mathcal W_P$. For such $\xi$ and $\zeta$,
\begin{equation}\label{eq:nc-positive-third}
 \Gamma'''_\zeta\ge2\quad\text{on every open plateau of $\zeta$ contained in $(0,K_P)$.}
\end{equation}
\end{lemma}

The proof (Appendix~\ref{sec:an-nc}) bounds $|\Psi^\xi_\zeta(t,m)-4-6m^2|$ by an absolute
constant times $\normCfour{\xi-\xic}+\Wone(\zeta,\delta_0)+\sqrt t$, by
\eqref{eq:plat-jet-general} and \eqref{eq:plat-small-time}.

\begin{lemma}[Confinement]\label{lem:nc-confinement}
\begin{thmlist}
\item For every weak neighbourhood $\mathcal W$ of $\delta_0$ there is $\varepsilon>0$ such
 that $\zeta_{\xi,\theta}\in\mathcal W$ for every covariance $\xi$ with
 $\normCfour{\xi-\xic}<\varepsilon$ and every $\theta\in[0,1]$.
\item For every $K\in(0,1)$ there are $\varepsilon>0$ and a weak neighbourhood
 $\mathcal W_K$ of $\delta_0$ such that
 \begin{equation}\label{eq:nc-confinement}
  \Gamma^\xi_\zeta(t)<t\qquad(K\le t\le1)
 \end{equation}
 for every covariance $\xi$ with $\normCfour{\xi-\xic}<\varepsilon$ and every
 $\zeta\in\mathcal W_K$. For such $\xi$ and $\zeta$, no point of $[K,1]$ minimizes $\Obs_\zeta$
 over $[0,1]$ or over $(0,1]$; in particular no point of $[K,1]$ is a contact of~$\zeta$.
\end{thmlist}
\end{lemma}

\begin{proof}
(a) Suppose not. Then there are an open set $\mathcal U\subset\mathcal W$ containing
$\delta_0$, covariances $\xi_n$ with $\normCfour{\xi_n-\xic}\to0$, and constraints $\theta_n$
such that $\zeta_n=\zeta_{\xi_n,\theta_n}\notin\mathcal U$. By compactness of
$\Pspace([0,1])$ a subsequence converges weakly to some $\zeta\notin\mathcal U$. The measure
$\delta_0$ is feasible for every constraint, so $P^{\xi_n}_{\zeta_n}\le P^{\xi_n}_{\delta_0}$,
and joint continuity gives $P^{\xic}_\zeta\le P^{\xic}_{\delta_0}$. Thus $\zeta$ minimizes
$P^{\xic}$, and $\zeta=\delta_0$ by uniqueness, a contradiction.

(b) The continuous function $t-\Gamma_c(t)$ has a positive minimum on $[K,1]$. If the claim
failed, there would be covariances $\xi_n\to\xic$ in $C^4$, measures
$\zeta_n\Rightarrow\delta_0$ and points $t_n\in[K,1]$ with
$\Gamma^{\xi_n}_{\zeta_n}(t_n)\ge t_n$; by uniform convergence of $\Gamma^{\xi_n}_{\zeta_n}$
to $\Gamma_c$, a limit point $t$ of $(t_n)$ would satisfy $\Gamma_c(t)\ge t$. Given
\eqref{eq:nc-confinement}, $\Obs'_\zeta(t)=\frac{a(t)}2(t-\Gamma^\xi_\zeta(t))>0$ on $[K,1]$,
so $\Obs_\zeta$ is strictly increasing there and takes smaller values just to the left of
$K$; no point of $[K,1]$ minimizes it.
\end{proof}

\paragraph{Choice of constants.}
We fix $K\in(0,\min\{K_P,\frac1{20}\}]$, whose value Appendix~\ref{sec:dat-constants} specifies, take $\varepsilon$ and $\mathcal W_K$ from
Lemma~\ref{lem:nc-confinement}(b) for this $K$, and put
$\mathcal W=\mathcal W_P\cap\mathcal W_K$. By Lemma~\ref{lem:nc-confinement}(a) we can then
choose $\varepsilon_1\in(0,\min\{\varepsilon_P,\varepsilon\}]$ such that, for every
covariance $\xi$ with $\normCfour{\xi-\xic}<\varepsilon_1$:
\begin{steplist}
\item\label{item:nc-positive} every $\zeta\in\mathcal W$ satisfies $\Psi^\xi_\zeta\ge2$ on
 $[0,K]\times[0,1]$, hence \eqref{eq:nc-positive-third} on $(0,K)$, and
 \eqref{eq:nc-confinement}, and has no contact in $[K,1]$;
\item\label{item:nc-confined} every constrained minimizer $\zeta_{\xi,\theta}$,
 $\theta\in[0,1]$, lies in $\mathcal W$; in particular all its positive contacts lie in
 $(0,K)$.
\end{steplist}
Everything below holds for every such choice of $K$, $\mathcal W$ and $\varepsilon_1$;
Appendix~\ref{sec:dat-constants} makes the choice used for Theorem~\ref{thm:nearcrit}(b),(c),
choosing $K$ first and then $\mathcal W$ and $\varepsilon_1$ for it. None of
Lemmas~\ref{lem:nc-positive-third}, \ref{lem:nc-confinement}, \ref{lem:nc-gap} and the choice
of constants uses the value of $\beta_3$ or the condition $\xi''(0)>1$: they hold on the whole
ball $\normCfour{\xi-\xic}<\varepsilon_1$.

\begin{lemma}[No gaps between contacts below $K$]\label{lem:nc-gap}
Let $\normCfour{\xi-\xic}<\varepsilon_1$ and $\zeta\in\mathcal W$.
\begin{thmlist}
\item If $0<u<v<K$ are local minimum points of $\Obs_\zeta$, then $\zeta((u,v))>0$.
\item If $\zeta=\zetas$ is the unconstrained minimizer of $P^\xi$ and
 $v\in\supp\zetas\cap(0,K)$, then $\zetas((0,v))>0$.
\end{thmlist}
\end{lemma}

\begin{proof}
(a) By Lemma~\ref{lem:plat-obstacle}(b), $\Gamma_\zeta(u)=u$, $\Gamma_\zeta(v)=v$ and
$\Gamma_\zeta'(u),\Gamma_\zeta'(v)\le1$. If $\zeta((u,v))=0$, then $(u,v)$ is a plateau of
$\zeta$ inside $(0,K)$, so $\Gamma_\zeta'''\ge2$ there by \eqref{eq:nc-positive-third}, and
$\Gamma_\zeta'$ is strictly convex on $(u,v)$. This contradicts the trapezoid principle,
Lemma~\ref{lem:plat-obstacle}(d): the identity $\Gamma_\zeta(v)-\Gamma_\zeta(u)=v-u$ makes the
mean of $\Gamma_\zeta'$ over $(u,v)$ equal to one, while strict convexity and the endpoint
bounds force it below one.

(b) By Lemma~\ref{lem:ons-zero-in-support}, $0$ and $v$ are support points of $\zetas$ with
$\Gamma_*(0)=0$, $\Gamma_*(v)=v$ and $\Gamma_*'(0),\Gamma_*'(v)\le1$. If $\zetas((0,v))=0$,
the argument of (a) applies with $u=0$.
\end{proof}

\subsection{The near-critical class and Kondor's limit}\label{sec:nc-class}

Recall that $\Kcl(\varepsilon)$ is the set of covariances $\xi(t)=\sum_{p\ge2}\beta_p^2t^p$
with radius of convergence greater than one and
\begin{equation}\label{eq:nc-class}
 \beta_3=0,\qquad\xi''(0)=2\beta_2^2>1,\qquad\normCfour{\xi-\xic}<\varepsilon .
\end{equation}
For instance $\xi(t)=bt^2+et^4$ with $b>\frac12$, $e\ge0$ and $24(b-\frac12+e)<\varepsilon$
belongs to $\Kcl(\varepsilon)$, and so does the Sherrington--Kirkpatrick covariance
$\beta^2t^2/2$ when $1<\beta^2<1+\varepsilon$, since its distance to $\xic$ is $\beta^2-1$.
Odd degrees $p\ge5$ are allowed; only the cubic coefficient is excluded, and covariances with
a cubic term are treated in Section~\ref{sec:atom-odd}. Below, $\varepsilon_1\le\frac14$ and
$K\le\frac1{20}$ are the constants fixed in Section~\ref{sec:nc-reference}. The global input,
the next proposition, holds whatever $\beta_3$; Section~\ref{sec:atom-odd} uses it with a cubic
term.

\begin{proposition}[Interval support near criticality]\label{prop:nc-support}
Let $\xi(t)=\sum_{p\ge2}\beta_p^2t^p$ have radius of convergence greater than one, with
$\xi''(0)>1$ and $\normCfour{\xi-\xic}<\varepsilon_1$.
\begin{thmlist}
\item $\xi''(t)>1-\normCfour{\xi-\xic}$ for $t\in[-1,1]$; in particular $\xi$ is strictly
 convex on $[-1,1]$.
\item $\supp\zetas=[0,v_*]$ for some $v_*\in(0,K)$. In particular zero is an accumulation
 point of $\supp\zetas$.
\end{thmlist}
\end{proposition}

\begin{proof}
(a) For $0\le r\le1$, all terms of $\xi''(r)$ are nonnegative, so $\xi''(r)\ge\xi''(0)>1$. In
$\xi''(-r)$ only the terms of odd degree are negative, so
\[
 \xi''(-r)\ge2\beta_2^2-\sum_{p\ge3}p(p-1)\beta_p^2=2\xi''(0)-\xi''(1)
 >2-\bigl(1+\normCfour{\xi-\xic}\bigr),
\]
because $\xi''(0)>1$ and $|\xi''(1)-1|\le\normCfour{\xi-\xic}$. The lower bound is positive,
since $\varepsilon_1\le\frac14$.

(b) By Lemma~\ref{lem:ons-zero-in-support}, zero lies in $\supp\zetas$, every support point of
$\zetas$ is a contact, and $\Gamma_*'(u)\le1$ at every support point $u$. The measure $\zetas$
is not $\delta_0$: for $\delta_0$ the Parisi equation has the solution
$\log\cosh x+\frac12(\xi'(1)-\xi'(t))$, so $\Gamma'_{\delta_0}(0)=\xi''(0)>1$. So
$\supp\zetas$ has a positive point, and by item~\ref{item:nc-confined} of the choice of
constants (with $\theta=0$), every positive support point lies in $(0,K)$. Hence
$v_*=\max\supp\zetas\in(0,K)$. Suppose some $w\in(0,v_*)$ is not in the support. Put
$u=\max(\supp\zetas\cap[0,w])$ and $v=\min(\supp\zetas\cap[w,v_*])$; then $u<w<v$ and
$\zetas((u,v))=0$. If $u>0$, then $u$ and $v$ are contacts in $(0,K)$, which contradicts
Lemma~\ref{lem:nc-gap}(a); if $u=0$, this contradicts Lemma~\ref{lem:nc-gap}(b). So
$\supp\zetas=[0,v_*]$, and zero is an accumulation point of the support.
\end{proof}

\begin{proposition}[The near-critical class]\label{prop:nc-class}
Every $\xi\in\Kcl(\varepsilon_1)$ has the following properties.
\begin{thmlist}
\item $\xi''(t)>1-\varepsilon_1$ for $t\in[-1,1]$; in particular $\xi$ is strictly convex on
 $[-1,1]$.
\item $\supp\zetas=[0,v_*]$ for some $v_*\in(0,K)$. In particular \condH\ holds, and
 Theorem~\ref{thm:local} applies to $\xi$.
\item $\Dstar\ge2$.
\item There is $\theta_0\in(0,1)$ such that, for $0<\theta<\theta_0$,
 $\supp\zeta_\theta=\{0\}\cup[q_\theta,v_\theta]$ with $0<q_\theta<v_\theta<K$, and the first
 positive contact $q_\theta$ is marginal. As $\theta\downarrow0$,
 \begin{equation}\label{eq:nc-onset}
  \frac{q_\theta}\theta\longrightarrow\frac{3\sqrt\azero}{\Dstar},\qquad
  \frac{\zeta_\theta(\{q_\theta\})}\theta\longrightarrow\frac12 .
 \end{equation}
\end{thmlist}
\end{proposition}

Part~(b) is the global information that the local theory lacks: Theorem~\ref{thm:local} takes
\condH\ as a hypothesis, and here it follows from the comparison with the critical reference.
Part~(d) rests on the following description of $\zeta_\theta$, valid for every constraint and
used again in Sections~\ref{sec:nc-dat} and~\ref{sec:nc-stability}; as in
Lemma~\ref{lem:cons-variational}(e), $q_\theta$ is the least point of the positive contact set
$C_\theta$.

\begin{lemma}[Shape of the constrained minimizers]\label{lem:nc-regular}
Let $\xi\in\Kcl(\varepsilon_1)$ and $\theta\in(0,1)$. Then $\zeta_\theta(\{0\})=\theta$, and
the positive contact set of $\zeta_\theta$ is an interval $[q_\theta,v_\theta]\subset(0,K)$,
possibly reduced to the point $q_\theta$, with
$\supp\zeta_\theta=\{0\}\cup[q_\theta,v_\theta]$. If $v_\theta>q_\theta$, then
$\Gamma_\theta(t)=t$ on $[q_\theta,v_\theta]$ and the first positive contact is marginal. If
$v_\theta=q_\theta$, then $\zeta_\theta=\theta\delta_0+(1-\theta)\delta_{q_\theta}$.
\end{lemma}

In its proof, given in Appendix~\ref{sec:dat-constants}, Proposition~\ref{prop:nc-class}(b), through Proposition~\ref{prop:ons-equilibrium}(b), makes the constraint
active, and by Lemma~\ref{lem:nc-gap}(a) no interval of zero mass separates two positive
contacts, so $C_\theta$ is an interval carrying the whole positive support.

\begin{proof}[Proof of Proposition~\ref{prop:nc-class}]
(a) and the first assertion of (b) are Proposition~\ref{prop:nc-support}, since
$\normCfour{\xi-\xic}<\varepsilon_1$. Zero is an accumulation point of the support, so
\condH\ holds; since $\xi'''(0)=6\beta_3^2=0$ and $\beta_2>0$, all hypotheses of
Theorem~\ref{thm:local} hold.

(c) Evaluate \eqref{eq:plat-psi} for $\zetas$ at $t=0$ and $m=0$. There $X_0=0$, so
$U_0=\Ustar$, $W_0=\Wstar$ and $V_0=0$ by parity, while $a'(0)=\xi'''(0)=0$ and
$a''(0)=\xi''''(0)$. Thus $\Psi^\xi_{\zetas}(0,0)=\xi''''(0)\Ustar^2+\azero^3\Wstar^2=\Dstar$.
Since $\zetas\in\mathcal W$ by item~\ref{item:nc-confined} of the choice of constants,
item~\ref{item:nc-positive} gives $\Dstar\ge2$.

(d) By Proposition~\ref{prop:ons-onset}(a), $\zeta_\theta\Rightarrow\zetas$ as
$\theta\downarrow0$. If $v_{\theta_n}=q_{\theta_n}$ along a sequence $\theta_n\downarrow0$, then
by Lemma~\ref{lem:nc-regular}
$\zeta_{\theta_n}=\theta_n\delta_0+(1-\theta_n)\delta_{q_{\theta_n}}$, and every weak limit
point of this sequence is a Dirac mass. But $\supp\zetas=[0,v_*]$ with $v_*>0$. Hence
$v_\theta>q_\theta$ for all small $\theta>0$, and Lemma~\ref{lem:nc-regular} gives the support
and the marginality. The limits \eqref{eq:nc-onset} are Theorem~\ref{thm:local}(a).
\end{proof}

\begin{proposition}[Kondor's limit]\label{prop:nc-kondor}
As $\normCfour{\xi-\xic}\to0$ within $\Kcl(\varepsilon_1)$, one has $\zetas\Rightarrow\delta_0$,
$\Wstar\to-2$, $\Dstar\to4$, $\rhostar\to2$ and
\[
 c=\frac{9\azero^3}{80\Dstar^3}\longrightarrow\frac9{5120} .
\]
\end{proposition}

\begin{proof}
As $\normCfour{\xi-\xic}\to0$ we have $\azero\to1$, $\xi''''(0)\to0$ and, by
Theorem~\ref{thm:local}, $\Ustar=\azero^{-1/2}\to1$. Lemma~\ref{lem:nc-confinement}(a) with
$\theta=0$ gives $\zetas\Rightarrow\delta_0$, so
$\|a\alpha_*-1\|_{L^1}\le\|a-1\|_{L^1}+\frac54\int s\,\zetas(ds)\to0$. Apply
\eqref{eq:plat-jet-general} with $k=3$ to the pairs $(\xi,\zetas)$ and $(\xic,\delta_0)$, for
which $\tilde u=\tanh$ and $\tanh'''(0)=-2$: this gives $\Wstar=\partial_x^3u(0,0)\to-2$. Hence
$\Dstar=\xi''''(0)\Ustar^2+\azero^3\Wstar^2\to4$, $\rhostar=\Dstar/(2\sqrt\azero)\to2$, and
$c\to9/(80\cdot64)=9/5120$.
\end{proof}

The value $9/5120$ is Kondor's coefficient~\cite{Kon83} for the truncated model near the
critical temperature, $\Phi(n)=\feq-9n^5/5120$ in the replica notation, as reported by Parisi
and Rizzo~\cite[Section~IV]{PR09}. Here it is a limit: at $\xi=\xic$, $\zetas=\delta_0$ is
feasible for every constraint, so $p(\theta)=\feq$ and the fifth-order coefficient is zero;
Theorem~\ref{thm:local} holds at a fixed covariance, with no uniformity as $\xi\to\xic$.

\begin{proof}[Proof of Theorem~\ref{thm:nearcrit}\textup{(a)}]
Take $\varepsilon_*=\varepsilon_1\in(0,\frac14]$. Proposition~\ref{prop:nc-class}(b),(c) give
$\supp\zetas=[0,v_*]$, hence \condH, and $\Dstar\ge2$; Theorem~\ref{thm:local} applies, and
Proposition~\ref{prop:nc-kondor} gives $c\to9/5120$. For $1<\beta^2<1+\varepsilon_*$ the
covariance $\beta^2t^2/2$ lies in $\Kcl(\varepsilon_*)$, at distance $\beta^2-1$ from $\xic$,
and its constant $c$ is $\cb$, since $\azero=\beta^2$ and
$\frac9{640}\beta^3\rhob^{-3}=\frac9{640}\beta^3\bigl(2\beta/\Dstar\bigr)^3
=\frac{9\beta^6}{80\Dstar^3}$. Hence $\cb\to9/5120$ as $\beta\downarrow1$.
\end{proof}

For the Sherrington--Kirkpatrick covariance with $1<\beta^2<1+\varepsilon_1$,
Proposition~\ref{prop:nc-class}(b) thus gives \condH\ without the accumulation theorem
\cite[Proposition~5.1]{Lop26a}; Remark~\ref{rem:fifth-sk-inputs} compares this with Zhou's
theorem and records the resulting inputs of Theorem~\ref{thm:sk} near $\beta=1$.

\subsection{A cubic term: proof of Theorem~\ref*{thm:atom}(c)}\label{sec:atom-odd}

We now compute the correction $S(\delta)$ of Theorem~\ref{thm:atom}(a) in a class where
$m_0>0$ and the constrained minimizers above $m_0$ can be described. For $\varepsilon>0$ let
$\Kclodd(\varepsilon)$ be the set of covariances $\xi(t)=\sum_{p\ge2}\beta_p^2t^p$ with radius
of convergence greater than one and
\begin{equation}\label{eq:atom-odd-class}
 \beta_3>0,\qquad \xi''(0)=2\beta_2^2>1,\qquad \normCfour{\xi-\xic}<\varepsilon ;
\end{equation}
it differs from $\Kcl(\varepsilon)$ only in the condition on $\beta_3$. For example,
$\xi(t)=bt^2+\gamma t^3+et^4$ with $b>\frac12$, $\gamma>0$, $e\ge0$ and
$2(b-\frac12)+6\gamma+24e<\varepsilon$ belongs to $\Kclodd(\varepsilon)$, since then
$\normCfour{\xi-\xic}\le2(b-\frac12)+6\gamma+24e$, and
Proposition~\ref{prop:atom-odd-structure}(b) gives its atom $m_0=3\gamma/(2b)^{3/2}$ when
$\varepsilon\le\varepsilon_3$, where we put $\varepsilon_3=\varepsilon_1\in(0,\frac14]$, the
constant of Section~\ref{sec:nc-reference}. Throughout this subsection $\azero=\xi''(0)$,
$\Ustar$ and $\Wstar=\Phi_{\zetas,xxxx}(0,0)$ are as before, and
\begin{equation}\label{eq:atom-odd-D}
 \Dstar:=\xi''''(0)\Ustar^2-6\xi''(0)\xi'''(0)m_0\Ustar^3
   +\xi''(0)^3\bigl(\Wstar^2+6m_0^2\Ustar^4\bigr),
 \qquad c:=\frac{9\azero^3}{80\Dstar^3}.
\end{equation}
If $\xi'''(0)=0$ and $m_0=0$, this $\Dstar$ is the constant of Theorem~\ref{thm:local}.

The proofs of Sections~\ref{sec:fifth} and~\ref{sec:ons} carry over with $\theta$ replaced by
$\theta-m_0$, with two changes. First, the curvature of $\Gamma_\theta$ at zero,
$\Gamma_\theta''(0+)=\xi'''(0)\bar U_\theta^2-2\azero^2\theta \bar U_\theta^3$ with
$\bar U_\theta=\Phi_{\zeta_\theta,xx}(0,0)$, is no longer proportional to $\theta$: its leading
terms cancel at $\theta=m_0$. The onset is then decided by the rate at which $\bar U_\theta$
approaches $\Ustar$, and marginality supplies that rate
(Proposition~\ref{prop:atom-odd-onset}(a)). Second, the Legendre step changes as explained
after Proposition~\ref{prop:atom-slope}: $\Lamhat(m_0+u)\sim m_0cu^5$, and the exponent of the
correction is $5/4$. With Theorem~\ref{thm:nearcrit}(a), this describes the rate near $\feq$
for every covariance with $\xi''(0)>1$ in a small enough $C^4$-neighbourhood of $\xic$: a law
of order $\delta^{6/5}$ when $\beta_3=0$, a linear law with a correction of order
$\delta^{5/4}$ when $\beta_3>0$.

Proposition~\ref{prop:atom-odd-structure} (Appendix~\ref{app:cubic}) supplies the structure
for $\xi\in\Kclodd(\varepsilon_3)$: $\supp\zetas=[0,v_*]$ (Proposition~\ref{prop:nc-support}),
the atom $m_0\in(0,1)$ and the closed form of $\Dstar$
(Proposition~\ref{prop:atom-atom}(b)), $\zeta_\theta=\zetas$ for $\theta\le m_0$, and, for
$m_0<\theta<\theta_3$ with some $\theta_3\in(m_0,1)$ and at the levels in $(\feq,f_3)$,
$f_3=p(\theta_3)$, a support $\{0\}\cup[q_\theta,v_\theta]$ with a marginal first positive
contact.
It uses Section~\ref{sec:nc-reference} only through items~\ref{item:nc-positive}
and~\ref{item:nc-confined} and Lemma~\ref{lem:nc-gap}, and Section~\ref{sec:nc-class} only through Proposition~\ref{prop:nc-support}. The next proposition is the counterpart of
Theorem~\ref{thm:local}(a),(b) above the atom; part~(a) is the additional input.

\begin{proposition}\label{prop:atom-odd-onset}
Let $\xi\in\Kclodd(\varepsilon_3)$. Then $\Dstar\ge2$, and as $\theta\downarrow m_0$:
\begin{thmlist}
\item $q_\theta\to0$ and
 $\bar U_\theta:=\Phi_{\zeta_\theta,xx}(0,0)=\azero^{-1/2}\bigl(1+\tfrac{\Dstar}{12}q_\theta^2
 +o(q_\theta^2)\bigr)$;
\item the first contact and its atom satisfy
 \begin{equation}\label{eq:atom-odd-onset}
  \frac{q_\theta}{\theta-m_0}\longrightarrow\frac{3\sqrt{\azero}}{\Dstar},\qquad
  \frac{\zeta_\theta(\{q_\theta\})}{\theta-m_0}\longrightarrow\frac12 ;
 \end{equation}
\item with $c$ as in \eqref{eq:atom-odd-D},
 \begin{gather}
  p'(\theta)=5c(\theta-m_0)^4(1+o(1)),\qquad
  p(\theta)-\feq=c(\theta-m_0)^5(1+o(1)),\label{eq:atom-odd-p}\\
  \Lambda(\theta)-\theta\feq=m_0c(\theta-m_0)^5(1+o(1)),\qquad
  \Lambda'(\theta)-\feq=5m_0c(\theta-m_0)^4(1+o(1)).\label{eq:atom-odd-Lambda}
 \end{gather}
\end{thmlist}
\end{proposition}

The proof (Appendix~\ref{app:cubic}) follows Sections~\ref{sec:ons} and~\ref{sec:fifth} step
by step: a positive third derivative on the plateau below the first contact, three identities
at the first contact, the onset, the atom, and the envelope identity integrated from $m_0$.
The Legendre step then follows from the exact formula \eqref{eq:atom-exact}.

\begin{proposition}\label{prop:atom-odd-rate}
Let $\xi\in\Kclodd(\varepsilon_3)$. As $\delta\downarrow0$,
\begin{equation}\label{eq:atom-odd-rate}
 \Sigma(\feq+\delta)=-m_0\delta-\frac45\,(5m_0c)^{-1/4}\,\delta^{5/4}\,(1+o(1)),
 \qquad
 \frac45\,(5m_0c)^{-1/4}=\frac{8}{5\sqrt3}\Bigl(\frac{\Dstar}{\azero}\Bigr)^{3/4}m_0^{-1/4}.
\end{equation}
At the level $f=\feq+\delta$, $\theta_f-m_0=(\delta/(5m_0c))^{1/4}(1+o(1))$,
$q_f=\frac{3\sqrt{\azero}}{\Dstar}(\theta_f-m_0)(1+o(1))$ and
$\zeta_f(\{q_f\})=\frac12(\theta_f-m_0)(1+o(1))$.
\end{proposition}

\begin{proof}[Proof of Theorem~\ref{thm:atom}\textup{(c)}]
The support, the atom and the closed form of $\Dstar$ are
Proposition~\ref{prop:atom-odd-structure}(b), and $\Dstar\ge2$ is
Proposition~\ref{prop:atom-odd-onset}. The asymptotics of $\Sigma$, of $\theta_f$, $q_f$ and
$\zeta_f(\{q_f\})$ are Proposition~\ref{prop:atom-odd-rate}, and the first positive contact at
every level in $(\feq,f_3)$ is marginal by Proposition~\ref{prop:atom-odd-structure}(d).
\end{proof}

\begin{remark}[Scope]\label{rem:atom-odd-scope}
All these asymptotics are taken at a fixed covariance, and no uniformity in $\xi$ is asserted.
In particular the coefficient of $\delta^{5/4}$ in \eqref{eq:atom-odd-rate} is of order
$m_0^{-1/4}$ and diverges as $\beta_3\downarrow0$. Remark~\ref{rem:mo-fm} records which bounds
on $\Prob(N^{-1}\log\ZN\ge\feq+\delta)$ follow from \eqref{eq:atom-odd-rate}.
\end{remark}

\subsection{The de Almeida--Thouless threshold}\label{sec:nc-dat}

Theorem~\ref{thm:local} shows that the first positive contact is marginal for small $\theta$.
For every covariance of the class we now determine where this regime ends: at the constraint
where the simplest candidate, a measure with two atoms, becomes the minimizer.

For $w\ge0$, $\theta\in[0,1]$, a standard Gaussian $Z$ and continuous $F$ with
$|F(x)|\le Ce^{|x|}$, write
\[
 \E_{w,\theta}F=\frac{\E\bigl[F(\sqrt wZ)\cosh^\theta(\sqrt wZ)\bigr]}
 {\E\cosh^\theta(\sqrt wZ)},\qquad
 \Tfrak(w,\theta)=\E_{w,\theta}\tanh^2,\qquad
 \Sfrak(w,\theta)=\E_{w,\theta}\sech^4 .
\]
For $\theta\in[0,1]$ and $r\in(0,1)$ let $\zeta^{\theta,r}=\theta\delta_0+(1-\theta)\delta_r$.
On $[r,1]$ its Parisi equation has the explicit solution $\log\cosh x+\frac12(\xi'(1)-\xi'(t))$,
and on $[0,r)$ its distribution function is the constant $\theta$, so by the plateau law
\eqref{eq:plat-law} the law of $X_r$ is the $\cosh^\theta$ tilt of $\mathcal N(0,\xi'(r))$.
Consequently (Lemma~\ref{lem:C-two-atom})
\begin{equation}\label{eq:nc-two-atom}
 \Gamma_{\zeta^{\theta,r}}(r)=\Tfrak(\xi'(r),\theta),\qquad
 \Gamma_{\zeta^{\theta,r}}'(r)=\xi''(r)\,\Sfrak(\xi'(r),\theta).
\end{equation}
The contact identity $\Gamma(r)=r$ of the two-atom measure is therefore the
replica-symmetric fixed-point equation $q=\Tfrak(\xi'(q),\theta)$ at replica number $\theta$,
the general-$\xi$ form of \cite[(10)]{PR09}, and the Plefka inequality \eqref{eq:intro-plefka} at
its contact, $\xi''(q)\Sfrak(\xi'(q),\theta)\le1$, is the general-$\xi$ form of the
de Almeida--Thouless stability condition \cite{dAT78} taken at that replica number, whose
equality case for the Sherrington--Kirkpatrick covariance is \cite[(22)]{PR09}. We write
$\tau=\azero-1$ for the distance to criticality.

\begin{proposition}[The de Almeida--Thouless threshold]\label{prop:nc-dat-threshold}
There are absolute constants $K_{\rm rs}\in(0,\frac14]$ and $C$ such that the following holds
for every $\xi\in\Kcl(\varepsilon_*)$, with $\varepsilon_*$ as fixed in
Appendix~\ref{sec:dat-constants}. Put $\tau=\xi''(0)-1\in(0,\varepsilon_*)$.
\begin{thmlist}
\item For each $\theta\in[0,1]$ the equation $q=\Tfrak(\xi'(q),\theta)$ has exactly one
 solution $\qrs_\theta$ in $(0,K_{\rm rs}]$. It satisfies $\qrs_\theta\le2\tau$, and
 $\theta\mapsto\qrs_\theta$ is continuously differentiable.
\item Put $\mathfrak A(\theta)=\xi''(\qrs_\theta)\,\Sfrak(\xi'(\qrs_\theta),\theta)$. There is
 a unique $\theta_c\in(0,1)$ with $\mathfrak A>1$ on $[0,\theta_c)$, $\mathfrak A(\theta_c)=1$
 and $\mathfrak A<1$ on $(\theta_c,1]$.
\item If $0<\theta<\theta_c$, then $\supp\zeta_\theta=\{0\}\cup[q_\theta,v_\theta]$ with
 $0<q_\theta<v_\theta<\frac1{20}$, and the first positive contact is marginal:
 $\xi''(q_\theta)\E\Phi_{\zeta_\theta,xx}(q_\theta,X_{q_\theta})^2=1$.
\item If $\theta_c\le\theta<1$, then $\zeta_\theta=\theta\delta_0+(1-\theta)\delta_{\qrs_\theta}$,
 its positive contact set is $\{\qrs_\theta\}$, and with $q=\qrs_\theta$
 \[
  1-\xi''(q)\,\E\Phi_{\zeta_\theta,xx}(q,X_q)^2=1-\mathfrak A(\theta),
 \]
 which vanishes at $\theta=\theta_c$ and is positive for $\theta>\theta_c$.
\item The threshold satisfies
 \begin{equation}\label{eq:nc-dat-asymptotics}
  \Bigl|\theta_c-\tfrac23\Bigl(1+\tfrac{\xi''''(0)}4\Bigr)\tau\Bigr|\le C\tau^2 .
 \end{equation}
\end{thmlist}
\end{proposition}

For the Sherrington--Kirkpatrick covariance with $1<\beta^2<1+\varepsilon_*$, (e) reads
$\theta_c=\frac23(\beta^2-1)+O((\beta^2-1)^2)$; for $bt^2+et^4$ in the class it reads
$\theta_c=\frac23(1+6e)(2b-1)+O((2b-1)^2)$; odd degrees $p\ge5$ do not enter the leading
coefficient. As $C$ is absolute over the class, the quartic correction $1+\frac14\xi''''(0)$ is
resolved whenever $\xi''''(0)$ is large compared with $\tau$. The bound $v_\theta<\frac1{20}$
in~(c) comes from $K\le\frac1{20}$. Figure~\ref{fig:dat} shows the two regimes; in terms of
levels the proposition reads as follows.

\begin{corollary}[The marginal levels]\label{cor:nc-dat-levels}
Let $\xi\in\Kcl(\varepsilon_*)$ and $f_c=\Lambar'_+(\theta_c)$. Then $\feq<f_c<f_1$ and
$\theta_{f_c}=\theta_c$. A level $f\in(\feq,f_1)$ has a marginal first positive contact if and
only if $f\le f_c$. For $f\in(f_c,f_1)$ the measure $\zeta_f$ has exactly two atoms and the
first positive contact satisfies the strict Plefka inequality \cite[(1.8)]{Bou26b}.
Consequently, within $(\feq,f_1)$, the strict Plefka hypothesis of \cite[Theorem~1]{Bou26b}
holds exactly on $(f_c,f_1)$ and fails exactly on $(\feq,f_c]$.
\end{corollary}

\begin{proof}
The covariance lies in $\Kcl(\varepsilon_1)$, so it satisfies \condH\ by
Proposition~\ref{prop:nc-class}(b), and Theorem~\ref{thm:local}(c) applies: $\Lambda$ is $C^1$
and strictly convex on $(0,1)$, $\Lambda'$ maps $(0,1)$ increasingly onto $(\feq,f_1)$, and
$\theta_f=(\Lambda')^{-1}(f)$ for every $f\in(\feq,f_1)$. Hence
$f_c=\Lambda'(\theta_c)\in(\feq,f_1)$, $\theta_{f_c}=\theta_c$, and $\theta_f\le\theta_c$ if and
only if $f\le f_c$. By Proposition~\ref{prop:nc-dat-threshold}, the first positive contact of
$\zeta_f=\zeta_{\theta_f}$ is marginal when $\theta_f<\theta_c$, by~(c), and when
$\theta_f=\theta_c$, by~(d) with $\mathfrak A(\theta_c)=1$; when $\theta_f>\theta_c$, the
measure has two atoms and $1-\mathfrak A(\theta_f)>0$. The condition \cite[(1.8)]{Bou26b} is the
strict form of the Plefka inequality \eqref{eq:intro-plefka} at the first positive contact.
\end{proof}

Every $\xi\in\Kcl(\varepsilon_*)$ satisfies \cite[Assumption~1.1]{Bou26b}, which asks for a
series with nonnegative coefficients and radius of convergence greater than one that is
strictly convex on $[-1,1]$; strict convexity is Proposition~\ref{prop:nc-support}(a).

\paragraph{Structure of the proof of Proposition~\ref*{prop:nc-dat-threshold}.}
The proof, in Appendix~\ref{app:dat}, separates an explicit computation along the two-atom
branch from a convexity argument that identifies the minimizer.
\begin{steplist}
\item \emph{The replica-symmetric branch.} The functions $\Tfrak$ and $\Sfrak$ are smooth in
 $(w,\theta)$, with explicit jets at $w=0$ obtained from the heat equation
 (Lemma~\ref{lem:C-moments}). Writing $\xi'(q)=q\ell(q)$ with
 $\ell(q)=\int_0^1\xi''(sq)\,ds$, the fixed-point equation becomes $\Upsilon(q,\theta)=0$ for a
 smooth $\Upsilon$ with $\Upsilon(0,\theta)=\tau$ and $\partial_q\Upsilon\le-\frac12$ near
 $q=0$, which gives the root $\qrs_\theta\in(0,2\tau]$. Along the root,
 $\mathfrak A(\theta)-1=\qrs_\theta\,\mathfrak W(\theta)$ with $\mathfrak W$ strictly
 decreasing, $\mathfrak W(0)>0>\mathfrak W(1)$, and two Taylor expansions locate the zero
 $\theta_c$ of $\mathfrak W$ (Lemma~\ref{lem:C-rs}). The nonnegative coefficients control
 $\ell$ and the derivatives of $\xi''$ near zero uniformly over the class
 (Lemma~\ref{lem:dat-class-bounds}), so $K_{\rm rs}$ and $C$ are absolute
 (Lemma~\ref{lem:dat-branch}).
\item \emph{Necessity.} By Lemma~\ref{lem:nc-regular}, a constrained minimizer either has
 interval positive support with a marginal first contact, or has the form $\zeta^{\theta,r}$
 with $r\in(0,K)$. In the second case $r$ is a contact, so \eqref{eq:nc-two-atom} forces $r$ to
 be the only root of the fixed-point equation in $(0,K)\subset(0,K_{\rm rs}]$, namely
 $\qrs_\theta$ (Lemma~\ref{lem:C-rs}), and $\mathfrak A(\theta)\le1$. Hence $\zeta_\theta$ has
 interval positive support, contained in $(0,K)$, when $\theta<\theta_c$.
\item \emph{Sufficiency.} Conversely, if $\mathfrak A(\theta)\le1$, the obstacle of
 $\zeta^{\theta,\qrs_\theta}$ has its strict minimum at $\qrs_\theta$: $\Gamma'$ is strictly
 convex on both plateaus $(0,\qrs_\theta)$ and $(\qrs_\theta,K)$ by
 \eqref{eq:nc-positive-third}, $\Gamma(t)<t$ on $[K,1]$ by \eqref{eq:nc-confinement}, and in
 the equality case the right second derivative $\Gamma''(\qrs_\theta+)$ is negative
 (Lemma~\ref{lem:C-two-atom}). Convexity of $P$ and the first variation along mixing paths
 \cite[Lemma~3.2 and Example~3.4]{JT17} then make $\zeta^{\theta,\qrs_\theta}$ the constrained
 minimizer (Lemma~\ref{lem:C-sufficiency}).
\end{steplist}
No monotonicity of $\theta\mapsto\zeta_\theta$ is needed: the only monotone quantity is the
explicit scalar $\mathfrak W=(\mathfrak A-1)/\qrs_\theta$ along the replica-symmetric branch, whose sign is that of $\mathfrak A-1$.

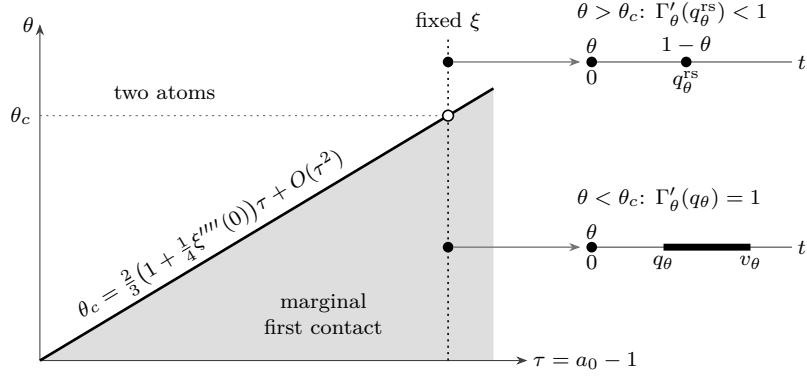
\begin{figure}[tp]
\centering
\begin{tikzpicture}[x=1cm, y=1cm, font=\scriptsize,
  >={Stealth[length=1.5mm,width=1.1mm]},
  axis/.style={->, line width=0.4pt, draw=black!75},
  threshold/.style={line width=1.0pt, draw=black},
  fixed/.style={line width=0.6pt, draw=black, dash pattern=on 0.8pt off 1.6pt},
  guide/.style={line width=0.4pt, draw=black!60, dash pattern=on 0.8pt off 1.6pt},
  link/.style={->, line width=0.5pt, draw=black!55},
  supp/.style={line width=2.6pt, draw=black},
  numline/.style={line width=0.4pt, draw=black!75},
  dot/.style={circle, fill=black, inner sep=0pt, minimum size=4pt},
  odot/.style={circle, draw=black, fill=white, line width=0.6pt, inner sep=0pt,
    minimum size=4pt},
]
\fill[black!13] (0,0) -- (6.0,3.6) -- (6.0,0) -- cycle;
\draw[axis] (0,0) -- (6.45,0);
\node[anchor=west, inner sep=2pt] at (6.45,0) {$\tau=\azero-1$};
\draw[axis] (0,0) -- (0,4.45) node[anchor=east, inner sep=2pt] {$\theta$};
\draw[threshold] (0,0) -- node[sloped, anchor=south, pos=0.40, inner sep=2pt]
  {$\theta_c=\frac23\bigl(1+\frac14\xi''''(0)\bigr)\tau+O(\tau^2)$} (6.0,3.6);
\node[align=center] at (1.65,3.55) {two atoms};
\node[align=center] at (3.75,0.62) {marginal\\ first contact};
\draw[fixed] (5.4,0) -- (5.4,4.3);
\node[anchor=south, inner sep=1.5pt] at (5.4,4.3) {fixed $\xi$};
\draw[guide] (0,3.24) -- (5.4,3.24);
\node[anchor=east, inner sep=3pt] at (0,3.24) {$\theta_c$};
\node[odot] at (5.4,3.24) {};
\node[dot] (P2) at (5.4,3.95) {};
\node[dot] (P1) at (5.4,1.5) {};
\draw[link] (P2) -- (7.15,3.95);
\draw[link] (P1) -- (7.15,1.5);
\draw[numline] (7.3,3.95) -- (9.95,3.95) node[anchor=west, inner sep=2pt] {$t$};
\node[dot] at (7.3,3.95) {};
\node[dot] at (8.55,3.95) {};
\node[anchor=south, inner sep=3pt] at (7.3,3.95) {$\theta$};
\node[anchor=south, inner sep=3pt] at (8.55,3.95) {$1-\theta$};
\node[anchor=north, inner sep=3pt] at (7.3,3.95) {$0$};
\node[anchor=north, inner sep=3pt] at (8.55,3.95) {$\qrs_\theta$};
\node[anchor=south west, inner sep=0pt] at (7.1,4.42)
  {$\theta>\theta_c$: $\Gamma_\theta'(\qrs_\theta)<1$};
\draw[numline] (7.3,1.5) -- (9.95,1.5) node[anchor=west, inner sep=2pt] {$t$};
\node[dot] at (7.3,1.5) {};
\draw[supp] (8.25,1.5) -- (9.4,1.5);
\node[anchor=south, inner sep=3pt] at (7.3,1.5) {$\theta$};
\node[anchor=north, inner sep=3pt] at (7.3,1.5) {$0$};
\node[anchor=north, inner sep=3pt] at (8.25,1.5) {$q_\theta$};
\node[anchor=north, inner sep=3pt] at (9.4,1.5) {$v_\theta$};
\node[anchor=south west, inner sep=0pt] at (7.1,1.97)
  {$\theta<\theta_c$: $\Gamma_\theta'(q_\theta)=1$};
\end{tikzpicture}
\caption{The de Almeida--Thouless threshold of Theorem~\ref{thm:nearcrit}(b),
schematically, along a one-parameter family in $\Kcl(\varepsilon_*)$ with fixed
$\xi''''(0)$ (for instance the Sherrington--Kirkpatrick covariances, $\tau=\beta^2-1$), in
the plane of $\tau=\azero-1$ and the constraint $\theta$. The threshold is drawn by its
leading term. Below it (grey), $\supp\zeta_\theta=\{0\}\cup[q_\theta,v_\theta]$ and the first
positive contact is marginal (Proposition~\ref{prop:nc-dat-threshold}(c)). For
$\theta_c\le\theta<1$, the minimizer $\zeta_\theta=\theta\delta_0+(1-\theta)\delta_{\qrs_\theta}$
has exactly two atoms, and its only positive contact is strictly Plefka stable when
$\theta>\theta_c$ (Proposition~\ref{prop:nc-dat-threshold}(d)). The insets show
$\supp\zeta_\theta$ on either side of the threshold at one covariance (dotted line), which
crosses it once, with the masses of the atoms at zero and, above the threshold, at
$\qrs_\theta$. Gap exclusion
(Figure~\ref{fig:plateau}(a)) needs two positive contacts; beyond $\theta_c$ there is only
one, and the Plefka inequality at it is strict. In terms of levels, the first positive
contact is marginal exactly on $(\feq,f_c]$ (Corollary~\ref{cor:nc-dat-levels}). For the
Sherrington--Kirkpatrick covariances the leading slope $\frac23$ in $\tau$ is $\frac43$ in
$1-T$, the slope of the line $n_{\rm dAT}=\frac43(1-T)+O((1-T)^2)$ of Parisi and Rizzo
\cite[(22)--(23)]{PR09}; a positive $\xi''''(0)$ multiplies it by $1+\frac14\xi''''(0)$.}
\label{fig:dat}
\end{figure}

\begin{remark}\label{rem:nc-dat}
(1) For the Sherrington--Kirkpatrick model, Parisi and Rizzo located this threshold by a
replica computation at replica number $n$ \cite[Section~III, (10), and Section~III.C, (22)--(23)]{PR09}. With
$\xi(t)=\beta^2t^2/2$ and $n=\theta$, their equations (10) and (22) are the equation of
Proposition~\ref{prop:nc-dat-threshold}(a) and $\mathfrak A(\theta)=1$, and their expansion
$n_{\rm dAT}=\frac43(1-T)+O((1-T)^2)$ agrees with (e), since
$\tau=\beta^2-1=2(1-T)+O((1-T)^2)$ and $\xi''''(0)=0$; they describe the
replica-symmetry-breaking solution below the line in \cite[Section~IV]{PR09}. For
$1<\beta^2<1+\varepsilon_*$, Proposition~\ref{prop:nc-dat-threshold} is the constrained-Parisi
form, near criticality, of this prediction: the two-atom measure is the constrained minimizer
exactly when $\theta\ge\theta_c$, and below $\theta_c$ the minimizer carries an interval
starting at a marginal first contact. A positive $\xi''''(0)$ multiplies the leading term of
$\theta_c$ by $1+\frac14\xi''''(0)$. For the Sherrington--Kirkpatrick covariance,
Theorem~\ref{thm:sk-global}(d) extends the proposition to every $\beta>1$, with the same
$\mathfrak A$ and a threshold $\theta_c(\beta)$ that agrees with the present one for
$1<\beta^2<1+\varepsilon_*$; that extension rests on the unrefereed input of
Proposition~\ref{prop:skg-oai} (Section~\ref{sec:skg}).

(2) Every admissible $\theta_0$ in Theorem~\ref{thm:local}(a) satisfies
$\theta_0\le\theta_c$, because for $\theta\ge\theta_c$ the measure $\zeta_\theta$ has two
atoms and its only positive contact is $\qrs_\theta$, by~(d). By~(c) and
Lemma~\ref{lem:nc-regular}, $\theta_0=\theta_c$ is admissible for $\xi\in\Kcl(\varepsilon_*)$. In particular the range of constraints on which
the marginal onset of Theorem~\ref{thm:local} operates shrinks to zero as $\azero\downarrow1$,
in line with the absence of uniformity noted after Proposition~\ref{prop:nc-kondor}.
\end{remark}

\begin{proof}[Proof of Theorem~\ref{thm:nearcrit}\textup{(b)}]
Let $\varepsilon_*$ be as fixed in Appendix~\ref{sec:dat-constants}; it does not exceed
$\varepsilon_1$, so part~(a) holds with it. Proposition~\ref{prop:nc-dat-threshold} is proved
in Appendix~\ref{app:dat}. Its part (b) gives $\theta_c\in(0,1)$, and its parts (c) and (d) give the two regimes, which determine
$\theta_c$ uniquely, and the strict Plefka stability above $\theta_c$; (e) gives the asymptotics with
an absolute implied constant, including the form stated for $\xi(t)=\beta^2t^2/2$, where
$\azero=\beta^2$ and $\xi''''(0)=0$; Corollary~\ref{cor:nc-dat-levels} gives the levels.
\end{proof}

\subsection{Stability above the first contact}\label{sec:nc-stability}

The derivation of the fifth-order law at all temperatures by Parisi and Rizzo assumes that the
order parameter $q(x)$ is unperturbed above its first plateau \cite[Appendix~B]{PR09}, and they
predict that it changes there only by $O(n^5)$ \cite[Section~IV]{PR09}. The proof of
Theorem~\ref{thm:local} does not assume that the measure is unperturbed above its first contact: it uses the plateau below the first contact, the
interval of contacts that starts there (Proposition~\ref{prop:ons-onset}(b)), which makes that contact marginal, and $\Wone(\zeta_\theta,\zetas)\to0$. For every covariance of the class, the next proposition shows that above its
first contact the constrained measure agrees with $\zetas$ up to $O(\theta^5)$, and gives the
onset limits with rates.

Fix $\xi\in\Kcl(\varepsilon_*)$; it lies in $\Kcl(\varepsilon_1)$. For small $\theta>0$,
$\supp\zeta_\theta=\{0\}\cup[q_\theta,v_\theta]$ by Proposition~\ref{prop:nc-class}(d), and
$\supp\zetas=[0,v_*]$. Write $\alpha_\theta$ and $\alphas$ for the distribution functions of
$\zeta_\theta$ and $\zetas$, $\Delta_\theta=\alpha_\theta-\alphas$,
$x_\theta=\alpha_\theta(q_\theta)$, $v_-=\min\{v_\theta,v_*\}$ and
$v_+=\max\{v_\theta,v_*\}$. The measure $\Shift_q\zeta=\zeta|_{(q,1]}+\zeta([0,q])\delta_q$
has distribution function $\ind_{[q,1]}\alpha_\zeta$, so the parts of the two measures above
the first contact are compared by
\[
 L_\theta=\Wone\bigl(\Shift_{q_\theta}\zeta_\theta,\Shift_{q_\theta}\zetas\bigr)
 =\int_{q_\theta}^1|\Delta_\theta(t)|\,dt .
\]
For a distribution function $\alpha$ on $[0,1]$, its quantile function is
$Q(x)=\inf\{t:\alpha(t)\ge x\}$, $x\in(0,1]$; $Q_\theta$ and $Q_*$ belong to $\zeta_\theta$ and
$\zetas$.

\begin{proposition}[Stability above the first contact]\label{prop:nc-above-contact}
For every $\xi\in\Kcl(\varepsilon_*)$, as $\theta\downarrow0$:
\begin{thmlist}
\item $L_\theta=O(\theta^5)$;
\item $|v_\theta-v_*|=O(\theta^5)$ and
 $\sup\{|\Delta_\theta(t)|:t\in[q_\theta,v_-)\cup[v_+,1]\}=O(\theta^5)$; in particular
 $x_\theta=\zetas([0,q_\theta])+O(\theta^5)$;
\item $\sup_{x_\theta<x<1}|Q_\theta(x)-Q_*(x)|=O(\theta^5)$;
\item $q_\theta=\dfrac{3\sqrt\azero}{\Dstar}\,\theta\,(1+O(\theta))$,
 $\zeta_\theta(\{q_\theta\})=\dfrac\theta2\,(1+O(\theta))$, and
 \[
  \Wone(\zeta_\theta,\zetas)=\frac{5\sqrt\azero}{4\Dstar}\,\theta^2\,(1+O(\theta)).
 \]
\end{thmlist}
The implied constants depend on $\xi$.
\end{proposition}

By~(d), the transport distance between $\zeta_\theta$ and $\zetas$ has exact order $\theta^2$,
and all of it except $O(\theta^5)$ is carried by $[0,q_\theta]$, where $\zeta_\theta$ has its
plateau and its atom. For the Sherrington--Kirkpatrick covariance with
$1<\beta^2<1+\varepsilon_*$, statement~(c) is the constrained-Parisi form, near criticality,
of the prediction $q(x)=q_{\rm free}(x)+O(n^5)$ for $x_c<x<1$ of Parisi and Rizzo
\cite[Section~IV]{PR09}, with $n=\theta$, $x_c=x_\theta$, $q=Q_\theta$ and
$q_{\rm free}=Q_*$. Their prediction is stated at every temperature; the argument below uses
the smallness of the support near criticality. Their remark that $q(x)$ is continuous at $x_c$
concerns the constrained order parameter and corresponds to $Q_\theta(x_\theta+)=q_\theta$,
which follows from $\supp\zeta_\theta=\{0\}\cup[q_\theta,v_\theta]$. Part~(b) extends the
$O(\theta^5)$ agreement to the end of the plateau, and~(d) gives the onset values
$x_\theta=\frac32\theta+O(\theta^2)$ and $q_\theta=\frac3{2\rhostar}\theta+O(\theta^2)$ with
rates.

\paragraph{Structure of the proof of Proposition~\ref*{prop:nc-above-contact}.}
The proof, in Appendix~\ref{app:stability}, compares the two measures through the law of the
diffusion at the common point $q_\theta$, which is a contact of both: $q_\theta$ is the first
positive contact of $\zeta_\theta$, and for small $\theta$ it is an interior support point of
$\zetas$, so that $\Gamma_\theta(q_\theta)=q_\theta=\Gamma_*(q_\theta)$.
\begin{steplist}[label=\textup{(S\arabic*)}, ref=(S\arabic*), widest=S4]
\item \emph{Relations on the supports.} On the interior of each support interval
 $\Gamma'\equiv1$, and the one-sided formulas for $\Gamma''$ in Lemma~\ref{lem:plat-ito}(b)
 force $\alpha_\zeta=N_\zeta/\sigma_\zeta$, with $N_\zeta$ and $\sigma_\zeta$ as in
 \eqref{eq:plat-NS}, computed along the diffusion of $\zeta$
 (Lemma~\ref{lem:D-support-relation}). Above $q_\theta$ the difference $\Delta_\theta$ is
 therefore driven by the jets of the two measures, which see only $\Delta_\theta$ on $[t,1]$
 and differ by $O(L_\theta)$, and by the laws of the two diffusions at time $t$.
\item \emph{The plateau acts through one scalar.} Law differences are written through the
 backward Kolmogorov function of the diffusion of $\zetas$ \eqref{eq:D-duality}. Their part
 over $[0,q_\theta]$ is the difference of the laws of $X_{q_\theta}$ under the two measures,
 tested against an even function $H$. Up to $O(\theta^5+\theta^3L_\theta)$, this difference is
 a combination of $H''(0)$ and $H''''(0)$ whose two coefficients do not depend on $H$, and
 both are controlled, up to $O(\theta^4)$, by the first moment
 $I_\theta=\int_0^{q_\theta}\Delta_\theta(s)\,s\,ds$ (Lemma~\ref{lem:D-plateau}). The reason
 is Gaussian: for a centred Gaussian of variance $v$, a change $dv$ changes the fourth moment
 by $6v\,dv$, so at the relevant order the plateau changes the law at $q_\theta$ only through
 its variance.
\item \emph{Closure by the common contact.} Testing the identity
 $\Gamma_\theta(q_\theta)=\Gamma_*(q_\theta)$ with the even function
 $\Phi_{\zetas,x}(q_\theta,\cdot)^2$ gives $I_\theta=O(\theta^4+\theta L_\theta)$; this is the
 $3/2$ onset law, since $I_\theta=q_\theta^2(\frac\theta2-\frac{\rhostar q_\theta}3)+O(q_\theta^4)$.
 Every even test function then sees a law difference $O(\theta^5+\theta L_\theta)$ at time
 $q_\theta$ (Lemma~\ref{lem:D-contact-closure}), and the bound propagates to all times in
 $[q_\theta,K]$ at the cost of a term $C_0tL_\theta$ (Lemma~\ref{lem:D-propagation}).
\item \emph{Contraction.} The support relation then bounds $\Delta_\theta$ pointwise on
 $(q_\theta,v_-)$, and the bound $\Gamma''\le-1$ on a plateau of height one inside $(0,K)$
 bounds $|v_\theta-v_*|$ (\eqref{eq:D-pointwise} and \eqref{eq:D-endpoint}). Together,
 $L_\theta\le2C_0K\,L_\theta+2C(\theta^5+\theta L_\theta)$, where $C_0$ depends only on the
 $C^4$ ball of radius $\frac14$ about $\xic$ and $C$ on the covariance. Since $K$ is chosen in
 Appendix~\ref{sec:dat-constants} with $2C_0K\le\frac14$, this gives $L_\theta=O(\theta^5)$,
 and parts (b)--(d) follow, with a lower bound $\frac12$ on the slopes of $\alpha_\theta$
 and $\alphas$ on their support intervals for the quantiles.
\end{steplist}

\needspace{3\baselineskip}
\begin{remark}\label{rem:nc-above-contact}
(1) The $3/2$ onset law is the vanishing of $I_\theta$ at order $\theta^3$, so the proof
of~(d) recovers the onset limits \eqref{eq:nc-onset}, with rates, given $q_\theta\asymp\theta$
and the support structure of Proposition~\ref{prop:nc-class}.

(2) The argument is perturbative near criticality through the smallness of $K$. It gives no
uniformity in the covariance, and no statement for the Sherrington--Kirkpatrick model away
from criticality, where $\supp\zetas$ is not contained in a short interval $[0,K]$ and the
contraction step does not apply.
\end{remark}

\begin{proof}[Proof of Theorem~\ref{thm:nearcrit}\textup{(c)}]
Propositions~\ref{prop:nc-dat-threshold} and~\ref{prop:nc-above-contact} hold on
$\Kcl(\varepsilon_*)$ with $\varepsilon_*$ fixed in Appendix~\ref{sec:dat-constants}, and
parts (c) and (d) of Proposition~\ref{prop:nc-above-contact}, proved in
Appendix~\ref{app:stability}, are the assertions of Theorem~\ref{thm:nearcrit}(c).
\end{proof}

\section{The Sherrington--Kirkpatrick model at every temperature: proof of
Theorem~\ref*{thm:sk-global}}\label{sec:skg}

Section~\ref{sec:nc} took two global facts from closeness to the critical reference: the third
derivative of $\Gamma_\zeta$ is positive on plateaus below a fixed level, so no plateau separates
two contacts there, and no contact lies above that level. For the Sherrington--Kirkpatrick
covariance one inequality replaces both at every temperature: $\Gamma_\zeta'$ is strictly convex
on every plateau of positive height, for every $\zeta$ (Theorem~\ref{thm:sk-global}(a)). With the
trapezoid principle (Lemma~\ref{lem:plat-obstacle}(d)) it excludes a plateau between any two
contacts (Theorem~\ref{thm:sk-global}(b)), and this forces the positive support of a constrained
minimizer to be a point or an interval (Theorem~\ref{thm:sk-global}(c)). When it is a point, the
minimizer has two atoms, and a sign-pattern argument shows that the replica-symmetric contact
conditions characterize it (Proposition~\ref{prop:skg-two-atom}). Which case occurs is then
decided by one scalar along the replica-symmetric branch, the de Almeida--Thouless function
$\mathfrak A$, whose monotonicity is an elementary inequality for Gaussian integrals
(Lemma~\ref{lem:skg-branch}).

Section~\ref{sec:skg-curvature} states the imported inequality and proves
Theorem~\ref{thm:sk-global}(a); the approximation that transfers the inequality to every measure
is in Appendix~\ref{sec:H-transfer}. Sections~\ref{sec:skg-gap}, \ref{sec:skg-dichotomy}
and~\ref{sec:skg-threshold} prove parts (b), (c), and (d) with (e).

Throughout this section $\xi(t)=\beta^2t^2/2$, so that $a=\xi''\equiv\beta^2$, $a'=a''=0$ and
$\xi'(t)=\beta^2t$. We keep the notation of Section~\ref{sec:plat}: $U_t$, $V_t$ and $W_t$ are the
second to fourth spatial derivatives of $\Phi_\zeta$ along the diffusion $X=X^\zeta$ of
\eqref{eq:intro-diffusion}. For $\theta\in(0,1)$ we write $\Gamma_\theta=\Gamma_{\zeta_\theta}$ and
$\Obs_\theta=\Obs_{\zeta_\theta}$, and $\Tfrak$, $\Sfrak$ are the functions \eqref{eq:intro-TS}.

\subsection{Curvature on plateaus}\label{sec:skg-curvature}

The input comes from a proof that the zero-temperature Sherrington--Kirkpatrick order parameter
has full support \cite{OAI26c}, an unrefereed preprint. That proof approximates the
zero-temperature problem by step-function order parameters at a finite inverse temperature $M$,
and we use the finite-$M$ statement. Let $D\ge1$ be an integer, let $M>0$, and let
$c_0\le c_1\le c_2\le\cdots$ be positive numbers with $c_j\le M$ for every $j$. Put
$\gamma(t)=c_j$ for $t\in[j/D,(j+1)/D)$, $j=0,\dots,D-1$, let $\Phi^\gamma$ solve
\begin{equation}\label{eq:skg-oai-pde}
 \partial_t\Phi^\gamma=-\tfrac12\bigl(\Phi^\gamma_{xx}+\gamma\,(\Phi^\gamma_x)^2\bigr)
 \quad\text{on }[0,1],\qquad \Phi^\gamma(1,x)=M^{-1}\log(2\cosh Mx),
\end{equation}
and let $Y$ solve $dY_t=\gamma(t)\Phi^\gamma_x(t,Y_t)\,dt+dB_t$ with $Y_0=0$.

\begin{proposition}[Finite-step coercivity {\cite[Section~3 and Proposition~4.2]{OAI26c}}]
\label{prop:skg-oai}
In this setting let $0<i\le D$ be an integer, and put $r=c_i\Phi^\gamma_x(i/D,\cdot)$, $v=r_x$,
$w=r_{xx}$ and $z=r_{xxx}$. Then
\[
 \E\bigl[(z^2-12vw^2+6v^4)(Y_{i/D})\bigr]\ge\frac1{1000}\,\E\,v(Y_{i/D})^4 .
\]
\end{proposition}

In \cite{OAI26c} this inequality is Proposition~4.2, stated with the constant
$1661507/1521000000>1/1000$ for an arbitrary partition into steps; Lemma~3.4 there verifies its
hypotheses for these finite-step data at a fixed positive time inside a step, and the proof of
\cite[Proposition~5.1]{OAI26c} combines the two. The equal-step form with the constant $1/1000$
displayed above is also a lemma of the Lean formalization that accompanies \cite{OAI26c}. We use
only this form. Although \cite{OAI26c} concerns zero temperature, Proposition~\ref{prop:skg-oai}
is a statement at positive temperature: $M$ is fixed, and the change of variables
$\Phi^\gamma(t,x)=\beta^{-1}\{\Phi_\zeta(t,\beta x)+\log2\}$, $\gamma=\beta\alpha_\zeta$, $M=\beta$
turns \eqref{eq:skg-oai-pde} into the Parisi equation \eqref{eq:intro-pde} of the
Sherrington--Kirkpatrick covariance (Appendix~\ref{sec:H-transfer}). The proofs below of every
result of this section except Lemma~\ref{lem:skg-branch} and Proposition~\ref{prop:skg-branch}
use Proposition~\ref{prop:skg-oai}, an unrefereed input; for parts (b)--(e), a qualitative
consequence suffices, which \cite{Lop26a} also supplies (Remark~\ref{rem:skg-curvature}).

The next proposition transfers the inequality from step functions at grid times to every measure
and every time in a plateau. The constant is absolute and $M=\beta$ stays fixed, so the only limit
needed is the weak continuity of expectations along the diffusion.

\begin{proposition}[Plateau coercivity]\label{prop:skg-plateau}
Let $\beta>0$ and $\zeta\in\Pspace([0,1])$, and let $(u,v)\subset(0,1)$ be an open interval with
$\zeta((u,v))=0$ and height $m=\zeta([0,u])>0$. Then for every $t\in(u,v)$
\begin{equation}\label{eq:skg-plateau}
 \E\bigl[W_t^2-12mU_tV_t^2+6m^2U_t^4\bigr]\ge\frac1{1000}\,m^2\,\E U_t^4 .
\end{equation}
\end{proposition}

The proof, in Appendix~\ref{sec:H-transfer}, approximates $\zeta$ by finitely atomic measures
that keep the height $m$ near a rational time $t$ and whose distribution functions are step
functions on a uniform rational grid, applies Proposition~\ref{prop:skg-oai} at $t$, passes to
the limit with \cite[Proposition~1(ii)]{AC15a}, and reaches every $t$ by continuity.

\begin{proof}[Proof of Theorem~\ref{thm:sk-global}\textup{(a)}]
Let $m>0$ and $u<t<v$. By Lemma~\ref{lem:plat-ito}(c) with $a\equiv\beta^2$,
$\Gamma_\zeta'''(t)=\beta^6\E(W_t^2-12mU_tV_t^2+6m^2U_t^4)$, and
Proposition~\ref{prop:skg-plateau} gives
$\Gamma_\zeta'''(t)\ge\frac1{1000}\beta^6m^2\E U_t^4$. This is positive, because $U_t>0$
by~\eqref{eq:plat-bounds}.

Let $m=0$. Then $\alpha_\zeta=0$ on $[0,v)$, so $X_t=\beta B_t$ is Gaussian for $t<v$,
and~\eqref{eq:plat-second} gives $\Gamma_\zeta''(t)=\beta^4\E V_t^2$ on $(u,v)$. If this vanished
at some $t\in(u,v)$, the continuous function $\Phi_{\zeta,xxx}(t,\cdot)$ would vanish
identically, because the law of $X_t$ has a positive density. Then $\Phi_{\zeta,x}(t,\cdot)$
would be affine and, being bounded, constant, so that $\Phi_{\zeta,xx}(t,\cdot)\equiv0$,
contradicting~\eqref{eq:plat-bounds}.
\end{proof}

\begin{remark}\label{rem:skg-curvature}
Near zero and for measures close to $\zetas$, Lemma~\ref{lem:plat-local-plateau} gives
$\Gamma_\zeta'''\ge\Dstar/2$ on plateaus, for every covariance with $\xi'''(0)=0$. Part (a) needs
no closeness and holds on every plateau of positive height of every measure, but only for this
covariance: the inequality of \cite{OAI26c} uses that $\xi''$ is constant. For every probability
measure whose distribution function equals a constant $m>0$ on an open interval $(u,v)$, Lopatto
proves that $\Gamma'''>0$ at each point of the interval where $\Gamma''=0$
(\cite[Proposition~5.1]{Lop26a} in version~4). Part (a) removes the condition $\Gamma''=0$ and
gives a quantitative lower bound.

The proofs of parts (b)--(e) below use the case $m>0$ of part (a) only through a qualitative
consequence that Lopatto's result already gives. By that result every critical point of
$\Gamma'$ in $(u,v)$ is a strict local minimum, so $\Gamma'$ has no interior local maximum: it is
strictly monotone, or strictly decreasing and then strictly increasing (cf.\
\cite[Corollary~5.6]{Lop26a} in version~4). By continuity, $\Gamma'(s)<\max\{\Gamma'(u),
\Gamma'(v)\}$ for $u<s<v$, and $\Gamma'(s_2)\le\max\{\Gamma'(s_1),\Gamma'(s_3)\}$ for
$u\le s_1<s_2<s_3\le v$. The first inequality replaces strict convexity in the proof of
Lemma~\ref{lem:plat-obstacle}(d), and the second is all that the proof of
Lemma~\ref{lem:H-signs} uses. Hence parts (b)--(e) and Corollary~\ref{cor:skg-interval} also
follow with \cite{Lop26a} (version~4) in place of Proposition~\ref{prop:skg-oai}; for $\zetas$
this is the gap-exclusion argument of \cite[Proposition~6.1]{Lop26a} in version~4. Only the quantitative
bound in part (a) needs \cite{OAI26c}.

The constant $1/1000$ is that of Proposition~\ref{prop:skg-oai}. The statement of
\cite[Proposition~4.2]{OAI26c} has the constant $1661507/1521000000$; using it, with the
hypotheses verified by \cite[Lemma~3.4]{OAI26c} in the proof given there rather than the
formalized equal-step form, the same proof gives the correspondingly larger bound.
\end{remark}

\subsection{Gap exclusion and the support of the Parisi measure}\label{sec:skg-gap}

\begin{proof}[Proof of Theorem~\ref{thm:sk-global}\textup{(b)}]
Let $0<u<v<1$ be local minimum points of $\Obs_\zeta$. By Lemma~\ref{lem:plat-obstacle}(b),
$\Gamma_\zeta(u)=u$, $\Gamma_\zeta(v)=v$, $\Gamma_\zeta'(u)\le1$ and $\Gamma_\zeta'(v)\le1$.
Suppose that $\zeta((u,v))=0$, so that $(u,v)$ is a plateau of some height $m\in[0,1]$. If
$m>0$, part (a) makes $\Gamma_\zeta'$ strictly convex on $(u,v)$, contradicting
Lemma~\ref{lem:plat-obstacle}(d). If $m=0$, part (a) makes $\Gamma_\zeta'$ strictly increasing on
$(u,v)$; since $\Gamma_\zeta'$ is continuous on $[u,v]$, $\Gamma_\zeta'<\Gamma_\zeta'(v)\le1$ on
$(u,v)$, and $v-u=\Gamma_\zeta(v)-\Gamma_\zeta(u)=\int_u^v\Gamma_\zeta'<v-u$, again a
contradiction.

For the second statement, $\Gamma_\zeta(0)=0$ (Lemma~\ref{lem:plat-ito}(a)),
$\Gamma_\zeta'(0)\le1$ by Lemma~\ref{lem:plat-obstacle}(c), and $v$ satisfies the contact
conditions of Lemma~\ref{lem:plat-obstacle}(b). The argument above applies with $u=0$, which
Lemma~\ref{lem:plat-obstacle}(d) allows.
\end{proof}

This is the mechanism of Lemma~\ref{lem:plat-local-gap} and of Section~\ref{sec:nc-reference},
with the sign of $\Gamma'''$ supplied by part (a) on every plateau. Applied to $\zetas$, it gives
the structure of the Parisi measure.

\begin{corollary}[Interval support]\label{cor:skg-interval}
Let $\beta>0$. Then $\supp\zetas=[0,q_P]$ for some $q_P\in[0,1)$, and $\Obs_*$ attains its
minimum over $[0,1]$ exactly on $\supp\zetas$. If $\beta>1$, then $q_P>0$; in particular
\condH\ holds and $\zetas(\{0\})=0$.
\end{corollary}

\begin{proof}
Put $S=\supp\zetas$ and let $C_P$ be the set of minimum points of $\Obs_*$ on $[0,1]$. By
Lemma~\ref{lem:ons-zero-in-support}, $0\in S\subset C_P$, and $1\notin C_P$ by
Lemma~\ref{lem:plat-obstacle}(a), so $S\subset[0,1)$.
Every point of $C_P$ is a local minimum point of $\Obs_*$, and $\Obs_*(t)\ge\Obs_*(0)$ for all $t$
because $0\in C_P$.

\emph{$S$ has no gaps.} Let $s_1<s_2$ be points of $S$ with $(s_1,s_2)\cap S=\emptyset$, so that
$\zetas((s_1,s_2))=0$ and $s_2\in(0,1)$. If $s_1>0$, part (b) gives $\zetas((s_1,s_2))>0$; if
$s_1=0$, its second statement gives the same. Both contradict the choice of $s_1,s_2$. A closed
subset of $[0,1)$ that contains $0$ and has no gaps is an interval $[0,q_P]$.

\emph{$C_P=S$.} Let $c\in C_P\setminus S$. Then $q_P<c<1$ and $\zetas((q_P,c))=0$. If $q_P>0$,
this contradicts part (b) applied to $q_P$ and $c$, and if $q_P=0$, its second statement applied
to $c$.

\emph{$\beta>1$.} If $q_P=0$, then $\zetas=\delta_0$, for which
$\Phi_{\delta_0}(t,x)=\log\cosh x+\frac{\beta^2}2(1-t)$, so $\Phi_{\delta_0,xx}(0,0)=1$ and
$\Gamma_{\delta_0}'(0)=\beta^2>1$. This contradicts Lemma~\ref{lem:plat-obstacle}(c), which
applies because $0\in C_P$. Hence $q_P>0$, zero is an accumulation point of $S$, and
$\zetas(\{0\})=0$ by Proposition~\ref{prop:ons-equilibrium}(b).
\end{proof}

For $\beta>1$ the conclusion $\supp\zetas=[0,q_P]$, $\zetas(\{0\})=0$ is Lopatto's theorem
\cite[Theorem~1.1 and Lemma~6.2]{Lop26a}. Version~4 of \cite{Lop26a} excludes gaps in the
support by the same trapezoid argument (Proposition~6.1 there), with the curvature input of its
Proposition~5.1 in place of part (a). The equality $C_P=S$, which says that the obstacle of
$\zetas$ has no minimum point outside the support, is not stated there. The corollary gives a
third proof of \condH\ for this covariance, besides \cite[Proposition~5.1]{Lop26a} and
\cite[Theorem~5.1]{Pei26}; as written it depends on the unrefereed
Proposition~\ref{prop:skg-oai}, which \cite{Lop26a} (version~4) can replace
(Remark~\ref{rem:skg-curvature}).

\subsection{Point or interval}\label{sec:skg-dichotomy}

\begin{proof}[Proof of Theorem~\ref{thm:sk-global}\textup{(c)}]
Let $\beta>1$ and $\theta\in(0,1)$. By Corollary~\ref{cor:skg-interval}, $\zetas(\{0\})=0<\theta$,
so Lemma~\ref{lem:cons-variational}(c) gives $\zeta_\theta(\{0\})=\theta$, and
Lemma~\ref{lem:cons-variational}(e) applies: the positive contact set $C_\theta$ is a nonempty
compact subset of $(0,1)$ that contains $S=\supp\zeta_\theta\cap(0,1]$. The set $S$ is nonempty,
since $\zeta_\theta((0,1])=1-\theta>0$, and closed, since it is contained in the compact set
$C_\theta\subset(0,1)$.

Every point of $C_\theta$ minimizes $\Obs_\theta$ over $(0,1]$ and lies in $(0,1)$, so it is a
local minimum point of $\Obs_\theta$. By part (b),
\begin{equation}\label{eq:skg-no-gap}
 \zeta_\theta((c_1,c_2))>0\qquad\text{for all}\ c_1<c_2\ \text{in}\ C_\theta .
\end{equation}
Let $c\in C_\theta\setminus S$. If $c<\min S$, then $\zeta_\theta((c,\min S))=0$ with both
endpoints in $C_\theta$, contradicting \eqref{eq:skg-no-gap}; the case $c>\max S$ is the same.
Otherwise $c$ lies in a component $(a,b)$ of $[\min S,\max S]\setminus S$, with $a,b\in S\subset
C_\theta$ and $\zeta_\theta((a,b))=0$, again a contradiction. Hence $C_\theta=S$. By
\eqref{eq:skg-no-gap}, $S$ has no gaps, so $S=[q_\theta,v_\theta]$, where
$q_\theta=\min C_\theta$ is the first positive contact.

If $q_\theta<v_\theta$, then $\Obs_\theta$ is constant on $[q_\theta,v_\theta]$, so
$\Obs_\theta'=-\frac12\beta^2(\Gamma_\theta-\id)$ vanishes there; thus $\Gamma_\theta=\id$ on
$[q_\theta,v_\theta]$, and $\Gamma_\theta'=1$ there because $\Gamma_\theta$ is continuously
differentiable. If $q_\theta=v_\theta$, then $\supp\zeta_\theta=\{0,q_\theta\}$ and
$\zeta_\theta(\{0\})=\theta$.
\end{proof}

In the second alternative the minimizer is a two-atom measure
$\zeta^{\theta,r}=\theta\delta_0+(1-\theta)\delta_r$ of Section~\ref{sec:nc-dat}, and
Lemma~\ref{lem:C-two-atom} evaluates $\Gamma$ and $\Gamma'$ at its atom $r$. The next proposition
is the global form, for this covariance, of the necessity and sufficiency steps of
Proposition~\ref{prop:nc-dat-threshold}; on the near-critical class, sufficiency is
Lemma~\ref{lem:C-sufficiency}, which assumes $r\le2\tau$.

\begin{proposition}[The two-atom criterion]\label{prop:skg-two-atom}
Let $\beta>1$ and $\theta\in(0,1)$.
\begin{thmlist}
\item If the second alternative of Theorem~\ref{thm:sk-global}\textup{(c)} holds, with
 $\zeta_\theta=\theta\delta_0+(1-\theta)\delta_r$, then $r=\Tfrak(\beta^2r,\theta)$ and
 $\beta^2\Sfrak(\beta^2r,\theta)\le1$.
\item Conversely, if $r\in(0,1)$ satisfies $r=\Tfrak(\beta^2r,\theta)$ and
 $\beta^2\Sfrak(\beta^2r,\theta)<1$, then $\zeta_\theta=\theta\delta_0+(1-\theta)\delta_r$, its
 positive contact set is $\{r\}$, and this contact is strictly Plefka stable.
\end{thmlist}
\end{proposition}

\begin{proof}
(a) The atom $r$ is the first positive contact, so \eqref{eq:cons-contact} gives
$\Gamma_\theta(r)=r$ and $\Gamma_\theta'(r)\le1$, and Lemma~\ref{lem:C-two-atom} with
$\xi'(r)=\beta^2r$ and $\xi''\equiv\beta^2$ rewrites these conditions as stated.

(b) Let $\zeta=\zeta^{\theta,r}$ and $h=\Gamma_\zeta-\id$, which is continuously differentiable on
$[0,1]$ (Lemma~\ref{lem:cons-standard}(b)). By Lemma~\ref{lem:C-two-atom}, $h(r)=0$ and
$h'(r)=\beta^2\Sfrak(\beta^2r,\theta)-1<0$; also $h(0)=0$ and $h(1)=\Gamma_\zeta(1)-1<0$. The
intervals $(0,r)$ and $(r,1)$ are plateaus of $\zeta$, of heights $\theta>0$ and $1$. By
Theorem~\ref{thm:sk-global}(a), $h'=\Gamma_\zeta'-1$ is strictly convex on each, hence convex on
$[0,r]$ and on $[r,1]$ by continuity. Lemma~\ref{lem:H-signs}(a) on $[0,r]$ gives $h>0$ on
$(0,r)$, and Lemma~\ref{lem:H-signs}(b) on $[r,1]$ gives $h<0$ on $(r,1]$. Since
$\Obs_\zeta'=-\frac12\beta^2h$, the obstacle is strictly decreasing on $[0,r]$ and strictly
increasing on $[r,1]$, so $\Obs_\zeta(t)>\Obs_\zeta(r)$ for every $t\in[0,1]\setminus\{r\}$.

Let $\nu\in\Pspace([0,1])$ with $\nu(\{0\})\ge\theta$. Since $\Obs_\zeta\ge\Obs_\zeta(r)$ on
$(0,1]$,
\[
 \begin{aligned}
 \int\Obs_\zeta\,d(\nu-\zeta)
 &\ge\nu(\{0\})\Obs_\zeta(0)+(1-\nu(\{0\}))\Obs_\zeta(r)
 -\theta\Obs_\zeta(0)-(1-\theta)\Obs_\zeta(r)\\
 &=(\nu(\{0\})-\theta)\bigl(\Obs_\zeta(0)-\Obs_\zeta(r)\bigr)\ge0 .
 \end{aligned}
\]
By Lemma~\ref{lem:cons-standard}(d), $s\mapsto P_{(1-s)\zeta+s\nu}$ is convex on $[0,1]$ with
right derivative $\int\Obs_\zeta\,d(\nu-\zeta)\ge0$ at $s=0$, so $P_\nu\ge P_\zeta$. Hence
$\zeta$ minimizes $P$ on $\mathcal A_\theta$, and $\zeta_\theta=\zeta$ by the uniqueness in
Lemma~\ref{lem:cons-variational}(a). Its positive contact set is $\{r\}$, because
$\Obs_\zeta>\Obs_\zeta(r)$ elsewhere on $(0,1]$, and the Plefka quantity at $r$ is
$1-\Gamma_\zeta'(r)=-h'(r)>0$.
\end{proof}

By the uniqueness of $\zeta_\theta$, at most one $r\in(0,1)$ satisfies both conditions of~(b).

\subsection{The threshold}\label{sec:skg-threshold}

Proposition~\ref{prop:skg-two-atom} reduces the location of the two alternatives to the
replica-symmetric branch. With $\mathfrak R(w,\theta)=\Tfrak(w,\theta)/w$, which is smooth on
$[0,\infty)\times[0,1]$ with $\mathfrak R(0,\theta)=1$ (Lemma~\ref{lem:C-moments}), the
fixed-point equation $q=\Tfrak(\beta^2q,\theta)$ for $q\in(0,1)$ reads
$\mathfrak R(\beta^2q,\theta)=\beta^{-2}$. Near criticality, Lemma~\ref{lem:C-rs} analysed the
branch by Taylor expansion at $w=0$; at a general temperature we need the following monotonicity
statements, which involve no Parisi measure and no input from \cite{OAI26c}.

\begin{lemma}[The tilted Gaussian]\label{lem:skg-branch}
For every $\theta\in[0,1]$ and $w>0$:
\begin{thmlist}
\item $\Tfrak(w,\theta)>w\,\partial_w\Tfrak(w,\theta)$; equivalently
 $\partial_w\mathfrak R(w,\theta)<0$;
\item $\partial_\theta\Tfrak(w,\theta)>0$, $\partial_\theta\Sfrak(w,\theta)<0$ and
 $\partial_w\Sfrak(w,\theta)<0$;
\item $\Sfrak(w,0)>\mathfrak R(w,0)$ and $\Sfrak(w,1)<\mathfrak R(w,1)$.
\end{thmlist}
\end{lemma}

The proof is in Appendix~\ref{sec:H-gauss}. For (a), an integration by parts against the tilted
law turns $\Tfrak-w\partial_w\Tfrak$ into the mean of a function of $x$ alone that does not
involve $w$, and that function is positive for $x\ne0$ at every $\theta\in[0,1]$. Part (b)
consists of covariance inequalities between monotone functions of $|x|$, and part (c) follows
from the Gaussian Poincar\'e inequality at $\theta=0$ and from (a) at $\theta=1$.

\begin{proposition}[The replica-symmetric branch]\label{prop:skg-branch}
Let $\beta>1$.
\begin{thmlist}
\item For each $\theta\in[0,1]$ the equation $\mathfrak R(w,\theta)=\beta^{-2}$ has exactly one
 solution $w_\theta\in(0,\infty)$, and the solutions of $q=\Tfrak(\beta^2q,\theta)$ in $(0,1)$
 are exactly $q=\qrs_\theta:=w_\theta/\beta^2$. The map $\theta\mapsto\qrs_\theta$ is
 continuously differentiable on $[0,1]$ \textup{(}one-sided at the endpoints\textup{)} and
 strictly increasing.
\item $\mathfrak A(\theta)=\beta^2\Sfrak(w_\theta,\theta)$ is continuously differentiable and
 strictly decreasing on $[0,1]$, and $\mathfrak A(0)>1>\mathfrak A(1)$. Hence there is a unique
 $\theta_c(\beta)\in(0,1)$ with $\mathfrak A>1$ on $[0,\theta_c(\beta))$,
 $\mathfrak A(\theta_c(\beta))=1$ and $\mathfrak A<1$ on $(\theta_c(\beta),1]$.
\end{thmlist}
\end{proposition}

\begin{proof}
(a) The function $\mathfrak R(\cdot,\theta)$ is continuous on $[0,\infty)$, equals $1$ at zero,
satisfies $\mathfrak R\le1/w$ because $\Tfrak\le1$, and is strictly decreasing on $(0,\infty)$
by Lemma~\ref{lem:skg-branch}(a). It therefore takes the value $\beta^{-2}\in(0,1)$ exactly once.
For $q\in(0,1)$, $q=\Tfrak(\beta^2q,\theta)$ if and only if $\mathfrak R(\beta^2q,\theta)=\beta^{-2}$,
that is, $\beta^2q=w_\theta$; and $w_\theta/\beta^2=\Tfrak(w_\theta,\theta)<1$. Since
$\mathfrak R$ is smooth and $\partial_w\mathfrak R<0$, the implicit function theorem makes
$\theta\mapsto w_\theta$ continuously differentiable, with
$dw_\theta/d\theta=-\partial_\theta\mathfrak R/\partial_w\mathfrak R>0$, because
$\partial_\theta\mathfrak R=\partial_\theta\Tfrak/w>0$ by Lemma~\ref{lem:skg-branch}(b).

(b) By the chain rule,
$\mathfrak A'(\theta)=\beta^2\bigl(\partial_w\Sfrak\cdot dw_\theta/d\theta+\partial_\theta\Sfrak\bigr)<0$,
by Lemma~\ref{lem:skg-branch}(b) and (a). At the root, $\mathfrak R(w_\theta,\theta)=\beta^{-2}$, so
$\mathfrak A(\theta)-1=\beta^2(\Sfrak-\mathfrak R)(w_\theta,\theta)$, and
Lemma~\ref{lem:skg-branch}(c) at $w=w_0$ and $w=w_1$ gives the signs at the endpoints.
\end{proof}

Proposition~\ref{prop:skg-branch} is the first part of Theorem~\ref{thm:sk-global}(d), with
$\mathfrak A(\theta)=\xi''(\qrs_\theta)\Sfrak(\xi'(\qrs_\theta),\theta)$ as in
Proposition~\ref{prop:nc-dat-threshold}.

\begin{proof}[Proof of Theorem~\ref{thm:sk-global}\textup{(d)}]
Fix $\beta>1$ and write $\theta_c=\theta_c(\beta)$. By Corollary~\ref{cor:skg-interval},
$\zetas(\{0\})=0$, so part (c) applies at every $\theta\in(0,1)$.

\emph{$0<\theta<\theta_c$.} Suppose the second alternative of (c) held, with atom $r$. By
Proposition~\ref{prop:skg-two-atom}(a), $r\in(0,1)$ solves $r=\Tfrak(\beta^2r,\theta)$ and
$\beta^2\Sfrak(\beta^2r,\theta)\le1$. By Proposition~\ref{prop:skg-branch}(a), $r=\qrs_\theta$, so
$\mathfrak A(\theta)\le1$, which contradicts $\mathfrak A(\theta)>1$. Hence the first alternative
holds.

\emph{$\theta_c<\theta<1$.} Then $\mathfrak A(\theta)<1$, so $r=\qrs_\theta$ satisfies the
hypotheses of Proposition~\ref{prop:skg-two-atom}(b): $\zeta_\theta=\theta\delta_0+(1-\theta)\delta_r$
with positive contact set $\{r\}$, and by Lemma~\ref{lem:C-two-atom} the Plefka quantity at $r$ is
$1-\Gamma_\theta'(r)=1-\mathfrak A(\theta)>0$.

\emph{$\theta=\theta_c$.} Here we use the global optimality just established and continuity, not
a condition at the marginal atom. Take $\theta_n\downarrow\theta_c$ with $\theta_n<1$. By the
previous case, $\zeta_{\theta_n}=\theta_n\delta_0+(1-\theta_n)\delta_{r_n}$ with
$r_n=\qrs_{\theta_n}$, and $r_n\to r_c:=\qrs_{\theta_c}\in(0,1)$ by
Proposition~\ref{prop:skg-branch}(a). Hence
$\zeta_{\theta_n}\to\theta_c\delta_0+(1-\theta_c)\delta_{r_c}$ weakly. By
Lemma~\ref{lem:cons-variational}(a), $\zeta_{\theta_n}\to\zeta_{\theta_c}$ weakly, so
$\zeta_{\theta_c}=\theta_c\delta_0+(1-\theta_c)\delta_{r_c}$. This is the second alternative of
(c), so the only positive contact is $r_c$, and by Lemma~\ref{lem:C-two-atom}
$\Gamma_{\theta_c}'(r_c)=\mathfrak A(\theta_c)=1$: the Plefka quantity vanishes.

\emph{Regularity in $\beta$.} The implicit function theorem for
$\mathfrak R(w,\theta)=\beta^{-2}$, with $\partial_w\mathfrak R<0$, makes $(\beta,\theta)\mapsto
w_\theta(\beta)$ continuously differentiable on $(1,\infty)\times[0,1]$. Hence
$(\beta,\theta)\mapsto\beta^2\Sfrak(w_\theta(\beta),\theta)$ is continuously differentiable with
negative $\theta$-derivative, and the implicit function theorem applied to its zero gives
$\theta_c\in C^1(1,\infty)$.

\emph{Near criticality.} If $1<\beta^2<1+\varepsilon_*$, then $\xi\in\Kcl(\varepsilon_*)$
(Remark~\ref{rem:fifth-sk-inputs}). The root of Proposition~\ref{prop:nc-dat-threshold}(a) is a
solution of $q=\Tfrak(\beta^2q,\theta)$ in $(0,K_{\rm rs}]\subset(0,1)$, hence equals $\qrs_\theta$,
and the function $\mathfrak A$ there is the present one. Its unique zero, the threshold of
Theorem~\ref{thm:nearcrit}(b), is therefore $\theta_c(\beta)$.
\end{proof}

\begin{proof}[Proof of Theorem~\ref{thm:sk-global}\textup{(e)}]
By Corollary~\ref{cor:skg-interval}, $\zetas(\{0\})=0$ and $\zetas\ne\delta_0$, so
Corollary~\ref{cor:cons-level} applies: $\Lambda'$ maps $(0,1)$ increasingly and
homeomorphically onto $(\feq,f_1)$, and $\theta_f=(\Lambda')^{-1}(f)$. Hence
$f_c(\beta)=\Lambda'(\theta_c(\beta))\in(\feq,f_1)$, and $\theta_f\le\theta_c(\beta)$ if and
only if $f\le f_c(\beta)$. Apply (c) and (d) at $\theta=\theta_f$. If $\theta_f<\theta_c(\beta)$,
the first alternative holds, and every positive contact of $\zeta_f$ is marginal. If
$\theta_f\ge\theta_c(\beta)$, then $\zeta_f$ has exactly two atoms and a single positive contact
$q_f=\qrs_{\theta_f}$, at which the Plefka quantity is $1-\mathfrak A(\theta_f)$; it vanishes if
$\theta_f=\theta_c(\beta)$ and is positive otherwise. This gives the three equivalences and the
statement for $f\in(f_c(\beta),f_1)$. The strict Plefka condition \cite[(1.8)]{Bou26b} is the
strict form of the Plefka inequality at the first positive contact
\eqref{eq:fifth-strict-plefka}, and $\xi(t)=\beta^2t^2/2$ satisfies
\cite[Assumption~1.1]{Bou26b}, which asks for a series with nonnegative coefficients and radius
of convergence greater than one that is strictly convex on $[-1,1]$.
\end{proof}

\begin{remark}\label{rem:skg-scope}
(1) For $0<\theta<\theta_c(\beta)$ the first alternative of Theorem~\ref{thm:sk-global}(c) is the
conclusion of Theorem~\ref{thm:local}(a), so for this covariance $\theta_0=\theta_c(\beta)$ is
admissible there; by Theorem~\ref{thm:sk-global}(d), as in Remark~\ref{rem:nc-dat}(2), no
larger $\theta_0$ is.

(2) Theorem~\ref{thm:sk-global} gives no rate and no information on $\zeta_\theta$ above its
first contact beyond the interval structure. The bound $O(\theta^5)$ of
Theorem~\ref{thm:nearcrit}(c) uses a contraction argument that needs the support of $\zetas$ to
be short (Remark~\ref{rem:nc-above-contact}), and it is not extended here.

(3) The proofs use from \cite{OAI26c} only Proposition~\ref{prop:skg-oai}. The zero-temperature
statements of \cite{OAI26c}, its regularity lemmas on compact time strips and its approximation
with terminal smoothings $M_n\to\infty$ are not used: at positive temperature the terminal datum
$\log\cosh$ is already smooth.
\end{remark}

\section{Earlier work and open questions}\label{sec:disc}

Credits that concern a single statement are given where it is used. Here we collect the rest of
the rigorous work behind the paper (Section~\ref{sec:disc-prior}) and record its inputs and the
questions it leaves open (Section~\ref{sec:disc-open}).

\subsection{Earlier work}\label{sec:disc-prior}

Talagrand proved Guerra's bound for $\E\ZN^\theta$ \cite[Section~6]{Tal07} and explained the
matching lower bound for even covariances \cite[Section~7]{Tal07}. Chen, Guionnet, Ko,
Lacroix-A-Chez-Toine and Mourrat prove \condFM\ for every $\xi$ finite on $\R$
\cite[Theorem~2.1]{CGKLM}, with the Guerra-type step written out for convex $\xi$, and relate the
Sherrington--Kirkpatrick case to the multiscale formula of Contucci and Mingione \cite{CM19};
Aronow and Lopatto give a quantitative version for that model, for $0<\theta\le\frac12$
\cite[Theorem~1.1]{AL26}. Boursier notes that Talagrand's results identify $-\Sigma$ as the
upper-tail rate \cite[Remark~1.3]{Bou26b} (see also \cite[Section~1.4]{Bou26a}) and proves, for
$\xi$ strictly convex on $[-1,1]$, a direct lower bound for $\log\ZN$ at levels satisfying the strict Plefka condition
\cite[Remark~1.4, Proposition~3.12]{Bou26b}; the large-deviation principle is therefore not new,
and we add the behaviour of $\Sigma$ near $\feq$ and the constrained minimizers behind it. On the
Parisi side we use \cite{AC15b,AC15a,Che17,JT17}. For the Sherrington--Kirkpatrick model, Zhou
proved full replica symmetry breaking near criticality, with an interval of support at the bottom
\cite{Zho26}, Lopatto proved that for every $\beta>1$ the Parisi measure is supported on
$[0,q_\beta]$, with a density in $C^\infty([0,q_\beta))$ and a single atom at $q_\beta$
\cite[Theorem~1.1]{Lop26a}, and the non-strict sign $\Phi_{\zeta,xxx}(t,x)\le0$ for $x>0$ is
\cite[(7.6)]{Lop26a}, with other proofs in \cite[Lemma~5.4]{Pei26} and
\cite[Lemma~3.4 and Appendix~B.3]{Che26} (Remark~\ref{rem:plat-sign}). For every measure with a
plateau of positive height, Lopatto proves that $\Gamma'''>0$ wherever $\Gamma''=0$ on the
plateau \cite[Proposition~5.1]{Lop26a} in version~4, which already suffices for
Theorem~\ref{thm:sk-global}(b)--(e) (Remark~\ref{rem:skg-curvature}); the unconditional strict
convexity of Theorem~\ref{thm:sk-global}(a) comes from the finite-step inequality of the unrefereed preprint
\cite[Proposition~4.2]{OAI26c}, the core of its curvature estimate
\cite[Proposition~5.1]{OAI26c} for the zero-temperature order parameter. For every fixed $\beta>1$ the unrefereed preprints
\cite{OAI26a,OAI26b} identify the scale $N^{1/6}$ of the fluctuations of $\FNb$ and their
limiting law (Section~\ref{sec:disc-open}).

\paragraph{Existing results re-proved or extended.}
Items 2, 5 and 6 are the adapted boxes of Figure~\ref{fig:roadmap}, and items 3 and 4 belong to
its bottom band; items 8 and 9 concern Theorem~\ref{thm:sk-global}, which the figure does not
show.
\begin{itemize}[leftmargin=*,itemsep=1pt,topsep=2pt]
\item Zero lies in $\supp\zetas$ \cite[Theorem~1(i)]{AC15a}. That proof uses
 \cite[Theorem~5]{AC15a} at the smallest support point; when that point is isolated, the
 inequality $\Gamma_*'\le1$ there can be taken from \cite[Corollary~3.10]{JT17}, whose derivative
 formulas are stated without proof \cite[Lemma~8.4]{JT17}. We give a direct proof
 (Lemma~\ref{lem:ons-zero-in-support}), as does \cite[Proposition~2.1 (P6) and Appendix~A.1]{Pei26}.
 For $\Gamma'\le1$ on the support we use the second-order condition for the obstacle at interior
 contacts and the endpoint inequality at zero (Lemma~\ref{lem:plat-obstacle}(b),(c)) and the
 first-variation characterization of the support \cite[Corollary~3.6]{JT17}.
\item The smoothness of $\alphas$ and its quotient formula on an interval $[t_1,t_2)$ of the support,
 $t_1=0$ included, \cite[Theorem~2(ii) and its proof]{AC15a} are obtained at zero from \condH,
 which under $\beta_2>0$ and $\xi'''(0)=0$ produces an interval $[0,\kappa]$ of support
 (Proposition~\ref{prop:ons-equilibrium}(a)), with
 $\rhostar=\Dstar/(2\sqrt{\azero})$, re-deriving $\Ustar=1/\beta$ and $\zetas(\{0\})=0$ of
 \cite[Lemmas~6.1 and~6.2]{Lop26a} for the Sherrington--Kirkpatrick model
 (Proposition~\ref{prop:ons-equilibrium}). For that model $\rhob=\beta^5\Wstar^2/2>0$ also follows
 from \cite[Theorem~B.1 and (B.4)]{Lop26b}, which uses interval support, and from
 \cite[Section~7]{AL26}, and it is \cite[Lemma~5.8]{OAI26a}, an unrefereed preprint
 (Remark~\ref{rem:fifth-sk-inputs}).
\item The non-strict sign is strengthened to $\Phi_{\zeta,xxx}(t,x)\le-\kappa_tx$ for
 $0\le x\le\frac12$ and $t<1$, uniformly in $\zeta$, for every covariance
 (Lemma~\ref{lem:plat-fourth-sign}); with $\beta_2>0$ it gives $\Dstar>0$. At the origin and
 for the Sherrington--Kirkpatrick equation, a strict bound with unspecified constants is also
 \cite[Lemma~5.3]{OAI26a}, an unrefereed preprint (Remark~\ref{rem:plat-sign}).
\item The convexity, saturation and first-contact statements of \cite[Lemma~3.2]{Bou26b} are
 followed in Section~\ref{sec:cons-variational} without the strict convexity of $\xi$ on
 $[-1,1]$ assumed there, with the uniqueness of the dual minimizer and the level identity added;
 the contact inequality $\Gamma'(q)\le1$ at dual minimizers is noted in \cite[Remark~1.2]{Bou26b}.
 The envelope identity \eqref{eq:intro-envelope} at every $\theta$
 (Section~\ref{sec:cons-envelope}) is ours; for the zero-temperature Laplace transform,
 continuous differentiability with an envelope formula is \cite[Lemma~5.3]{CGKLM}.
\item The slope identity stated in \cite[Remark~1.4]{Bou26a} is proved as a right derivative with
 an exact remainder (Proposition~\ref{prop:atom-slope}). The atom formula of
 Theorem~\ref{thm:atom}(b) is implicit in \cite[Proposition~3]{AC15a} with the support conditions
 (for spherical models: \cite[Theorem~6(ii)]{AC15a}), and the atom for $\xi''(0)<1$ is
 \cite[Theorem~1]{AC15a}, whose argument Proposition~\ref{prop:atom-atom}(a) repeats.
\item Zhou's gap exclusion for the Sherrington--Kirkpatrick model near criticality \cite{Zho26} is
 adapted to the constrained problem for mixed covariances near $\xic$
 (Section~\ref{sec:nc-reference}); Section~\ref{sec:plat-gap} gives a version at the origin that
 needs no closeness to $\xic$.
\item Toulouse's identity $\Ffield''(0)=1/\beta$ \cite{Tou80}, proved for the
 Sherrington--Kirkpatrick model as a second-order limit in \cite[Theorem~3.1(b),(d)]{Pei26} using
 \cite[Lemma~6.1]{Lop26a}, is extended to the value $\azero^{-1/2}$ for the covariances of
 Theorem~\ref{thm:field}, together with the $|h|^{10/3}$ term; $\Ffield\in C^1$ with
 $\Ffield'=\mfield$ is the routine envelope argument of Lemma~\ref{lem:fld-continuity}.
\item The finite-step coercivity inequality of \cite[Proposition~4.2]{OAI26c}, stated there for
 step functions and used there at zero temperature, is transferred in its equal-step form
 (Proposition~\ref{prop:skg-oai}) to every measure at positive temperature
 (Proposition~\ref{prop:skg-plateau}, Appendix~\ref{sec:H-transfer}), which gives
 Theorem~\ref{thm:sk-global}(a).
\item Lopatto's interval support of the Sherrington--Kirkpatrick Parisi measure
 \cite[Theorem~1.1 and Lemma~6.2]{Lop26a} is re-derived from Theorem~\ref{thm:sk-global}(b), with
 the additional statement that the obstacle of $\zetas$ is minimal only on its support
 (Corollary~\ref{cor:skg-interval}); the proof depends on the unrefereed input of
 Proposition~\ref{prop:skg-oai}.
\end{itemize}

\subsection{Inputs and open questions}\label{sec:disc-open}

\paragraph{What the results rest on.}
Apart from \condH\ and \condFM, which appear as hypotheses where they are needed,
Theorems~\ref{thm:local}--\ref{thm:nearcrit} and Corollary~\ref{cor:pure-p} rest on refereed
results \cite{AC15a,AC15b,Che17,JT17} and the present paper. The statements at every $\beta>1$ for
the Sherrington--Kirkpatrick model, in Theorem~\ref{thm:sk} and the last sentence of
Theorem~\ref{thm:field}, take \condH\ from a preprint, \cite[Proposition~5.1]{Lop26a} or the
computer-assisted alternative proof \cite[Theorem~5.1]{Pei26}; for $\beta$ close to $1$ they rest
on refereed results (Remark~\ref{rem:fifth-sk-inputs}), and \condFM\ is a refereed result for this
covariance \cite{CM19} (Remark~\ref{rem:mo-fm} for other covariances). Beyond these cases and
$\Kcl(\varepsilon_*)$, with $\varepsilon_*$ as in Theorem~\ref{thm:nearcrit}(a), we do not know a proof of \condH\ for covariances with $\beta_2>0$ and
$\xi'''(0)=0$ (Section~\ref{sec:intro-constrained}); any such proof makes Theorems~\ref{thm:local}
and~\ref{thm:field} apply to the covariances it covers. The proof of Theorem~\ref{thm:sk-global}
given here uses neither \cite{Lop26a} nor \cite{Pei26}. Besides the present paper and
\cite{AC15a,AC15b,Che17,JT17}, it rests on one input from an unrefereed preprint, the
finite-step inequality \cite[Proposition~4.2]{OAI26c} in the equal-step form of
Proposition~\ref{prop:skg-oai}, which is also proved in the Lean formalization accompanying that
preprint; the statements about the replica-symmetric branch in Theorem~\ref{thm:sk-global}(d)
(Proposition~\ref{prop:skg-branch}) do not use it. Parts (b)--(e) need only a qualitative
consequence of that input, which \cite[Proposition~5.1]{Lop26a} in version~4 also gives
(Remark~\ref{rem:skg-curvature}); only the quantitative bound in part (a) needs \cite{OAI26c}.

\paragraph{Rates and deviations that shrink with $N$.}
The $o(1)$ terms of Theorems~\ref{thm:sk} and~\ref{thm:local} come from the convergences
$\Wone(\zeta_\theta,\zetas)\to0$ and $q_\theta\to0$ of Proposition~\ref{prop:ons-onset}(a), which
follow from compactness and the uniqueness of $\zetas$ and come without a rate; along our route a
rate for the fifth-order law requires a quantitative form of them. Such a rate is also the first
missing ingredient for deviations that shrink with $N$, which Aronow and Lopatto reach down to
order $N^{-2/3}$ \cite[Theorem~1.5]{AL26}, with a quantitative \condFM\ for $\theta\le\frac12$
\cite[Theorem~1.1]{AL26}:
our method would run the change of measure of Lemma~\ref{lem:cons-gartner-ellis} at a constraint
$\theta_N\downarrow0$, with every error small compared with $\Lambda(\theta_N)-\theta_N\feq$, of
order $\theta_N^6$. As all limits are taken at a fixed covariance with $N\to\infty$ first,
Theorem~\ref{thm:sk} says nothing about the fluctuations of $N^{-1}\log\ZN$ on the scale
$N^{-5/6}$, at which $N\delta^{6/5}$ is of order one. On that scale, for every fixed $\beta>1$,
the unrefereed preprint \cite[Theorem~1.1]{OAI26a} proves that the standard deviation of $\FNb$
is $N^{1/6+o(1)}$, and its companion \cite[Theorem~1.1]{OAI26b} that $\FNb$, centered and divided
by its standard deviation, converges in law to a nondegenerate limit. The proof in
\cite{OAI26a} bounds the centered log-Laplace transform
$K_N(s)=\log\E\exp\{s(\log\ZN-\E\log\ZN)\}$ (its completed free energy has the law of
$\log\ZN$) between $c\,N^{20m}$ and $C\,N^{100m}$ at $s=N^{-1/6+5m}$, for fixed
$0<m<\frac1{300}$ \cite[Proposition~7.5]{OAI26a}. This is consistent with
Theorem~\ref{thm:local}(b) in a formal sense only: replacing $K_N(s)$ by
$N\{\Lambda(s)-s\feq\}=N\cb s^6(1+o(1))$, which \condFM\ justifies at fixed $s$ as
$N\to\infty$ but not at $s=N^{-1/6+5m}$, gives $\cb N^{30m}$, inside that window.

\paragraph{Zero temperature.}
The $o(1)$ terms are not known to be uniform as $\beta\to\infty$, so Theorem~\ref{thm:sk} gives
no constant for the ground-state energy, whose six-fifths upper-tail behaviour is a physics
prediction recorded in \cite[(1.9)]{CGKLM}, who also prove that the ground-state rate is quadratic
near its minimum in a field and superquadratic without one \cite[Theorem~1.3]{CGKLM}; Aronow and
Lopatto prove two-sided bounds of this order with unspecified
constants \cite[Theorem~1.6 and (1.14)]{AL26}. The constant would require control of our $o(1)$
terms uniformly in large $\beta$, or a plateau calculus for the zero-temperature Parisi problem.

\paragraph{Away from zero.}
The local calculus controls the contacts of $\zeta_\theta$ only in a fixed interval
$(0,K_{\rm loc})$, where every plateau of a measure close to $\zetas$ has $\Gamma'''\ge\Dstar/2$
(Lemma~\ref{lem:plat-local-plateau}). It leaves open whether, for small $\theta$, every positive
contact of $\zeta_\theta$ is marginal, as the first one is, and whether the quantile functions of
$\zeta_\theta$ and $\zetas$ differ by $O(\theta^5)$ above the first contact, as Parisi and Rizzo
predict at every temperature below the critical one \cite[Section~IV]{PR09}. Near criticality,
closeness to $\xic$ supplies the missing global control: for every $K\in(0,1)$ and every
covariance close enough to $\xic$, no constrained minimizer has a contact in $[K,1]$
(Lemma~\ref{lem:nc-confinement}). On $\Kcl(\varepsilon_*)$, which contains the
Sherrington--Kirkpatrick model for $1<\beta^2<1+\varepsilon_*$, Theorem~\ref{thm:nearcrit}(b),(c)
answer both questions: for $0<\theta<\theta_c$ the positive contacts of $\zeta_\theta$ form an
interval $[q_\theta,v_\theta]$ with $q_\theta<v_\theta$, on which $\Gamma_\theta(t)=t$
(Lemma~\ref{lem:nc-regular}), so that $\Gamma_\theta'=1$ at each of them; and the bound of~(c)
comes from a contraction argument that needs the support of $\zetas$ to be short
(Remark~\ref{rem:nc-above-contact}). For the Sherrington--Kirkpatrick model at every $\beta>1$,
Theorem~\ref{thm:sk-global}(c),(d) answer the first question, with an unrefereed input: for
$0<\theta<\theta_c(\beta)$ the positive contacts form an interval on which $\Gamma_\theta=\id$.
The second question remains open for this model away from criticality, and both remain open for
other covariances outside $\Kcl(\varepsilon_*)$.

\paragraph{TAP states.}
The levels in $(\feq,\feq+\bar\delta)$, which carry the six-fifths law, lie outside the scope of
Boursier's theorem on the existence rate of TAP states \cite[Theorem~1]{Bou26b}, which assumes the
strict Plefka condition \cite[(1.8)]{Bou26b}: under the hypotheses of Theorem~\ref{thm:local},
hence for the Sherrington--Kirkpatrick model at every $\beta>1$, the first positive contact is
marginal at every such level, and the Plefka factor $1-\Gamma'_{\zeta_f}$ vanishes on a whole
interval of contacts (Corollary~\ref{cor:fifth-plefka}). On $\Kcl(\varepsilon_*)$, which contains
the Sherrington--Kirkpatrick model for $1<\beta^2<1+\varepsilon_*$, that condition holds, within
$(\feq,f_1)$, exactly on $(f_c,f_1)$ (Corollary~\ref{cor:nc-dat-levels}), and for the
Sherrington--Kirkpatrick model at every $\beta>1$ exactly on $(f_c(\beta),f_1)$, with every
positive contact marginal at the levels in $(\feq,f_c(\beta)]$
(Theorem~\ref{thm:sk-global}(e), with an unrefereed input). Determining the existence
rate at marginal levels requires an argument that does not rely on a strict Plefka gap. We state
no result on TAP states here; they are the subject of a companion paper \cite{Pei26b}.

\appendix
\section{Routine proofs for Sections~\ref*{sec:cons}--\ref*{sec:ons} and~\ref*{sec:nc-reference}}\label{app:analytic}

This appendix contains the proofs that the main text defers: routine estimates, and direct
proofs of standard facts in the form used here. Appendix~\ref{sec:an-cons} proves the
convex-duality facts of Section~\ref{sec:cons}. Appendix~\ref{sec:plat-jets} states and proves
the quantitative estimates on the Parisi equation behind the analytic facts of
Section~\ref{sec:plat} (jet stability in the measure and in the covariance, bounds in time, and
a coupling lemma), and proves the plateau law and the It\^o identities of
Section~\ref{sec:plat-ito}. Appendix~\ref{sec:an-zero} proves that zero lies in the support of
the Parisi measure, and Appendix~\ref{sec:an-nc} proves the positivity of the plateau
polynomial near criticality. The tools are convexity and the first-variation formula, the
maximum principle, Gronwall's inequality and It\^o's formula. The main text uses only the
statements; those of Lemmas~\ref{lem:plat-jets} and~\ref{lem:plat-coupling} and of the estimates \eqref{eq:plat-jet-stability}--\eqref{eq:plat-origin} are given in Appendix~\ref{sec:plat-jets}.

\subsection{The constrained problem}\label{sec:an-cons}

The standing assumption and notation are those of Section~\ref{sec:cons}.

\begin{proof}[Proof of Lemma~\ref{lem:cons-standard}]
Parts (a)--(d) are taken from the literature, with two remarks on scope. First, the papers
\cite{AC15a,AC15b,Che17} state their models under a summability condition such as
$\sum_p2^p\beta_p^2<\infty$, which our radius assumption need not imply; there it serves only
to define the Hamiltonian. The cited statements concern the Parisi equation and functional,
which involve only finitely many derivatives of $\xi$ on $[0,1]$ ($\xi'$, $\xi''$ and, in \cite[Proposition~3]{AC15a}, $\xi'''$); these are bounded here, with
$\xi''>0$ on $(0,1]$. Second, the statements of Auffinger and Chen, made along a time-changed
Brownian motion, hold along the diffusion \eqref{eq:intro-diffusion}, and their function
$\Gamma_\zeta$ of \cite[(18)]{AC15a} is ours, by the Girsanov argument recalled in the
paragraph on analytic facts at the beginning of Section~\ref{sec:plat}.

(a) Existence, continuity and uniform convergence of the spatial derivatives are
\cite[Proposition~1(i)]{AC15a}, and the two bounds are \cite[Proposition~2, (15)--(16)]{AC15a}.
For finitely atomic $\zeta$ the Cole--Hopf formulas \cite[(11)--(12)]{AC15a} show that
$\Phi_\zeta(t,\cdot)$ is even, and evenness passes to general $\zeta$ by the uniform
convergence.

(b) Differentiability, continuity of $\Gamma_\zeta'$ and its formula are
\cite[Proposition~3, (20)]{AC15a}. The uniform convergence of $\Gamma_{\zeta_n}$ and
$\Gamma'_{\zeta_n}$ is \cite[Proposition~1(ii)]{AC15a}, applied to the polynomials $y_1^2$ and
$y_2^2$; that of $\Obs_{\zeta_n}$ follows by integration. Since $X_0=0$ and
$\Phi_{\zeta,x}(0,\cdot)$ is odd, $\Gamma_\zeta(0)=\Phi_{\zeta,x}(0,0)^2=0$. The bounds
$0\le\Gamma_\zeta\le1$ follow from (a), and $\Gamma_\zeta(1)=\E\tanh^2X_1<1$.

(c) This is \cite[Theorem~2 and Corollary~1]{AC15b}.

(d) At $s=0$ the right derivative is the first-variation formula \cite[Theorem~2]{Che17}, which gives the
second expression in \eqref{eq:cons-first-variation}. Applied at the base point $\zeta_s$ in the directions $\nu-\zeta_s=(1-s)(\nu-\zeta)$ and $\zeta-\zeta_s=-s(\nu-\zeta)$, the same formula gives both one-sided derivatives at every $s\in(0,1)$ and the left derivative at $s=1$, and all equal the second expression in \eqref{eq:cons-first-variation}. Write
$\delta\alpha=\alpha_\nu-\alpha_\zeta$. By \eqref{eq:intro-obstacle} the integrand is
$-\Obs_{\zeta_s}'\,\delta\alpha$, and Fubini's theorem gives
\[
 -\int_0^1\Obs_{\zeta_s}'(t)\,\delta\alpha(t)\,dt
 =-\int d(\nu-\zeta)(r)\int_r^1\Obs_{\zeta_s}'(t)\,dt
 =\int\Obs_{\zeta_s}\,d(\nu-\zeta),
\]
because $\Obs_{\zeta_s}(1)=0$. As $s\mapsto\zeta_s$ is weakly continuous, (b) shows that the
derivative is continuous in $s$. Convexity follows from (c).

(e) Apply (d) to the segment from $\nu$ to $\mu$. Since $0\le\Gamma\le1$ and $0\le t\le1$, we
have $|\Gamma(t)-t|\le1$, so the derivative in \eqref{eq:cons-first-variation} is at most
$\frac12\xi''(1)\int_0^1|\alpha_\mu-\alpha_\nu|\,dt$ in absolute value. The same estimate,
with $\sup_{[0,1]}\xi''$ in place of $\frac12\xi''(1)$, appears in the proof
of~\cite[Lemma~3.2]{Bou26b}.

(f) Since $\Gamma_\zeta\in C^1[0,1]$, differentiating \eqref{eq:intro-obstacle} gives the first
two identities in \eqref{eq:cons-obstacle} and $\Obs_\zeta\in C^2[0,1]$. Moreover
$\Obs_\zeta'(1)=\frac12\xi''(1)\{1-\Gamma_\zeta(1)\}>0$ by (b), so $\Obs_\zeta(t)<\Obs_\zeta(1)$
for $t<1$ close to $1$. At a minimum point $t\in(0,1)$ of $\Obs_\zeta$ on $(0,1]$ we have
$\Obs_\zeta'(t)=0$ and $\Obs_\zeta''(t)\ge0$. Since $\xi''(t)>0$, the first gives
$\Gamma_\zeta(t)=t$, and the second then gives $\Gamma_\zeta'(t)\le1$.

(g) By (d), the right derivative at $s=0$ of $s\mapsto P_{(1-s)\zetas+s\nu}$ is
$\int\Obs_*\,d(\nu-\zetas)$, and it is nonnegative because $\zetas$ minimizes $P$. Taking
$\nu=\delta_t$ gives $\Obs_*(t)\ge o_*$. So $\Obs_*-o_*$ is a nonnegative continuous function
with $\zetas$-integral zero, and it vanishes on $\supp\zetas$. This is
\cite[Corollary~3.6]{JT17}. Conversely, under the stated condition, convexity of
$s\mapsto P_{(1-s)\zeta+s\nu}$ and (d) give
$P_\nu\ge P_\zeta+\int\Obs_\zeta\,d(\nu-\zeta)\ge P_\zeta$ for every $\nu$, so $\zeta$
minimizes $P$ and equals $\zetas$ by (c).
\end{proof}

\begin{proof}[Proof of Lemma~\ref{lem:cons-variational}]
(a) The set $\mathcal A_\theta$ is convex. It is weakly closed, because
$\zeta\mapsto\zeta(\{0\})$ is upper semicontinuous, and hence compact. By
Lemma~\ref{lem:cons-standard}(c),(e), $P$ attains its minimum on $\mathcal A_\theta$ at exactly
one point $\zeta_\theta$. If $\theta\le m_0$, then $\zetas\in\mathcal A_\theta$, so
$\zeta_\theta=\zetas$ and $p(\theta)=\feq$. The sets $\mathcal A_\theta$ decrease in $\theta$,
so $p$ is nondecreasing and $p\ge p(0)=\feq$. The measure $\theta\delta_0+(1-\theta)\zetas$ lies
in $\mathcal A_\theta$, at $\Wone$-distance $\theta\int_0^1(1-\alphas)\,dt\le\theta$ from
$\zetas$, so \eqref{eq:cons-lipschitz} gives the upper bound on $p(\theta)$.

Let $\theta_n\to\theta$ in $[0,1)$, and put $\lambda_n=(\theta_n-\theta)^+/(1-\theta)$. The
measure $\lambda_n\delta_0+(1-\lambda_n)\zeta_\theta$ lies in $\mathcal A_{\theta_n}$ and within
$\Wone$-distance $\lambda_n$ of $\zeta_\theta$, so $\limsup_np(\theta_n)\le p(\theta)$ by
\eqref{eq:cons-lipschitz}. Let $\zeta'$ be the weak limit of a subsequence of
$(\zeta_{\theta_n})$. Since $\{0\}$ is closed,
$\zeta'(\{0\})\ge\limsup_n\zeta_{\theta_n}(\{0\})\ge\theta$, and
$P_{\zeta'}=\lim P_{\zeta_{\theta_n}}\le p(\theta)$ along the subsequence. Thus $\zeta'$
minimizes $P$ on $\mathcal A_\theta$, so $\zeta'=\zeta_\theta$ and $P_{\zeta'}=p(\theta)$. By
compactness, $\zeta_{\theta_n}\Rightarrow\zeta_\theta$ and $p(\theta_n)\to p(\theta)$.

(b) Let $\theta_1,\theta_2\in[0,1)$ and $t\in(0,1)$ with
$\theta=t\theta_1+(1-t)\theta_2>0$. The probability measure
\begin{equation}\label{eq:cons-mixture}
 \nu=\frac{t\theta_1\zeta_{\theta_1}+(1-t)\theta_2\zeta_{\theta_2}}{\theta}
 \quad\text{satisfies}\quad
 \nu(\{0\})\ge\frac{t\theta_1^2+(1-t)\theta_2^2}{\theta}\ge\theta
\end{equation}
(as in \cite[proof of Lemma~3.2]{Bou26b}), so $\nu\in\mathcal A_\theta$, and convexity of $P$ gives
$\Lambda(\theta)\le\theta P_\nu\le t\Lambda(\theta_1)+(1-t)\Lambda(\theta_2)$. Hence $\Lambda$ is
convex on $[0,1)$. By (a), $\Lambda(\theta)/\theta=p(\theta)\to\feq$ as $\theta\downarrow0$,
and $\Lambda(0)=0$; this gives $\Lambar'_+(0)=\feq$. If $\mu\in\mathcal A_{1-h}$ with
$0<h\le1$, then $\mu((0,1])\le h$, so $\Wone(\mu,\delta_0)=\int_0^1\mu((t,1])\,dt\le h$. Since
$\delta_0\in\mathcal A_{1-h}$, \eqref{eq:cons-lipschitz} gives
\[
 (1-h)\bigl(P_{\delta_0}-\tfrac12\xi''(1)h\bigr)\le\Lambda(1-h)\le(1-h)P_{\delta_0}.
\]
Hence $\Lambda(\theta)\to P_{\delta_0}$ as $\theta\uparrow1$, so $\Lambar(1)=P_{\delta_0}$ and
$\Lambar$ is continuous on $[0,1]$; with the value $+\infty$ outside $[0,1]$ it is closed and
convex. The same bounds give
$\{\Lambar(1)-\Lambar(1-h)\}/h\le P_{\delta_0}+\frac12\xi''(1)$. By convexity these quotients
increase to $\Lambar'_-(1)$ as $h\downarrow0$, so $f_1<\infty$.

(c) Let $m_0<\theta<1$ and suppose that $\zeta_\theta(\{0\})>\theta$. For every $\nu$, the
measure $(1-s)\zeta_\theta+s\nu$ lies in $\mathcal A_\theta$ for small $s>0$, so
\[
 P_{\zeta_\theta}\le P_{(1-s)\zeta_\theta+s\nu}\le(1-s)P_{\zeta_\theta}+sP_\nu ,
\]
that is, $P_{\zeta_\theta}\le P_\nu$. Then $\zeta_\theta=\zetas$, whose mass at zero is
$m_0<\theta$, a contradiction. Hence $\zeta_\theta(\{0\})=\theta$; this also holds at
$\theta=m_0$, where $\zeta_{m_0}=\zetas$.

Let $m_0\le\theta_1<\theta_2<1$. If $p(\theta_1)=p(\theta_2)$, then
$\zeta_{\theta_2}\in\mathcal A_{\theta_1}$ attains $p(\theta_1)$, so
$\zeta_{\theta_2}=\zeta_{\theta_1}$. This is impossible, because $\zeta_{\theta_i}(\{0\})=\theta_i$.
So $p$ is strictly increasing on $[m_0,1)$.

Let $\theta_1\ne\theta_2$ in $[m_0,1)$, let $t\in(0,1)$, and put
$\theta=t\theta_1+(1-t)\theta_2$, so that $m_0<\theta<1$. For the measure $\nu$ of
\eqref{eq:cons-mixture}, saturation gives
\[
 \nu(\{0\})=\frac{t\theta_1^2+(1-t)\theta_2^2}{\theta}
 =\theta+\frac{t(1-t)(\theta_1-\theta_2)^2}{\theta}>\theta=\zeta_\theta(\{0\}).
\]
So $\nu\in\mathcal A_\theta$ differs from $\zeta_\theta$, uniqueness gives
$\theta P_\nu>\Lambda(\theta)$, and the convexity inequality of (b) is strict.

(d) Let $\feq<f<f_1$ and $\varphi(\theta)=\Lambar(\theta)-\theta f$. By (b), $\varphi$ is closed
and convex, so it attains its minimum on $[0,1]$, and
$\varphi'_+(0)=\feq-f<0<f_1-f=\varphi'_-(1)$. Hence no minimizer lies at $0$ or at $1$. On
$[0,m_0]$ we have $\varphi(\theta)=\theta(\feq-f)$ by (a), which is strictly decreasing, so the
minimizers lie in $[m_0,1)$. They form an interval on which $\varphi$ is constant, that is,
on which $\Lambda$ is affine. By (c) this interval is a single point.

(e) Let $m_0<\theta<1$. By (c), $\mu_\theta(\{0\})=0$. For $\nu\in\Pspace([0,1])$, the segment
from $\zeta_\theta$ to $\theta\delta_0+(1-\theta)\nu$ stays in $\mathcal A_\theta$. By
minimality and \eqref{eq:cons-first-variation}, its right derivative at $\zeta_\theta$, which is
$(1-\theta)\int\Obs_\theta\,d(\nu-\mu_\theta)$, is nonnegative. Taking $\nu=\delta_t$ gives
$\Obs_\theta(t)\ge o_\theta$ for every $t\in[0,1]$. Since
$\int(\Obs_\theta-o_\theta)\,d\mu_\theta=0$ and $\Obs_\theta$ is continuous,
$\Obs_\theta=o_\theta$ on $\supp\mu_\theta$. Suppose that $\Obs_\theta(0)=o_\theta$. Then for
every $\nu$,
\[
 \int\Obs_\theta\,d(\nu-\zeta_\theta)
 =\int\Obs_\theta\,d\nu-\theta\,\Obs_\theta(0)-(1-\theta)o_\theta
 =\int(\Obs_\theta-o_\theta)\,d\nu\ge0,
\]
and Lemma~\ref{lem:cons-standard}(g) gives $\zeta_\theta=\zetas$, which contradicts
$\zeta_\theta(\{0\})=\theta>m_0$. Hence $\Obs_\theta(0)>o_\theta$.

Consequently $0\notin\supp\mu_\theta$. So $o_\theta=\min_{(0,1]}\Obs_\theta$, and
$\supp\zeta_\theta\cap(0,1]=\supp\mu_\theta\subset C_\theta$; in particular $C_\theta$ is
nonempty. Since $o_\theta=\min_{(0,1]}\Obs_\theta<\Obs_\theta(0)$, also $C_\theta=\{t\in[0,1]:\Obs_\theta(t)=o_\theta\}$; this set is closed and does not
contain $0$. It is therefore compact and bounded away from zero, and by
Lemma~\ref{lem:cons-standard}(f) it does not contain $1$. Thus $C_\theta\subset(0,1)$ and
$q_\theta=\min C_\theta\in(0,1)$. Since $\supp\mu_\theta\subset C_\theta\subset[q_\theta,1)$,
we get $\zeta_\theta((0,q_\theta))=0$. A point $r\in(0,q_\theta)$ is not in $C_\theta$, so
$\Obs_\theta(r)>o_\theta=\Obs_\theta(q_\theta)$, and the same holds at $r=0$. Finally,
$q_\theta\in(0,1)$ is a minimum point of $\Obs_\theta$ on $(0,1]$, and
Lemma~\ref{lem:cons-standard}(f),(b) give the last two relations in \eqref{eq:cons-contact}.
\end{proof}

\begin{proof}[Proof of Lemma~\ref{lem:cons-regular}]
(a) For $0\le\theta<1$ we have $\Lamhat(\theta)=\theta\{p(\theta)-\feq\}\ge0$, with equality on
$[0,m_0]$ by Lemma~\ref{lem:cons-variational}(a). If $m_0<\theta<1$, then
$p(\theta)>p(m_0)=\feq$ by Lemma~\ref{lem:cons-variational}(c). Also
$\Lamhat(1)=P_{\delta_0}-\feq\ge0$ by Lemma~\ref{lem:cons-variational}(b). If $m_0=1$, then
$\zetas=\delta_0$ and $\Lamhat(1)=0$; if $m_0<1$, then $\delta_0\ne\zetas$, and uniqueness of the
minimizer gives $\Lamhat(1)>0$.

(b) This is Lemma~\ref{lem:cons-variational}(c).

(c) An affine piece of $\Lambar$ on $[a,b]\subset[m_0,1]$ with $a<b$ would be affine on
$[a,(a+b)/2]\subset[m_0,1)$, which contradicts the strict convexity in
Lemma~\ref{lem:cons-variational}(c).

(d) Let $m_0<1$. Finiteness of $f_1$ is Lemma~\ref{lem:cons-variational}(b). By convexity of
$\Lambar$ and (a),
\[
 f_1=\Lambar'_-(1)\ge\frac{\Lambar(1)-\Lambar(m_0)}{1-m_0}
 =\feq+\frac{\Lamhat(1)}{1-m_0}>\feq .
\]
Let $f\in(\feq,f_1)$. By Lemma~\ref{lem:cons-variational}(d),
$\theta_f\in(0,1)\cap[m_0,1)$. Suppose that $\theta_f=m_0$. Then $0<m_0<1$, and since $\theta_f$
minimizes the convex function $\Lambar(\theta)-\theta f$, we have $f\in\partial\Lambar(m_0)$.
By Proposition~\ref{prop:cons-envelope}, $\Lambda$ is differentiable at the interior point $m_0$,
with $\Lambda'(m_0)=p(m_0)+m_0p'(m_0)=\feq$. So $\partial\Lambar(m_0)=\{\feq\}$, a contradiction.
Hence $\theta_f>m_0$. If $m_0=1$, then Lemma~\ref{lem:cons-variational}(a),(b) give
$\Lambar(\theta)=\theta\feq$ on $[0,1]$, so $f_1=\feq$.
\end{proof}

\begin{proof}[Proof of Corollary~\ref{cor:cons-level}]
By Proposition~\ref{prop:cons-envelope}, $\Lambda'$ is continuous on $(0,1)$. It is strictly
increasing on $(m_0,1)$, because $\Lambda$ is strictly convex on $[m_0,1)$
(Lemma~\ref{lem:cons-variational}(c)). As $\theta\uparrow1$, $\Lambda'(\theta)$ increases to
$\Lambar'_-(1)=f_1$, by convexity and the continuity of $\Lambar$ at $1$
(Lemma~\ref{lem:cons-variational}(b)). If $m_0>0$, then $\Lambda'(\theta)\to\Lambda'(m_0)=\feq$
as $\theta\downarrow m_0$, as in the proof of Lemma~\ref{lem:cons-regular}(d). If $m_0=0$, then
for $0<\theta<\frac12$ convexity gives
\[
 \feq=\Lambar'_+(0)\le\Lambda'(\theta)\le\frac{\Lambda(2\theta)-\Lambda(\theta)}{\theta}
 =2p(2\theta)-p(\theta),
\]
and the right side tends to $\feq$ by Lemma~\ref{lem:cons-variational}(a). Therefore $\Lambda'$
maps $(m_0,1)$ onto $(\feq,f_1)$. For $f$ in this interval and $\theta=(\Lambda')^{-1}(f)$,
the convex function $\Lambar(\vartheta)-\vartheta f$ has derivative zero at $\vartheta=\theta$, so $\theta$
minimizes it and $\theta=\theta_f$ by Lemma~\ref{lem:cons-variational}(d). Finally,
$\Lambda'=p+\theta p'$ and \eqref{eq:cons-envelope} at $\theta_f\in(m_0,1)$ give
\eqref{eq:cons-level}.
\end{proof}

\subsection{Jets, coupling and plateau identities}\label{sec:plat-jets}

This subsection states and proves the quantitative estimates announced in the paragraph on
analytic facts of Section~\ref{sec:plat}, and proves Lemmas~\ref{lem:plat-law}
and~\ref{lem:plat-ito}. The standing assumptions and notation are those of
Section~\ref{sec:plat}; as stated there, Lemmas~\ref{lem:plat-jets}
and~\ref{lem:plat-coupling} assume in addition $\beta_2>0$.

The next lemma quantifies how the spatial derivatives depend on the measure, on time and,
when the covariance varies, on $\xi$. Let $\mathfrak F$ be a set of covariances as in
Section~\ref{sec:plat} on which $\xi''$ is bounded above and below by positive constants on
$[0,1]$, uniformly over $\mathfrak F$. For $\xi,\tilde\xi\in\mathfrak F$ and
$\zeta,\tilde\zeta\in\Pspace([0,1])$ write $a=\xi''$, $\tilde a=\tilde\xi''$,
$\alpha=\alpha_\zeta$, $\tilde\alpha=\alpha_{\tilde\zeta}$, $u=\Phi_{\zeta,x}$ and
$\tilde u=\tilde\Phi_{\tilde\zeta,x}$, where $\tilde\Phi$ solves \eqref{eq:intro-pde} with
$\tilde\xi$ in place of $\xi$.

\begin{lemma}[Jets]\label{lem:plat-jets}
\begin{thmlist}
\item For a fixed covariance there are constants $C_k$, depending only on $\xi$ and $k$,
 such that for all $\zeta,\nu\in\Pspace([0,1])$
 \begin{equation}\label{eq:plat-jet-stability}
  \|\partial_x^k(\Phi_\zeta-\Phi_\nu)\|_\infty\le C_k\|a\|_\infty\Wone(\zeta,\nu)
  \qquad(k\ge1).
 \end{equation}
\item The constants $B_k$ in \eqref{eq:plat-bounds} can be chosen uniformly over
 $\mathfrak F$. For every $k\ge0$ there is $C_k'$, depending only on $k$ and
 $\mathfrak F$, such that
 \begin{equation}\label{eq:plat-jet-general}
  \|\partial_x^k(u-\tilde u)\|_\infty
  \le C_k'\bigl(\|a-\tilde a\|_{L^1[0,1]}+\|a\alpha-\tilde a\tilde\alpha\|_{L^1[0,1]}\bigr).
 \end{equation}
\item For every $k\ge1$ there is $L_k$, depending only on $\xi$ and $k$, such that for
 every $\zeta$ and all $s,t\in[0,1]$
 \begin{equation}\label{eq:plat-time-lip}
  \sup_{x\in\R}|\partial_x^k\Phi_\zeta(t,x)-\partial_x^k\Phi_\zeta(s,x)|\le L_k|t-s| .
 \end{equation}
\item For every $\zeta$ and $t\in[0,1]$,
 \begin{equation}\label{eq:plat-small-time}
  \E X_t^2\le2\|a\|_\infty t+2\|a\|_\infty^2t^2 .
 \end{equation}
\item For all $\zeta,\nu\in\Pspace([0,1])$, $t\in[0,1]$ and $k\ge1$,
 \begin{equation}\label{eq:plat-origin}
  \E\bigl|\partial_x^k\Phi_\zeta(t,X^\zeta_t)-\partial_x^k\Phi_\nu(0,0)\bigr|
  \le B_{k+1}\bigl(2\|a\|_\infty t+2\|a\|_\infty^2t^2\bigr)^{1/2}+L_kt
  +C_k\|a\|_\infty\Wone(\zeta,\nu).
 \end{equation}
\end{thmlist}
\end{lemma}

\begin{proof}
(b) First let $\zeta$ and $\tilde\zeta$ have finitely many atoms. On each closed strip
between consecutive atoms of either measure, $u$ solves
\[
 u_t+\tfrac a2u_{xx}+a\alpha uu_x=0,\qquad u(1,x)=\tanh x,
\]
classically, and $\tilde u$ solves the same equation with $\tilde a$, $\tilde\alpha$. For
$k\ge2$, differentiating $k$ times in $x$ gives a linear backward equation for
$\partial_x^ku$, with drift $a\alpha u$, zeroth-order coefficient $(k+1)a\alpha u_x$, and a
forcing that is a polynomial in derivatives of order less than $k$. Since $|u|\le1$ and
$0<u_x\le1$, the maximum principle and induction on $k$ bound $\partial_x^ku$ uniformly
over $\mathfrak F$. The difference $w=u-\tilde u$ has terminal value zero and solves
\[
 w_t+\tfrac a2w_{xx}+a\alpha uw_x+a\alpha\tilde u_xw
 =-\tfrac{a-\tilde a}2\,\tilde u_{xx}-(a\alpha-\tilde a\tilde\alpha)\,\tilde u\tilde u_x ,
\]
whose drift and zeroth-order coefficient are bounded uniformly. The maximum principle
bounds $\|w\|_\infty$ by a constant times the time integral of the supremum of the
right side, which is \eqref{eq:plat-jet-general} for $k=0$. Differentiating this
equation and inducting gives every $k$; at each step the forcing involves only finitely
many of the derivatives already bounded, so the constants stay uniform over $\mathfrak F$
and over all measures. Both statements pass to general measures
by~\cite[Proposition~1(i)]{AC15a}, applied at each of the two covariances, because weak
convergence of measures on $[0,1]$ implies $L^1$ convergence of their distribution
functions.

(a) Take $\mathfrak F=\{\xi\}$ in (b). Then
$\|a(\alpha_\zeta-\alpha_\nu)\|_{L^1}\le\|a\|_\infty\Wone(\zeta,\nu)$, and
\eqref{eq:plat-jet-general} for $u=\Phi_{\zeta,x}$ becomes \eqref{eq:plat-jet-stability}
after the index shift $k\mapsto k-1$.

(c) For a measure with finitely many atoms the Parisi equation holds classically on each
closed strip between consecutive atoms, and there
$\partial_t\partial_x^k\Phi=-\partial_x^{k-1}\{\frac a2\Phi_{xxx}+a\alpha\Phi_x\Phi_{xx}\}$.
This is bounded by a constant $L_k$ that depends only on $\xi$ and $B_1,\dots,B_{k+2}$.
Since $\partial_x^k\Phi$ is continuous in time across atoms~\cite[Proposition~5,
(64)]{AC15a}, \eqref{eq:plat-time-lip} holds for such measures, and it passes to every
$\zeta$ by \eqref{eq:plat-jet-stability}.

(d) The drift of $X$ is bounded by $\|a\|_\infty$, since $\alpha\le1$ and
$|\Phi_{\zeta,x}|\le1$, and the martingale part has variance $\int_0^ta\le\|a\|_\infty t$.

(e) Split the difference at $(t,0)$ and at $(0,0)$. The three pieces are bounded by
$B_{k+1}\E|X_t|$, by $L_kt$ (part (c)) and by $C_k\|a\|_\infty\Wone(\zeta,\nu)$
(part (a)); then use $\E|X_t|\le(\E X_t^2)^{1/2}$ and \eqref{eq:plat-small-time}.
\end{proof}

Part (b), uniform over a family of covariances, is needed only near criticality
(Section~\ref{sec:nc-reference}); elsewhere the covariance is fixed and
\eqref{eq:plat-jet-stability} suffices. The next lemma gives continuity of
$\Gamma_\zeta$ and $\Gamma_\zeta'$ in the pair $(\xi,\zeta)$, uniformly in time. It is used to
pass It\^o identities from finitely atomic measures to general ones
(Proposition~\ref{prop:ons-equilibrium} and Lemma~\ref{lem:atom-odd-slope}) and near
criticality (Section~\ref{sec:nc-reference}).

\begin{lemma}[Coupling]\label{lem:plat-coupling}
Let $\xi_n,\xi\in\mathfrak F$ and $\zeta_n,\zeta\in\Pspace([0,1])$, with $\xi_n''\to\xi''$
uniformly on $[0,1]$ and $\zeta_n\to\zeta$ weakly. Let $X^n$ and $X$ be the diffusions of
the pairs $(\xi_n,\zeta_n)$ and $(\xi,\zeta)$, driven by the same Brownian motion, and let
$\Gamma_n$, $\Gamma$ be their functions $\Gamma$. Then
$\E\sup_{t\le1}|X^n_t-X_t|^2\to0$, and $\Gamma_n\to\Gamma$ and $\Gamma_n'\to\Gamma'$
uniformly on $[0,1]$.
\end{lemma}

\begin{proof}
Write $a_n=\xi_n''$, $\alpha_n=\alpha_{\zeta_n}$, $u_n$ for the corresponding
$\Phi_{\zeta_n,x}$, and $b_n=a_n\alpha_nu_n$, $b=a\alpha u$ for the drifts. Since
$|u_n|\le1$ and $|u_x|\le1$,
\[
 |b_n(s,y)-b(s,x)|\le|a_n\alpha_n-a\alpha|(s)+\|a\|_\infty\|u_n-u\|_\infty
 +\|a\|_\infty|y-x| .
\]
The martingale parts differ by $\int_0^t(\sqrt{a_n}-\sqrt a)\,dB$, whose squared maximum
has expectation at most $4\int_0^1|\sqrt{a_n}-\sqrt a|^2$ by Doob's inequality. Gronwall's
inequality applied to $\Delta_n(t)=\E\sup_{s\le t}|X^n_s-X_s|^2$ gives
\[
 \Delta_n(1)\le3e^{3\|a\|_\infty^2}\Bigl\{\bigl(\|a_n\alpha_n-a\alpha\|_{L^1}
 +\|a\|_\infty\|u_n-u\|_\infty\bigr)^2+4\int_0^1|\sqrt{a_n}-\sqrt a|^2\Bigr\},
\]
which tends to zero by \eqref{eq:plat-jet-general} with $k=0$ and because
$\|a_n\alpha_n-a\alpha\|_{L^1}\le\|a_n-a\|_\infty+\|a\|_\infty\Wone(\zeta_n,\zeta)$. Since
$|u|\le1$, $0<u_x\le1$ and $|u_{xx}|\le B_3$,
\[
 |\Gamma_n(t)-\Gamma(t)|\le2\|u_n-u\|_\infty+2\E|X^n_t-X_t|,
\]
and, with $\Gamma'=a\E u_x(t,X_t)^2$,
\[
 |\Gamma_n'(t)-\Gamma'(t)|\le\|a_n-a\|_\infty
 +2\|a\|_\infty\bigl(\|\partial_x(u_n-u)\|_\infty+B_3\,\E|X^n_t-X_t|\bigr).
\]
Both bounds are uniform in $t$ and tend to zero.
\end{proof}

We turn to the two lemmas of Section~\ref{sec:plat-ito}.

\begin{proof}[Proof of Lemma~\ref{lem:plat-law}]
On $[s,t)$ we have $\alpha_\zeta=m$. If $m=0$, the drift of $X$ vanishes there and the
claim is the Gaussian transition law. If $m>0$, put $g=e^{m\Phi_\zeta}$. For $s<r<t'<t$, $\Phi_\zeta(r,\cdot)$ is the Cole--Hopf transform of $\Phi_\zeta(t',\cdot)$ (as in the proof of Lemma~\ref{lem:plat-ito}(c)); since $\Phi_\zeta$ is continuous on $[0,1]\times\R$ with $|\Phi_{\zeta,x}|\le1$ (Lemma~\ref{lem:cons-standard}(a)), letting $t'\uparrow t$ and $r\downarrow s$ shows that $g(r,\cdot)$ is the Gaussian average of $g(t,\cdot)$ with variance $\xi'(t)-\xi'(r)$ for $s\le r\le t$. Hence $g_r+\frac a2g_{xx}=0$ on $[s,t)\times\R$, and $g$ is continuous on $[s,t]\times\R$. Let $M_r=x+\int_s^r\sqrt{a}\,dB$.
Then $g(r,M_r)/g(s,x)$, $s\le r\le t$, is a positive martingale with mean one. Under the
probability measure that it defines at time $t$, Girsanov's theorem gives $M$ the drift
$a\,\partial_x\log g=am\Phi_{\zeta,x}$, which is the drift of $X$ on $[s,t)$. The drift is
Lipschitz in space, so the law of $X$ on $[s,t]$ started from $x$ is the law of $M$ under
the new measure. The density of $X_t$ is therefore $g(t,y)/g(s,x)$ times the
$\mathcal N(x,\xi'(t)-\xi'(s))$ density of $M_t$.
\end{proof}

\begin{proof}[Proof of Lemma~\ref{lem:plat-ito}]
(a) is~\cite[Proposition~3, (20)]{AC15a}, transferred to $X$ as explained in the
paragraph on analytic facts of Section~\ref{sec:plat}, together with
$\Gamma(0)=\Phi_{\zeta,x}(0,0)^2=0$ by oddness. (b) The one-sided derivatives
are~\cite[Proposition~3, (21)--(22)]{AC15a}; the right derivative involves $\zeta([0,t])$ and the left one $\zeta([0,t))$. At $t=0$ we have $X_0=0$, $V_0=0$ and $a'(0)=\xi'''(0)$, which gives
\eqref{eq:plat-zero}, and \eqref{eq:plat-jump} is the difference of the two expressions.
(c) On an open plateau $I$ the function $\Phi_\zeta$ is a Cole--Hopf transform of its value
at a later time of $I$, hence smooth on $I\times\R$, and the stochastic differentials of
$U$ and $V$ stated after \eqref{eq:plat-third-pde} hold. Their martingale parts are true
martingales by \eqref{eq:plat-bounds}.
Taking expectations,
\[
 \frac d{dt}\E U_t^2=a\E\bigl(V_t^2-2mU_t^3\bigr),\qquad
 \frac d{dt}\E\bigl(V_t^2-2mU_t^3\bigr)=a\E\bigl(W_t^2-12mU_tV_t^2+6m^2U_t^4\bigr),
\]
with continuous right sides. Differentiating $\Gamma'=a\E U_t^2$ twice gives
\eqref{eq:plat-second} and \eqref{eq:plat-third}.
\end{proof}

\subsection{Zero lies in the support}\label{sec:an-zero}

Here $X$ is the diffusion of $\zetas$, $U_t=\Phi_{\zetas,xx}(t,X_t)$ and
$V_t=\Phi_{\zetas,xxx}(t,X_t)$, and $\Gamma_*=\Gamma_{\zetas}$, $\Obs_*=\Obs_{\zetas}$.

\begin{proof}[Proof of Lemma~\ref{lem:ons-zero-in-support}]
By Lemma~\ref{lem:cons-standard}(g), with $o_*=\int\Obs_*\,d\zetas$, we have
$o_*=\min_{[0,1]}\Obs_*$, and the minimum is attained at every support point. By
Lemma~\ref{lem:cons-standard}(f), $\Obs_*$ is twice continuously differentiable, with
$\Obs_*'(t)=-\frac12\xi''(t)\{\Gamma_*(t)-t\}$, and $1$ is not a minimum point of $\Obs_*$
on $(0,1]$; so $\Obs_*(1)>o_*$ and $1\notin\supp\zetas$.

\emph{Positive support points.} A point $u\in\supp\zetas\cap(0,1)$ is an interior minimum
of $\Obs_*$, and~\eqref{eq:plat-contact} gives $\Gamma_*(u)=u$ and $\Gamma_*'(u)\le1$.

\emph{Zero lies in the support.} Suppose that $t_0:=\min\supp\zetas>0$. By the two
preceding steps, $t_0<1$ and $\Gamma_*(t_0)=t_0$. The distribution function $\alphas$
vanishes on $[0,t_0)$. Hence the drift of $X$ vanishes there, and for $0<t<t_0$ the variable
$X_t=\int_0^t\sqrt{\xi''(s)}\,dB_s$ is centered Gaussian with variance $\xi'(t)>0$. The
interval $(0,t_0)$ is a plateau of height $m=0$, and~\eqref{eq:plat-second} gives
\[
 \Gamma_*''(t)=\xi'''(t)\,\E U_t^2+\xi''(t)^2\,\E V_t^2\qquad(0<t<t_0).
\]
Here $\xi'''(t)\ge0$ and $\xi''(t)>0$. Moreover $\E V_t^2>0$. Indeed, the continuous
function $\Phi_{\zetas,xxx}(t,\cdot)$ is not identically zero. Otherwise
$\Phi_{\zetas,xx}(t,\cdot)$ would be the positive constant $\Phi_{\zetas,xx}(t,0)$, and by
oddness $\Phi_{\zetas,x}(t,x)=x\,\Phi_{\zetas,xx}(t,0)$ would be unbounded, which
contradicts $|\Phi_{\zetas,x}|\le1$ (Lemma~\ref{lem:cons-standard}(a)). So $\Phi_{\zetas,xxx}(t,\cdot)$ is nonzero on a nonempty
open set, which has positive probability under the Gaussian law of $X_t$. Hence
$\Gamma_*''>0$ on $(0,t_0)$. The function $h=\Gamma_*-\id$ is therefore continuous and
strictly convex on $[0,t_0]$, with $h(t_0)=0$ and $h(0)=\Phi_{\zetas,x}(0,0)^2=0$ by
oddness. Thus $h<0$ on $(0,t_0)$, so $\Obs_*'=-\frac12\xi''h>0$ there and
$\Obs_*(0)<\Obs_*(t_0)=o_*$. This contradicts $\Obs_*\ge o_*$, so $0\in\supp\zetas$.

\emph{The point zero.} By oddness $\Gamma_*(0)=0$, and $\Obs_*(0)=o_*$ because
$0\in\supp\zetas$. If $\xi''(0)>0$, the endpoint inequality,
Lemma~\ref{lem:plat-obstacle}(c), gives $\Gamma_*'(0)\le1$: the difference
$\Obs_*(t)-\Obs_*(0)=-\frac14\xi''(0)\{\Gamma_*'(0)-1\}t^2+o(t^2)$ is nonnegative. If
$\xi''(0)=0$, then $\Gamma_*'(0)=\xi''(0)\Ustar^2=0$.
\end{proof}

\subsection{Positivity of the plateau polynomial near criticality}\label{sec:an-nc}

We use the notation of Section~\ref{sec:nc}: the covariance is recorded as a superscript,
$\Phi_c(t,x)=\log\cosh x+\frac12(1-t)$ is the solution for the reference pair
$(\xic,\delta_0)$, and $\Psi^\xi_\zeta$ is the plateau polynomial \eqref{eq:plat-psi} computed
with $\xi$. As in the paragraph on varying the covariance in Section~\ref{sec:nc-reference}, the
covariances with $\normCfour{\xi-\xic}<\frac14$, together with $\xic$, form a family
$\mathfrak F$ as in Lemma~\ref{lem:plat-jets}(b), with $\frac34<a<\frac54$ on $[0,1]$.

\begin{proof}[Proof of Lemma~\ref{lem:nc-positive-third}]
Let $\normCfour{\xi-\xic}<\frac14$, let $\zeta$ be a probability measure, and let
$k\in\{2,3,4\}$. The derivative $\partial_x^k\Phi_c(t,x)=(\log\cosh)^{(k)}(x)$ does not depend
on $t$, and $(\log\cosh)^{(k+1)}$ is bounded by an absolute constant $L$. Hence
\[
 \bigl|\partial_x^k\Phi^\xi_\zeta(t,X_t)-\partial_x^k\Phi_c(0,0)\bigr|
 \le\|\partial_x^k(\Phi^\xi_\zeta-\Phi_c)\|_\infty+L|X_t| .
\]
We apply the estimate \eqref{eq:plat-jet-general} to the pairs $(\xi,\zeta)$ and $(\xic,\delta_0)$, for
which $(\tilde a,\tilde\alpha)=(1,1)$. Since
$\|a\alpha_\zeta-1\|_{L^1}\le\|a-1\|_{L^1}+\frac54\int_0^1(1-\alpha_\zeta)\,dt$ and
$\int_0^1(1-\alpha_\zeta)\,dt=\int s\,\zeta(ds)=\Wone(\zeta,\delta_0)$, the first term is at most
$C(\normCfour{\xi-\xic}+\Wone(\zeta,\delta_0))$. By \eqref{eq:plat-small-time},
$\E|X_t|\le(2\|a\|_\infty t+2\|a\|_\infty^2t^2)^{1/2}\le3\sqrt t$. Now $\Psi^\xi_\zeta(t,m)$ is a
polynomial in $a(t),a'(t),a''(t),m$ and in expectations of monomials of degree at most four in
$(U_t,V_t,W_t)$, all of which are bounded uniformly; at the reference values $(1,0,0)$ and
$(1,0,-2)$ it equals $4+6m^2$. Therefore
\[
 \bigl|\Psi^\xi_\zeta(t,m)-4-6m^2\bigr|
 \le C'\bigl(\normCfour{\xi-\xic}+\Wone(\zeta,\delta_0)+\sqrt t\bigr)
 \qquad(0\le t\le1,\ 0\le m\le1),
\]
with an absolute constant $C'$. Choose $K_P$ with $C'\sqrt{K_P}\le\frac23$, and then
$\varepsilon_P$ and $\mathcal W_P$ so that the other two terms contribute at most $\frac43$.
Then $\Psi^\xi_\zeta\ge2$ on $[0,K_P]\times[0,1]$, and \eqref{eq:nc-positive-third} follows from
Lemma~\ref{lem:plat-ito}(c).
\end{proof}

\section{Fractional moments and the free energy}\label{app:guerra}

The probability statements, Theorem~\ref{thm:sk}(c), Theorem~\ref{thm:local}(d),
Theorem~\ref{thm:atom}(d) and the last assertion of Corollary~\ref{cor:pure-p}, pass from
$\Lambda$ to $N^{-1}\log\ZN$ through \condFM. This appendix collects that spin-system side: the
tilting lemma and the places where \condFM\ enters, \condFM\ for the Sherrington--Kirkpatrick
model from \cite{CM19}, and Guerra's bound, the upper half of \condFM, for $\xi$ convex on
$[-1,1]$.

\subsection{From fractional moments to the rate}\label{sec:cons-moments}

That \condFM\ makes $-\Sigma$ the upper-tail rate of $N^{-1}\log\ZN$ is noted in
\cite[Remark~1.3]{Bou26b}; \cite[Remark~1.4 and Proposition~3.12]{Bou26b} prove, for $\xi$
strictly convex on $[-1,1]$, a direct lower bound at regular levels that satisfy the strict
Plefka condition. We record the tilting
argument because it shows what it needs from the variational side: differentiability of
$\Lambda$ at the dual minimizer, which Corollary~\ref{cor:cons-level} gives at every level.
Remark~\ref{rem:mo-fm} lists the sources of \condFM\ and which bounds need it.

\begin{lemma}[Upper deviations of the free energy]\label{lem:cons-gartner-ellis}
Assume \condFM. Then for every $f\in(\feq,f_1)$,
\begin{equation}\label{eq:cons-ldp}
 \lim_{N\to\infty}\frac1N\log\Prob\bigl(N^{-1}\log\ZN\ge f\bigr)=\Sigma(f).
\end{equation}
The interval $(\feq,f_1)$ is nonempty if and only if $\zetas\ne\delta_0$.
\end{lemma}

\begin{proof}
The last sentence is Lemma~\ref{lem:cons-regular}(d). Let $f\in(\feq,f_1)$, so that
$\zetas\ne\delta_0$ and Corollary~\ref{cor:cons-level} applies.

\emph{Upper bound.} For $\theta\in(0,1)$, Markov's inequality gives
$\Prob(N^{-1}\log\ZN\ge f)\le e^{-N\theta f}\E\ZN^\theta$, so
$\limsup_NN^{-1}\log\Prob(N^{-1}\log\ZN\ge f)\le\Lambda(\theta)-\theta f$. At
$\theta=\theta_f$ the right side is $\Sigma(f)$.

\emph{Lower bound.} Let $\varepsilon>0$ with $f'=f+2\varepsilon<f_1$, and put
$\theta=\theta_{f'}\in(m_0,1)$, so that $\Lambda'(\theta)=f'$ by
Corollary~\ref{cor:cons-level}. Tilt by $Z_N^\theta$: for events $E$ put
$\mathbb Q_N(E)=\E[\ZN^\theta\ind_E]/\E\ZN^\theta$. For $0<\lambda<\min\{\theta,1-\theta\}$,
Markov's inequality under $\mathbb Q_N$ and \condFM\ give
\begin{align*}
 \limsup_N\frac1N\log\mathbb Q_N\bigl(N^{-1}\log\ZN\ge f'+\varepsilon\bigr)
  &\le\Lambda(\theta+\lambda)-\Lambda(\theta)-\lambda(f'+\varepsilon),\\
 \limsup_N\frac1N\log\mathbb Q_N\bigl(N^{-1}\log\ZN\le f'-\varepsilon\bigr)
  &\le\Lambda(\theta-\lambda)-\Lambda(\theta)+\lambda(f'-\varepsilon).
\end{align*}
Since $\Lambda'(\theta)=f'$, both right sides equal $-\lambda\varepsilon+o(\lambda)$ as
$\lambda\downarrow0$. Fix $\lambda$ so that both are negative. Then $\mathbb Q_N(A_N)\to1$ for
$A_N=\{|N^{-1}\log\ZN-f'|<\varepsilon\}$. On $A_N$ we have
$\ZN^\theta\le e^{N\theta(f'+\varepsilon)}$ and $N^{-1}\log\ZN>f+\varepsilon$. Hence
$\Prob(N^{-1}\log\ZN\ge f)\ge\Prob(A_N)
\ge\mathbb Q_N(A_N)\,\E\ZN^\theta\,e^{-N\theta(f'+\varepsilon)}$, and \condFM\ gives
$\liminf_NN^{-1}\log\Prob(N^{-1}\log\ZN\ge f)\ge\Lambda(\theta)-\theta f'-\theta\varepsilon
\ge\Sigma(f')-\varepsilon$. The function $\Sigma$ is an infimum of affine functions, and it is
finite because $\Lambar$ is lower semicontinuous and finite on $[0,1]$; so $\Sigma$ is concave
and continuous on $\R$. Letting $\varepsilon\downarrow0$ gives the lower bound $\Sigma(f)$.
\end{proof}

This is the tilting step behind the G\"artner--Ellis theorem~\cite{DZ98}, written for the
one-sided event because \condFM\ concerns only $\theta\in(0,1)$. For the
Sherrington--Kirkpatrick covariance the field \eqref{eq:intro-covariance} is the original
Hamiltonian, summed over $i<j$, completed by an independent diagonal term, as
in~\cite[Section~1.1]{AL26}. The next lemma shows that the completion is invisible on the scale
of large deviations; the proof of Theorem~\ref{thm:sk}(c) applies it with $I=(\feq,f_1)$.

\begin{lemma}[The original Sherrington--Kirkpatrick model]\label{lem:cons-original-sk}
Let $\xi(t)=\beta^2t^2/2$ with $\beta>0$, and let
$\FNb=\log\sum_{\sigma\in\{-1,1\}^N}\exp\bigl(\beta N^{-1/2}\sum_{i<j}g_{ij}\sigma_i\sigma_j\bigr)$
with independent standard Gaussian $g_{ij}$. For all $f\in\R$, $\eta>0$ and $N\ge1$,
\begin{equation}\label{eq:cons-sk-sandwich}
\begin{aligned}
 \Prob\bigl(N^{-1}\FNb-\log2\ge f\bigr)
 &\le\Prob\bigl(N^{-1}\log\ZN\ge f-\eta\bigr)+e^{-\eta^2N^2/\beta^2},\\
 \Prob\bigl(N^{-1}\log\ZN\ge f+\eta\bigr)
 &\le\Prob\bigl(N^{-1}\FNb-\log2\ge f\bigr)+e^{-\eta^2N^2/\beta^2}.
\end{aligned}
\end{equation}
Consequently, if $\lim_NN^{-1}\log\Prob(N^{-1}\log\ZN\ge f)=\Sigma(f)$ for every $f$ in an open
interval $I$, then $\lim_NN^{-1}\log\Prob(N^{-1}\FNb-\log2\ge f)=\Sigma(f)$ for every $f\in I$.
\end{lemma}

\begin{proof}
Let $D_N$ be a centered Gaussian variable with variance $\frac12$, independent of
$(g_{ij})_{i<j}$, and put
$H_N(\sigma)=\beta N^{-1/2}\sum_{i<j}g_{ij}\sigma_i\sigma_j+\beta D_N$. For
$\sigma,\sigma'\in\{-1,1\}^N$ with $R=N^{-1}\sigma\cdot\sigma'$,
$\E H_N(\sigma)H_N(\sigma')
=\frac{\beta^2}N\sum_{i<j}\sigma_i\sigma'_i\sigma_j\sigma'_j+\frac{\beta^2}2
=\frac{\beta^2}{2N}(N^2R^2-N)+\frac{\beta^2}2=N\xi(R)$.
So $H_N$ is a field as in \eqref{eq:intro-covariance}, and
$\ZN=2^{-N}\sum_\sigma e^{H_N(\sigma)}$ has the law of the partition function. On this
probability space $N^{-1}\log\ZN=N^{-1}\FNb-\log2+\beta D_N/N$ identically. If
$N^{-1}\FNb-\log2\ge f$ and $N^{-1}\log\ZN<f-\eta$, then $\beta D_N<-\eta N$; if
$N^{-1}\log\ZN\ge f+\eta$ and $N^{-1}\FNb-\log2<f$, then $\beta D_N>\eta N$. Since $\beta D_N$
is centered Gaussian with variance $\beta^2/2$, each of these events has probability at most
$e^{-\eta^2N^2/\beta^2}$, and \eqref{eq:cons-sk-sandwich} follows.

For the consequence, let $f\in I$ and let $\eta>0$ be small enough that $f\pm\eta\in I$. The
function $\Sigma$ is finite and continuous on $\R$ (proof of
Lemma~\ref{lem:cons-gartner-ellis}), and the error terms in \eqref{eq:cons-sk-sandwich} decay
faster than every exponential. The first inequality therefore gives
$\limsup_NN^{-1}\log\Prob(N^{-1}\FNb-\log2\ge f)\le\Sigma(f-\eta)$, and the second gives
$\liminf_NN^{-1}\log\Prob(N^{-1}\FNb-\log2\ge f)\ge\Sigma(f+\eta)$. Let $\eta\downarrow0$.
\end{proof}

\begin{remark}[Where the fractional-moment formula enters]\label{rem:mo-fm}
\begin{steplist}
\item \emph{The upper bound.} The upper bound in the proof of
 Lemma~\ref{lem:cons-gartner-ellis} uses only
 $\limsup_NN^{-1}\log\E\ZN^\theta\le\Lambda(\theta)$ at $\theta_f$, which for $\xi$ convex on
 $[-1,1]$ is Proposition~\ref{prop:A-guerra}. For such $\xi$,
 $\limsup_NN^{-1}\log\Prob(N^{-1}\log\ZN\ge f)\le\Sigma(f)$ for $\feq<f<f_1$ without \condFM,
 and if such $\xi$ also satisfies the hypotheses of Theorem~\ref{thm:local}, its item~(c)
 gives, again without \condFM,
 \[
  \liminf_{\delta\downarrow0}\delta^{-6/5}\liminf_{N\to\infty}
  \Bigl(-\frac1N\log\Prob\bigl(N^{-1}\log\ZN\ge\feq+\delta\bigr)\Bigr)
  \ge\frac56\,(6c)^{-1/5}.
 \]
\item \emph{The lower bound.} It uses \condFM\ at $\theta_{f+2\varepsilon}$ and
 $\theta_{f+2\varepsilon}\pm\lambda$, which lie
 in any prescribed neighbourhood of $\theta_f$ once $\varepsilon$ and $\lambda$ are small, by
 continuity of $(\Lambda')^{-1}$. So \eqref{eq:cons-ldp} holds at $f$ once \condFM\ holds near
 $\theta_f$, and since $\theta_f\downarrow m_0$ as $f\downarrow\feq$
 (Corollary~\ref{cor:cons-level}), limits as $f\downarrow\feq$, such as the second display of
 Theorem~\ref{thm:local}(d), need \condFM\ only on $(m_0,\theta_1)$ for some $\theta_1>m_0$. For
 the Sherrington--Kirkpatrick covariance with $\beta>1$, where \condH\ holds
 \cite[Proposition~5.1]{Lop26a} (an alternative proof is \cite[Theorem~5.1]{Pei26}) and hence
 $m_0=0$ by Theorem~\ref{thm:local}, \cite[Theorem~1.1]{AL26} for $\theta\le\frac12$ suffices
 for these limits.
\item \emph{Sources.} The upper half of \condFM\ is Guerra's bound \cite[Section~6]{Tal07},
 proved for $\xi$ convex on $[-1,1]$ in Proposition~\ref{prop:A-guerra}, and
 \cite[Section~7]{Tal07} explains how the proof of the
 Parisi formula is modified to give the lower half for even $\xi$. By the preprint
 \cite[Theorem~2.1 and its proof]{CGKLM}, \condFM\ holds for every $\xi$ finite on $\R$, with the
 Guerra-type upper bound written out for convex $\xi$ and sketched otherwise, and the lower half
 \cite[Lemma~2.6]{CGKLM} written as modifications of the Aizenman--Sims--Starr scheme; this is
 the only source we know for covariances with a term of odd degree, such as the pure $p$-spin
 models with odd $p$ and the mixtures of Theorem~\ref{thm:atom}(c). For the
 Sherrington--Kirkpatrick covariance, \condFM\ is the one-level case of \cite[Theorem~2.1]{CM19}
 (Lemma~\ref{lem:A-fm-sk}), and the preprint \cite[Theorem~1.1]{AL26} gives a quantitative
 version for $0<\theta\le\frac12$.
\item \emph{A cubic term.} The covariances of Theorem~\ref{thm:atom}(c) are convex on $[-1,1]$
 (Proposition~\ref{prop:nc-support}(a), as $\varepsilon_3=\varepsilon_1$), so by (i) the bound
 $\limsup_NN^{-1}\log\Prob(N^{-1}\log\ZN\ge\feq+\delta)\le\Sigma(\feq+\delta)$,
 $0<\delta<f_1-\feq$, holds for them without \condFM; when $\xi$ is finite on $\R$, \condFM\
 from \cite[Theorem~2.1]{CGKLM} and Theorem~\ref{thm:atom}(d) turn \eqref{eq:atom-odd-rate} into
 the exponent of $\Prob(N^{-1}\log\ZN\ge\feq+\delta)$.
\item \emph{Even $p$.} For the pure $p$-spin models with even $p$, $\xi$ is convex on $[-1,1]$,
 $m_0>0$ (Corollary~\ref{cor:pure-p}) and $P_{\zetas}=\feq$. Taking $\zeta=\zetas$ in
 \eqref{eq:A-guerra}, Markov's inequality gives
 $\Prob(N^{-1}\log\ZN\ge\feq+\delta)\le e^{-N\theta\delta}$ for all $N$ and
 $\theta\in(0,1)$ with $\theta\le m_0$, hence
 $\limsup_NN^{-1}\log\Prob(N^{-1}\log\ZN\ge\feq+\delta)\le-\zetas(\{0\})\,\delta$. This immediate
 consequence of Guerra's bound \cite[Section~6]{Tal07} is not claimed as new
 (Section~\ref{sec:intro-perturbations}).
\end{steplist}
\end{remark}

\subsection{The fractional-moment formula for the Sherrington--Kirkpatrick model}
\label{sec:A-sk}

\begin{lemma}[The fractional-moment formula for the Sherrington--Kirkpatrick model]
\label{lem:A-fm-sk}
Let $\xi(t)=\beta^2t^2/2$ with $\beta>0$. Then \condFM\ holds: for every $\theta\in(0,1)$,
$\lim_{N\to\infty}N^{-1}\log\E\ZN^\theta=\Lambda(\theta)$.
\end{lemma}

Contucci and Mingione proved a variational formula for the pressure of a multiscale
Sherrington--Kirkpatrick model \cite[Theorem~2.1]{CM19}; Chen, Guionnet, Ko,
Lacroix-A-Chez-Toine and Mourrat observe that its one-scale case contains \condFM\ for this model
without external field \cite{CGKLM}. We write out the dictionary, using an identity for
Poisson--Dirichlet weights, which turns the one-scale pressure into a fractional moment, and an
approximation lemma, which turns the trial set of \cite{CM19} into the constraint
$\zeta(\{0\})\ge\theta$. (The upper half of Lemma~\ref{lem:A-fm-sk} is also a case of
Proposition~\ref{prop:A-guerra}, since $\xi$ is convex.)

\paragraph{Finitely supported measures.}
This proof and that of Proposition~\ref{prop:A-guerra} reduce to measures with finitely many
atoms, for which the Parisi equation \eqref{eq:intro-pde} is solved by Cole--Hopf transforms
\cite[(11)--(12)]{AC15a}. Let $\zeta\in\Pspace([0,1])$ have finite support, let
$0=q_0<q_1<\dots<q_L=1$ be points with $\supp\zeta\subset\{q_0,\dots,q_L\}$, and put
$y_j=\zeta([0,q_{j-1}])$ for $1\le j\le L$, so that $\zeta([0,t])=y_j$ for
$t\in[q_{j-1},q_j)$. If $y_1=\zeta(\{0\})>0$, then for $1\le j\le L$ and $x\in\R$
\begin{equation}\label{eq:A-cole-hopf}
 \Phi_\zeta(q_{j-1},x)=\frac1{y_j}\log\E\exp\Bigl\{y_j\,
 \Phi_\zeta\bigl(q_j,\,x+\sqrt{\xi'(q_j)-\xi'(q_{j-1})}\;g\bigr)\Bigr\},
 \qquad \Phi_\zeta(1,x)=\log\cosh x,
\end{equation}
where $g$ is a standard Gaussian variable. With $\psi(t)=t\xi'(t)-\xi(t)$, so that
$\psi'(t)=t\xi''(t)$ and $\psi(0)=0$, the second term of $P_\zeta$ in
\eqref{eq:intro-functional} is
\begin{equation}\label{eq:A-penalty}
 \frac12\int_0^1t\xi''(t)\,\zeta([0,t])\,dt
 =\frac12\sum_{j=1}^Ly_j\bigl\{\psi(q_j)-\psi(q_{j-1})\bigr\}.
\end{equation}

\paragraph{The Poisson--Dirichlet identity.}
Let $0<\theta<1$, let $u_1>u_2>\cdots$ be the points of a Poisson process on $(0,\infty)$ with
mean measure $\theta x^{-1-\theta}\,dx$, and let $\nu_\alpha=u_\alpha/\sum_{\alpha'\ge1}u_{\alpha'}$,
$\alpha\ge1$, be the Poisson--Dirichlet weights with parameter $\theta$ (see \cite{Pan13}). If
$Z>0$ is a random variable with $\E Z^\theta<\infty$ and $(Z_\alpha)_{\alpha\ge1}$ are
independent copies of $Z$, independent of $(\nu_\alpha)$, then
\begin{equation}\label{eq:A-pd}
 \E\log\sum_{\alpha\ge1}\nu_\alpha Z_\alpha=\frac1\theta\log\E Z^\theta .
\end{equation}
This is the case of one level of \cite[(121)--(122)]{CM19}; see also \cite{Pan13}. (For one
level it also follows from the marking and mapping theorems for Poisson processes: the points
$u_\alpha Z_\alpha$ form a Poisson process with mean measure
$\E Z^\theta\cdot\theta x^{-1-\theta}dx$, the image of the original one under multiplication by
$c_Z=(\E Z^\theta)^{1/\theta}$; so $\sum_\alpha u_\alpha Z_\alpha$ has the law of
$c_Z\sum_\alpha u_\alpha$, and \eqref{eq:A-pd} follows on subtracting the finite quantity
$\E\log\sum_\alpha u_\alpha$.)

\begin{lemma}[Approximation by finitely supported measures]\label{lem:A-approx}
Let $0<\theta<1$, and let $\mathcal D_\theta$ be the set of finitely supported
$\zeta\in\Pspace([0,1])$ with $\zeta(\{0\})\ge\theta$ and $\zeta([0,t])<1$ for every $t<1$.
Every $\zeta\in\Pspace([0,1])$ with $\zeta(\{0\})\ge\theta$ is the limit in $\Wone$ of a
sequence in $\mathcal D_\theta$. Consequently
$p(\theta)=\inf\{P_\zeta:\zeta\in\mathcal D_\theta\}$.
\end{lemma}

\begin{proof}
Let $\alpha(t)=\zeta([0,t])$. For an integer $n$ with $1-1/n\ge\theta$, define $\alpha_n(1)=1$
and $\alpha_n(t)=\min\{\alpha(k/n),1-1/n\}$ for $t\in[k/n,(k+1)/n)$, $0\le k\le n-1$. Then
$\alpha_n$ is nondecreasing and right-continuous, so it is the distribution function of a
measure $\zeta_n$ supported in $\{0,1/n,\dots,1\}$. Moreover
$\zeta_n(\{0\})=\min\{\alpha(0),1-1/n\}\ge\theta$ and $\alpha_n\le1-1/n$ on $[0,1)$, so
$\zeta_n\in\mathcal D_\theta$. For $t\in[k/n,(k+1)/n)$,
$0\le\alpha(t)-\alpha_n(t)\le\alpha(\frac{k+1}n)-\alpha(\frac kn)+\frac1n$,
and integration gives $\Wone(\zeta_n,\zeta)\le2/n$. Since $\mathcal D_\theta$ lies in the
constraint set $\{\zeta(\{0\})\ge\theta\}$, the last assertion follows from the Lipschitz bound
\eqref{eq:cons-lipschitz}.
\end{proof}

\begin{proof}[Proof of Lemma~\ref{lem:A-fm-sk}]
Fix $\theta\in(0,1)$. We use \cite{CM19} with one scale: in their notation $r=1$, $\zeta_0=\theta$,
$\gamma_0=0$ and $\gamma_1=\beta/\sqrt2$ \cite[(6), (8)]{CM19}.

\emph{The pressure is a fractional moment.} With one scale the auxiliary index $\alpha$ ranges
over $\N$, and two distinct indices meet at level $0$, where $\gamma_0=0$ \cite[(7),
(9)]{CM19}. The Hamiltonian \cite[(10)]{CM19} is therefore
$H_N(\sigma,\alpha)=N^{-1/2}\sum_{i,j=1}^Ng_{ij}(\alpha)\sigma_i\sigma_j$, where the families
$(g_{ij}(\alpha))_{i,j\le N}$, $\alpha\in\N$, are independent and consist of independent centered
Gaussian variables with variance $\gamma_1^2$. For each $\alpha$ the field
$\sigma\mapsto H_N(\sigma,\alpha)$ is centered Gaussian with covariance
$\gamma_1^2N^{-1}(\sigma\cdot\sigma')^2=N\xi(N^{-1}\sigma\cdot\sigma')$, exactly as in
\eqref{eq:intro-covariance}; the diagonal terms $i=j$ are included. The weights of
\cite[(15)]{CM19} are the Poisson--Dirichlet weights with parameter $\theta$, independent of the
Hamiltonian. Put $Z_\alpha=2^{-N}\sum_\sigma e^{H_N(\sigma,\alpha)}$; these are independent copies
of $\ZN$, and $\E\ZN^\theta\le(\E\ZN)^\theta<\infty$. By \eqref{eq:A-pd}, the pressure
\cite[(14)]{CM19} is
$p_N=N^{-1}\E\log\sum_{\alpha\ge1}\nu_\alpha\,2^NZ_\alpha=\log2+(N\theta)^{-1}\log\E\ZN^\theta$.

\emph{The trial functional is the Parisi functional.} A point of the trial set of
\cite[Definition~2.2]{CM19} consists of an integer $k\ge1$, numbers
$0=x_{-1}<x_0<x_1<\dots<x_k=1$, one of which, say $x_{m'}$ with $0\le m'<k$, equals $\theta$, and
numbers $0=q_0\le q_1\le\dots\le q_k=1$ \cite[(20), (21), (27)]{CM19}; we write $x_j$ for their
$\xi_j$. The coupling levels \cite[(25)--(26)]{CM19} are $\tilde\gamma_j=0$ for $j\le m'$ and
$\tilde\gamma_j=\gamma_1$ for $j>m'$. We read off three facts from the functional
\cite[(28)--(30)]{CM19}. First, the Gaussian increments with index $j\le m'$ vanish; the
increment with index $m'+1$ has variance $2\gamma_1^2q_{m'+1}=\xi'(q_{m'+1})$, and for $j>m'+1$ the
increment with index $j$ has variance $2\gamma_1^2(q_j-q_{j-1})=\xi'(q_j)-\xi'(q_{j-1})$.
Second, the recursion \cite[(29)]{CM19} integrates out the increment with index $j$ with the
exponent $x_{j-1}$, for $j=k,k-1,\dots,1$ (we read $J_p$ in \cite[(28)]{CM19} as $J_j$, and
the range of \cite[(29)]{CM19} as $1\le j\le k$); in particular the increment from $0$ to $q_{m'+1}$
carries the exponent $x_{m'}=\theta$. Third, the terminal function is $\log2+\log\cosh$, and the
deterministic term of \cite[(30)]{CM19} is
$\frac12\bigl\{\theta\gamma_1^2q_{m'+1}^2+\sum_{j=m'+1}^{k-1}x_j\gamma_1^2(q_{j+1}^2-q_j^2)\bigr\}$.

Let $\mu$ be the probability measure on $[0,1]$ whose distribution function equals $\theta$ on
$[0,q_{m'+1})$, equals $x_j$ on $[q_j,q_{j+1})$ for $m'<j<k$, and equals $1$ at $1$. Integrating
out a Gaussian variable of variance zero changes nothing, so the first two facts show that the
recursion of \cite[(28)--(29)]{CM19} is the recursion \eqref{eq:A-cole-hopf} for $\mu$, shifted
by $\log2$; hence its value is $\log2+\Phi_\mu(0,0)$. Since
$\gamma_1^2(b^2-a^2)=\int_a^bt\xi''(t)\,dt$, the deterministic term equals
$\frac12\int_0^1t\xi''(t)\mu([0,t])\,dt$. So the functional of \cite[(30)]{CM19} at this point
equals $\log2+P_\mu$. The measure $\mu$ is finitely supported, $\mu(\{0\})\ge\theta$, and
$\mu([0,t])\le x_{k-1}<1$ for $t<1$; that is, $\mu\in\mathcal D_\theta$.

Conversely, every $\mu\in\mathcal D_\theta$ arises in this way. Let $v_0<v_1<\dots<v_n$ be the
values of $\mu([0,t])$ for $t\in[0,1)$, so that $v_0=\mu(\{0\})\ge\theta$ and $v_n<1$, and let
$0=s_0<s_1<\dots<s_n<s_{n+1}=1$ be such that $\mu([0,t])=v_i$ for $t\in[s_i,s_{i+1})$. If
$v_0=\theta$, take $k=n+1$, $x_j=v_j$ for $j\le n$, and $q_j=s_j$ for $1\le j\le k$. If
$v_0>\theta$, take $k=n+2$, $x_0=\theta$, $x_j=v_{j-1}$ for $1\le j\le n+1$, $q_1=0$, and
$q_j=s_{j-1}$ for $2\le j\le k$. In both cases $m'=0$, and the construction of the previous
paragraph returns $\mu$.

\emph{Conclusion.} The infimum of the trial functional is therefore
$\log2+\inf\{P_\mu:\mu\in\mathcal D_\theta\}=\log2+p(\theta)$, by Lemma~\ref{lem:A-approx}.
By \cite[Theorem~2.1]{CM19}, $p_N$ converges to this infimum, so
$(N\theta)^{-1}\log\E\ZN^\theta\to p(\theta)$. Multiplying by $\theta$ gives \condFM.
\end{proof}

\subsection{Guerra's bound for fractional moments}\label{sec:A-guerra-bound}

Guerra's interpolation \cite{Gue03}, which Talagrand extended to fractional moments for convex
covariances \cite[Section~6]{Tal07}, gives the upper half of \condFM\ at every $N$. We include
the proof in our normalization; it shows where the constraint $\zeta(\{0\})\ge\theta$ and
convexity of $\xi$ at negative overlaps enter.

\begin{proposition}[Guerra's bound for fractional moments]\label{prop:A-guerra}
Let $\xi(t)=\sum_{p\ge2}\beta_p^2t^p$ have radius of convergence greater than one, and assume
that $\xi$ is convex on $[-1,1]$. Then for every $N\ge1$, every $\theta\in(0,1)$ and every
$\zeta\in\Pspace([0,1])$ with $\zeta(\{0\})\ge\theta$,
\begin{equation}\label{eq:A-guerra}
 \frac1{N\theta}\log\E\ZN^\theta\le P_\zeta .
\end{equation}
Consequently $N^{-1}\log\E\ZN^\theta\le\Lambda(\theta)$ for every $N\ge1$ and
$\theta\in(0,1)$.
\end{proposition}

As the coefficients of $\xi$ are nonnegative, $\xi$ is convex on $[0,1]$, so the hypothesis
concerns negative overlaps only; it holds for even $\xi$, in particular for the
Sherrington--Kirkpatrick covariance and the pure $p$-spin models with even $p$.

\paragraph{The interpolation.}
Fix $N\ge1$, $\theta\in(0,1)$ and a finitely supported $\zeta$ with $\zeta(\{0\})\ge\theta$;
choose points $q_0,\dots,q_L$ and exponents $y_1,\dots,y_L$ as in \eqref{eq:A-cole-hopf}, and put
$y_0=\theta$ and $y_{L+1}=1$. Then
\begin{equation}\label{eq:A-exponents}
 \theta=y_0\le y_1\le y_2\le\dots\le y_L\le y_{L+1}=1,
\end{equation}
and the first inequality is the constraint $y_1=\zeta(\{0\})\ge\theta$.

The randomness is arranged in $L+1$ independent levels. Level $0$ carries the field $H_N$ of
\eqref{eq:intro-covariance}. For $1\le j\le L$, level $j$ carries a vector
$z^j=(z^j_i)_{i\le N}$ of independent centered Gaussian variables with variance
$\xi'(q_j)-\xi'(q_{j-1})$. Write $\E_j$ for the expectation over level $j$ alone. For
$u\in[0,1]$ and $\sigma\in\{-1,1\}^N$ put
$H_u(\sigma)=\sqrt u\,H_N(\sigma)+\sqrt{1-u}\,\sum_{j=1}^Lz^j\cdot\sigma$,
let $Y_L=2^{-N}\sum_\sigma e^{H_u(\sigma)}$, define recursively
$Y_{j-1}=(\E_jY_j^{y_j})^{1/y_j}$ for $j=L,L-1,\dots,1$, and put
\begin{equation}\label{eq:A-interpolation}
 \varphi(u)=\frac1{N\theta}\log\E_0Y_0^\theta .
\end{equation}
Thus $Y_j$ depends on the levels $0,\dots,j$ only, and the physical disorder at level $0$ is
integrated with the exponent $y_0=\theta$.

At $u=1$ the field $H_u=H_N$ does not involve the levels $1,\dots,L$, so $Y_0=\ZN$ and
$\varphi(1)=(N\theta)^{-1}\log\E\ZN^\theta$. At $u=0$ we have $H_u(\sigma)=\sum_i\sigma_i\bar z_i$
with $\bar z_i=\sum_jz^j_i$, so $Y_L=\prod_i\cosh\bar z_i$. The coordinates are independent, and
induction on $j$ with \eqref{eq:A-cole-hopf} gives
$Y_j=\prod_i\exp\Phi_\zeta\bigl(q_j,\sum_{l\le j}z^l_i\bigr)$ for $j=L,\dots,0$. Hence
$Y_0=e^{N\Phi_\zeta(0,0)}$ is deterministic, and $\varphi(0)=\Phi_\zeta(0,0)$. By
\eqref{eq:A-penalty}, inequality \eqref{eq:A-guerra} for this $\zeta$ reads
\begin{equation}\label{eq:A-reduced}
 \varphi(1)-\varphi(0)\le-\frac12\sum_{j=1}^Ly_j\bigl\{\psi(q_j)-\psi(q_{j-1})\bigr\}.
\end{equation}
It follows from a formula for $\varphi'$, which needs no convexity.

\begin{lemma}[Derivative of the interpolation]\label{lem:A-derivative}
In the setting above, $\varphi$ is continuous on $[0,1]$ and differentiable on $(0,1)$. For
every $u\in(0,1)$ there are probability measures $\mu_{u,0},\dots,\mu_{u,L}$ on
$\{-1,1\}^N\times\{-1,1\}^N$ such that
\begin{equation}\label{eq:A-derivative}
 \varphi'(u)=-\frac12\sum_{j=1}^Ly_j\bigl\{\psi(q_j)-\psi(q_{j-1})\bigr\}
 -\frac12\sum_{l=0}^L(y_{l+1}-y_l)\sum_{\sigma,\tau}\mu_{u,l}(\sigma,\tau)\,
 \Delta_l\bigl(R(\sigma,\tau)\bigr),
\end{equation}
where $R(\sigma,\tau)=N^{-1}\sigma\cdot\tau$ and
$\Delta_l(r)=\xi(r)-\xi(q_l)-\xi'(q_l)(r-q_l)$.
\end{lemma}

Under $\mu_{u,l}$ the two configurations $\sigma,\tau$ are drawn with the same randomness at the
levels $0,\dots,l$; the measures are defined in \eqref{eq:A-pair-measures} below.

\begin{proof}[Proof of Proposition~\ref{prop:A-guerra}]
Let first $\zeta$ be finitely supported. By \eqref{eq:A-exponents} every coefficient
$y_{l+1}-y_l$ in \eqref{eq:A-derivative} is nonnegative. Each $\Delta_l(r)$ is the gap between
$\xi$ and its tangent line at $q_l\in[0,1]$, so it is nonnegative for $r\in[-1,1]$ because $\xi$
is convex there; and $R(\sigma,\tau)\in[-1,1]$. Hence
$\varphi'(u)\le-\frac12\sum_jy_j\{\psi(q_j)-\psi(q_{j-1})\}$ for $u\in(0,1)$, and the mean value
theorem gives \eqref{eq:A-reduced}, that is, \eqref{eq:A-guerra}. For a general $\zeta$ with
$\zeta(\{0\})\ge\theta$, Lemma~\ref{lem:A-approx} gives finitely supported $\zeta_n$ with
$\zeta_n(\{0\})\ge\theta$ and $\Wone(\zeta_n,\zeta)\to0$; apply \eqref{eq:A-guerra} to $\zeta_n$
and let $n\to\infty$, using \eqref{eq:cons-lipschitz}. Taking the infimum over $\zeta$ gives
$(N\theta)^{-1}\log\E\ZN^\theta\le p(\theta)$, and multiplying by $\theta$ gives the last
assertion.
\end{proof}

The term with $l=0$ in \eqref{eq:A-derivative} comes from pairs of configurations that share
only the physical disorder, and its coefficient is $y_1-y_0=\zeta(\{0\})-\theta$: the
constraint $\zeta(\{0\})\ge\theta$ in the definition of $\Lambda$ is exactly what makes it
nonnegative. Convexity is used only through $\Delta_l(R)\ge0$ at negative overlaps $R$.

\begin{proof}[Proof of Lemma~\ref{lem:A-derivative}]
Write $H_u=\sum_{\ell=0}^LX_\ell$ with $X_0=\sqrt u\,H_N$ and
$X_\ell(\sigma)=\sqrt{1-u}\,z^\ell\cdot\sigma$ for $\ell\ge1$; the process $X_\ell$ is a function
of level $\ell$, and these processes are independent.

\emph{Regularity.} Put $M=\max_\sigma|H_N(\sigma)|+\sum_{j,i}|z^j_i|$, which has exponential
moments of every order. Then $|H_u(\sigma)|\le M$ for all $u$ and $\sigma$, so
$e^{-M}\le Y_L\le e^M$. A power mean with exponent in $(0,1]$ lies between the geometric and
the arithmetic mean, so by induction
$\exp(-\E_{j+1}\cdots\E_LM)\le Y_j\le\E_{j+1}\cdots\E_Le^M$ for every $j$ and $u$. Hence every
quantity below, and its derivative in $u$ on a compact subset of $(0,1)$, is bounded by
$Ce^{CM}$ for a constant $C$. Dominated convergence therefore justifies the differentiations
under the expectations and the Gaussian integrations by parts below, and it gives continuity of
$\varphi$ on $[0,1]$.

\emph{Weights.} For $0\le j\le L$ put $V_j=Y_j^{y_j}/\E_jY_j^{y_j}$; it depends on the levels
$0,\dots,j$, and $\E_jV_j=1$. Let $\pi(\sigma)=2^{-N}e^{H_u(\sigma)}/Y_L$ be the Gibbs weights.
For a function $F$ of all levels put $\mathcal A_LF=F$ and
$\mathcal A_{j-1}F=\E_j[V_j\,\mathcal A_jF]$ for $j=L,\dots,1$; thus $\mathcal A_jF$ depends on
the levels $0,\dots,j$, and $\mathcal A_j1=1$. Since the levels are independent, with
$V=V_0V_1\cdots V_L$,
\begin{equation}\label{eq:A-tower}
 \E[VF]=\E\bigl[V_0\cdots V_j\,\mathcal A_jF\bigr]\qquad(0\le j\le L).
\end{equation}
For $0\le l\le L$ define
\begin{equation}\label{eq:A-pair-measures}
 \mu_{u,l}(\sigma,\tau)=\E\bigl[V_0\cdots V_l\;\mathcal A_l\pi(\sigma)\;\mathcal A_l\pi(\tau)\bigr].
\end{equation}
These numbers are nonnegative and sum to $\E[V_0\cdots V_l]=1$.

\emph{The derivative.} Differentiating the recursion gives
\begin{gather*}
 \partial_u\log Y_L=\sum_\sigma\pi(\sigma)\,\partial_uH_u(\sigma),\qquad
 \partial_u\log Y_{j-1}=\E_j\bigl[V_j\,\partial_u\log Y_j\bigr]\quad(1\le j\le L),\\
 \partial_u\log\E_0Y_0^\theta=\theta\,\E_0\bigl[V_0\,\partial_u\log Y_0\bigr].
\end{gather*}
By \eqref{eq:A-tower} with $j=0$,
\[
 N\varphi'(u)=\sum_{\ell=0}^L\E\Bigl[V\sum_\sigma\pi(\sigma)G_\ell(\sigma)\Bigr],\qquad
 G_0=\frac{H_N}{2\sqrt u},\quad G_\ell(\sigma)=-\frac{z^\ell\cdot\sigma}{2\sqrt{1-u}}\ (\ell\ge1).
\]

\emph{Integration by parts at level $\ell$.} Fix $\ell$ and condition on the other levels. The
pair $(G_\ell,X_\ell)$ is a centered Gaussian process, and the integrand depends on level $\ell$
only through the values $X_\ell(\tau)$, $\tau\in\{-1,1\}^N$. Gaussian integration by parts gives
\[
 \E\Bigl[V\sum_\sigma\pi(\sigma)G_\ell(\sigma)\Bigr]
 =\sum_{\sigma,\tau}c_\ell(\sigma,\tau)\,
 \E\Bigl[\frac{\partial(V\pi(\sigma))}{\partial X_\ell(\tau)}\Bigr],
 \qquad c_\ell(\sigma,\tau)=\E\,G_\ell(\sigma)X_\ell(\tau).
\]
To compute the derivative, note that $Y_j$ does not depend on level $\ell$ when $j<\ell$. For
$j\ge\ell$, induction from $j=L$, where $\partial\log Y_L/\partial X_\ell(\tau)=\pi(\tau)$, gives
$\partial\log Y_j/\partial X_\ell(\tau)=\mathcal A_j\pi(\tau)$, because $\E_j$ does not
integrate level $\ell$ when $j>\ell$. Consequently $\partial\log V_j/\partial X_\ell(\tau)$ is
$0$ for $j<\ell$; it is $y_\ell\,\mathcal A_\ell\pi(\tau)$ for $j=\ell$, since the normalization
$\E_\ell Y_\ell^{y_\ell}$ does not depend on level $\ell$; and it is
$y_j\{\mathcal A_j\pi(\tau)-\mathcal A_{j-1}\pi(\tau)\}$ for $j>\ell$. Together with
$\partial\log\pi(\sigma)/\partial X_\ell(\tau)=\ind_{\sigma=\tau}-\pi(\tau)$, and after
regrouping with $y_{L+1}=1$ and $\mathcal A_L\pi=\pi$,
\[
 \frac{\partial(V\pi(\sigma))}{\partial X_\ell(\tau)}
 =V\pi(\sigma)\Bigl[\ind_{\sigma=\tau}-\sum_{l=\ell}^L(y_{l+1}-y_l)\,\mathcal A_l\pi(\tau)\Bigr].
\]
Since $\mathcal A_l\pi(\tau)$ depends on the levels $0,\dots,l$ only, \eqref{eq:A-tower} gives
$\E[V\pi(\sigma)\,\mathcal A_l\pi(\tau)]=\mu_{u,l}(\sigma,\tau)$. The diagonal value
$\bar c_\ell=c_\ell(\sigma,\sigma)$ does not depend on $\sigma$ (see below), and
$\sum_\sigma\E[V\pi(\sigma)]=1$. Hence
\begin{equation}\label{eq:A-ibp}
 \E\Bigl[V\sum_\sigma\pi(\sigma)G_\ell(\sigma)\Bigr]
 =\bar c_\ell-\sum_{l=\ell}^L(y_{l+1}-y_l)\sum_{\sigma,\tau}\mu_{u,l}(\sigma,\tau)\,
 c_\ell(\sigma,\tau).
\end{equation}

\emph{Covariances.} With $R=R(\sigma,\tau)$,
\[
 c_0(\sigma,\tau)=\tfrac12\E H_N(\sigma)H_N(\tau)=\tfrac N2\,\xi(R),\qquad
 c_\ell(\sigma,\tau)=-\tfrac N2\bigl\{\xi'(q_\ell)-\xi'(q_{\ell-1})\bigr\}R\quad(\ell\ge1),
\]
and $R(\sigma,\sigma)=1$; so $\bar c_0=\frac N2\xi(1)$, and
$\bar c_\ell=-\frac N2\{\xi'(q_\ell)-\xi'(q_{\ell-1})\}$ for $\ell\ge1$. Write
$\mu_{u,l}[F]=\sum_{\sigma,\tau}\mu_{u,l}(\sigma,\tau)\,F(R(\sigma,\tau))$. Summing
\eqref{eq:A-ibp} over $\ell$, the diagonal values add up to
$\frac N2\{\xi(1)-\xi'(1)\}=-\frac N2\psi(1)$, because $\xi'(q_0)=0$. In the remaining terms
with $\ell\ge1$ we exchange the sums over $\ell$ and $l$; since
$\sum_{\ell=1}^l\{\xi'(q_\ell)-\xi'(q_{\ell-1})\}=\xi'(q_l)$, which vanishes for $l=0$,
\[
 \varphi'(u)=-\frac12\psi(1)
 -\frac12\sum_{l=0}^L(y_{l+1}-y_l)\,\mu_{u,l}\bigl[\xi(R)-\xi'(q_l)R\bigr].
\]
Finally, $\xi(r)-\xi'(q_l)r=\Delta_l(r)-\psi(q_l)$, and summation by parts with $\psi(q_0)=0$,
$q_L=1$ and $y_{L+1}=1$ gives
$\sum_{l=0}^L(y_{l+1}-y_l)\psi(q_l)=\psi(1)-\sum_{j=1}^Ly_j\{\psi(q_j)-\psi(q_{j-1})\}$. Since
each $\mu_{u,l}$ is a probability measure, \eqref{eq:A-derivative} follows.
\end{proof}

\subsection{Normalizations}\label{sec:A-normalizations}

\begin{remark}\label{rem:A-normalizations}
\begin{steplist}
\item \emph{Terminal value.} Our terminal value $\log\cosh x$ in \eqref{eq:intro-pde} and our
 partition function $\ZN=2^{-N}\sum_\sigma e^{H_N(\sigma)}$ are those of \cite[(1.2)--(1.4)]{Bou26b} and, with the
 uniform reference measure on spins, of \cite[(2.2), (2.4)]{CGKLM}. With $\log(2\cosh x)$, as in
 \cite{CPS23}, \cite[(1.4)--(1.5)]{Bou26a}, \cite[(1.2)]{AL26} and \cite[(28)]{CM19}, or with
 $\log2$ added to the functional, as in \cite[(1.3)]{Lop26a}, the Parisi functional, $\feq$ and
 $p(\theta)$ exceed ours by $\log2$. This matches the passage from $\ZN$ to
 $\sum_\sigma e^{H_N(\sigma)}=2^N\ZN$: $\Lambda(\theta)$ becomes $\Lambda(\theta)+\theta\log2$,
 and $\Sigma$ is unchanged once $f$ is shifted by $\log2$. Thus the constrained minimum of
 Aronow and Lopatto is $P_\beta(s)=p(s)+\log2$, with $s$ in the role of $\theta$
 \cite[(1.2)]{AL26}.
\item \emph{The Sherrington--Kirkpatrick Hamiltonian.} For $\xi(t)=\beta^2t^2/2$, a field with
 covariance exactly $N\xi(N^{-1}\sigma\cdot\sigma')$ is obtained by summing over all ordered
 pairs, the diagonal included, for instance $\beta(2N)^{-1/2}\sum_{i,j=1}^Ng_{ij}\sigma_i\sigma_j$ with independent standard
 Gaussian $g_{ij}$ \cite[(1.1)]{Lop26a}, or the
 one-scale Hamiltonian of \cite[(10)]{CM19} with $\gamma_1=\beta/\sqrt2$. The original model
 $\FNb$ sums over $i<j$ only; its completion $\FNb+\beta D_N$ \cite[Section~1.1]{AL26}, with $D_N$
 centered Gaussian of variance $\frac12$ and independent of $\FNb$ (as in the proof of
 Lemma~\ref{lem:cons-original-sk}), has the law of $\log\ZN+N\log2$. Hence, for $0<\theta<1$,
 \begin{equation}\label{eq:A-original-moments}
  \E e^{\theta\FNb}=e^{-\theta^2\beta^2/4}\,2^{\theta N}\,\E\ZN^\theta ,
 \end{equation}
 and Lemma~\ref{lem:A-fm-sk} holds for $\FNb$ with $\Lambda(\theta)$ replaced by
 $\Lambda(\theta)+\theta\log2$.
\end{steplist}
\end{remark}

\section{Dictionary to the replica predictions}\label{app:dictionary}

Parisi and Rizzo \cite{PR09} state their predictions for the Sherrington--Kirkpatrick model in
the language of replicas and in physical units, in which the free energy of a sample is
$-T\log Z$. This appendix translates them into our variables and records which statement proves
each of them, and in what scope; equation and section numbers of \cite{PR09} refer to its arXiv
version~2. Their replica free energy is $\Phi(n)=-(\beta nN)^{-1}\log\E Z^n$, where $Z$ sums over
all $2^N$ configurations without the factor $2^{-N}$, and they measure deviations $\Delta f<0$ of
the free energy per spin. With their coefficient $\cPR<0$ of $n^5$ in $\Phi(n)$
\cite[(33), (B3)]{PR09}, the dictionary is
\begin{equation}\label{eq:intro-pr-dictionary}
 n=\theta,\qquad p(\theta)-\feq=-\beta\bigl\{\Phi(n)-\Phi(0)\bigr\},\qquad
 \delta=\beta|\Delta f|,\qquad \cb=-\beta\,\cPR .
\end{equation}
In detail, with $\xi(t)=\beta^2t^2/2$ and $T=1/\beta$: the replica number $n\in(0,1)$ is the
constraint $\theta$ on the atom at zero, and their maximization over Parisi's order parameter
$q(x)$ is our constrained minimization \eqref{eq:B-functional}; $\Phi(n)$ tends to
$-\beta^{-1}\{\log2+p(n)\}$ as $N\to\infty$ \eqref{eq:B-replica}; $q(x)$ is the quantile function
of the measure, so that $\qdot=1/\rhostar$; their solution $f(x,y)$ of the Parisi equation in the
variable $x$ is $\beta^{-1}\Phi_\zeta(q(x),\beta y)$ up to an additive constant \eqref{eq:B-f};
their field, written $h$ in \cite{PR09} and $H$ here, is our $h/\beta$ \eqref{eq:B-field}. Their replica free energy always carries its
argument, $\Phi(n)$, and the solution $\Phi_\zeta$ of the Parisi equation its measure.

\paragraph{What is proved.}
Table~\ref{tab:B-summary} lists the predictions of \cite{PR09} on the bottom of the Parisi
measure. The first five hold for the Sherrington--Kirkpatrick model at every $\beta>1$. They use
\condH, which holds by Lopatto's theorem \cite[Proposition~5.1]{Lop26a} (an alternative proof is
\cite[Theorem~5.1]{Pei26}; both are preprints, and for $1<\beta^2<1+\varepsilon_*$, with $\varepsilon_*$ as in Theorem~\ref{thm:nearcrit}(a), refereed inputs and this paper suffice, Remark~\ref{rem:fifth-sk-inputs}); the probabilistic form of the six-fifths law also uses \condFM, which
holds by Lemma~\ref{lem:A-fm-sk}. ``Leading term'' means the first term of the expansion as the
small parameter tends to zero, with a relative error $o(1)$ and no rate; the small parameter is
taken at fixed $\beta$, except for the de Almeida--Thouless line, whose small parameter is
$\beta^2-1$ and whose error is $O((\beta^2-1)^2)$.
``Near criticality'' means: proved for $1<\beta^2<1+\varepsilon_*$, with $\varepsilon_*$ as in
Theorem~\ref{thm:nearcrit}(b), for the constrained Parisi problem; Parisi and Rizzo state these
predictions at every temperature below the critical one.

\begin{table}[tp]
\centering
\small
\renewcommand{\arraystretch}{1.15}
\newcommand{\dictcell}{\raggedright\hspace{0pt}}
\begin{tabular}{@{}p{0.30\textwidth}p{0.31\textwidth}p{0.33\textwidth}@{}}
\toprule
\dictcell Prediction in \cite{PR09} & \dictcell In our notation & \dictcell Status here\tabularnewline
\midrule
\dictcell plateau height and end, (B2) and Section~IV
 & \dictcell $q_\theta\sim\frac{3\theta}{2\rhob}$,\ \ $\zeta_\theta([0,q_\theta])\sim\frac{3\theta}2$
 & \dictcell leading terms; Theorem~\ref{thm:sk}(a)\tabularnewline
\dictcell fifth-order law, (B3) and (33)
 & \dictcell $p(\theta)-\feq\sim\cb\,\theta^5$
 & \dictcell leading term; Theorem~\ref{thm:sk}(a)\tabularnewline
\dictcell six-fifths law, (4) and Section~IV
 & \dictcell $\Sigma(\feq+\delta)\sim-\frac56(6\cb)^{-1/5}\delta^{6/5}$
 & \dictcell leading term; Theorem~\ref{thm:sk}(b),(c)\tabularnewline
\dictcell identities (B9)--(B10)
 & \dictcell $\Ustar=1/\beta$,\ \ $\Wstar<0$,\ \ $\rhob=\beta^5\Wstar^2/2$
 & \dictcell exact; Proposition~\ref{prop:ons-equilibrium}, Lemma~\ref{lem:plat-fourth-sign},
  Theorem~\ref{thm:sk} (also \cite[Lemma~6.1]{Lop26a}, \cite[(B.4)]{Lop26b}, \cite[Section~7]{AL26})\tabularnewline
\dictcell small field, (B4)--(B5)
 & \dictcell asymptotics of $\qmin$ and $\Ffield$
 & \dictcell leading terms; Theorem~\ref{thm:field}, Corollary~\ref{cor:dict-field}\tabularnewline
\dictcell Kondor's coefficient, Section~IV
 & \dictcell $\cb\to9/5120$ as $\beta\downarrow1$
 & \dictcell as a limit only; Theorem~\ref{thm:nearcrit}(a)\tabularnewline
\dictcell de Almeida--Thouless line, (22)--(23)
 & \dictcell the threshold $\theta_c$ of the marginal regime
 & \dictcell near criticality: leading term; Theorem~\ref{thm:nearcrit}(b)\tabularnewline
\dictcell $q(x)=q_{\rm free}(x)+O(n^5)$ for $x>x_c$, Section~IV
 & \dictcell $\sup_{x>x_\theta}|Q_\theta(x)-Q_*(x)|=O(\theta^5)$
 & \dictcell near criticality; Theorem~\ref{thm:nearcrit}(c)\tabularnewline
\dictcell plateau expansion (B1); $\ddot q(0)=0$, (B8)
 & \dictcell ---
 & \dictcell neither proved nor used\tabularnewline
\bottomrule
\end{tabular}
\caption{The predictions of Parisi and Rizzo \cite{PR09} for the bottom of the Parisi measure,
their form in the notation of this paper, and their status here. Equation and section numbers
in the first column are those of \cite{PR09}.}
\label{tab:B-summary}
\end{table}

\subsection{Replica free energy and fractional moments}\label{sec:B-normalization}

For a sample $J$ with partition function $Z_J$ and free energy per spin
$f_J=-(\beta N)^{-1}\log Z_J$, $\Phi(n)=-(\beta nN)^{-1}\log\E Z_J^n$ \cite[(1)]{PR09}, with the
limit $N\to\infty$ understood. That $Z_J$ carries no factor $2^{-N}$ is visible in the terminal
value $\beta^{-1}\log2\cosh\beta y$ of \cite[(27)]{PR09} and in the high-temperature value
$f_{\rm RS}=-\beta/4-\log2/\beta=-\beta^{-1}\{\log2+P_{\delta_0}\}$ \cite[Section~III]{PR09},
since $P_{\delta_0}=\xi(1)/2=\beta^2/4$. Take $Z_J=e^{\FNb}$, the original model of
Theorem~\ref{thm:sk}(c). In the coupling of Lemma~\ref{lem:cons-original-sk},
$\log\ZN=\FNb-N\log2+\beta D_N$ with $D_N\sim\mathcal N(0,\frac12)$ independent of $\FNb$, so
$\E e^{n\FNb}=e^{-n^2\beta^2/4}\,2^{nN}\,\E\ZN^n$, and Lemma~\ref{lem:A-fm-sk} gives, for $0<n<1$,
\begin{equation}\label{eq:B-replica}
 \lim_{N\to\infty}\Phi(n)=\lim_{N\to\infty}\Bigl(-\frac1{\beta nN}\log\E Z_J^n\Bigr)
 =-\frac1\beta\bigl\{\log2+p(n)\bigr\}.
\end{equation}
Adding diagonal terms to the Hamiltonian multiplies $Z_J$ by $e^G$, with $G$ Gaussian of bounded
variance, independent of the configuration and of the other couplings; this does not change the
limit. As $p$ is continuous at
zero with $p(0)=\feq$ (Lemma~\ref{lem:cons-variational}(a)), \eqref{eq:B-replica} tends to
$-\beta^{-1}\{\log2+\feq\}$ as $n\downarrow0$: their typical free energy $f_{\rm typ}=\Phi(0)$
\cite[Section~I]{PR09}, the limit of $\E f_J$ by the Parisi formula \cite{Tal06}. This is the
second identity of \eqref{eq:intro-pr-dictionary}. Likewise $f_J\le f_{\rm typ}-|\Delta f|$
exactly when $N^{-1}\FNb-\log2\ge\feq+\beta|\Delta f|$: a deviation $\Delta f<0$ of the free
energy per spin is the deviation $\delta=\beta|\Delta f|$ of $N^{-1}\log\ZN$ above $\feq$, up to
the term $\beta D_N/N$, which Lemma~\ref{lem:cons-original-sk} controls. Positive replica numbers
thus describe free energies per spin below the typical one \cite[Section~I]{PR09}, that is, the
upper deviations studied here.

Parisi and Rizzo relate the sample complexity $L(\varphi)=\lim_NN^{-1}\log P_N(\varphi)$, where
$P_N$ is the density of $f_J$ \cite[(5)]{PR09}, to $\Phi$ by a Legendre transform:
$L(\varphi)=\beta n\varphi-\beta n\Phi(n)$ at the $n$ with $\varphi=\partial_n\{n\Phi(n)\}$
\cite[(8)--(9)]{PR09}. Put $f=-\beta\varphi-\log2$. By \eqref{eq:B-replica},
$\Lambda(\theta)=\theta p(\theta)=-\beta\theta\Phi(\theta)-\theta\log2$, so
$\Lambda(\theta)-\theta f=\beta\theta\{\varphi-\Phi(\theta)\}$, and
$\varphi=\partial_n\{n\Phi(n)\}$ at $n=\theta$ is $\Lambda'(\theta)=f$. Their (8)--(9) thus state
$L(\varphi)=\Lambda(\theta_f)-\theta_ff=\Sigma(f)$ with $\Lambda'(\theta_f)=f$, which is
Corollary~\ref{cor:cons-level}; Lemma~\ref{lem:cons-gartner-ellis} proves the corresponding
statement, under \condFM, for the upper tail $\{N^{-1}\log\ZN\ge f\}$ instead of a density.

\subsection{The order parameter and the Parisi equation}\label{sec:B-order}

\paragraph{The order parameter.}
At $n=0$ it is the quantile function $q(x)=Q_*(x)=\inf\{t:\alphas(t)\ge x\}$, $x\in(0,1]$, of the
Parisi measure. Under the hypotheses of Theorem~\ref{thm:local}, $\alphas$ is $C^1$ near zero with
$\alphas'(0)=\rhostar>0$ (Proposition~\ref{prop:ons-equilibrium}), so $\qdot=Q_*'(0)=1/\rhostar$,
and $\qdot=1/\rhob$ for the Sherrington--Kirkpatrick model. For $0<n<1$, Kondor's approach, which
Parisi and Rizzo use below the de Almeida--Thouless line \cite[Section~IV]{PR09}, takes a
nondecreasing $q(x)$ on $n\le x\le1$; extended by zero to $(0,n)$, it is the quantile function of
a measure $\zeta$ with $\zeta(\{0\})\ge n$, and every such measure arises in this way. An interval
on which $q(x)$ is constant is an atom of $\zeta$, and a plateau of $\alpha_\zeta$ in our sense is
a jump of $q(x)$.

\paragraph{The functional.}
Parisi and Rizzo write $\Phi(n)=\max_qF_n[q]$ for an explicit functional $F_n$ that also contains
an external field, written $h$ there and $H$ here \cite[Section~IV]{PR09}. Let $\zeta$ be the measure whose quantile function
is $q$, extended by zero, and recall
$P^h_\zeta=\Phi_\zeta(0,h)-\frac12\int_0^1t\xi''(t)\zeta([0,t])\,dt$ from
Theorem~\ref{thm:field}. Then
\begin{equation}\label{eq:B-functional}
 -\beta F_n[q]-\log2=P^{\beta H}_\zeta .
\end{equation}
Indeed, $F_n[q]$ is the sum of $-\frac\beta4\{1-2q(1)+\int_n^1q(x)^2dx\}$ and
$-(\beta n)^{-1}\log\E e^{\beta nf(n,H+\sqrt{q(n)}Z)}$, with $Z$ standard Gaussian. By
\eqref{eq:B-f} at $x=n$ and the Cole--Hopf formula on the first plateau $[0,q(n))$, on which
$\zeta([0,t])=n$, the second term is
$-\beta^{-1}\{\Phi_\zeta(0,\beta H)+\log2-\frac{\beta^2}2(1-q(1))\}$; since $q=0$ on $(0,n)$ and
$\int_0^1t\,\zeta([0,t])\,dt=\frac12\{1-\int_0^1q(x)^2dx\}$, the sum is
$-\beta^{-1}\{P^{\beta H}_\zeta+\log2\}$. So, for $H=0$, maximizing $F_n$ over order parameters on
$[n,1]$ is minimizing $P_\zeta$ under $\zeta(\{0\})\ge n$, which defines $p(n)$; the factor
$-\beta$ turns the maximum into a minimum. For constant $q(x)\equiv q$, $F_n$ is their
replica-symmetric functional \cite[Section~III]{PR09} and $\zeta$ is the two-atom measure
$n\delta_0+(1-n)\delta_q$ of Section~\ref{sec:nc-dat}.

\paragraph{The function $f(x,y)$.}
It solves $\partial_xf=-\frac12\dot q(x)\{\partial_y^2f+\beta x(\partial_yf)^2\}$ with
$f(1,y)=\beta^{-1}\log2\cosh\beta y$ \cite[(26)--(27)]{PR09}. Where $q$ is strictly increasing,
$\zeta([0,q(x)])=x$, and as $\xi''=\beta^2$ this is the Parisi equation \eqref{eq:intro-pde} of
$\zeta$ in the variables $x$ and $y=b/\beta$, $b$ being the space variable of $\Phi_\zeta$; where
$q$ is constant neither side moves, and across a jump of $q$ the equation is read at fixed $x$, as
the Parisi equation over an interval that $\zeta$ does not charge. At the top $q(1)$ of the
support, $\Phi_\zeta(q(1),b)=\log\cosh b+\frac12\{\xi'(1)-\xi'(q(1))\}$; comparing terminal
values,
\begin{equation}\label{eq:B-f}
 \beta f(x,y)=\Phi_\zeta\bigl(q(x),\beta y\bigr)+\log2-\frac{\beta^2}2\bigl(1-q(1)\bigr).
\end{equation}
So $\partial_yf(x,y)=\partial_b\Phi_\zeta(q(x),\beta y)$ is their $m(x,y)$ \cite[(29)]{PR09}, and
their Lagrange multiplier $P(x,y)$ \cite[(30)--(31)]{PR09} is the density of $X_{q(x)}/\beta$ for
the diffusion \eqref{eq:intro-diffusion} of $\zeta$ started at $\beta H$: (30) is its
Fokker--Planck equation in the time $t=q(x)$, and (31) is the plateau law \eqref{eq:plat-law} at
the time $q(n)$, with the Gaussian centred at the starting point. Their stationarity condition
$q(x)=\int P(x,y)\,m(x,y)^2\,dy$ \cite[(28)]{PR09} therefore reads
$\Gamma_{\zeta,\beta H}(q(x))=q(x)$, with $\Gamma_{\zeta,h}$ as in \eqref{eq:fld-obstacle}; at
$H=0$ this is the contact equation of minimizers on their support
(Lemma~\ref{lem:cons-standard}(f),(g) and Lemma~\ref{lem:cons-variational}(e)), and
Lemma~\ref{lem:fld-diffusion}(c) gives it for $\zeta_h$. The function $\Gamma_{\zeta,\beta H}$ is
the right side of (28), not the order parameter, and the two agree on the support. In the
replica-symmetric case this contact equation
is their (10), and its marginal case their de Almeida--Thouless condition (22), see
\eqref{eq:B-dat}.

\subsection{The predictions}\label{sec:B-predictions}

We follow Table~\ref{tab:B-summary}, with $\xi(t)=\beta^2t^2/2$ and $\beta>1$, so that
$\azero=\beta^2$ and, as $\xi''''=0$, $\Dstar=\beta^6\Wstar^2$.

\paragraph{The first plateau and the fifth-order law.}
The constrained order parameter has a plateau at height $q_0=\frac32n\qdot+o(n)$
\cite[Section~IV and Appendix~B, (B2)]{PR09}, ending at $x_c=\frac32n+O(n^2)$
\cite[Section~IV]{PR09}. As $q_0$ is the first positive contact $q_\theta$ and
$x_c=\zeta_\theta([0,q_\theta])=\theta+\zeta_\theta(\{q_\theta\})$, Theorem~\ref{thm:sk}(a) gives
\begin{equation}\label{eq:B-plateau}
 q_\theta=\frac32\,\theta\,\qdot\,\bigl(1+o(1)\bigr),\qquad
 \zeta_\theta([0,q_\theta])=\frac32\,\theta\,\bigl(1+o(1)\bigr)\qquad(\theta\downarrow0),
\end{equation}
and Theorem~\ref{thm:local}(a) gives the same for every covariance it covers, with $\qdot=1/\rhostar$; Theorem~\ref{thm:nearcrit}(c) gives relative errors $O(\theta)$ on $\Kcl(\varepsilon_*)$.
By \eqref{eq:intro-pr-dictionary}, $\beta T=1$ and $\qdot=1/\rhob$, the expansion
\cite[(B3)]{PR09} with the exact coefficient \cite[(33)]{PR09},
\begin{equation}\label{eq:B-fifth}
 \Phi(n)=\Phi(0)+\cPR\,n^5+O(n^7),\qquad
 \cPR=-\frac9{640}\,\beta^5\bigl(T\qdot\bigr)^3 ,
\end{equation}
predicts $p(\theta)-\feq=-\beta\cPR\,\theta^5+O(\theta^7)$ with
$-\beta\cPR=\frac9{640}\beta^3\rhob^{-3}=\cb$; we prove neither the order of the remainder nor the
absence of a term $n^6$ \cite[Section~IV]{PR09}. Under the hypotheses of Theorem~\ref{thm:local},
its part~(b) gives $c=\frac9{640}\azero^{3/2}\rhostar^{-3}$, the same expression with $\azero$ in
place of $\beta^2$. With a cubic term (Theorem~\ref{thm:atom}(c)), Parisi's $x(q)=\zetas([0,q])$ has
$x(0)=m_0$ and $\dot q(x(0)+)=1/\rhostar$ (Lemma~\ref{lem:atom-odd-slope}); with
$\theta-x(0)$ in the role of $n$, Proposition~\ref{prop:atom-odd-onset} and that lemma give
$q_\theta\sim\frac32(\theta-x(0))\,\dot q(x(0)+)$ and
$p(\theta)-\feq\sim\frac9{640}\azero^{3/2}\dot q(x(0)+)^3(\theta-x(0))^5$, the form of (B2)--(B3), derived in
\cite{PR09} for the Sherrington--Kirkpatrick model ($x(0)=0$), with $\dot q(x(0)+)$ in place of
$\qdot$. In terms of the plateau height,
$p(\theta)-\feq\sim\frac{\Dstar^2\sqrt\azero}{2160}\,q_\theta^5$, the fifth-order effect of a small
plateau anticipated in \cite[Appendix~B]{PR09}.

\begin{remark}[Kondor's coefficient is a limit]\label{rem:B-kondor}
Near the critical temperature, Kondor \cite{Kon83} found that $\Phi(n)$ starts at order $n^5$
\cite[(3)]{PR09}, with $\Phi(n)=\Phi(0)-\frac9{5120}n^5$ for the truncated model, as reported in
\cite[Section~IV]{PR09}; in \cite[Figure~3]{PR09}, $\cPR\beta^{-5}=-\frac9{5120}$ at $T=1$,
consistent with (33) and the leading term $\qdot=\frac12$ of their series (B6) at $T=1$. Here it
is the limit $\cb\to\frac9{5120}$, equivalently $\rhob\to2$, as $\beta\downarrow1$
(Theorem~\ref{thm:nearcrit}(a); Proposition~\ref{prop:nc-kondor}: $\Wstar\to-2$, $\Dstar\to4$),
and only a limit: at $\beta=1$, $\zetas=\delta_0$ (Section~\ref{sec:nc-reference}), \condH\ fails, the constraint is inactive
($\zeta_\theta=\delta_0$, $p(\theta)=\feq$ for every $\theta$) and the coefficient of $\theta^5$ is
zero; correspondingly, Parisi and Rizzo find that on the line $T=1$ the replica free energy is
$f_{\rm RS}$, the value of the solution $q=0$, for $n<2$ \cite[Section~III]{PR09}. The $o(1)$
terms of Theorem~\ref{thm:sk} are not uniform in $\beta$, nor those of Theorem~\ref{thm:local} in $\xi$; on
$\Kcl(\varepsilon_*)$ the marginal regime ends at $\theta_c$ of order $\azero-1$
(Theorem~\ref{thm:nearcrit}(b), Remark~\ref{rem:nc-dat}).
\end{remark}

\paragraph{The six-fifths law.}
For a deviation $\Delta f=\varphi-f_{\rm typ}<0$ fixed as $N\to\infty$ \cite[Section~I]{PR09},
the order of limits of Theorem~\ref{thm:sk}(c), Parisi and Rizzo predict \cite[(4)]{PR09}
\begin{equation}\label{eq:B-rate}
 P_N(\varphi)\propto\exp\bigl\{-Na_{6/5}\,(f_{\rm typ}-\varphi)^{6/5}\bigr\},\qquad
 a_{6/5}=5\beta\,6^{-6/5}|\cPR|^{-1/5},
\end{equation}
written in \cite[Section~IV]{PR09} as
$L(\varphi)=-5\beta\,6^{-6/5}|\cPR|^{-1/5}|\Delta f|^{6/5}+O(|\Delta f|^{8/5})$, where the coefficient, also called $a_{6/5}$ there, is the negative of ours. As $\delta=\beta|\Delta f|$ and $\cb=\beta|\cPR|$,
\[
 \frac56\,(6\cb)^{-1/5}\delta^{6/5}=5\cdot6^{-6/5}\,\beta\,|\cPR|^{-1/5}\,|\Delta f|^{6/5}
 =a_{6/5}\,|\Delta f|^{6/5},
\]
so Theorem~\ref{thm:sk}(b),(c) give the leading term of the exponent, for the upper tail of
$N^{-1}\log\ZN$ and of $N^{-1}\FNb-\log2$. We prove no correction term, and nothing on the
fluctuations of order $N^{-5/6}$ that \cite[Section~VI]{PR09} relate to \eqref{eq:B-rate} by a
matching argument described there as not rigorous.

\paragraph{The identities of Crisanti and Rizzo.}
Parisi and Rizzo use three exact statements at $n=H=0$ due to Crisanti and Rizzo \cite{CR02}:
$\ddot q(0)=0$, $f^{(0,2)}(0,0)=1$ and $f^{(0,4)}(0,0)=-(2/(T\qdot))^{1/2}$
\cite[(B8)--(B10)]{PR09}, $f^{(0,k)}$ being the $k$-th derivative in $y$. Under the hypotheses of
Theorem~\ref{thm:local}, $Q_*(x)\to0$ as $x\downarrow0$, so \eqref{eq:B-f} for $\zetas$ at $x=0$
gives $f^{(0,k)}(0,0)=\beta^{k-1}\partial_b^k\Phi_{\zetas}(0,0)$, that is,
\begin{equation}\label{eq:B-crisanti-rizzo}
 f^{(0,2)}(0,0)=\beta\Ustar,\qquad f^{(0,4)}(0,0)=\beta^3\Wstar .
\end{equation}
So (B9) is $\Ustar=1/\beta$ (Proposition~\ref{prop:ons-equilibrium} with $\azero=\beta^2$), and
(B10) says that $\Wstar<0$ (Lemma~\ref{lem:plat-fourth-sign}) and
$\beta^6\Wstar^2=2\beta/\qdot=2\beta\rhob$, that is, $\rhob=\beta^5\Wstar^2/2$
(Theorem~\ref{thm:sk}; Remark~\ref{rem:fifth-sk-inputs} lists other derivations of the value and
the positivity of $\rhob$). We neither use nor prove (B8), on $Q_*''(0)$: our results give only
$\alphas\in C^1$ near zero.

\paragraph{A small external field.}
Their field $H$ and temperature $T=1/\beta$ enter the standard Gibbs weight
$\exp\{\beta\sum_{i<j}J_{ij}\sigma_i\sigma_j+\beta H\sum_i\sigma_i\}$, with independent
$\mathcal N(0,1/N)$ couplings $J_{ij}$, the convention implied by \cite[(27) and Section~III]{PR09}; in our variables
\begin{equation}\label{eq:B-field}
 h=\beta H,\qquad f(H)=-T\bigl\{\log2+\Ffield(\beta H)\bigr\},\qquad q_0=\qmin(\beta H),
\end{equation}
where $q_0$ is the height of their plateau at $n=0$. Indeed, for $\xi(t)=\beta^2t^2/2$ the
process $H_N$ of \eqref{eq:intro-covariance} can be realized as
$\beta\sum_{i<j}J_{ij}\sigma_i\sigma_j+\beta D_N$, with $D_N\sim\mathcal N(0,\frac12)$ independent
of $\sigma$ and of the $J_{ij}$ (the coupling of Lemma~\ref{lem:cons-original-sk}); as
$D_N/N\to0$ almost surely, \eqref{eq:fld-free-energy} shows that $f(H)$ is the limiting free
energy per spin of this model in the field.

\begin{corollary}[The small-field laws]\label{cor:dict-field}
Let $\xi(t)=\beta^2t^2/2$ with $\beta>1$, let $T=1/\beta$, $h=\beta H$ and $f(H)$ be as in
\eqref{eq:B-field}, and let $\qdot=1/\rhob$ with $\rhob$ as in Theorem~\ref{thm:sk}. Then, as
$H\to0$,
\begin{align}
 \qmin(\beta H)&=\Bigl(\frac{3T\qdot}2\Bigr)^{1/3}|H|^{2/3}\,(1+o(1)),\label{eq:fld-pr-b4}\\
 f(H)&=f(0)-\frac{H^2}2+\frac3{20}\Bigl(\frac9{4T\qdot}\Bigr)^{1/3}|H|^{10/3}\,(1+o(1)).
 \label{eq:fld-pr-b5}
\end{align}
These are the leading terms of (B4) and (B5) of \cite{PR09}.
\end{corollary}

\begin{proof}
By Corollary~\ref{cor:fld-examples}(a), Theorems~\ref{thm:local} and~\ref{thm:field} apply. Here
$\azero=\beta^2$, and Theorem~\ref{thm:local} gives
$\Dstar=2\sqrt\azero\,\rhostar=2\beta\rhob=2/(T\qdot)$. With $h=\beta H$,
Theorem~\ref{thm:field}(b) gives
$\qmin(\beta H)^3=\frac{3\beta^2H^2}{\azero\Dstar}(1+o(1))=\frac32T\qdot H^2(1+o(1))$, which is
\eqref{eq:fld-pr-b4}. Since $T\azero^{-1/2}\beta^2=1$ and $T\azero^{-7/6}\beta^{10/3}=1$,
Theorem~\ref{thm:field}(d) gives
$f(H)-f(0)=-T(\Ffield(\beta H)-\Ffield(0))
=-\frac{H^2}2+\frac3{40}(9\Dstar)^{1/3}|H|^{10/3}(1+o(1))$, and
$\frac3{40}(9\Dstar)^{1/3}=\frac3{40}(18/(T\qdot))^{1/3}=\frac3{20}(9/(4T\qdot))^{1/3}$, which is
\eqref{eq:fld-pr-b5}.
\end{proof}

The quadratic term $-H^2/2$, a unit susceptibility throughout the spin-glass phase that
\cite[Appendix~B]{PR09} describe as a known result, is Toulouse's identity
$\Ffield''(0)=1/\beta$ \cite{Tou80}; Theorem~\ref{thm:field}(d) gives $\Ffield''(0)=\azero^{-1/2}$
in general, and Section~\ref{sec:intro-perturbations} recalls its history.

\paragraph{The de Almeida--Thouless line.}
The replica-symmetric overlap $q$ solves \cite[(10)]{PR09} and is stable above the line
\cite[(22)]{PR09}. Both are expectations of functions of $\beta y$, with $y$ a centred Gaussian of
variance $q$ tilted by $\cosh^n\beta y$; so $\beta y$ has the $\cosh^n$-tilted law of
$\mathcal N(0,w)$, $w=\beta^2q=\xi'(q)$, whose expectation is $\E_{w,n}$ of
Section~\ref{sec:nc-dat}, and (10) and (22) read
\begin{equation}\label{eq:B-dat}
 q=\Tfrak\bigl(\xi'(q),n\bigr),\qquad \xi''(q)\,\Sfrak\bigl(\xi'(q),n\bigr)=1 ,
\end{equation}
the contact equation of the two-atom measure $n\delta_0+(1-n)\delta_q$ and its marginal case
\eqref{eq:nc-two-atom}. For $1<\beta^2<1+\varepsilon_*$, Theorem~\ref{thm:nearcrit}(b) gives
$\theta_c=\frac23(\beta^2-1)+O((\beta^2-1)^2)$, whose leading term is that of their line
$n_{\rm dAT}=\frac43(1-T)+O((1-T)^2)$ \cite[(23)]{PR09}, since
$\frac23(\beta^2-1)=\frac43(1-T)+O((1-T)^2)$. Their small parameter $1-T$ is not the distance
$\azero-1$ to criticality used in Section~\ref{sec:nc}.

\paragraph{Stability above the first plateau.}
With $q=Q_\theta$, $q_{\rm free}=Q_*$ and $x_c=x_\theta=\zeta_\theta([0,q_\theta])$, the
prediction $q(x)=q_{\rm free}(x)+O(n^5)$ for $x_c<x<1$ \cite[Section~IV]{PR09} is the bound
$\sup_{x_\theta<x<1}|Q_\theta(x)-Q_*(x)|=O(\theta^5)$ of Theorem~\ref{thm:nearcrit}(c), and the
continuity of the constrained $q(x)$ at $x_c$ noted there says that $q_\theta$ is not isolated
from above in $\supp\zeta_\theta$, which on $\Kcl(\varepsilon_*)$ follows from
$\supp\zeta_\theta=\{0\}\cup[q_\theta,v_\theta]$, $0<\theta<\theta_c$
(Proposition~\ref{prop:nc-dat-threshold}(c)). For $1<\beta^2<1+\varepsilon_*$ both are thus proved
near criticality, in their constrained-Parisi form (Section~\ref{sec:nc-stability}); Parisi and
Rizzo state them at every temperature below the critical one.

\subsection{The plateau expansion}\label{sec:B-expansion}

Parisi and Rizzo obtain (B2)--(B5) by expanding their functional along a one-parameter family
\cite[Appendix~B]{PR09}, assuming that $q(x)$ is unperturbed above its first plateau: $q(x)=q_0$
for $n<x<x_{\rm plateau}$ and $q(x)=q_{\rm free}(x)$ for $x>x_{\rm plateau}$, where
$q_{\rm free}(x_{\rm plateau})=q_0$. As a measure, this trial order parameter is
$\tilde\zeta^{\theta,r}=\theta\delta_0+(\alphas(r)-\theta)\delta_r+\zetas|_{(r,1]}$, with $r>0$
and $\alphas(r)\ge\theta$: the mass of $\zetas$ on $[0,r]$ is moved into an atom $\theta$ at zero
and an atom at $r$. Their expansion \cite[(B1)]{PR09} of $\Phi(n,H,q_0)-\Phi(0,0,0)$ (written with $h$ for $H$ there), stated up
to terms of order six in $q_0$ with $n=O(q_0)$ and $H^2=O(q_0^3)$, becomes by
\eqref{eq:B-functional}, with $h=\beta H$, $\theta=n$, $r=q_0$ and $\rho=\rhob=1/\qdot$,
\begin{equation}\label{eq:B-expansion}
 P^h_{\tilde\zeta^{\theta,r}}-\feq=\frac{h^2}{2\beta}+\frac{\theta h^2r}2-\frac{\rho h^2r^2}4
 +\beta^3\Bigl(\frac{\theta^2r^3}6-\frac{5\rho\theta r^4}{24}+\frac{\rho^2r^5}{15}\Bigr)+\cdots .
\end{equation}
At $h=0$ and $r=u\theta/\rho$ the right side is $\beta^3\rho^{-3}\theta^5g(u)$, with
$g(u)=\frac{u^3}6-\frac{5u^4}{24}+\frac{u^5}{15}$ and $g'(u)=\frac16u^2(u-1)(2u-3)$. The trial
measure is admissible when $\alphas(r)\ge\theta$, that is, $u\ge1$ to leading order, and on
$[1,\infty)$ the function $g$ is smallest at $u=\frac32$, where $g(\frac32)=\frac9{640}$: this is
(B2)--(B3), with $r=\frac{3\theta}{2\rho}$, the value $\cb\theta^5$ and an atom
$\alphas(r)-\theta\approx\frac\theta2$ at $r$. At $\theta=0$ the minimum over $r$ is at
$r^3=3h^2/(2\beta^3\rho)=3h^2/(\azero\Dstar)$, with value
$\frac{h^2}{2\beta}-\frac{\beta^3\rho^2r^5}{10}$: this is (B4)--(B5), the constants of
Theorem~\ref{thm:field}(b),(d), and the atom at $r$ carries all of $\alphas(r)$. Within the
ansatz, the constants thus come from minimizing over the height of the first plateau (our minimum
is their maximum of $\Phi$).

We neither prove nor use \eqref{eq:B-expansion}. The ansatz presupposes the marginal first
contact of Theorem~\ref{thm:local}(a): were $\tilde\zeta^{\theta,r}$ the constrained minimizer,
its support would contain $[r,\kappa]$ for small $r$, since $[0,\kappa]\subset\supp\zetas$, and
$\Gamma=\id$ would hold there (Lemmas~\ref{lem:cons-variational}(e)
and~\ref{lem:cons-standard}(f)), so that $\Gamma'(r)=1$. We prove this marginality
(Proposition~\ref{prop:ons-onset}) and obtain the constants from the envelope identity and the
universal cubic (Section~\ref{sec:intro-ideas}), with no information on the measure above its
first contact. To leading order the minimizer has the shape of the trial measure at the bottom of
its support (Figure~\ref{fig:bottom}(a)); on $\Kcl(\varepsilon_*)$ it also agrees with $\zetas$
above its first contact, up to $O(\theta^5)$ (Theorem~\ref{thm:nearcrit}(c)).

\section{Proofs for the small field}\label{app:field}

This appendix proves the auxiliary lemmas of Section~\ref{sec:fld}, in the order
Lemma~\ref{lem:fld-diffusion}(a),(b), Lemma~\ref{lem:fld-ito} (which uses only part~(a)),
Lemma~\ref{lem:fld-diffusion}(c) (which uses Lemma~\ref{lem:fld-ito}), and
Lemmas~\ref{lem:fld-continuity}, \ref{lem:fld-plateau} and~\ref{lem:fld-taylor}.
Lemmas~\ref{lem:fld-diffusion} and~\ref{lem:fld-ito} extend to the diffusion \eqref{eq:fld-sde}
facts that are standard for $h=0$; Lemma~\ref{lem:fld-diffusion}(c) rests on
\cite[Proposition~1.1]{JT17}, and Lemma~\ref{lem:fld-continuity} concerns the dependence on $h$.

\begin{proof}[Proof of Lemma~\ref{lem:fld-diffusion}\textup{(a),(b)}]
(a) Put $M_t=\int_0^t\sqrt{a(s)}\,dB_s$ and $\mathsf b(t,x)=a(t)\alpha_\zeta(t)\Phi_{\zeta,x}(t,x)$.
By the facts of Section~\ref{sec:fld-setting}, $\mathsf b$ is Borel, $|\mathsf b|\le\|a\|_\infty$ and
$|\mathsf b(t,x)-\mathsf b(t,y)|\le\|a\|_\infty|x-y|$. For each continuous path $M$, the integral
equation $Y_t=h+\int_0^t\mathsf b(s,Y_s+M_s)\,ds$ has a unique continuous solution (Picard
iteration, and Gronwall's inequality for uniqueness), and the solutions of \eqref{eq:fld-sde} are
exactly the processes $X=Y+M$. Thus $X^{\zeta,h}=\Pi_h(M)$ for a measurable map $\Pi_h$. Since
$\Phi_{\zeta,x}(t,\cdot)$ is odd, $\mathsf b(t,-x)=-\mathsf b(t,x)$, so $-\Pi_h(M)$ solves the
equation with data $(-h,-M)$, that is, $-\Pi_h(M)=\Pi_{-h}(-M)$. As $-M$ has the law of $M$,
$X^{\zeta,-h}$ has the law of $-X^{\zeta,h}$. The functions $\Phi_{\zeta,x}^2$ and
$\Phi_{\zeta,xx}^2$ are even in $x$, which gives the two identities of expectations. Evenness of
$\Phi_\zeta(0,\cdot)$ gives $P^{-h}_\zeta=P^h_\zeta$, and uniqueness of the minimizer gives
$\zeta_{-h}=\zeta_h$.

(b) Let $\mathcal G_n=\{k/n:0\le k\le n\}\cup\{q\}=\{s_0<s_1<\dots<s_{N_n}\}$, and move the mass of
$\zeta$ on each $(s_{i-1},s_i]$ to $s_i$:
$\zeta_n=\sum_{i=1}^{N_n}\zeta((s_{i-1},s_i])\,\delta_{s_i}$. This is a probability measure because
$\zeta(\{0\})=0$. Every atom moves by at most $1/n$, so $\zeta_n\to\zeta$ weakly. If $s_i<q$, then
$(s_{i-1},s_i]\subset(0,q)$ carries no mass, so $\zeta_n$ has no atom in $[0,q)$; let $t_1\ge q$
be its first atom. The Cole--Hopf form with $m=0$ on $[0,t_1)$ and the semigroup property of
Gaussian convolution give \eqref{eq:fld-heat} for $\zeta_n$, and it passes to $\zeta$ by the
uniform convergence of $\Phi_{\zeta_n}$ to $\Phi_\zeta$. On $[0,q)$ the drift in
\eqref{eq:fld-sde} vanishes because $\alpha_\zeta=0$ there, so
$X^{\zeta,h}_q=h+M_q\sim\mathcal N(h,\xi'(q))$, as $\xi'(0)=0$.
\end{proof}

The It\^o identities are proved for finitely atomic $\zeta$, by It\^o's formula between atoms,
and extended to general $\zeta$ by coupling the diffusions.

\begin{proof}[Proof of Lemma~\ref{lem:fld-ito}]
(a) The derivatives $\partial_x^j\Phi_\zeta$, $j\ge1$, are continuous and bounded on
$[0,1]\times\R$, and $X$ has continuous paths; apply dominated convergence.

\emph{Step 1: finitely atomic measures.} Let $\zeta$ have finitely many atoms, and let
$[t_l,t_{l+1})$ be an interval on which $\alpha_\zeta=m_l$. Put $\varphi=\Phi_\zeta(t_{l+1},\cdot)$;
it is smooth, its derivatives of positive order are bounded, and $|\varphi(y)|\le|\varphi(0)|+|y|$.
For $\tau>0$ the Gaussian convolution of $e^{m_l\varphi}$ (of $\varphi$ if $m_l=0$) with variance
$\tau$ is $C^\infty$ in $(\tau,x)$, positive, and solves the heat equation. Since
$\xi'(t_{l+1})-\xi'(s)$ is analytic and positive for $s<t_{l+1}$, the Cole--Hopf form shows that
$\Phi_\zeta$ is $C^\infty$ on $[t_l,t_{l+1})\times\R$ and solves the Parisi equation there
classically, with $\alpha_\zeta=m_l$. Hence $u=\Phi_{\zeta,x}$ satisfies
$u_t+\frac a2u_{xx}+am_luu_x=0$ there, and the generator of $X$ on this interval is
$\frac a2\partial_{xx}+am_lu\,\partial_x$. It\^o's formula on $[s_1,s_2]\subset[t_l,t_{l+1})$,
together with the equations obtained by differentiating the equation for $u$ in $x$, gives
\[
 du(t,X_t)=\sqrt a\,U_t\,dB_t,\qquad
 dU_t=-m_laU_t^2\,dt+\sqrt a\,V_t\,dB_t,\qquad
 dV_t=-3m_laU_tV_t\,dt+\sqrt a\,W_t\,dB_t .
\]
These use only the equation, not the starting point. All stochastic integrands are bounded, so
the martingale parts have mean zero, and
\begin{align}
 \E u(s_2,X_{s_2})^2-\E u(s_1,X_{s_1})^2&=\int_{s_1}^{s_2}a\,\E U_s^2\,ds,\label{eq:fld-I1}\\
 \E U_{s_2}^2-\E U_{s_1}^2&=\int_{s_1}^{s_2}a\,\E\bigl(V_s^2-2m_lU_s^3\bigr)\,ds,\label{eq:fld-I2}\\
 \E\bigl(V^2-2m_lU^3\bigr)_{s_2}-\E\bigl(V^2-2m_lU^3\bigr)_{s_1}
 &=\int_{s_1}^{s_2}a\,\E\bigl(W_s^2-12m_lU_sV_s^2+6m_l^2U_s^4\bigr)\,ds.\label{eq:fld-I3}
\end{align}
By part~(a), \eqref{eq:fld-I1} and \eqref{eq:fld-I2} extend to $s_2=t_{l+1}$. Summing over the
intervals, \eqref{eq:fld-I1} holds for all $0\le s_1\le s_2\le1$, and so does \eqref{eq:fld-I2}
with $m_l$ replaced by $\alpha_\zeta(s)$ inside the integral.

\emph{Step 2: general measures.} Let $S\subset(0,1)$ be finite, let
$\mathcal G_n=\{k/n:0\le k\le n\}\cup S$, and define a finitely atomic measure $\zeta_n$ by its
distribution function $\alpha_n(s)=\alpha_\zeta(\max\{g\in\mathcal G_n:g\le s\})$ for $s<1$ and
$\alpha_n(1)=1$. Since $\alpha_\zeta(s-1/n)\le\alpha_n(s)\le\alpha_\zeta(s)$, $\alpha_n\to\alpha_\zeta$ at
every continuity point of $\alpha_\zeta$; hence $\zeta_n\to\zeta$ weakly and
$\Wone(\zeta_n,\zeta)\to0$. If $\alpha_\zeta=m$ on an open interval $I$ and $s_1<s_2$ are points of
$S\cap I$, then $\alpha_n=m$ on $[s_1,s_2]$. Let $X^n$ solve \eqref{eq:fld-sde} for $\zeta_n$ with
the same Brownian motion and $X^n_0=h$ (Lemma~\ref{lem:fld-diffusion}(a)), and put
$u_n=\Phi_{\zeta_n,x}$ and $u=\Phi_{\zeta,x}$. The martingale parts cancel, and since
$\alpha_n\le1$, $|u|\le1$ and $|u_x|\le1$,
\[
 |X^n_t-X_t|\le\|a\|_\infty\Bigl(\|u_n-u\|_\infty+\Wone(\zeta_n,\zeta)
 +\int_0^t|X^n_s-X_s|\,ds\Bigr).
\]
Gronwall's inequality gives, on every path,
$\sup_t|X^n_t-X_t|\le\|a\|_\infty e^{\|a\|_\infty}\bigl(\|u_n-u\|_\infty+\Wone(\zeta_n,\zeta)\bigr)$,
which tends to zero. Then, for $j\ge1$,
\[
 \sup_t\bigl|\partial_x^j\Phi_{\zeta_n}(t,X^n_t)-\partial_x^j\Phi_\zeta(t,X_t)\bigr|
 \le\|\partial_x^j(\Phi_{\zeta_n}-\Phi_\zeta)\|_\infty+B_{j+1}\sup_t|X^n_t-X_t|\longrightarrow0 .
\]
All these quantities are bounded by the constants $B_j$, so expectations of polynomials in them
converge uniformly in $t$; also $\alpha_n\to\alpha_\zeta$ almost everywhere. By dominated
convergence in the time integrals, \eqref{eq:fld-I1} and \eqref{eq:fld-I2} (with
$\alpha_\zeta(s)$ in the integrand) hold for $\zeta$ and all $0\le s_1\le s_2\le1$, and
\eqref{eq:fld-I3} holds for $\zeta$ on $[s_1,s_2]$ whenever $\alpha_\zeta=m$ on an open interval
containing $[s_1,s_2]$ (take $S=\{s_1,s_2\}$ and $m_l=m$).

\emph{Step 3: conclusions.} (b) By \eqref{eq:fld-I1} and part~(a),
$\Gamma(s_2)-\Gamma(s_1)=\int_{s_1}^{s_2}a\,\E U_s^2\,ds$ with a continuous integrand. (c) Let
$A(t)=\E U_t^2$. By part~(a) and right continuity of $\alpha_\zeta$, the integrand
$a(s)\E(V_s^2-2\alpha_\zeta(s)U_s^3)$ in \eqref{eq:fld-I2} is right continuous, so $A$ has the right
derivative $a(t)\E(V_t^2-2\alpha_\zeta(t)U_t^3)$ at every $t\in[0,1)$. Since $\Gamma'=aA$ and
$a\in C^1$, part~(c) follows. (d) On $I$ both integrands in \eqref{eq:fld-I2} and
\eqref{eq:fld-I3} are continuous by part~(a). Hence $A$ and $R(t)=\E(V_t^2-2mU_t^3)$ are $C^1$ on
$I$, with $A'=aR$ and $R'=a\E(W^2-12mUV^2+6m^2U^4)$. Differentiating $\Gamma'=aA$ twice gives
\eqref{eq:fld-second} and \eqref{eq:fld-third}.
\end{proof}

Part~(c) rests on \cite[Proposition~1.1]{JT17}; the main step is to identify the weak solution of
the Parisi equation used there with $\Phi_\zeta$.

\begin{proof}[Proof of Lemma~\ref{lem:fld-diffusion}\textup{(c)}]
We use \cite[Proposition~1.1]{JT17}, applied with $\beta=1$ and $\xi_0=\xi$; its hypothesis
$\xi_0(1+\epsilon)<\infty$ holds because the radius of convergence exceeds one. In \cite{JT17}
the functional and the diffusion are built from the unique weak solution $u_\zeta$ of the Parisi
equation in the class of continuous functions on $[0,1]\times\R$ with essentially bounded weak
derivative $\partial_xu_\zeta$ \cite[Section~8.1 and Proposition~8.1]{JT17}. We first check that
$u_\zeta=\Phi_\zeta$. Let $\zeta$ be finitely atomic. Then $\Phi_\zeta$ is continuous on
$[0,1]\times\R$ with $|\Phi_{\zeta,x}|\le1$, and by Step~1 of the proof of
Lemma~\ref{lem:fld-ito} it is $C^\infty$ on each strip $[t_l,t_{l+1})\times\R$ and solves the
Parisi equation classically there. Hence
$0\le-\partial_t\Phi_\zeta=\frac a2(\Phi_{\zeta,xx}+m_l\Phi_{\zeta,x}^2)\le\|a\|_\infty$ on each
strip, and $\Phi_\zeta(\cdot,x)$ is Lipschitz on $[0,1]$. Let
$\varphi\in C_c^\infty((0,1]\times\R)$. Integrating by parts once in $t$ over $[0,1]$ and twice in
$x$, and using $\Phi_\zeta(1,x)=\log\cosh x$ and $\varphi=0$ near $t=0$, we can write the weak
formulation of \cite[Section~8.1]{JT17} as
\[
 \int_0^1\!\!\int_\R\Bigl(\partial_t\Phi_\zeta+\frac a2\bigl\{
 \Phi_{\zeta,xx}+\alpha_\zeta\Phi_{\zeta,x}^2\bigr\}\Bigr)\varphi\,dx\,dt=0,
\]
which holds because the integrand vanishes off the finitely many atom times. Thus $\Phi_\zeta$ is
a weak solution, and $\Phi_\zeta=u_\zeta$ by uniqueness. For general $\zeta$, take finitely atomic
$\zeta_n$ with $\Wone(\zeta_n,\zeta)\to0$, as in Step~2 of the proof of Lemma~\ref{lem:fld-ito}.
Then $\Phi_{\zeta_n}\to\Phi_\zeta$ uniformly, and $u_{\zeta_n}\to u_\zeta$ uniformly by
\cite[Proposition~8.2]{JT17}, which is stated for the same metric. Hence $u_\zeta=\Phi_\zeta$.
Consequently the Parisi functional of \cite{JT17} at field $h$ is $P^h_\zeta$, its minimizer is
$\zeta_h$, its diffusion is \eqref{eq:fld-sde}, and its function $G_\mu$
\cite[(1.1.1)]{JT17} is $\Obs_{\zeta,h}$.

Let $h\ge0$. By \cite[Proposition~1.1]{JT17} (see also \cite[Corollary~3.6]{JT17}; these results were obtained independently in \cite{Che17}, see \cite[Remark~1.3]{JT17}), $\zeta_h$ gives
full mass to the closed set of minimizers of $\Obs_{\zeta_h,h}$ over $[0,1]$, so $\supp\zeta_h$
is contained in that set. The same proposition also states \eqref{eq:fld-sc} at every support
point, but its inequality is derived \cite[Corollary~3.10]{JT17} from derivative formulas whose
proofs are omitted there \cite[Lemma~8.4]{JT17}. At points of $(0,1)$ we derive
\eqref{eq:fld-sc} from the support statement instead. By Lemma~\ref{lem:fld-ito}(b),
$\Gamma=\Gamma_{\zeta_h,h}$ is $C^1$, so $\Obs=\Obs_{\zeta_h,h}$ is $C^2$ on $[0,1]$ with
$\Obs'=-\frac a2(\Gamma-t)$ and $\Obs''=-\frac{a'}2(\Gamma-t)-\frac a2(\Gamma'-1)$. At an interior
minimum, $\Obs'=0\le\Obs''$; since $a>0$, this gives \eqref{eq:fld-sc}, as in
Section~\ref{sec:plat-obstacle}.

For $h<0$, part~(a) gives $\zeta_h=\zeta_{-h}$ and $\Gamma_{\zeta_h,h}=\Gamma_{\zeta_{-h},-h}$, hence
$\Obs_{\zeta_h,h}=\Obs_{\zeta_{-h},-h}$, and the expectation in \eqref{eq:fld-sc} is the same for
$h$ and $-h$; so the statements for $-h>0$ transfer.

If $h\ne0$, then $\Obs_{\zeta_h,h}$ is $C^1$ with
$\Obs_{\zeta_h,h}'(0)=-\frac\azero2\Gamma_{\zeta_h,h}(0)=-\frac\azero2\Phi_{\zeta_h,x}(0,h)^2<0$,
because $X_0=h$ and $\Phi_{\zeta_h,x}(0,\cdot)$ is odd and strictly increasing. So zero is not a
minimizer of $\Obs_{\zeta_h,h}$ and does not lie in $\supp\zeta_h$; for $h>0$ this is also
\cite[Lemma~3.9]{JT17}.
\end{proof}

\begin{proof}[Proof of Lemma~\ref{lem:fld-continuity}]
If $h_n\to h$ and $\zeta_n\to\zeta$ weakly, then $P^{h_n}_{\zeta_n}\to P^h_\zeta$: indeed
$|\Phi_{\zeta_n}(0,h_n)-\Phi_\zeta(0,h)|\le\|\Phi_{\zeta_n}-\Phi_\zeta\|_\infty+|h_n-h|$, which tends
to zero, and $\alpha_{\zeta_n}\to\alpha_\zeta$ in $L^1$. Let $h_n\to h$, and let $\zeta'$ be a weak
limit point of $(\zeta_{h_n})$; the space of probability measures on $[0,1]$ is compact. Along the
corresponding subsequence, for every $\nu$,
$P^h_{\zeta'}=\lim P^{h_n}_{\zeta_{h_n}}\le\lim P^{h_n}_\nu=P^h_\nu$, so $\zeta'=\zeta_h$ by
uniqueness. Hence $h\mapsto\zeta_h$ is weakly continuous. By Lemma~\ref{lem:fld-diffusion}(a),
$\zeta_{-h}=\zeta_h$ and $\Ffield$ is even; since $\Phi_{\zeta_h,x}(0,\cdot)$ is odd,
$\mfield(-h)=-\mfield(h)$.

Fix $h$ and $k\ne0$. Comparing each minimizer in the other problem,
\[
 \Phi_{\zeta_{h+k}}(0,h+k)-\Phi_{\zeta_{h+k}}(0,h)
 \le\Ffield(h+k)-\Ffield(h)
 \le\Phi_{\zeta_h}(0,h+k)-\Phi_{\zeta_h}(0,h).
\]
By the mean value theorem both bounds equal $k\,\Phi_{\zeta,x}(0,y)$, with
$\zeta\in\{\zeta_h,\zeta_{h+k}\}$ and $y$ between $h$ and $h+k$. Since $\Phi_{\zeta,xx}\le1$,
$|\Phi_{\zeta,x}(0,y)-\Phi_{\zeta_h,x}(0,h)|\le\|\Phi_{\zeta,x}-\Phi_{\zeta_h,x}\|_\infty+|y-h|$,
which tends to zero as $k\to0$ by jet stability and the continuity just proved. Dividing by $k$
(the two bounds have the same limit for either sign of $k$) gives $\Ffield'(h)=\mfield(h)$, and
the same estimate shows that $\mfield$ is continuous.
\end{proof}

Lemma~\ref{lem:fld-plateau} holds because, for small $t$ and $|h|$ and for $\zeta$ near $\zetas$,
the diffusion stays near the origin with high probability and the jets along it are close to those
of $\Phi_{\zetas}$ at $(0,0)$.

\begin{proof}[Proof of Lemma~\ref{lem:fld-plateau}]
Write $X=X^{\zeta,h}$ and, for $k=2,3,4$, $j_k=\partial_x^k\Phi_{\zetas}(0,0)$. Thus $j_2=\Ustar$,
$j_4=\Wstar$, and $j_3=0$ because $\Phi_{\zetas}(0,\cdot)$ is even. Let $Q_m(t,x;\zeta)$ be the
expression defining $\Psi_{\zeta,h}(t,m)$, the right side of \eqref{eq:plat-psi}, with the
expectations removed and $U,V,W$ replaced by $\partial_x^k\Phi_\zeta(t,x)$, $k=2,3,4$; thus
$\Psi_{\zeta,h}(t,m)=\E Q_m(t,X_t;\zeta)$. Since $a'(0)=\xi'''(0)=0$ and $a''(0)=\xi''''(0)$, its
value at the reference point is
\[
 Q_m^*:=\xi''''(0)j_2^2+\azero^3\bigl(j_4^2+6m^2j_2^4\bigr)=\Dstar+6m^2\azero^3\Ustar^4 .
\]
The function $Q_m$ is a polynomial in $(a,a',a'',m)$ and in the three jets, which range over a
fixed compact set by the uniform bounds of Section~\ref{sec:fld-setting}. Hence
$|Q_m|\le Q_{\max}$ and, with a constant $L$ uniform in $m\in[0,1]$,
\[
 |Q_m(t,x;\zeta)-Q_m^*|\le L\Bigl(\delta(t)+\sum_{k=2}^4|\partial_x^k\Phi_\zeta(t,x)-j_k|\Bigr),
 \qquad \delta(t)=|a(t)-\azero|+|a'(t)|+|a''(t)-a''(0)| ,
\]
where $\delta(t)\to0$ as $t\to0$. Put
$\varepsilon(\zeta)=\max_{2\le k\le4}\|\partial_x^k(\Phi_\zeta-\Phi_{\zetas})\|_\infty$, which is at
most $C'\Wone(\zeta,\zetas)$ by jet stability, with $C'=\max_{2\le k\le4}C_k\|a\|_\infty$, and
\[
 \omega(r)=\max_{2\le k\le4}\ \sup_{0\le t\le r,\ |x|\le r}
 \bigl|\partial_x^k\Phi_{\zetas}(t,x)-j_k\bigr|,
\]
which tends to zero as $r\to0$ by continuity of the jets. The drift in \eqref{eq:fld-sde} is
bounded by $\|a\|_\infty$, so, as in \eqref{eq:plat-small-time},
$\E|X_t-h|^2\le2\|a\|_\infty t+2\|a\|_\infty^2t^2$. Let $0<r\le1$, $t\le r$ and $|h|\le r/2$.
Chebyshev's inequality gives
$\Prob(|X_t|>r)\le\Prob(|X_t-h|>r/2)\le8(\|a\|_\infty t+\|a\|_\infty^2t^2)/r^2$, and on the event
$\{|X_t|\le r\}$ each jet differs from $j_k$ by at most $\varepsilon(\zeta)+\omega(r)$. Hence
\begin{equation}\label{eq:fld-plateau-error}
 \sup_{0\le m\le1}\bigl|\Psi_{\zeta,h}(t,m)-Q_m^*\bigr|
 \le L\bigl(\delta(t)+3\varepsilon(\zeta)+3\omega(r)\bigr)
 +\frac{16Q_{\max}}{r^2}\bigl(\|a\|_\infty t+\|a\|_\infty^2t^2\bigr).
\end{equation}
Given $\eta>0$, choose $r\in(0,1]$ with $3L\omega(r)\le\eta/4$, and then $t_0\in(0,r]$ such that
the terms in $\delta(t)$ and in $t$ are each at most $\eta/4$ for $t\le t_0$. The left side of
\eqref{eq:fld-plateau-error} is then at most $\eta$ whenever $t\le t_0$, $|h|\le r/2$ and
$3LC'\Wone(\zeta,\zetas)\le\eta/4$. This proves the first claim. For the second, take
$\eta=\Dstar/2$, which is positive by Remark~\ref{rem:fld-nondegeneracy}, and put
$K_{\rm fld}=\min\{t_0,1/2\}$, $h_0=r/2$ and $\varepsilon_{\rm fld}=\eta/(12LC')$. If $\alpha_\zeta=m$ on an open
interval $I\subset(0,K_{\rm fld})$, then, by Lemma~\ref{lem:fld-ito}(d), $\Gamma_{\zeta,h}'''=\Psi_{\zeta,h}(\cdot,m)\ge Q_m^*-\Dstar/2\ge\Dstar/2$
on $I$.
\end{proof}

\begin{proof}[Proof of Lemma~\ref{lem:fld-taylor}]
Let $T_{k-1}g$ denote the Taylor polynomial of degree $k-1$ at zero of a function $g$. If
$\|g^{(k)}\|_\infty\le C''$, Taylor's theorem gives $|g(x)-T_{k-1}g(x)|\le C''|x|^k/k!$ on
$\R$. We use $\E X^2=h^2+\sigma^2$, $\E X^4=h^4+6h^2\sigma^2+3\sigma^4$,
$\E X^6=h^6+15h^4\sigma^2+45h^2\sigma^4+15\sigma^6$ and $\E|X|^k\le C_k''(h^k+\sigma^k)$; since
$h,\sigma\le1$, higher powers are absorbed into lower ones.

For \eqref{eq:fld-t1}, write $v=p+r$ with $p(x)=c_1x+c_3x^3/6+c_5x^5/120$. The remainder $r$ is
odd and $|r'(x)|\le C|x|^6/6!$. Since $Z$ and $-Z$ have the same law,
$\E r(h+\sigma Z)=\frac12\E[r(\sigma Z+h)-r(\sigma Z-h)]$, which is at most
$h\,\E\sup_{|s|\le h}|r'(\sigma Z+s)|=O(h(\sigma^6+h^6))$ in absolute value. Also
$\E p(X)=c_1h+\frac{c_3}6(h^3+3h\sigma^2)+\frac{c_5}{120}(h^5+10h^3\sigma^2+15h\sigma^4)$.

For \eqref{eq:fld-t2}, $v'^2$ is even, with
$T_5(v'^2)(x)=c_1^2+c_1c_3x^2+(\frac14c_3^2+\frac1{12}c_1c_5)x^4$ and remainder bounded by a
constant times $x^6$.

For \eqref{eq:fld-t3}, $v^2$ is even, with
$T_7(v^2)(x)=c_1^2x^2+\frac13c_1c_3x^4+(\frac1{36}c_3^2+\frac1{60}c_1c_5)x^6$ and remainder bounded
by a constant times $x^8$; here $|v|\le1$ is used. Take expectations and subtract $\sigma^2$
times the expansion of $\E v'(X)^2$ before truncation. The terms of order $\sigma^2$ and
$\sigma^4$ cancel, and the coefficient of $\sigma^6$ is
\[
 15\Bigl(\frac1{36}c_3^2+\frac1{60}c_1c_5\Bigr)-3\Bigl(\frac14c_3^2+\frac1{12}c_1c_5\Bigr)
 =-\frac13c_3^2 .
\]
Every remaining polynomial term contains $h^2\sigma^2$ or $h^4$, and the truncation errors are
$O(h^8+\sigma^8+\sigma^2h^6)$.
\end{proof}

\section{A cubic term near criticality: proofs}\label{app:cubic}

We prove the statements of Section~\ref{sec:atom-odd} for $\xi\in\Kclodd(\varepsilon_3)$,
$\varepsilon_3=\varepsilon_1$, with $K$, $\mathcal W$ as in Section~\ref{sec:nc-reference}
and $\Dstar$, $c$ as in \eqref{eq:atom-odd-D}. Proposition~\ref{prop:atom-odd-structure}
describes $\zetas$ and $\zeta_\theta$; Proposition~\ref{prop:atom-odd-onset} then follows
Sections~\ref{sec:fifth} and~\ref{sec:ons} with $\theta-m_0$ in place of $\theta$, and
Proposition~\ref{prop:atom-odd-rate} follows from \eqref{eq:atom-exact}; Lemma~\ref{lem:atom-odd-slope} expresses the constants
through the slope of $\alphas$ to the right of its jump at zero.

\subsection{The equilibrium measure and the constrained minimizers}\label{sec:cubic-structure}

\begin{proposition}[The measures with a cubic term]\label{prop:atom-odd-structure}
Every $\xi\in\Kclodd(\varepsilon_3)$, where $\varepsilon_3=\varepsilon_1\in(0,\frac14]$, has
the following properties.
\begin{thmlist}
\item $\xi''\ge1-\varepsilon_3$ on $[-1,1]$; in particular $\xi$ is strictly convex on
 $[-1,1]$.
\item $\supp\zetas=[0,v_*]$ for some $v_*\in(0,K)$. Moreover $\Ustar=\azero^{-1/2}$,
 \[
  m_0=\frac{\xi'''(0)}{2\xi''(0)^{3/2}}=\frac{3\beta_3^2}{(2\beta_2^2)^{3/2}}\in(0,1),
  \qquad
  \Dstar=\frac{\xi''''(0)}{\azero}-\frac{3\xi'''(0)^2}{2\azero^2}+\azero^3\Wstar^2 .
 \]
\item For $0\le\theta\le m_0$, $\zeta_\theta=\zetas$ and $p(\theta)=\feq$. For
 $m_0<\theta<1$, $\zeta_\theta(\{0\})=\theta$ and $p(\theta)>\feq$, the first positive
 contact $q_\theta$ of $\zeta_\theta$ exists and satisfies $\zeta_\theta((0,q_\theta))=0$,
 $\Gamma_\theta(q_\theta)=q_\theta$ and $\Gamma_\theta'(q_\theta)\le1$, and $p$ is
 differentiable at $\theta$ with
 $p'(\theta)=\Obs_{\zeta_\theta}(0)-\Obs_{\zeta_\theta}(q_\theta)$. Also $\feq<f_1$, and
 every $f\in(\feq,f_1)$ has $\theta_f\in(m_0,1)$.
\item As $\theta\downarrow m_0$, $p(\theta)\to\feq$ and $\zeta_\theta\Rightarrow\zetas$.
 There is $\theta_3\in(m_0,1)$ such that, for $m_0<\theta<\theta_3$,
 \[
  \supp\zeta_\theta=\{0\}\cup[q_\theta,v_\theta],\qquad 0<q_\theta<v_\theta<K,\qquad
  \Gamma_\theta'(q_\theta)=\xi''(q_\theta)\E U_{q_\theta}^2=1 .
 \]
 With $f_3=p(\theta_3)>\feq$, every $f\in(\feq,f_3)$ has $\theta_f\in(m_0,\theta_3)$, so
 its first positive contact is marginal.
\end{thmlist}
\end{proposition}

\begin{proof}
Let $\xi\in\Kclodd(\varepsilon_3)$. Since $\xi''(0)>1$ and
$\normCfour{\xi-\xic}<\varepsilon_3=\varepsilon_1$, the comparison with the critical
reference applies to $\xi$, whatever $\beta_3$: Proposition~\ref{prop:nc-support} and
items~\ref{item:nc-positive} and~\ref{item:nc-confined} of the choice of constants in
Section~\ref{sec:nc-reference}.

(a) By Proposition~\ref{prop:nc-support}(a),
$\xi''>1-\normCfour{\xi-\xic}>1-\varepsilon_3\ge\frac34$ on $[-1,1]$.

(b) The first assertion is Proposition~\ref{prop:nc-support}(b). Zero is therefore an
accumulation point of the support, and Proposition~\ref{prop:atom-atom}(b) gives
$\Ustar=\azero^{-1/2}$ and the formula for $m_0$, which is positive since $\beta_3>0$. Also
$m_0<1$, because the open set $(0,v_*)$ meets the support and so has positive mass.
Inserting $\Ustar=\azero^{-1/2}$ and $\xi'''(0)=2\azero^{3/2}m_0$ into \eqref{eq:atom-odd-D}
gives $6\azero\xi'''(0)m_0\Ustar^3=3\xi'''(0)^2/\azero^2$ and
$6\azero^3m_0^2\Ustar^4=\frac32\xi'''(0)^2/\azero^2$, which is the closed form of $\Dstar$.

(c) For $\theta\le m_0$ the measure $\zetas$ is admissible, so $p(\theta)=\feq$, and
uniqueness gives $\zeta_\theta=\zetas$. For $m_0<\theta<1$, the activity of the constraint
and $p(\theta)>\feq$, that is $\Lamhat(\theta)>0$, are part of Lemma~\ref{lem:cons-regular},
and so are $\feq<f_1$ and $\theta_f\in(m_0,1)$, since $m_0<1$ by (b). Let
$\theta\in(m_0,1)$. By Corollary~\ref{cor:cons-level}, which applies since $m_0<1$,
$f=\Lambda'(\theta)$ lies in $(\feq,f_1)$ and $\theta=\theta_f$.
Lemma~\ref{lem:cons-variational}(e) applies at $\theta$ and gives the first contact
$q_\theta$ with the stated properties, and Proposition~\ref{prop:cons-envelope} gives the
derivative of $p$.

(d) For $\theta>m_0$ the proof of Proposition~\ref{prop:atom-slope} gives
$\feq\le p(\theta)\le\feq+\frac12\xi''(1)(\theta-m_0)/(1-m_0)$, so $p(\theta)\to\feq$, and every weak
subsequential limit of $\zeta_\theta$ is $\zetas$ by uniqueness; hence
$\zeta_\theta\Rightarrow\zetas$.

For $\theta$ close to $m_0$, the positive support of $\zeta_\theta$ has at least two points.
Otherwise, by (c), there would be $\theta_n\downarrow m_0$ with
$\zeta_{\theta_n}=\theta_n\delta_0+(1-\theta_n)\delta_{r_n}$, and a weak limit of the form
$m_0\delta_0+(1-m_0)\delta_r$, which is not $\zetas$ by (b). Positive support points of
$\zeta_\theta$ are contacts (Lemma~\ref{lem:cons-variational}(e)), and by
item~\ref{item:nc-confined} they lie in $(0,K)$ and $\zeta_\theta\in\mathcal W$. A gap between
two of them would be a plateau of $\zeta_\theta$ in $(0,K)$ bounded by two contacts, at which
$\Gamma_\theta=\id$ and $\Gamma_\theta'\le1$ by \eqref{eq:plat-contact};
item~\ref{item:nc-positive} and the trapezoid principle exclude it
(Lemma~\ref{lem:nc-gap}(a)). So the positive support is an interval $[r_\theta,v_\theta]$ with
$r_\theta<v_\theta<K$. The first contact satisfies $q_\theta\le r_\theta$. If
$q_\theta<r_\theta$, then $(q_\theta,r_\theta)$ is a support-free interval between two
contacts, excluded in the same way; thus $q_\theta=r_\theta$. Since $\Gamma_\theta=\id$ on
$[q_\theta,v_\theta]$ and $\Gamma_\theta'$ is continuous, $\Gamma_\theta'(q_\theta)=1$. This
holds for $m_0<\theta<\theta_3$ with a suitable $\theta_3\in(m_0,1)$.

Let $f_3=p(\theta_3)$, which exceeds $\feq$ by (c). Since $p$ is nondecreasing, $p'\ge0$ and
$f_3\le p(\theta_3)+\theta_3p'(\theta_3)=\Lambda'(\theta_3)<f_1$ by the proof of (c). Let
$f\in(\feq,f_3)$. As $p$ is nondecreasing and $\Lambar(1)=P_{\delta_0}\ge p(s)$ for every
$s$, we have $\Lambar(s)-sf\ge s(p(\theta_3)-f)>0$ for $s\in[\theta_3,1]$. Since
$p(s)\to\feq<f$ as $s\downarrow m_0$, the same quantity is negative for some
$s\in(m_0,\theta_3)$. Hence $\theta_f\in(m_0,\theta_3)$ by (c).
\end{proof}

\subsection{Onset above the atom}\label{sec:cubic-onset}

The new inputs are the identity \eqref{eq:atom-odd-id2}, which gives part~(a), and Step~3,
where the leading terms of $\Gamma_\theta''(0+)$ cancel at $\theta=m_0$ and part~(a), which
comes from marginality, controls what is left; Step~1 replaces
Lemma~\ref{lem:plat-local-plateau} by a direct jet argument.

\begin{proof}[Proof of Proposition~\ref{prop:atom-odd-onset}]
Throughout, $m_0<\theta<\theta_3$, $q=q_\theta$, $o(1)$ refers to $\theta\downarrow m_0$, and
$U_t,V_t,W_t$ are built from $\zeta_\theta$.

\emph{Step 1: the third derivative on the left plateau.} Fix $\rho\in(0,v_*)$. By
Proposition~\ref{prop:atom-odd-structure}(b), $\zetas((0,\rho))>0$, and weak convergence gives
$\zeta_\theta((0,\rho))>0$ for $\theta$ close to $m_0$. Positive support points are contacts,
so $q_\theta<\rho$. Hence $q_\theta\to0$.

On $[0,q)$ the cumulative mass of $\zeta_\theta$ is exactly $\theta$, so
\eqref{eq:plat-third} holds on $(0,q)$ with $m=\theta$:
\[
 \Gamma_\theta'''=a''\E U^2+3aa'\E\bigl(V^2-2\theta U^3\bigr)
   +a^3\E\bigl(W^2-12\theta UV^2+6\theta^2U^4\bigr).
\]
We claim that
\begin{equation}\label{eq:atom-odd-third}
 \epsilon_\theta:=\sup_{0<t<q_\theta}\bigl|\Gamma_\theta'''(t)-\Dstar\bigr|\longrightarrow0 .
\end{equation}
Let $F$ be a monomial of degree at most four in $(\Phi_{xx},\Phi_{xxx},\Phi_{xxxx})$, and
write $F_\theta(t,x)$ and $F_*(t,x)$ for its values built from $\zeta_\theta$ and $\zetas$. For
$0<t<q$,
\[
 \bigl|\E F_\theta(t,X_t)-F_*(0,0)\bigr|
 \le\|F_\theta-F_*\|_\infty+\E\bigl|F_*(t,X_t)-F_*(t,0)\bigr|+\bigl|F_*(t,0)-F_*(0,0)\bigr| .
\]
The first term is at most $C\Wone(\zeta_\theta,\zetas)$ by \eqref{eq:plat-jet-stability} and
the bounds \eqref{eq:plat-bounds}, and it tends to zero because
$\zeta_\theta\Rightarrow\zetas$. The second is at most $C(\E X_t^2)^{1/2}\le C\sqrt q$ by
\eqref{eq:plat-bounds} and \eqref{eq:plat-small-time}, uniformly in $\theta$. The third is
at most $Cq$ by \eqref{eq:plat-time-lip}. Finally $a,a',a''$ are continuous at zero,
$\Phi_{\zetas,xxx}(0,0)=0$ by parity, and $\theta\to m_0$. Therefore $\Gamma_\theta'''$
converges uniformly on $(0,q)$ to
\[
 \xi''''(0)\Ustar^2+3\azero\xi'''(0)\bigl(-2m_0\Ustar^3\bigr)
   +\azero^3\bigl(\Wstar^2+6m_0^2\Ustar^4\bigr)=\Dstar ,
\]
which proves \eqref{eq:atom-odd-third}. The interval $(0,q)\subset(0,K)$ is a plateau of
$\zeta_\theta$, and $\zeta_\theta\in\mathcal W$ by item~\ref{item:nc-confined}, so
item~\ref{item:nc-positive} gives $\Gamma_\theta'''\ge2$ there; letting $\theta\downarrow m_0$
in \eqref{eq:atom-odd-third} proves $\Dstar\ge2$.

\emph{Step 2: three identities.} By Proposition~\ref{prop:atom-odd-structure}(d), every point
of $[q,v_\theta]$ is a support point of $\zeta_\theta$, hence a contact
(Lemma~\ref{lem:cons-variational}(e)), so $\Gamma_\theta=\id$ on $[q,v_\theta]$
by~\eqref{eq:plat-contact}; and $\zeta_\theta((0,q))=0$. So
Corollary~\ref{cor:plat-first-plateau} applies to $\zeta_\theta$ with $m=\theta$ and
$D=\Dstar$, and its error parameter is the quantity $\epsilon_\theta$ of
\eqref{eq:atom-odd-third}. Its part~(a), with \eqref{eq:atom-at-zero}, gives
\begin{equation}\label{eq:atom-odd-Gpp0}
 \Gamma_\theta''(0+)=-\frac{2\Dstar}3\,q+O(\epsilon_\theta q).
\end{equation}
Next, $\Gamma_\theta'(0)=\azero \bar U_\theta^2$ by \eqref{eq:atom-at-zero} and
$\Gamma_\theta'(q)=1$, and $\Gamma_\theta''$ extends continuously to $[0,q]$
(Lemma~\ref{lem:plat-kernel}), so
\begin{equation}\label{eq:atom-odd-id2}
 1-\azero \bar U_\theta^2=\Gamma_\theta'(q)-\Gamma_\theta'(0)
 =q\,\Gamma_\theta''(0+)+\int_0^q(q-s)\,\Gamma_\theta'''(s)\,ds
 =-\frac{\Dstar}6\,q^2+O(\epsilon_\theta q^2).
\end{equation}
Thus $\azero \bar U_\theta^2=1+\frac{\Dstar}6q^2+o(q^2)$, which is (a). Finally,
\eqref{eq:atom-at-zero} for $\zeta_\theta$, whose atom at zero is $\theta$, reads
\begin{equation}\label{eq:atom-odd-id3}
 \Gamma_\theta''(0+)=\xi'''(0)\bar U_\theta^2-2\azero^2\theta \bar U_\theta^3 .
\end{equation}

\emph{Step 3: the onset.} By Proposition~\ref{prop:atom-odd-structure}(b),
$\xi'''(0)=2\azero^{3/2}m_0$, so \eqref{eq:atom-odd-id3} becomes
\[
 \Gamma_\theta''(0+)=2\azero^{3/2}\bar U_\theta^2\bigl(m_0-\theta\sqrt{\azero}\,\bar U_\theta\bigr)
 =-2\azero^{3/2}\bar U_\theta^2(\theta-m_0)
  -2\azero^{3/2}\bar U_\theta^2\,\theta\bigl(\sqrt{\azero}\,\bar U_\theta-1\bigr).
\]
The two leading terms $2\azero^{3/2}\bar U_\theta^2m_0$ and $2\azero^{3/2}\bar U_\theta^2\theta$ cancel
at $\theta=m_0$, and the second term on the right is of the size of
$\sqrt{\azero}\bar U_\theta-1$. This is where (a) is needed: it gives
$\sqrt{\azero}\bar U_\theta-1=O(q^2)$ and $2\azero^{3/2}\bar U_\theta^2=2\sqrt{\azero}(1+O(q^2))$.
Hence
$\Gamma_\theta''(0+)=-2\sqrt{\azero}(\theta-m_0)(1+O(q^2))+O(q^2)$. Comparing with
\eqref{eq:atom-odd-Gpp0},
\[
 2\sqrt{\azero}\,(\theta-m_0)(1+o(1))=\frac{2\Dstar}3\,q+o(q),
\]
and since $\Dstar>0$ this gives $q_\theta/(\theta-m_0)\to3\sqrt{\azero}/\Dstar$, the first
limit in \eqref{eq:atom-odd-onset}. In particular $q_\theta$ is of the order of
$\theta-m_0$.

\emph{Step 4: the atom at the first contact.} Part~(b) of
Corollary~\ref{cor:plat-first-plateau}, which rests on the jump formula \eqref{eq:plat-jump}
and on $\Gamma_\theta'\equiv1$ on $[q,v_\theta]$, gives
$|2a(q)^2\zeta_\theta(\{q\})\E U_q^3-\frac{\Dstar}3q|\le\frac53\epsilon_\theta q$. As in
Step~1, $a(q)^2\E U_q^3\to\azero^2\Ustar^3=\sqrt{\azero}$. Therefore
$\zeta_\theta(\{q\})=\frac{\Dstar}{6\sqrt{\azero}}\,q\,(1+o(1))$, and the first limit in
\eqref{eq:atom-odd-onset} gives the second.

\emph{Step 5: the envelope derivative.} By Proposition~\ref{prop:atom-odd-structure}(c) and
the definition of the obstacle,
\begin{equation}\label{eq:atom-odd-envelope}
 p'(\theta)=\Obs_{\zeta_\theta}(0)-\Obs_{\zeta_\theta}(q)=\frac12\int_0^qa(s)h(s)\,ds,
 \qquad h=\Gamma_\theta-\id\ \text{on }[0,q].
\end{equation}
Part~(c) of Corollary~\ref{cor:plat-first-plateau} gives
$|\int_0^qh-\Dstar q^4/72|\le\epsilon_\theta q^4/72$ and
$|h|\le(\frac{\Dstar}6+\frac56\epsilon_\theta)q^3$ on $[0,q]$. Also
$|a(s)-\azero|\le\|a'\|_\infty q$ on $[0,q]$. Splitting $a=\azero+(a-\azero)$ in
\eqref{eq:atom-odd-envelope},
\[
 p'(\theta)=\frac{\azero}2\int_0^qh+O(q^5)=\frac{\azero\Dstar}{144}\,q_\theta^4\,(1+o(1)).
\]
By \eqref{eq:atom-odd-onset}, $q_\theta^4=81\azero^2\Dstar^{-4}(\theta-m_0)^4(1+o(1))$, so
\[
 p'(\theta)=\frac{9\azero^3}{16\Dstar^3}\,(\theta-m_0)^4(1+o(1))=5c\,(\theta-m_0)^4(1+o(1)).
\]

\emph{Step 6: integration.} By Step 5 there is $\theta_1\in(m_0,\theta_3)$ with $|p'|\le1$
on $(m_0,\theta_1]$. For $m_0<\eta'<\theta\le\theta_1$, the function $p$ is differentiable on
$[\eta',\theta]$ with bounded derivative, hence absolutely continuous, and
$p(\theta)-p(\eta')=\int_{\eta'}^\theta p'$. Let $\eta'\downarrow m_0$; then $p(\eta')\to\feq$ by
Proposition~\ref{prop:atom-odd-structure}(d). Writing $p'(s)=5c(s-m_0)^4(1+r(s))$ with
$r(s)\to0$,
\[
 \bigl|p(\theta)-\feq-c(\theta-m_0)^5\bigr|\le c(\theta-m_0)^5\sup_{m_0<s\le\theta}|r(s)| ,
\]
which is \eqref{eq:atom-odd-p}. Since $\Lambda=\theta p$ and $\theta\to m_0>0$,
\[
 \Lambda(\theta)-\theta\feq=\theta\bigl(p(\theta)-\feq\bigr)=m_0c(\theta-m_0)^5(1+o(1)),
\]
and
$\Lambda'(\theta)-\feq=(p(\theta)-\feq)+\theta p'(\theta)
=c(\theta-m_0)^5(1+o(1))+5\theta c(\theta-m_0)^4(1+o(1))
=5m_0c(\theta-m_0)^4(1+o(1))$. This is \eqref{eq:atom-odd-Lambda}.
\end{proof}

\subsection{The Legendre step}\label{sec:cubic-rate}

The remainder $S$ in \eqref{eq:atom-exact} is now the conjugate of a fifth power.

\begin{proof}[Proof of Proposition~\ref{prop:atom-odd-rate}]
By \eqref{eq:atom-exact},
\begin{equation}\label{eq:atom-odd-S}
 \Sigma(\feq+\delta)=-m_0\delta-S(\delta),\qquad
 S(\delta)=\sup_{0\le u\le1-m_0}\bigl\{u\delta-\Lamhat(m_0+u)\bigr\},
\end{equation}
and by \eqref{eq:atom-odd-Lambda}, $\Lamhat(m_0+u)=Au^5(1+r(u))$ with $A:=m_0c$ and
$r(u)\to0$ as $u\downarrow0$. For $B>0$ let
$T_B(\delta)=\sup_{u\ge0}\{u\delta-Bu^5\}=\frac45(5B)^{-1/4}\delta^{5/4}$, attained at
$u_B=(\delta/5B)^{1/4}$.

Fix $\varepsilon\in(0,\frac12)$ and choose $\eta\in(0,1-m_0)$ with $|r|\le\varepsilon$ on
$(0,\eta]$; put $c_\eta=\Lamhat(m_0+\eta)>0$. For $u\ge\eta$, convexity of $\Lamhat$ and
$\Lamhat(m_0)=0$ give $\Lamhat(m_0+u)\ge(u/\eta)c_\eta$, so $u\delta-\Lamhat(m_0+u)<0$ when
$\delta<c_\eta/\eta$. For $u\in[0,\eta]$,
$A(1-\varepsilon)u^5\le\Lamhat(m_0+u)\le A(1+\varepsilon)u^5$. Take $\delta$ so small that
$\delta<c_\eta/\eta$ and $u_{A(1-\varepsilon)}\le\eta$. Then also
$u_{A(1+\varepsilon)}\le\eta$, and
\[
 0<T_{A(1+\varepsilon)}(\delta)\le S(\delta)\le T_{A(1-\varepsilon)}(\delta).
\]
Since $\varepsilon$ is arbitrary, $S(\delta)=\frac45(5m_0c)^{-1/4}\delta^{5/4}(1+o(1))$,
which is \eqref{eq:atom-odd-rate}. The closed form follows from
$5m_0c=\frac{9\azero^3m_0}{16\Dstar^3}$ and $(16/9)^{1/4}=2/\sqrt3$.

For the dual parameter, let $f=\feq+\delta$ with $\delta$ as above. Every
$\theta=m_0+u$ with $u\ge\eta$ has
$\Lamhat(\theta)-\theta\delta\ge-m_0\delta+u(c_\eta/\eta-\delta)>-m_0\delta\ge\Sigma(f)$,
so $\theta_f<m_0+\eta$, and $\theta_f\downarrow m_0$. By
Proposition~\ref{prop:atom-odd-structure}(c), $\theta_f\in(m_0,1)$ and $\Lambda$ is
differentiable at $\theta_f$; an interior minimizer of the convex function
$\vartheta\mapsto\Lambar(\vartheta)-\vartheta f$ satisfies $\Lambda'(\theta_f)=f$. By
\eqref{eq:atom-odd-Lambda}, $\delta=5m_0c(\theta_f-m_0)^4(1+o(1))$, which is the asymptotics
of $\theta_f$. For small $\delta$, $\theta_f\in(m_0,\theta_3)$, so $q_f=q_{\theta_f}$ is the
marginal first contact of Proposition~\ref{prop:atom-odd-structure}(d), and
Proposition~\ref{prop:atom-odd-onset}(b) gives the last two assertions.
\end{proof}

\subsection{The slope to the right of the jump}\label{sec:cubic-slope}

The constants have the same form as in Theorem~\ref{thm:local} once the slope of the
equilibrium distribution function is measured to the right of its jump at zero.

\begin{lemma}\label{lem:atom-odd-slope}
Let $\xi\in\Kclodd(\varepsilon_3)$ and $\alphas(t)=\zetas([0,t])$. Then
\[
 \alphas(t)=m_0+\rhostar t+o(t)\quad(t\downarrow0),\qquad \rhostar:=\frac{\Dstar}{2\sqrt{\azero}}
 \quad\text{(for $m_0=0$ this is \eqref{eq:ons-slope})}.
\]
Consequently $q_\theta/(\theta-m_0)\to3/(2\rhostar)$ and
$c=\frac9{640}\,\azero^{3/2}\rhostar^{-3}$.
\end{lemma}

\begin{proof}
Let $U_t,V_t,W_t$ be built from $\zetas$, and put $\mathcal U_2(t)=\E U_t^2$,
$\mathcal V_2(t)=\E V_t^2$ and $\mathcal U_3(t)=\E U_t^3$. For $0\le s\le t\le1$,
\begin{gather*}
 \mathcal U_2(t)-\mathcal U_2(s)=\int_s^ta\,\E\bigl(V_r^2-2\alphas(r)U_r^3\bigr)\,dr,\qquad
 \mathcal V_2(t)-\mathcal V_2(s)=\int_s^ta\,\E\bigl(W_r^2-6\alphas(r)U_rV_r^2\bigr)\,dr,\\
 \mathcal U_3(t)-\mathcal U_3(s)=\int_s^t3a\,\E\bigl(U_rV_r^2-\alphas(r)U_r^4\bigr)\,dr .
\end{gather*}
For measures with finitely many atoms these follow from It\^o's formula on each plateau,
where $dU_t=-ma U_t^2\,dt+\sqrt a\,V_t\,dB_t$ and
$dV_t=-3maU_tV_t\,dt+\sqrt a\,W_t\,dB_t$, and from continuity across the atoms; in general
they follow by approximation, using \eqref{eq:plat-jet-stability} and
Lemma~\ref{lem:plat-coupling}, as in the proof of Proposition~\ref{prop:ons-equilibrium}.
In the notation \eqref{eq:plat-NS}, write $N=N_{\zetas}=a'\mathcal U_2+a^2\mathcal V_2$ and
$\sigma=\sigma_{\zetas}=2a^2\mathcal U_3$; here $\sigma>0$ because
$\mathcal U_3\ge\mathcal U_2^{3/2}>0$. By Proposition~\ref{prop:atom-odd-structure}(b) and
Lemma~\ref{lem:ons-zero-in-support}, $\Gamma_*=\id$ on $[0,v_*]$, so $\Gamma_*'\equiv1$ on $[0,v_*)$ and the right derivative of
$\Gamma_*'$ vanishes there. By Lemma~\ref{lem:plat-ito}(b) this right derivative is
$N(t)-\alphas(t)\sigma(t)$, so
\[
 \alphas(t)=\frac{N(t)}{\sigma(t)}\qquad(0\le t<v_*).
\]
The integrands above are right continuous at $r=0$, so $\mathcal U_2$, $\mathcal V_2$ and
$\mathcal U_3$ have right derivatives at zero. With $X_0=0$, $V_0=0$, $\alphas(0)=m_0$ and
$a'(0)=\xi'''(0)$,
\[
 \begin{gathered}
 N(0)=a'(0)\Ustar^2,\qquad
 N'(0+)=a''(0)\Ustar^2-2\azero a'(0)m_0\Ustar^3+\azero^3\Wstar^2,\\
 \sigma(0)=2\azero^2\Ustar^3,\qquad
 \sigma'(0+)=4\azero a'(0)\Ustar^3-6\azero^3m_0\Ustar^4 .
 \end{gathered}
\]
Hence $N(0)/\sigma(0)=m_0$, and the right derivative of $\alphas$ at zero is
\[
 \frac{N'(0+)-m_0\sigma'(0+)}{\sigma(0)}
 =\frac{a''(0)\Ustar^2-6\azero a'(0)m_0\Ustar^3+\azero^3(\Wstar^2+6m_0^2\Ustar^4)}
       {2\azero^2\Ustar^3}
 =\frac{\Dstar}{2\azero^2\Ustar^3}=\frac{\Dstar}{2\sqrt{\azero}} .
\]
The two consequences follow from \eqref{eq:atom-odd-onset} and \eqref{eq:atom-odd-D}.
\end{proof}

On an interval $[t_1,t_2)$ of the support, smoothness of $\alphas$ and the quotient formula
$\alphas=N_{\zetas}/\sigma_{\zetas}$ are \cite[Theorem~2(ii) and its proof]{AC15a}; with
$t_1=0$ this also gives the right derivative of $\alphas$ at zero, where $\zetas$ has its atom,
and Lemma~\ref{lem:atom-odd-slope} computes its value. With it,
Proposition~\ref{prop:atom-odd-onset}(b) and the law for $p$ in \eqref{eq:atom-odd-p} read as
Theorem~\ref{thm:local}(a),(b) with $\theta$ replaced by $\theta-m_0$ (the laws for $\Lambda$
and $\Sigma$ change, Section~\ref{sec:atom-odd}):
$q_\theta\sim\frac3{2\rhostar}(\theta-m_0)$, an atom $\frac12(\theta-m_0)$ at $q_\theta$,
and $p(\theta)-\feq\sim\frac9{640}\azero^{3/2}\rhostar^{-3}(\theta-m_0)^5$. In particular
$\zeta_\theta([0,q_\theta])=\theta+\frac12(\theta-m_0)(1+o(1))$ and
$\alphas(q_\theta)=m_0+\rhostar q_\theta+o(q_\theta)$ agree up to $o(\theta-m_0)$: to leading
order the constrained distribution function rejoins $\alphas$ at its first contact, as drawn
in Figure~\ref{fig:bottom}(c).

\section{The de Almeida--Thouless threshold: proof of
\texorpdfstring{Proposition~\ref*{prop:nc-dat-threshold}}%
{Proposition \getrefnumber{prop:nc-dat-threshold}}}\label{app:dat}

This appendix proves Proposition~\ref{prop:nc-dat-threshold} for every
$\xi\in\Kcl(\varepsilon_*)$. By Lemma~\ref{lem:nc-regular}, the positive support of a
constrained minimizer is either an interval that starts at a marginal contact or a single point
$r$, and then the minimizer is $\zeta^{\theta,r}=\theta\delta_0+(1-\theta)\delta_r$. For such a
measure, $\Gamma$ and $\Gamma'$ at $r$ are explicit Gaussian integrals
(Lemma~\ref{lem:C-two-atom}), so the contact conditions at $r$ are the replica-symmetric
fixed-point equation $r=\Tfrak(\xi'(r),\theta)$ and the de Almeida--Thouless inequality
$\mathfrak A(\theta)\le1$. The fixed-point equation has exactly one small root $\qrs_\theta$,
and $\mathfrak A-1$ changes sign exactly once along it, at $\theta_c$ (Lemma~\ref{lem:C-rs}),
uniformly over the class. Conversely,
if $\mathfrak A(\theta)\le1$, the obstacle of $\zeta^{\theta,\qrs_\theta}$ has its strict
minimum at $\qrs_\theta$, and convexity of the Parisi functional makes this measure the
constrained minimizer (Lemma~\ref{lem:C-sufficiency}). No monotonicity of
$\theta\mapsto\zeta_\theta$ is used; the only monotone quantity is the scalar
$\mathfrak W=(\mathfrak A-1)/\qrs_\theta$, whose sign is that of $\mathfrak A-1$.
Section~\ref{sec:dat-constants} gives the bounds that the nonnegative coefficients of $\xi$
provide, proves Lemma~\ref{lem:nc-regular} and fixes the constants, in the order $K$, then
$(\mathcal W,\varepsilon_1)$, then $\varepsilon_*$, for this appendix and
Appendix~\ref{app:stability}. Section~\ref{sec:C-two-atom} computes the two-atom measures,
Section~\ref{sec:C-reduction} states the two remaining lemmas and proves the proposition, and
Sections~\ref{sec:C-branch} and~\ref{sec:C-sufficiency} prove these lemmas.

\subsection{The class and the choice of constants}\label{sec:dat-constants}

\begin{lemma}[Elementary bounds on the class]\label{lem:dat-class-bounds}
Let $0<\varepsilon\le\frac14$, $\xi\in\Kcl(\varepsilon)$, $a=\xi''$ and $\tau=\azero-1$. Then:
\begin{steplist}
\item $0<\tau<\varepsilon$, and $1<\azero\le a(t)\le1+\varepsilon$ for $t\in[0,1]$;
\item $0\le a'(t)\le\varepsilon t$, $0\le a''(t)\le\varepsilon$ and
 $0\le a(t)-\azero\le\varepsilon t^2/2$ for $t\in[0,1]$;
\item $0\le\xi^{(5)}(t)\le\varepsilon$ and $0\le a'(t)-\xi''''(0)t\le\varepsilon t^2/2$ for
 $t\in[0,\frac12]$;
\item the functions $\ell(q)=\int_0^1a(sq)\,ds$ and $\chi(q)=\int_0^1v\,\xi'''(vq)\,dv$
 satisfy $\xi'(q)=q\ell(q)$ and $\xi''(q)=\ell(q)+q\chi(q)$ on $[0,1]$, $\ell(0)=\azero$,
 $\ell'(0)=0$, $\chi(0)=0$, $\chi'(0)=\xi''''(0)/3$, $\chi\ge0$, $1<\ell\le1+\varepsilon$, and
 $|\ell'|,|\ell''|,|\chi|,|\chi'|,|\chi''|\le\varepsilon$ on $[0,\frac12]$.
\end{steplist}
\end{lemma}

\begin{proof}
All derivatives of $\xi$ are nonnegative and nondecreasing on $[0,1]$. Since $\xic''\equiv1$
and $\xic'''=\xic''''\equiv0$, the $C^4$ condition gives $|a-1|<\varepsilon$,
$0\le\xi'''<\varepsilon$ and $0\le\xi''''<\varepsilon$ on $[0,1]$. This gives (i), because
$\azero>1$ and $a$ is nondecreasing. Since $\xi'''(0)=6\beta_3^2=0$,
$a'(t)=\int_0^t\xi''''\le\varepsilon t$ and $a(t)-\azero=\int_0^ta'\le\varepsilon t^2/2$,
which is (ii). For (iii), compare the series termwise: for $p\ge5$ and $0\le t\le\frac12$,
$(p-4)t^{p-5}\le(p-4)2^{5-p}\le1$, so
\[
 \xi^{(5)}(t)=\sum_{p\ge5}\beta_p^2\,p(p-1)(p-2)(p-3)(p-4)\,t^{p-5}
 \le\sum_{p\ge4}\beta_p^2\,p(p-1)(p-2)(p-3)=\xi''''(1)<\varepsilon .
\]
Then $a'(t)-\xi''''(0)t=\int_0^t\int_0^s\xi^{(5)}$ gives the second bound. For (iv),
$\xi'(q)=\int_0^qa=q\ell(q)$, and Fubini's theorem gives
$\xi''(q)-\ell(q)=\int_0^1\int_{sq}^q\xi'''(u)\,du\,ds=q\,\chi(q)$. Differentiation under the
integral gives $\ell^{(j)}(q)=\int_0^1s^j\xi^{(j+2)}(sq)\,ds$ and
$\chi^{(j)}(q)=\int_0^1v^{j+1}\xi^{(j+3)}(vq)\,dv$, and the bounds follow from (i)--(iii). The
values at zero use $\xi'''(0)=0$.
\end{proof}

\begin{proof}[Proof of Lemma~\ref{lem:nc-regular}]
By Proposition~\ref{prop:nc-class}(b), \condH\ holds, so $\zetas(\{0\})=0$ by
Proposition~\ref{prop:ons-equilibrium}(b), and at every $\theta\in(0,1)$
Lemma~\ref{lem:cons-variational}(c),(e) applies: $\zeta_\theta(\{0\})=\theta$, and $C_\theta$ is a nonempty compact
subset of $(0,1)$ that contains $\supp\zeta_\theta\cap(0,1]$. By item~\ref{item:nc-confined} of
the choice of constants in Section~\ref{sec:nc-reference}, $\zeta_\theta\in\mathcal W$ and
$C_\theta\subset(0,K)$. Every point of $C_\theta$ is a global, hence local, minimum point of
$\Obs_{\zeta_\theta}$, so Lemma~\ref{lem:nc-gap}(a) gives $\zeta_\theta((u,v))>0$ whenever
$u<v$ lie in $C_\theta$.

First, $C_\theta$ is an interval. Otherwise some $t\in(q_\theta,\max C_\theta)$ is not in
$C_\theta$; then $u=\max(C_\theta\cap[q_\theta,t])$ and $v=\min(C_\theta\cap[t,1])$ satisfy
$u<t<v$ and $(u,v)\cap C_\theta=\emptyset$, so $\zeta_\theta((u,v))=0$, a contradiction. Write
$C_\theta=[q_\theta,v_\theta]$. Second, $\supp\zeta_\theta\cap(0,1]=C_\theta$. The positive
support is contained in $C_\theta$. If $q_\theta<v_\theta$, every interval with endpoints in
$C_\theta$ has positive mass, so every point of $(q_\theta,v_\theta)$, and by closedness every
point of $C_\theta$, is a support point. If $q_\theta=v_\theta$, the positive support is
nonempty, because $\zeta_\theta((0,1])=1-\theta>0$, and is contained in $\{q_\theta\}$. Since
$\zeta_\theta(\{0\})=\theta>0$, this proves $\supp\zeta_\theta=\{0\}\cup[q_\theta,v_\theta]$
and the two-atom form when $v_\theta=q_\theta$. Finally let $q_\theta<v_\theta$. Every
$t\in[q_\theta,v_\theta]\subset(0,1)$ is a minimum point of $\Obs_{\zeta_\theta}$, so
$\Gamma_\theta(t)=t$ by Lemma~\ref{lem:plat-obstacle}(b). Since $\Gamma_\theta$ is $C^1$,
$\Gamma_\theta'(q_\theta)=1$, that is,
$\xi''(q_\theta)\E\Phi_{\zeta_\theta,xx}(q_\theta,X_{q_\theta})^2=1$.
\end{proof}

\paragraph{Choice of constants.}
Put $\eta=\frac1{40}$. Let $K_{\rm rs}$ and $\tau_0$ be the absolute constants of
Lemma~\ref{lem:C-rs}. Appendix~\ref{app:stability} uses finitely many constants, each written
$C_0$, which depend only on the ball $\mathfrak B=\{\xi:\normCfour{\xi-\xic}<\frac14\}$, not
on $K$, the measure or $\theta$.
\emph{First}, fix $K>0$ with $K\le\min\{K_P,\frac1{20},K_{\rm rs}\}$, where $K_P$ is the
constant of Lemma~\ref{lem:nc-positive-third} (so that the choice of constants in
Section~\ref{sec:nc-reference} can be made with this $K$), and $C_0K\le\eta$ for each of
these constants $C_0$. Then
$2C_0K\le\frac14$, and the slope condition $\frac{3-C_0K}3-C_0K\ge\frac12$ used in
Appendix~\ref{sec:D-proof} holds. \emph{Second}, with this $K$ choose $\mathcal W$ and
$\varepsilon_1$ as in Section~\ref{sec:nc-reference}. Then intersect $\mathcal W$ with the weak
neighbourhood $\{\zeta:\Wone(\zeta,\delta_0)<r\}$ of $\delta_0$, where $r>0$ is the radius chosen
in Appendix~\ref{sec:D-setting} to obtain condition~\ref{item:D-K2} there; it depends only on
$\mathfrak B$ and $\eta$, not on $K$. Decrease $\varepsilon_1$ to at most $r$ and, using
Lemma~\ref{lem:nc-confinement}(a), further so that items~\ref{item:nc-positive}
and~\ref{item:nc-confined} of Section~\ref{sec:nc-reference} hold for the new $\mathcal W$. From
now on $K$, $\mathcal W$ and $\varepsilon_1$ denote these values.
\emph{Third}, since $\Wone$ metrizes the weak topology on $\Pspace([0,1])$, fix
$\varepsilon_*\in(0,\varepsilon_1]$ with $\varepsilon_*\le\min\{\eta,\tau_0,K/2\}$ such that
every probability measure on $[0,1]$ within $\Wone$-distance $2\varepsilon_*$ of $\delta_0$ lies
in $\mathcal W$. Theorem~\ref{thm:nearcrit}(a) holds with this $\varepsilon_*$, since
$\varepsilon_*\le\varepsilon_1$. Now let $\xi\in\Kcl(\varepsilon_*)$ and $\tau=\azero-1$. Then:
\begin{enumerate}[label=\textup{(N\arabic*)},ref=(N\arabic*),leftmargin=*,widest=4,
  topsep=3pt,itemsep=2pt]
\item\label{item:dat-C1} $0<\tau<\varepsilon_*\le\tau_0$ and $\xi\in\Kcl(\frac14)$, so
 Lemma~\ref{lem:C-rs} applies to $\xi$; moreover
 $2\tau<2\varepsilon_*\le K\le\min\{K_{\rm rs},\frac1{20}\}$;
\item\label{item:dat-C2} $\xi\in\Kcl(\varepsilon_1)$, so Propositions~\ref{prop:nc-support}
 and~\ref{prop:nc-class}, Lemma~\ref{lem:nc-regular}, and items~\ref{item:nc-positive}
 and~\ref{item:nc-confined} of the choice of constants apply to $\xi$;
\item\label{item:dat-C3} $\xi''\ge\azero>1$ on $[0,1]$ and $0\le\xi'''(t)\le\varepsilon_*t$ for
 $t\in[0,1]$, by Lemma~\ref{lem:dat-class-bounds}(i),(ii) with $\varepsilon=\varepsilon_*$;
\item\label{item:dat-C4} $\zeta^{\theta,r}\in\mathcal W$ for $\theta\in[0,1]$ and
 $r\in(0,2\tau]$, since
 $\Wone(\zeta^{\theta,r},\delta_0)=\int_0^1\zeta^{\theta,r}((t,1])\,dt=(1-\theta)r
 \le2\tau<2\varepsilon_*$.
\end{enumerate}

\paragraph{Notation.}
$\E_{w,\theta}$, $\Tfrak$ and $\Sfrak$ ($w\ge0$, $\theta\in[0,1]$) are as in
Section~\ref{sec:nc-dat}, and $\tau$, $\ell$, $\chi$ as in Lemma~\ref{lem:dat-class-bounds}. For a
fixed covariance, $\theta\in[0,1]$ and $r\in(0,1)$ we write
$\Gamma^{\theta,r}=\Gamma_{\zeta^{\theta,r}}$; superscripts refer to the two-atom measure, not
to the covariance.

\subsection{The two-atom measures}\label{sec:C-two-atom}

Above $r$ the Parisi equation of $\zeta^{\theta,r}$ is solved by a shifted $\log\cosh$, and below
$r$ the plateau law (Lemma~\ref{lem:plat-law}) identifies the law of $X_r$; this gives
$\Gamma^{\theta,r}$ near $r$ in closed form.

\begin{lemma}[Two-atom measures]\label{lem:C-two-atom}
Let $\xi(t)=\sum_{p\ge2}\beta_p^2t^p$ have radius of convergence greater than one and
$\beta_2>0$, let $\theta\in[0,1]$ and $r\in(0,1)$, and let $X$ be the diffusion
\eqref{eq:intro-diffusion} of $\zeta^{\theta,r}$. Then
\[
 \Gamma^{\theta,r}(r)=\Tfrak\bigl(\xi'(r),\theta\bigr),\qquad
 (\Gamma^{\theta,r})'(r)=\xi''(r)\,\E\Phi_{\zeta^{\theta,r},xx}(r,X_r)^2
 =\xi''(r)\,\Sfrak\bigl(\xi'(r),\theta\bigr),
\]
and the right second derivative satisfies
\begin{equation}\label{eq:C-right-second}
 (\Gamma^{\theta,r})''(r+)\le\xi'''(r)+\xi''(r)^2\bigl(10\,\Gamma^{\theta,r}(r)-2\bigr).
\end{equation}
\end{lemma}

\begin{proof}
Write $\Phi=\Phi_{\zeta^{\theta,r}}$. On $[r,1]$ the distribution function equals one, and
$\Phi(t,x)=\log\cosh x+\frac12(\xi'(1)-\xi'(t))$ solves the Parisi equation
\eqref{eq:intro-pde} there, because $\sech^2+\tanh^2=1$. Hence $\Phi_x(r,\cdot)=\tanh$,
$U:=\Phi_{xx}(r,\cdot)=\sech^2$ and $V:=\Phi_{xxx}(r,\cdot)=-2\sech^2\tanh$. The measure has
no mass in $(0,r)$ and mass $\theta$ at zero, so by \eqref{eq:plat-law} the law of $X_r$ has
density proportional to $e^{\theta\Phi(r,y)}\propto\cosh^\theta y$ with respect to
$\mathcal N(0,\xi'(r))$. Since a probability density integrates to one,
$\E F(X_r)=\E_{\xi'(r),\theta}F$ for every continuous $F$ with $|F(x)|\le Ce^{|x|}$. Taking
$F=\tanh^2$ gives the first identity. The second follows from
$(\Gamma^{\theta,r})'(r)=\xi''(r)\E U(X_r)^2$ (Lemma~\ref{lem:plat-ito}(a)) with
$U^2=\sech^4$.

By Lemma~\ref{lem:plat-ito}(b), the right derivative of $(\Gamma^{\theta,r})'$ at $r$ carries
the mass $\zeta^{\theta,r}([0,r])=1$:
\[
 (\Gamma^{\theta,r})''(r+)=\xi'''(r)\,\E U(X_r)^2+\xi''(r)^2\,\E\bigl(V(X_r)^2-2U(X_r)^3\bigr).
\]
Now $U^2\le1$, $V^2=4\sech^4\tanh^2\le4\tanh^2$, and
$U^3=(1-\tanh^2)^3\ge1-3\tanh^2$, because $(1-y)^3-(1-3y)=y^2(3-y)\ge0$ for $y\in[0,1]$.
Since $\E\tanh^2X_r=\Gamma^{\theta,r}(r)$ and $\xi'''\ge0$ on $[0,1]$ (the coefficients of
$\xi$ are nonnegative), this proves \eqref{eq:C-right-second}.
\end{proof}

The bound \eqref{eq:C-right-second} is needed only at a marginal atom, which below occurs only
for $\theta=\theta_c$: there $(\Gamma^{\theta,r})'(r)=1$, and first-order information alone
would not exclude that $\Gamma^{\theta,r}$ rises above the identity just after $r$, in which
case $r$ would not be a contact; \eqref{eq:C-right-second} excludes it.

\subsection{Reduction to the replica-symmetric branch}\label{sec:C-reduction}

The first lemma analyses the two conditions along the class. It involves no Parisi measure and
is proved in Section~\ref{sec:C-branch}, where the function $\varpi$ and the absolute constant
$\overline C\ge1$ that enter it are defined (Lemma~\ref{lem:dat-branch}).

\begin{lemma}[The replica-symmetric branch]\label{lem:C-rs}
Let $\overline C\ge1$ be the absolute constant of Lemma~\ref{lem:dat-branch}, and put
$K_{\rm rs}=\min\{\frac14,1/(2\overline C)\}$ and
$\tau_0=\min\{\frac12K_{\rm rs},1/(12\overline C^2)\}$. There is an absolute constant $C$ such
that the following holds for every $\xi\in\Kcl(\frac14)$ with $\tau=\azero-1\le\tau_0$.
\begin{thmlist}
\item For each $\theta\in[0,1]$ the equation $q=\Tfrak(\xi'(q),\theta)$ has exactly one solution
 $\qrs_\theta$ in $(0,K_{\rm rs}]$. It satisfies $\tau/\overline C\le\qrs_\theta\le2\tau$, and
 $\theta\mapsto\qrs_\theta$ is continuously differentiable on $[0,1]$ \textup{(}one-sided at
 the endpoints\textup{)}.
\item Put $\mathfrak A(\theta)=\xi''(\qrs_\theta)\,\Sfrak(\xi'(\qrs_\theta),\theta)$. Then
 $\mathfrak A(\theta)-1=\qrs_\theta\,\mathfrak W(\theta)$ with
 $\mathfrak W(\theta)=\varpi(\qrs_\theta,\theta)$; $\mathfrak W$ is strictly decreasing on
 $[0,1]$, $\mathfrak W(0)>0$ and $\mathfrak W(1)<0$. Hence there is a unique
 $\theta_c\in(0,1)$ with $\mathfrak A>1$ on $[0,\theta_c)$, $\mathfrak A(\theta_c)=1$ and
 $\mathfrak A<1$ on $(\theta_c,1]$.
\item $\bigl|\theta_c-\frac23\bigl(1+\frac14\xi''''(0)\bigr)\tau\bigr|\le C\tau^2$.
\end{thmlist}
\end{lemma}

The second lemma is the sufficiency half of the threshold; it is proved in
Section~\ref{sec:C-sufficiency}.

\begin{lemma}[Sufficiency of the de Almeida--Thouless condition]\label{lem:C-sufficiency}
Let $\xi\in\Kcl(\varepsilon_*)$, $\theta\in(0,1)$ and $r\in(0,2\tau]$, write
$\zeta=\zeta^{\theta,r}$ and $h(t)=\Gamma^{\theta,r}(t)-t$, and assume $h(r)=0$ and
$h'(r)\le0$. Then $h>0$ on $(0,r)$ and $h<0$ on $(r,1]$. Consequently
$\Obs_\zeta(t)>\Obs_\zeta(r)$ for every $t\in[0,1]\setminus\{r\}$, $\zeta_\theta=\zeta$, and
the positive contact set of $\zeta_\theta$ is $\{r\}$.
\end{lemma}

\begin{proof}[Proof of Proposition~\ref{prop:nc-dat-threshold}]
Let $\xi\in\Kcl(\varepsilon_*)$, so that \ref{item:dat-C1}--\ref{item:dat-C4} hold.

\emph{Parts~(a), (b) and~(e).} By \ref{item:dat-C1}, Lemma~\ref{lem:C-rs} applies to $\xi$.
Its part~(a) gives~(a), with the absolute constant $K_{\rm rs}\in(0,\frac14]$; its part~(b)
gives~(b); and its part~(c) is \eqref{eq:nc-dat-asymptotics}, with an absolute constant $C$.

\emph{Part~(c).} Let $0<\theta<\theta_c$, so that $\mathfrak A(\theta)>1$. By
\ref{item:dat-C2}, Lemma~\ref{lem:nc-regular} applies: the positive contact set of
$\zeta_\theta$ is an interval $[q_\theta,v_\theta]\subset(0,K)$, and
$\supp\zeta_\theta=\{0\}\cup[q_\theta,v_\theta]$. Suppose that $v_\theta=q_\theta$, and put
$r=q_\theta$. Lemma~\ref{lem:nc-regular} then gives $\zeta_\theta=\zeta^{\theta,r}$. As in the
proof of Lemma~\ref{lem:nc-regular}, Lemma~\ref{lem:cons-variational}(e) applies to
$\theta$, and the contact conditions \eqref{eq:cons-contact} at the first positive contact give
$\Gamma^{\theta,r}(r)=r$ and $(\Gamma^{\theta,r})'(r)\le1$. By Lemma~\ref{lem:C-two-atom} the
first condition says $r=\Tfrak(\xi'(r),\theta)$. Since $0<r<K\le K_{\rm rs}$ by
\ref{item:dat-C1}, the uniqueness in~(a) gives $r=\qrs_\theta$, and then the second condition
says $\mathfrak A(\theta)\le1$, a contradiction. Hence $v_\theta>q_\theta$. By
Lemma~\ref{lem:nc-regular}, $\Gamma_\theta(t)=t$ on $[q_\theta,v_\theta]$ and the first positive
contact is marginal, that is,
$\xi''(q_\theta)\E\Phi_{\zeta_\theta,xx}(q_\theta,X_{q_\theta})^2=\Gamma_\theta'(q_\theta)=1$;
moreover $0<q_\theta<v_\theta<K\le\frac1{20}$.

\emph{Part~(d).} Let $\theta_c\le\theta<1$, so that $\theta>0$, and put
$r=\qrs_\theta\in(0,2\tau]$. By Lemma~\ref{lem:C-two-atom} and~(a),
$\Gamma^{\theta,r}(r)=\Tfrak(\xi'(r),\theta)=r$ and
$(\Gamma^{\theta,r})'(r)=\mathfrak A(\theta)\le1$, with equality only for $\theta=\theta_c$,
by~(b). In the notation of Lemma~\ref{lem:C-sufficiency}, $h(r)=0$ and $h'(r)\le0$, so that
lemma gives $\zeta_\theta=\zeta^{\theta,r}$ with positive contact set $\{r\}$. By
Lemma~\ref{lem:C-two-atom}, the Plefka quantity \eqref{eq:intro-plefka} at this contact is
\[
 1-\xi''(r)\,\E\Phi_{\zeta_\theta,xx}(r,X_r)^2=1-(\Gamma^{\theta,r})'(r)=1-\mathfrak A(\theta),
\]
which vanishes at $\theta=\theta_c$ and is positive for $\theta>\theta_c$, by~(b).
\end{proof}

\subsection{Gaussian jets and the replica-symmetric branch}\label{sec:C-branch}

We prove Lemma~\ref{lem:C-rs}. The fixed-point equation divided by $q$, and the difference of
the two conditions divided by $q$ once more, are smooth functions of $(q,\theta)$ up to $q=0$,
with jets at $q=0$ given by the heat equation. The covariance enters only through $\ell$ and
$\chi$, which Lemma~\ref{lem:dat-class-bounds} controls uniformly over the class; Taylor's
formula then locates the root and the threshold.

\begin{lemma}[Gaussian jets]\label{lem:C-moments}
The functions $\Tfrak$ and $\Sfrak$ are $C^\infty$ on $[0,\infty)\times[0,1]$, with one-sided
derivatives on the boundary, and so are
\[
 \mathfrak R(w,\theta)=\frac{\Tfrak(w,\theta)}w,\qquad
 \mathfrak D(w,\theta)=\frac{\Sfrak(w,\theta)-\mathfrak R(w,\theta)}w ,
\]
extended continuously to $w=0$. For every $\theta\in[0,1]$,
\begin{equation}\label{eq:C-jets}
 \mathfrak R(0,\theta)=1,\qquad \partial_w\mathfrak R(0,\theta)=\theta-2,\qquad
 \mathfrak D(0,\theta)=-\theta,\qquad
 \partial_w\mathfrak D(0,\theta)=\tfrac43+3\theta-\theta^2 .
\end{equation}
\end{lemma}

\begin{proof}
Let $g(x,\theta)$ be smooth on $\R\times[0,1]$ with
$|\partial_x^k\partial_\theta^jg(x,\theta)|\le C_{k,j}(1+|x|)^je^{|x|}$ for all $k,j\ge0$, and
put $M_g(w,\theta)=\E g(\sqrt wZ,\theta)$ with $Z$ standard Gaussian. For $w>0$, write $M_g$ as
the integral of $g$ against the $\mathcal N(0,w)$ density $\varphi_w$. Since
$\partial_w\varphi_w=\frac12\varphi_w''$, differentiation under the integral and two
integrations by parts give $\partial_wM_g=\frac12M_{\partial_x^2g}$; the boundary terms vanish
because $\varphi_w$ has Gaussian decay. The right side is continuous on
$[0,\infty)\times[0,1]$ by dominated convergence, so
$M_g(w,\theta)-M_g(0,\theta)=\frac12\int_0^wM_{\partial_x^2g}(s,\theta)\,ds$, and the identity
$\partial_wM_g=\frac12M_{\partial_x^2g}$ extends to $w=0$. Differentiation in $\theta$ under the
expectation is also dominated. Iterating, every mixed partial derivative of $M_g$ exists and is
continuous, and
\begin{equation}\label{eq:C-heat-jets}
 \partial_w^kM_g(0,\theta)=2^{-k}\,\partial_x^{2k}g(0,\theta)\qquad(k\ge0).
\end{equation}
The functions $\cosh^\theta x$, $\tanh^2x\cosh^\theta x$ and $\sech^4x\cosh^\theta x$ satisfy
the growth hypothesis: each $x$-derivative of $\cosh^\theta x$ is a polynomial in
$(\theta,\tanh x)$ times $\cosh^\theta x\le e^{|x|}$, and each $\theta$-derivative contributes a
factor $\log\cosh x$, which lies in $[0,|x|]$. Since $\E\cosh^\theta(\sqrt wZ)\ge1$, the
functions $\Tfrak$ and $\Sfrak$ are smooth quotients. Moreover $\Tfrak(0,\theta)=0$ and
$\Sfrak(0,\theta)=1$. Hence $\mathfrak R(w,\theta)=\int_0^1\partial_w\Tfrak(sw,\theta)\,ds$ is
smooth, with $\mathfrak R(0,\theta)=\partial_w\Tfrak(0,\theta)$, and the same representation
applies to $\mathfrak D$ once $\Sfrak(0,\theta)=\mathfrak R(0,\theta)$ is checked below.

By \eqref{eq:C-heat-jets}, an even function with Taylor coefficients $g_0,g_2,g_4,g_6$ at
$x=0$ has $M_g(w)=g_0+g_2w+3g_4w^2+15g_6w^3+O(w^4)$ at fixed $\theta$. From
\begin{gather*}
 \log\cosh x=\tfrac{x^2}2-\tfrac{x^4}{12}+O(x^6),\qquad
 \tanh^2x=x^2-\tfrac23x^4+\tfrac{17}{45}x^6+O(x^8),\\
 \sech^4x=1-2x^2+\tfrac73x^4+O(x^6),
\end{gather*}
we get $\cosh^\theta x=1+\frac\theta2x^2+(\frac{\theta^2}8-\frac\theta{12})x^4+O(x^6)$, and
therefore
\begin{align*}
 \E\cosh^\theta(\sqrt wZ)
 &=1+\tfrac\theta2\,w+\bigl(\tfrac{3\theta^2}8-\tfrac\theta4\bigr)w^2+O(w^3),\\
 \E\bigl[\tanh^2\cosh^\theta\bigr](\sqrt wZ)
 &=w+\bigl(\tfrac{3\theta}2-2\bigr)w^2
   +\bigl(\tfrac{15\theta^2}8-\tfrac{25\theta}4+\tfrac{17}3\bigr)w^3+O(w^4),\\
 \E\bigl[\sech^4\cosh^\theta\bigr](\sqrt wZ)
 &=1+\bigl(\tfrac\theta2-2\bigr)w+\bigl(\tfrac{3\theta^2}8-\tfrac{13\theta}4+7\bigr)w^2+O(w^3).
\end{align*}
Dividing by the first series,
\[
 \Tfrak(w,\theta)=w+(\theta-2)w^2+\bigl(\theta^2-5\theta+\tfrac{17}3\bigr)w^3+O(w^4),\qquad
 \Sfrak(w,\theta)=1-2w+(7-2\theta)w^2+O(w^3).
\]
Thus $\mathfrak R=1+(\theta-2)w+(\theta^2-5\theta+\frac{17}3)w^2+O(w^3)$, so
$\Sfrak(0,\theta)=\mathfrak R(0,\theta)=1$, and
$\mathfrak D=-\theta+(\frac43+3\theta-\theta^2)w+O(w^2)$. All functions involved are smooth, so
these Taylor coefficients are the values in \eqref{eq:C-jets}.
\end{proof}

\paragraph{The two conditions as smooth functions.}
For $\xi\in\Kcl(\frac14)$, $q\in[0,\frac14]$ and $\theta\in[0,1]$ define
\begin{align*}
 \Upsilon(q,\theta)&=\ell(q)\,\mathfrak R\bigl(\xi'(q),\theta\bigr)-1,\\
 \varpi(q,\theta)&=\ell(q)^2\,\mathfrak D\bigl(\xi'(q),\theta\bigr)
 +\chi(q)\,\Sfrak\bigl(\xi'(q),\theta\bigr),
\end{align*}
with $\ell$, $\chi$ as in Lemma~\ref{lem:dat-class-bounds}(iv). For $bt^2+et^4$ one has
$\ell(q)=2b+4eq^2$ and $\chi(q)=8eq$; for the Sherrington--Kirkpatrick covariance
$\ell\equiv\beta^2$ and $\chi\equiv0$. Since $\xi'(q)=q\ell(q)\in[0,\frac5{16}]$, the arguments
of $\mathfrak R$, $\mathfrak D$, $\Sfrak$ lie in the compact set $[0,\frac12]\times[0,1]$, where
these functions are smooth by Lemma~\ref{lem:C-moments}.

\needspace{4\baselineskip}
\begin{lemma}[The branch functions]\label{lem:dat-branch}
There is an absolute constant $\overline C\ge1$ such that, for every $\xi\in\Kcl(\frac14)$, all
partial derivatives of $\Upsilon$ and $\varpi$ in $(q,\theta)$ of orders one and two are
bounded by $\overline C$ on $[0,\frac14]\times[0,1]$. For $0<q\le\frac14$,
\begin{align}
 \frac{\Tfrak(\xi'(q),\theta)}q-1&=\Upsilon(q,\theta),\label{eq:C-rs-identity}\\
 \xi''(q)\,\Sfrak(\xi'(q),\theta)-\frac{\Tfrak(\xi'(q),\theta)}q
 &=q\,\varpi(q,\theta).\label{eq:C-at-identity}
\end{align}
At $q=0$,
\begin{equation}\label{eq:C-zero-jets}
\begin{gathered}
 \Upsilon(0,\theta)=\tau,\qquad
 \partial_q\Upsilon(0,\theta)=-\azero^2(2-\theta),\qquad
 \partial_\theta\Upsilon(0,\theta)=0,\\
 \varpi(0,\theta)=-\azero^2\theta,\qquad
 \partial_\theta\varpi(0,\theta)=-\azero^2,\qquad
 \partial_q\varpi(0,\theta)=\azero^3\bigl(\tfrac43+3\theta-\theta^2\bigr)+\tfrac13\xi''''(0) .
\end{gathered}
\end{equation}
\end{lemma}

So the fixed-point equation is $\Upsilon=0$, and on its solutions the de Almeida--Thouless
quantity minus one is $q\varpi$.

\begin{proof}
By the chain rule, each derivative in question is a polynomial with universal coefficients in
the derivatives of order at most two of $\mathfrak R$, $\mathfrak D$, $\Sfrak$ at
$(\xi'(q),\theta)$ and in $\ell,\ell',\ell'',\chi,\chi',\chi'',\xi''(q),\xi'''(q)$. The former
are bounded by absolute constants on $[0,\frac12]\times[0,1]$, and the latter by
Lemma~\ref{lem:dat-class-bounds} with $\varepsilon=\frac14$. For the identities,
$\Tfrak(\xi'(q),\theta)=\xi'(q)\mathfrak R(\xi'(q),\theta)=q\ell(q)\mathfrak R(\xi'(q),\theta)$,
and, with $\xi''=\ell+q\chi$ and $\Sfrak-\mathfrak R=\xi'(q)\mathfrak D=q\ell\mathfrak D$,
\[
 \xi''\Sfrak-\ell\mathfrak R=\ell(\Sfrak-\mathfrak R)+q\chi\Sfrak
 =q\bigl(\ell^2\mathfrak D+\chi\Sfrak\bigr),
\]
where $\Sfrak$, $\mathfrak R$, $\mathfrak D$ are evaluated at $(\xi'(q),\theta)$. The jets
follow from \eqref{eq:C-jets}, $\Sfrak(0,\theta)=1$, $\frac{d}{dq}\xi'(q)=\xi''(q)$, and
$\ell(0)=\azero$, $\ell'(0)=0$, $\chi(0)=0$, $\chi'(0)=\xi''''(0)/3$.
\end{proof}

\begin{proof}[Proof of Lemma~\ref{lem:C-rs}]
Let $\xi\in\Kcl(\frac14)$ with $\tau\le\tau_0$. Since $\tau_0\le\frac12K_{\rm rs}$, we have
$2\tau\le K_{\rm rs}$.

(a) By \eqref{eq:C-zero-jets}, $\partial_q\Upsilon(0,\theta)\le-\azero^2<-1$. Since the first
and second derivatives of $\Upsilon$ are bounded by $\overline C$ and
$K_{\rm rs}\le1/(2\overline C)$,
$-\overline C\le\partial_q\Upsilon\le-\frac12$ on $[0,K_{\rm rs}]\times[0,1]$. Hence
\[
 \tau-\overline Cq\le\Upsilon(q,\theta)\le\tau-\tfrac q2\qquad(0\le q\le K_{\rm rs}).
\]
Since $2\tau\le K_{\rm rs}$, the strictly decreasing function $\Upsilon(\cdot,\theta)$ has
exactly one zero in $(0,K_{\rm rs}]$, and it lies in $[\tau/\overline C,2\tau]$. By
\eqref{eq:C-rs-identity} this zero is the only solution $\qrs_\theta$ of
$q=\Tfrak(\xi'(q),\theta)$ in $(0,K_{\rm rs}]$. Since $\partial_q\Upsilon\ne0$, the implicit
function theorem and uniqueness make $\theta\mapsto\qrs_\theta$ continuously differentiable,
with $d\qrs_\theta/d\theta=-\partial_\theta\Upsilon/\partial_q\Upsilon$ at
$(\qrs_\theta,\theta)$; at $\theta\in\{0,1\}$ the same formula for the one-sided derivative
follows from the mean value theorem on the convex set $[0,K_{\rm rs}]\times[0,1]$.

(b) At $q=\qrs_\theta$, \eqref{eq:C-rs-identity} gives $\Tfrak(\xi'(q),\theta)/q=1$, and
\eqref{eq:C-at-identity} gives
\[
 \mathfrak A(\theta)-1=\qrs_\theta\,\mathfrak W(\theta),\qquad
 \mathfrak W(\theta)=\varpi(\qrs_\theta,\theta).
\]
Since $\qrs_\theta>0$, the sign of $\mathfrak A-1$ is the sign of $\mathfrak W$, and it suffices
to show that $\mathfrak W$ is strictly decreasing on $[0,1]$ with
$\mathfrak W(0)>0>\mathfrak W(1)$. Because $\partial_\theta\Upsilon(0,\theta)=0$, we have
$|\partial_\theta\Upsilon(q,\theta)|\le\overline Cq$, hence
$|d\qrs_\theta/d\theta|\le2\overline C\qrs_\theta\le4\overline C\tau$. Also
$\partial_\theta\varpi(q,\theta)\le-\azero^2+\overline Cq\le-1+2\overline C\tau$ at
$q=\qrs_\theta$. Therefore
\[
 \frac{d\mathfrak W}{d\theta}
 =\partial_\theta\varpi+\partial_q\varpi\,\frac{d\qrs_\theta}{d\theta}
 \le-1+2\overline C\tau+4\overline C^2\tau\le-\tfrac12 ,
\]
since $\tau\le1/(12\overline C^2)$ and $\overline C\ge1$. With $\qrs_0\le2\tau$, Taylor's
formula at $(0,0)$ and \eqref{eq:C-zero-jets}, where $\varpi(0,0)=0$ and
$\partial_q\varpi(0,0)=\frac43\azero^3+\frac13\xi''''(0)\ge\frac43$, and the bound
$|\partial_q\varpi|\le\overline C$ with $\qrs_1\le2\tau$ give
\[
 \mathfrak W(0)\ge\partial_q\varpi(0,0)\,\qrs_0-\tfrac{\overline C}2\bigl(\qrs_0\bigr)^2
 \ge\qrs_0\bigl(\tfrac43-\overline C\tau\bigr)>0,\qquad
 \mathfrak W(1)\le\varpi(0,1)+\overline C\qrs_1\le-1+2\overline C\tau<0 ;
\]
both signs hold because $\overline C\tau\le1/(12\overline C)\le\frac1{12}$. Hence $\mathfrak W$
has a unique zero $\theta_c\in(0,1)$, which is the threshold of~(b).

(c) Write $q_c=\qrs_{\theta_c}\le2\tau$. Taylor's formula in $q$ at $(0,\theta_c)$, applied to
$\varpi(q_c,\theta_c)=\mathfrak W(\theta_c)=0$ and to $\Upsilon(q_c,\theta_c)=0$, gives by
\eqref{eq:C-zero-jets}
\[
 \azero^2\theta_c
 =\Bigl[\azero^3\bigl(\tfrac43+3\theta_c-\theta_c^2\bigr)+\tfrac13\xi''''(0)\Bigr]q_c+R_1,
 \qquad
 \azero^2(2-\theta_c)\,q_c=\tau+R_2,
\]
with $|R_1|,|R_2|\le\frac12\overline Cq_c^2\le2\overline C\tau^2$. Since $\azero\le\frac54$ and
$\xi''''(0)\le\frac14$ (Lemma~\ref{lem:dat-class-bounds}(i),(ii) with $\varepsilon=\frac14$),
the first identity gives $\theta_c\le C_1\tau$. As $\azero^2=(1+\tau)^2$, the second then gives
$|q_c-\frac\tau2|\le C_2\tau^2$. Substituting both into the first identity, every term other
than $(\frac43+\frac13\xi''''(0))q_c$ is at most $C_3\tau^2$. The constants $C_1,C_2,C_3$ are
absolute. Hence
$\theta_c=(\frac43+\frac13\xi''''(0))\frac\tau2+O(\tau^2)
=\frac23(1+\frac14\xi''''(0))\tau+O(\tau^2)$, with an absolute implied constant $C$.
\end{proof}

\subsection{Sufficiency}\label{sec:C-sufficiency}

We prove Lemma~\ref{lem:C-sufficiency} by a sign count for $h=\Gamma^{\theta,r}-\id$. On the
plateaus $(0,r)$ and $(r,K)$ of $\zeta^{\theta,r}$, of heights $\theta$ and $1$, the positive
third derivative of Section~\ref{sec:nc-reference} makes $h'$ strictly convex, and above $K$
confinement keeps $h$ negative. With $h(0)=h(r)=0$, $h'(r)\le0$, $h(K)<0$ and, when $h'(r)=0$,
the right curvature \eqref{eq:C-right-second}, this forces $h>0$ on $(0,r)$ and $h<0$ on
$(r,1]$. Since $\Obs_\zeta'=-\frac12\xi''h$, the obstacle decreases up to $r$ and increases
after it. For one-atom Parisi measures in a field, the same route, from the
de Almeida--Thouless inequality at the atom through the sign of $\Gamma-\id$ to optimality, is
\cite[Lemmas~4.1--4.2]{JT17}; for the Sherrington--Kirkpatrick covariance, \cite[Remark~1.2]{JT17}
identifies the contact conditions of a one-atom measure with the de Almeida--Thouless conditions.

\begin{proof}[Proof of Lemma~\ref{lem:C-sufficiency}]
By \ref{item:dat-C4}, $\zeta\in\mathcal W$, and by \ref{item:dat-C1}, $r\le2\tau<K$. We use
three facts. First, $h<0$ on $[K,1]$, by \eqref{eq:nc-confinement}, which holds for every
measure in $\mathcal W$ (item~\ref{item:nc-positive} of the choice of constants in
Section~\ref{sec:nc-reference}). Second, $h(0)=0$, since $\Gamma^{\theta,r}(0)=0$
(Lemma~\ref{lem:plat-ito}(a)). Third, the intervals $(0,r)$ and $(r,K)$ are plateaus of $\zeta$,
of heights $\theta$ and $1$, contained in $(0,K)$. By \eqref{eq:nc-positive-third}, which also
holds for every measure in $\mathcal W$, $h'''=(\Gamma^{\theta,r})'''\ge2$ on each of them.
Since $h'$ is continuous on $[0,1]$, it is strictly convex on $[0,r]$ and on $[r,K]$.

\emph{The sign of $h$ on $(0,r)$.} The set $\mathcal L=\{t\in[0,r]:h'(t)\le0\}$ is a closed
interval containing $r$, because $h'$ is convex and continuous and $h'(r)\le0$. If
$\mathcal L=[0,r]$, then $\int_0^rh'=h(r)-h(0)=0$ would force $h'\equiv0$ on $[0,r]$, which is
impossible for a strictly convex function. Thus $\mathcal L=[s,r]$ with $s>0$, and $h'>0$ on
$[0,s)$, so $h>0$ on $(0,s]$. Moreover $s<r$, since otherwise $h'>0$ on $[0,r)$ would give
$h(r)>h(0)=0$. On $(s,r)$ strict convexity gives $h'<\max\{h'(s),h'(r)\}\le0$, so $h$ decreases
strictly on $[s,r]$ to $h(r)=0$, and $h>0$ on $[s,r)$. Hence $h>0$ on $(0,r)$.

\emph{The sign of $h$ on $(r,1]$.} We first show that $h'<0$ on some interval $(r,r+\delta)$.
This is immediate if $h'(r)<0$. If $h'(r)=0$, then $(\Gamma^{\theta,r})'(r)=1$, and
\eqref{eq:C-right-second} with $\Gamma^{\theta,r}(r)=r\le\frac1{20}$, $\xi''\ge1$ and
$\xi'''(r)\le\varepsilon_*r<\frac14$ (by \ref{item:dat-C1} and \ref{item:dat-C3}) gives
\[
 h''(r+)=(\Gamma^{\theta,r})''(r+)\le\xi'''(r)+\xi''(r)^2\Bigl(\frac{10}{20}-2\Bigr)
 <\frac14-\frac32<0 .
\]
The set $\mathcal L'=\{t\in[r,K]:h'(t)\le0\}$ is therefore a closed interval $[r,s']$ with
$s'>r$. Strict convexity gives $h'<0$ on $(r,s')$, so $h<0$ on $(r,s']$. If $s'<K$, then $h'>0$
on $(s',K]$ and $h(K)<0$ by the first fact, so $h<0$ on $(s',K]$ as well. Together with the
first fact, $h<0$ on $(r,1]$.

\emph{The obstacle.} By \eqref{eq:intro-obstacle},
\[
 \Obs_\zeta(t)-\Obs_\zeta(r)=\frac12\int_t^r\xi''(u)h(u)\,du\quad(t<r),\qquad
 \Obs_\zeta(t)-\Obs_\zeta(r)=-\frac12\int_r^t\xi''(u)h(u)\,du\quad(t>r),
\]
and both are positive by the signs of $h$ just proved. So $\Obs_\zeta$ has its strict minimum
over $[0,1]$ at $r$.

\emph{Minimality.} Let $\nu$ be a probability measure on $[0,1]$ with $\nu(\{0\})\ge\theta$.
Since $\Obs_\zeta\ge\Obs_\zeta(r)$ and $\Obs_\zeta(0)-\Obs_\zeta(r)>0$,
\[
 \int\Obs_\zeta\,d\nu
 \ge\Obs_\zeta(r)+\nu(\{0\})\bigl(\Obs_\zeta(0)-\Obs_\zeta(r)\bigr)
 \ge\Obs_\zeta(r)+\theta\bigl(\Obs_\zeta(0)-\Obs_\zeta(r)\bigr)
 =\int\Obs_\zeta\,d\zeta .
\]
The segment $s\mapsto(1-s)\zeta+s\nu$ stays in the constrained set. Along it the Parisi
functional is convex, and its right derivative at $s=0$ is $\int\Obs_\zeta\,d(\nu-\zeta)$
(Lemma~\ref{lem:cons-standard}(c),(d); for this mixing direction the formula is also
\cite[Lemma~3.2 and Example~3.4]{JT17}). Convexity gives
\[
 P_\nu-P_\zeta\ge\int\Obs_\zeta\,d(\nu-\zeta)\ge0 .
\]
Thus $\zeta$ minimizes $P$ under the constraint $\zeta(\{0\})\ge\theta$, and the uniqueness in
Lemma~\ref{lem:cons-variational}(a) gives $\zeta_\theta=\zeta$. The positive contact set of
$\zeta_\theta$, the set of minimizers of $\Obs_{\zeta_\theta}$ over $(0,1]$, is $\{r\}$.
\end{proof}

\section{Stability above the first contact: proof of
\texorpdfstring{Proposition~\ref*{prop:nc-above-contact}}%
{Proposition \getrefnumber{prop:nc-above-contact}}}\label{app:stability}

This appendix proves Proposition~\ref{prop:nc-above-contact} along the four steps (S1)--(S4) described in
Section~\ref{sec:nc-stability}. Section~\ref{sec:D-setting} fixes the facts, conditions and
functions used below; its paragraph on uniform estimates, (U1)--(U4), is routine and can be
skipped on a first reading, but the radius $r$ chosen there enters the choice of constants in
Appendix~\ref{sec:dat-constants}. Section~\ref{sec:D-lemmas} proves the four lemmas
behind the steps, and Section~\ref{sec:D-proof} closes the contraction.

\subsection{Setting and uniform estimates}\label{sec:D-setting}

\paragraph{The covariance and the facts used.}
Throughout, $\xi\in\Kcl(\varepsilon_*)$, with $\varepsilon_*$ as fixed in
Appendix~\ref{sec:dat-constants}; this value does not depend on the covariance. The covariance
is then fixed and $\theta\downarrow0$. We use the notation of Section~\ref{sec:nc-stability}:
$\alpha_\theta$ and $\alphas$ are the distribution functions of $\zeta_\theta$ and $\zetas$,
$\Delta_\theta=\alpha_\theta-\alphas$, $x_\theta=\alpha_\theta(q_\theta)$,
$v_-=\min\{v_\theta,v_*\}$, $v_+=\max\{v_\theta,v_*\}$,
$L_\theta=\int_{q_\theta}^1|\Delta_\theta(t)|\,dt$, and $Q_\theta$, $Q_*$ are the quantile
functions. We write $\Gamma_\theta=\Gamma_{\zeta_\theta}$ and $\Gamma_*=\Gamma_{\zetas}$, and
$a=\xi''$. Since $\varepsilon_*\le\varepsilon_1$, the covariance lies in $\Kcl(\varepsilon_1)$,
and we use the following facts from Sections~\ref{sec:nc-reference} and~\ref{sec:nc-class},
with the value of $K$ fixed in Appendix~\ref{sec:dat-constants}.
\begin{enumerate}[label=\textup{(P\arabic*)},ref=(P\arabic*),leftmargin=*,topsep=3pt,
  itemsep=2pt]
\item\label{item:D-F1}
 $\supp\zetas=[0,v_*]$ with $0<v_*<K$, $\zetas(\{0\})=0$, $\Ustar=\azero^{-1/2}$ and
 $\Dstar\ge2$ (Propositions~\ref{prop:nc-support}(b) and~\ref{prop:nc-class}(b),(c), and
 Theorem~\ref{thm:local}; the bound $v_*<K$ is item~\ref{item:nc-confined} of the choice of
 constants in Section~\ref{sec:nc-reference}, with $\theta=0$).
 Every $t\in[0,v_*]$ satisfies $\Gamma_*(t)=t$ (Lemma~\ref{lem:ons-zero-in-support}).
\item\label{item:D-F2}
 There is $\theta_0>0$ such that, for $0<\theta<\theta_0$, $\zeta_\theta(\{0\})=\theta$,
 $\zeta_\theta((0,q_\theta))=0$ and $\supp\zeta_\theta=\{0\}\cup[q_\theta,v_\theta]$ with
 $0<q_\theta<v_\theta<K$, and every point of $[q_\theta,v_\theta]$ is a positive contact
 of $\zeta_\theta$ (Proposition~\ref{prop:nc-class}(d) and Lemma~\ref{lem:nc-regular}).
 Since a positive contact in $(0,1)$ is a minimum point of $\Obs_\theta$ on $(0,1]$,
 Lemma~\ref{lem:cons-standard}(f) gives $\Gamma_\theta(t)=t$ on $[q_\theta,v_\theta]$.
\item\label{item:D-F3}
 $q_\theta/\theta\to3\sqrt\azero/\Dstar\in(0,\infty)$ as $\theta\downarrow0$
 \eqref{eq:nc-onset}. In particular $q_\theta\to0$, so $q_\theta<v_*$ for all small $\theta$.
\end{enumerate}

\paragraph{Notation and constants.}
For a probability measure $\zeta$ on $[0,1]$ write $u^\zeta=\Phi_{\zeta,x}$,
$U^\zeta=\partial_xu^\zeta$, $V^\zeta=\partial_x^2u^\zeta$ and $W^\zeta=\partial_x^3u^\zeta$,
as functions of $(t,x)\in[0,1]\times\R$; the superscripts $\theta$ and $*$ refer to
$\zeta_\theta$ and $\zetas$, so that $\Ustar=U^*(0,0)$ and $\Wstar=W^*(0,0)$. Let $X^\zeta$ be the diffusion
\eqref{eq:intro-diffusion} of $\zeta$, with drift
$\omega^\zeta(t,x)=a(t)\alpha_\zeta(t)u^\zeta(t,x)$, and put $\nu^\zeta_t(F)=\E F(X^\zeta_t)$;
we write $\nu^\theta=\nu^{\zeta_\theta}$ and $\nu^*=\nu^{\zetas}$. For $F:\R\to\R$ let
$\|F\|_k=\max_{j\le k}\sup|F^{(j)}|$. Let $\mathfrak B$ be the set of covariances $\xi$ with
$\normCfour{\xi-\xic}<\frac14$; on $[0,1]$ these satisfy $\frac34<a<\frac54$ and
$|a'|,|a''|<\frac14$, so $\mathfrak B$ is a family of the kind considered in
Lemma~\ref{lem:plat-jets}(b), as in Section~\ref{sec:nc-reference}. Below, $C_0$ denotes constants
that depend only on $\mathfrak B$ and on derivative orders, which never exceed $12$. They do
not depend on the covariance in $\mathfrak B$, on the measure, on $K$ or on $\theta$. The
letter $C$ denotes constants that may also depend on the fixed covariance and on $K$, but not
on $\theta$. Both may change from line to line.

\paragraph{Uniform estimates.}
The following four estimates hold for every $\xi\in\mathfrak B$.

\smallskip\noindent\emph{(U1) Jets.} For every probability measure $\zeta$ and $k\ge1$,
$|\partial_x^ku^\zeta|\le C_0$ on $[0,1]\times\R$, and $t\mapsto\partial_x^ku^\zeta(t,x)$ is
$C_0$-Lipschitz, uniformly in $x$. These are \eqref{eq:plat-bounds} and
\eqref{eq:plat-time-lip}: by Lemma~\ref{lem:plat-jets}(b) the constants $B_k$ are uniform over
$\mathfrak B$, and so are the constants $L_k$, which the proof of
Lemma~\ref{lem:plat-jets}(c) bounds in terms of $\|a\|_\infty$ and $B_1,\dots,B_{k+2}$. Since
$u^\zeta(t,\cdot)$ is odd and $0<U^\zeta\le1$ \eqref{eq:plat-bounds}, $|u^\zeta(t,x)|\le|x|$
and $|V^\zeta(t,x)|\le C_0|x|$. On every open interval on which $\alpha_\zeta$ is continuous, the jets are continuously
differentiable in time and satisfy the $x$-derivatives of \eqref{eq:intro-pde} pointwise.
Indeed, for $j\ge1$ and $0\le t_1\le t_2\le1$,
\[
 \partial_x^j\Phi_\zeta(t_2,x)-\partial_x^j\Phi_\zeta(t_1,x)
 =-\int_{t_1}^{t_2}\partial_x^j\Bigl[\frac a2\bigl(\Phi_{\zeta,xx}
 +\alpha_\zeta\Phi_{\zeta,x}^2\bigr)\Bigr](s,x)\,ds .
\]
For measures with finitely many atoms this follows from the equation between atoms and the
continuity of the jets in time across atoms \cite[Proposition~5, (64)]{AC15a}. It passes to
every $\zeta$ because the jets converge uniformly along weakly convergent sequences of measures
\cite[Proposition~1(i)]{AC15a} and the distribution functions converge in $L^1$; and the
integrand is continuous in $s$ wherever $\alpha_\zeta$ is.

\smallskip\noindent\emph{(U2) Stability localized in time.} For two probability measures
$\zeta,\tilde\zeta$, $t\in[0,1]$ and $0\le k\le10$,
\begin{equation}\label{eq:D-local-jet-stability}
 \sup_{x\in\R}\bigl|\partial_x^k(u^\zeta-u^{\tilde\zeta})(t,x)\bigr|
 \le C_0\int_t^1|\alpha_\zeta-\alpha_{\tilde\zeta}|\,ds .
\end{equation}
This is the proof of \eqref{eq:plat-jet-general} with $\tilde a=a$, restricted to $[t,1]$.
For measures with finitely many atoms, the difference $w=u^\zeta-u^{\tilde\zeta}$ vanishes at
time one and solves
\[
 w_s+\tfrac a2w_{xx}+a\alpha_\zeta u^\zeta w_x+a\alpha_\zeta u^{\tilde\zeta}_xw
 =-a(\alpha_\zeta-\alpha_{\tilde\zeta})\,u^{\tilde\zeta}u^{\tilde\zeta}_x
\]
on $[t,1]\times\R$, between the atoms of either measure. Its coefficients are bounded
by $C_0$, so the maximum principle on $[t,1]\times\R$ bounds $|w|$ by $C_0$ times the time
integral of the forcing over $[t,1]$. The equations for $\partial_x^kw$ have forcings bounded
by $C_0|\alpha_\zeta-\alpha_{\tilde\zeta}|(s)$ plus $C_0$ times lower derivatives of $w$ at
times $s\ge t$, which by induction are at most $C_0\int_s^1|\alpha_\zeta-\alpha_{\tilde\zeta}|
\le C_0\int_t^1|\alpha_\zeta-\alpha_{\tilde\zeta}|$. This gives
\eqref{eq:D-local-jet-stability}, which passes to general measures as in
Lemma~\ref{lem:plat-jets}(b). Since $w$ and $\partial_x^2w$ are odd in $x$,
\eqref{eq:D-local-jet-stability} with $k=1,3$ gives
\begin{equation}\label{eq:D-odd-stability}
 \bigl|(u^\zeta-u^{\tilde\zeta})(t,x)\bigr|+\bigl|(V^\zeta-V^{\tilde\zeta})(t,x)\bigr|
 \le C_0|x|\int_t^1|\alpha_\zeta-\alpha_{\tilde\zeta}|\,ds .
\end{equation}
We use these with $\{\zeta,\tilde\zeta\}=\{\zeta_\theta,\zetas\}$ and $t\ge q_\theta$, where the
integral is at most $L_\theta$.

\smallskip\noindent\emph{(U3) Moments.} For every probability measure $\zeta$ and
$t\in[0,1]$, with $X=X^\zeta$,
\begin{equation}\label{eq:D-moments}
 \E\sup_{s\le t}|X_s|^{2k}\le C_0t^k\quad(k\le4),\qquad
 \bigl|\E X_t^2-\azero t\bigr|\le C_0t^2,\qquad
 \bigl|\E X_t^4-3\azero^2t^2\bigr|\le C_0t^3 .
\end{equation}
Write $X=M+A$ with $M_t=\int_0^t\sqrt{a}\,dB$ and $A_t=\int_0^t\omega^\zeta(s,X_s)\,ds$. The
process $M$ is a centred Gaussian martingale with $\E M_t^2=\xi'(t)$, and
$|\xi'(t)-\azero t|\le\frac12\sup|\xi'''|\,t^2\le t^2$. By (U1), $|\omega^\zeta(s,x)|\le2|x|$,
so Gronwall's inequality gives $\sup_{s\le t}|X_s|\le e^{2t}\sup_{s\le t}|M_s|$, and Doob's
inequality gives the first bound. Hence $\|A_t\|_{L^{2k}}\le C_0t^{3/2}$ for $k\le4$. Since
$\E M_t^4=3\xi'(t)^2=3\azero^2t^2+O(t^3)$, expanding $(M_t+A_t)^2$ and $(M_t+A_t)^4$ and using
H\"older's inequality gives the other two bounds; for instance
$|\E M_tA_t|\le\|M_t\|_{L^2}\|A_t\|_{L^2}\le C_0t^2$ and
$|\E M_t^3A_t|\le\|M_t\|_{L^4}^3\|A_t\|_{L^4}\le C_0t^3$.

\smallskip\noindent\emph{(U4) Duality.} Let $\zeta$ be a probability measure whose
distribution function is continuous outside a finite subset of $[0,1]$, let $t_1\in(0,1]$, let
$F$ be even with $\|F\|_{10}<\infty$, and put
$\psi(s,x)=\E\bigl[F(X^\zeta_{t_1})\bigm|X^\zeta_s=x\bigr]$ for $0\le s\le t_1$ and
$x\in\R$. The drift $\omega^\zeta$ is odd in $x$, has $x$-derivatives bounded by $C_0$ through order
$12$ by (U1), satisfies $|\omega^\zeta(s,x)|\le2|x|$, and is continuous in $(s,x)$, together
with these derivatives, for $s$ outside a finite set, since by (U1) and $u^\zeta(s,0)=0$ the
map $s\mapsto\partial_x^ku^\zeta(s,x)$ is $C_0(1+|x|)$-Lipschitz for $0\le k\le12$. For
$s\le\tau\le t_1$ let $X^{s,x}_\tau$ solve \eqref{eq:intro-diffusion} with $X^{s,x}_s=x$; by the
Markov property, $\psi(s,x)=\E F(X^{s,x}_{t_1})$. Pathwise, $X^{s,x}_\tau-\int_s^\tau\sqrt a\,dB$
solves an ordinary differential equation in $\tau$ whose vector field is continuous between the
points of discontinuity and has $x$-derivatives of orders $1$ to $12$ bounded by $C_0$. By
differentiable dependence on initial data, applied on each interval between these points,
$x\mapsto X^{s,x}_\tau$ is $C^{10}$, and its derivatives of orders $1$ to $10$ solve the
variational equations; they are therefore bounded by $C_0$, continuous in $(s,x)$, and tend to
those of the identity map as $\tau-s\to0$. By the chain rule and dominated convergence,
$\partial_x^k\psi$ is continuous on $[0,t_1]\times\R$ for $k\le10$ and
$\|\psi(s,\cdot)\|_{10}\le C_0\|F\|_{10}$. For $s<s'\le t_1$ put $Y=X^{s,x}$. The Markov
property gives $\psi(s,x)=\E\psi(s',Y_{s'})$, and It\^o's formula for $\psi(s',Y_\tau)$,
$s\le\tau\le s'$, whose martingale part has a bounded integrand, gives
\[
 \psi(s,x)-\psi(s',x)=\int_s^{s'}\E\,G_{s'}(\tau,Y_\tau)\,d\tau,\qquad
 G_{s'}(\tau,y)=\omega^\zeta(\tau,y)\psi_x(s',y)+\tfrac{a(\tau)}2\psi_{xx}(s',y).
\]
Differentiate this identity $k\le8$ times in $x$ under the integral. Gronwall's inequality gives
\[
 \E\sup_{s\le\tau\le s'}|Y_\tau-x|^2\le C_0(1+x^2)(s'-s),
\]
so the $k$-th derivative of the right side is at most $C_0\|F\|_{10}(1+|x|)(s'-s)$, and, divided by $s'-s$, it tends to
$\partial_x^kG_{s_0}(s_0,x)$ when $s$ and $s'$ tend to a point $s_0$ at which $\omega^\zeta$ is
continuous. Hence $\psi$ solves $\partial_s\psi+\omega^\zeta\psi_x+\frac a2\psi_{xx}=0$,
$\psi(t_1,\cdot)=F$, classically between the points of discontinuity, and there its
$x$-derivatives of order at most $8$ solve the differentiated equations classically; moreover
$s\mapsto\partial_x^k\psi(s,x)$ is $C_0\|F\|_{10}(1+|x|)$-Lipschitz on $[0,t_1]$ for $k\le8$,
so $\psi$ is continuous in $s$.
Since $\omega^\zeta$ is odd in $x$ and $F$ is even, $\psi(s,\cdot)$ is even, so
$\psi_x(s,0)=0$ and $|\psi_x(s,x)|\le C_0\|F\|_{10}|x|$. Let $\tilde\zeta$ be another
probability measure whose distribution function is continuous outside a finite set. It\^o's
formula for $\psi(s,X^{\tilde\zeta}_s)$ on the intervals between the points of discontinuity,
joined by continuity in $s$, and $\psi(0,0)=\nu^\zeta_{t_1}(F)$ give
\begin{equation}\label{eq:D-duality}
 \nu^{\tilde\zeta}_{t_1}(F)-\nu^\zeta_{t_1}(F)
 =\int_0^{t_1}\E\Bigl[\bigl(\omega^{\tilde\zeta}-\omega^\zeta\bigr)(s,X^{\tilde\zeta}_s)\,
 \psi_x\bigl(s,X^{\tilde\zeta}_s\bigr)\Bigr]\,ds .
\end{equation}
The martingale part has mean zero because $\psi_x$ grows at most linearly and
$X^{\tilde\zeta}$ has the moments \eqref{eq:D-moments}. We apply \eqref{eq:D-duality} with
$\zeta=\zetas$ and $\tilde\zeta=\zeta_\theta$, whose distribution functions are continuous
outside $\{v_*\}$ and $\{q_\theta,v_\theta\}$, respectively, by
Lemma~\ref{lem:D-support-relation} below.

\paragraph{Conditions on $K$ and $\varepsilon_*$.}
For $\xi=\xic$ and $\zeta=\delta_0$ one has $u=\tanh$, so $U(t,0)=1$ and $W(t,0)=-2$ for all
$t$. Appendix~\ref{sec:dat-constants} fixes, with $\eta=\frac1{40}$, first $K$, then the
neighbourhood $\mathcal W$ and the radius $\varepsilon_1$ of Section~\ref{sec:nc-reference}
(so that items~\ref{item:nc-positive} and~\ref{item:nc-confined} hold for this $K$), then
$\varepsilon_*$. This appendix uses three of its conditions: $C_0K\le\eta$ for each of the
finitely many constants $C_0$ that occur below, so that $2C_0K\le\frac14$; $\varepsilon_1\le r$
and $\Wone(\zeta,\delta_0)\le r$ for every $\zeta\in\mathcal W$; and
$\varepsilon_*\le\min\{\varepsilon_1,\eta\}$. Here $r>0$ is chosen as follows. By
\eqref{eq:plat-jet-general}, applied over $\mathfrak B$ to the pairs $(\xi,\zeta)$ and
$(\xic,\delta_0)$,
\[
 \|\partial_x^k(u^\zeta-\tanh)\|_\infty
 \le C_0\bigl(\|a-1\|_{L^1}+\|a\alpha_\zeta-1\|_{L^1}\bigr)
 \le C_0\bigl(2\normCfour{\xi-\xic}+\tfrac54\Wone(\zeta,\delta_0)\bigr)
 \qquad(k=1,3),
\]
since $\|a\alpha_\zeta-1\|_{L^1}\le\|a-1\|_{L^1}+\frac54\int_0^1(1-\alpha_\zeta)$ and
$\int_0^1(1-\alpha_\zeta)=\Wone(\zeta,\delta_0)$. We take $r$ so small that the right side is
at most $\eta$ when $\normCfour{\xi-\xic}<r$ and $\Wone(\zeta,\delta_0)\le r$. Then
$\xi\in\Kcl(\varepsilon_1)\subset\mathfrak B$, and item~\ref{item:nc-confined} of the choice of
constants in Section~\ref{sec:nc-reference},
$\tanh'(0)=1$, $\tanh'''(0)=-2$ and Lemma~\ref{lem:dat-class-bounds}(i),(ii) with
$\varepsilon=\varepsilon_*\le\eta$ give, for every $\theta\in[0,1]$ and
$\zeta=\zeta_{\xi,\theta}$:
\begin{enumerate}[label=\textup{(K\arabic*)},ref=(K\arabic*),leftmargin=*,topsep=3pt,
  itemsep=2pt]
\item\label{item:D-K1} every positive contact of $\zeta$ lies in $(0,K)$; in particular
 $v_\theta,v_*<K$;
\item\label{item:D-K2} $1-\eta\le U^\zeta(t,0)\le1$ and $|W^\zeta(t,0)+2|\le\eta$ for all
 $t\in[0,1]$;
\item\label{item:D-K3} $1<\azero\le a\le1+\eta$, $0\le a'(t)\le\eta t$ and $0\le a''\le\eta$
 on $[0,1]$.
\end{enumerate}

\paragraph{Consequences for $t\le K$.}
Let $\zeta=\zeta_{\xi,\theta}$ for some $\theta\in[0,1]$, and write $X=X^\zeta$. Since
$U^\zeta(t,\cdot)$ and $W^\zeta(t,\cdot)$ are even with second derivatives bounded by $C_0$,
\ref{item:D-K2} and \eqref{eq:D-moments} give $\E|U^\zeta(t,X_t)-1|\le\eta+C_0t$, and similarly
for $W^\zeta$; as $C_0K\le\eta$,
\begin{equation}\label{eq:D-K4}
 \E\bigl|U^\zeta(t,X_t)-1\bigr|\le2\eta,\qquad
 \E\bigl|W^\zeta(t,X_t)+2\bigr|\le2\eta\qquad(0\le t\le K).
\end{equation}
For $0\le t\le K$ let $N_\zeta(t)=a'(t)\E U^\zeta(t,X_t)^2+a(t)^2\E V^\zeta(t,X_t)^2$ and
$\sigma_\zeta(t)=2a(t)^2\E U^\zeta(t,X_t)^3$ be the functions \eqref{eq:plat-NS}, and put
$R_\zeta=N_\zeta/\sigma_\zeta$. They are the two coefficients of the one-sided derivatives of $\Gamma_\zeta'$ in
Lemma~\ref{lem:plat-ito}(b): the right derivative is $N_\zeta(t)-\alpha_\zeta(t)
\sigma_\zeta(t)$ and the left derivative is $N_\zeta(t)-\alpha_\zeta(t-)\sigma_\zeta(t)$.
Since $1-y^3\le3(1-y)$ for $y\in(0,1]$, \eqref{eq:D-K4} gives
$\E U^\zeta(t,X_t)^3\ge1-6\eta$. With \ref{item:D-K3}, $|V^\zeta(t,x)|\le C_0|x|$ and
\eqref{eq:D-moments},
\begin{equation}\label{eq:D-rho-sigma}
 0\le N_\zeta(t)\le C_0t,\qquad 1\le\sigma_\zeta(t)\le3\qquad(0\le t\le K).
\end{equation}
On an open plateau of height one inside $(0,K)$, \eqref{eq:plat-second} with $m=1$ and the
same bounds give, as $C_0K\le\eta$ and $\eta=\frac1{40}$,
\begin{equation}\label{eq:D-top-concave}
 \Gamma_\zeta''=a'\E U^2+a^2\E\bigl(V^2-2U^3\bigr)
 \le\eta K+(1+\eta)^2C_0K-2(1-6\eta)\le-1 .
\end{equation}

\subsection{Four lemmas}\label{sec:D-lemmas}

The first lemma identifies both distribution functions on their support intervals, and gives
the behaviour of $\alphas$ at zero that fixes the onset.

\begin{lemma}[Relations on the supports]\label{lem:D-support-relation}
Let $0<\theta<\theta_0$. Then $\alpha_\theta=R_{\zeta_\theta}$ on $(q_\theta,v_\theta)$ and
$\alphas=R_{\zetas}$ on $(0,v_*)$. Moreover $\alpha_\theta=\theta$ on $[0,q_\theta)$,
$\alpha_\theta=1$ on $[v_\theta,1]$ and $\alphas=1$ on $[v_*,1]$. Consequently $\alphas$ is
continuous on $[0,v_*)$ with $\alphas(t)\le C_0t$ there, $\alpha_\theta$ is continuous on
$(q_\theta,v_\theta)$ and right continuous at $q_\theta$, and
\begin{equation}\label{eq:D-alpha-star-linear}
 \alphas(t)=\rhostar t+O(t^2)\qquad(t\downarrow0),\qquad
 \rhostar=\frac{\Dstar}{2\sqrt\azero} .
\end{equation}
Finally, for all small $\theta>0$, $\theta/C\le q_\theta\le C\theta$ and $|\Delta_\theta|\le
C\theta$ on $[0,q_\theta)$.
\end{lemma}

\begin{proof}
By \ref{item:D-F2}, $\Gamma_\theta(t)=t$ on $[q_\theta,v_\theta]$, so
$\Gamma_\theta'\equiv1$ on $(q_\theta,v_\theta)$, and both one-sided derivatives of
$\Gamma_\theta'$ vanish there. By Lemma~\ref{lem:plat-ito}(b) they equal
$N_{\zeta_\theta}(t)-\alpha_\theta(t)\sigma_{\zeta_\theta}(t)$ and
$N_{\zeta_\theta}(t)-\alpha_\theta(t-)\sigma_{\zeta_\theta}(t)$; since
$\sigma_{\zeta_\theta}\ge1$ \eqref{eq:D-rho-sigma}, $\alpha_\theta(t-)=\alpha_\theta(t)=
R_{\zeta_\theta}(t)$. The same argument applies to $\zetas$ on $(0,v_*)$, because $\Gamma_*$
is the identity on $[0,v_*]$ by \ref{item:D-F1}. The values on the plateaus follow from
$\zeta_\theta(\{0\})=\theta$, $\zeta_\theta((0,q_\theta))=0$ and the supports. Expectations
along the diffusion of polynomials in the jets are continuous in time
(Section~\ref{sec:plat}), so $R_\zeta$ is continuous on $[0,K]$. Hence $\alphas$ is
continuous on $(0,v_*)$, and at zero because $\alphas(0)=\zetas(\{0\})=0$ and
$0\le\alphas=R_{\zetas}\le C_0t$ on $(0,v_*)$ by \eqref{eq:D-rho-sigma}. Likewise $\alpha_\theta$
is continuous on $(q_\theta,v_\theta)$ and right continuous at $q_\theta$.

For \eqref{eq:D-alpha-star-linear}, write $X=X^{\zetas}$. By evenness, (U1) and
\eqref{eq:D-moments}, $|U^*(t,x)-\Ustar|\le C_0(x^2+t)$, so $\E U^*(t,X_t)^2=\Ustar^2+O(t)$ and
$\E U^*(t,X_t)^3=\Ustar^3+O(t)$. Since $V^*(t,\cdot)$ is odd with $\partial_xV^*=W^*$,
$V^*(t,x)=W^*(t,0)x+O(|x|^3)$ and $W^*(t,0)=\Wstar+O(t)$, so
$\E V^*(t,X_t)^2=\azero\Wstar^2t+O(t^2)$ by \eqref{eq:D-moments}. By
Lemma~\ref{lem:dat-class-bounds}(iii), $a'(t)=\xi''''(0)\,t+O(t^2)$, and $a(t)=\azero+O(t^2)$
by Lemma~\ref{lem:dat-class-bounds}(ii). Hence
\[
 N_{\zetas}(t)=\bigl(\xi''''(0)\Ustar^2+\azero^3\Wstar^2\bigr)t+O(t^2)=\Dstar t+O(t^2),
 \qquad
 \sigma_{\zetas}(t)=2\azero^2\Ustar^3+O(t)=2\sqrt\azero+O(t),
\]
using $\Ustar=\azero^{-1/2}$, and \eqref{eq:D-alpha-star-linear} follows from
$\alphas=N_{\zetas}/\sigma_{\zetas}$. The bounds on $q_\theta$ follow from \ref{item:D-F3}.
On $[0,q_\theta)$ we have $\Delta_\theta=\theta-\alphas$ and $0\le\alphas\le C_0q_\theta$, so
$|\Delta_\theta|\le\theta+C_0q_\theta\le C\theta$.
\end{proof}

From now on $\theta$ is small, and we write $q=q_\theta$. The plateau of $\zeta_\theta$ is
measured by two integrals of $\Delta_\theta$ over $[0,q)$:
\begin{gather*}
 I_\theta=\int_0^q\Delta_\theta(s)\,s\,ds,\qquad
 \Pi_\theta=\int_0^q\Delta_\theta(s)\,k_\theta(s)\,ds,\qquad\text{where}\\
 k_\theta(s)=a(s)\,U^\theta(s,0)\,\E\bigl(X^{\zeta_\theta}_s\bigr)^2
 +\tfrac12\azero^3\,W^\theta(0,0)\,s^2,\qquad
 \gamma_\theta=\tfrac12\azero^3\,U^\theta(0,0).
\end{gather*}
Neither $k_\theta$ nor $\gamma_\theta$ depends on a test function.

\begin{lemma}[The plateau acts through its first moment]\label{lem:D-plateau}
For all small $\theta>0$ and every even $H$ with $\|H\|_{10}\le1$,
\begin{equation}\label{eq:D-plateau}
 \bigl|\nu^\theta_q(H)-\nu^*_q(H)-H''(0)\,\Pi_\theta-\gamma_\theta\,q\,H''''(0)\,I_\theta\bigr|
 \le C\bigl(\theta^5+\theta^3L_\theta\bigr).
\end{equation}
Moreover $|\Pi_\theta-\azero^2U^\theta(0,0)\,I_\theta|\le C\theta^4$.
\end{lemma}

\begin{proof}
Apply \eqref{eq:D-duality} with $\zeta=\zetas$, $\tilde\zeta=\zeta_\theta$, $t_1=q$ and $F=H$,
and let $\psi$ be the corresponding function. On $[0,q)$ the drifts differ by
\[
 \omega^{\zeta_\theta}-\omega^{\zetas}
 =a\bigl[\Delta_\theta\,u^\theta+\alphas\,(u^\theta-u^*)\bigr],
\]
and we treat the two terms separately; write $X=X^{\zeta_\theta}$ and
$\mathsf m_k(s)=\E X_s^k$.

\emph{The second term.} For $s<q$ we have $\alphas(s)\le C_0s$ and
$\int_s^1|\Delta_\theta|\le C\theta q+L_\theta$ by Lemma~\ref{lem:D-support-relation}. By
\eqref{eq:D-odd-stability}, $|\psi_x(s,x)|\le C_0|x|$ and \eqref{eq:D-moments}, this term
contributes at most $C\int_0^qs^2(\theta q+L_\theta)\,ds\le C(\theta^5+\theta^3L_\theta)$.

\emph{The first term.} It equals $\int_0^q\Delta_\theta(s)\,g(s)\,ds$ with
$g(s)=\E[a(s)\,u^\theta(s,X_s)\,\psi_x(s,X_s)]$. Put $\psi_j(s)=\partial_x^j\psi(s,0)$. The
function $x\mapsto u^\theta(s,x)\psi_x(s,x)$ is even, with sixth derivative bounded by $C_0$ by
(U1) and (U4), and its Taylor expansion at zero has the coefficients
$U^\theta(s,0)\psi_2(s)$ of $x^2$ and $\frac16(U^\theta(s,0)\psi_4(s)+W^\theta(s,0)\psi_2(s))$
of $x^4$. Since $\mathsf m_6(s)\le C_0s^3$,
\[
 g(s)=a(s)\Bigl[U^\theta(s,0)\,\psi_2(s)\,\mathsf m_2(s)
 +\tfrac16\bigl(U^\theta(s,0)\psi_4(s)+W^\theta(s,0)\psi_2(s)\bigr)\mathsf m_4(s)\Bigr]+O(s^3).
\]
To evaluate $\psi_2$, differentiate the backward equation twice in $x$ and set $x=0$. Since
$\psi_x(s,0)=\omega^{\zetas}(s,0)=\partial_x^2\omega^{\zetas}(s,0)=0$ by parity,
$\psi_2'(s)=-2\,\partial_x\omega^{\zetas}(s,0)\,\psi_2(s)-\frac{a(s)}2\,\psi_4(s)$,
where $0\le\partial_x\omega^{\zetas}(s,0)=a(s)\alphas(s)U^*(s,0)\le C_0s$. With $\psi_2(q)=H''(0)$,
$\psi_4(s)=H''''(0)+O(q-s)$ by (U4), and $a(s)=\azero+O(s)$, integration from $s$ to $q$ gives
\begin{equation}\label{eq:D-psi-two}
 \psi_2(s)=H''(0)+\tfrac{\azero}2\,H''''(0)\,(q-s)+O(q^2)\qquad(0\le s\le q).
\end{equation}
Now insert \eqref{eq:D-psi-two}, $U^\theta(s,0)=U^\theta(0,0)+O(s)$,
$W^\theta(s,0)=W^\theta(0,0)+O(s)$ (by (U1)), $\psi_4=H''''(0)+O(q)$, $\psi_2=H''(0)+O(q)$,
$a(s)=\azero+O(s)$, $\mathsf m_2(s)=\azero s+O(s^2)$ and $\mathsf m_4(s)=3\azero^2s^2+O(s^3)$
\eqref{eq:D-moments}. In the first summand of $g$ we keep
$a(s)U^\theta(s,0)\mathsf m_2(s)H''(0)$ exactly; the rest of that summand is
$\frac12\azero^3U^\theta(0,0)H''''(0)(q-s)s+O(qs^2+q^2s)$. The second summand equals
$\frac12\azero^3\bigl(U^\theta(0,0)H''''(0)+W^\theta(0,0)H''(0)\bigr)s^2+O(s^3+qs^2)$. The two
$H''''(0)$ contributions add up to $\gamma_\theta H''''(0)\bigl[(q-s)s+s^2\bigr]
=\gamma_\theta\,q\,H''''(0)\,s$, so
\[
 g(s)=H''(0)\,k_\theta(s)+\gamma_\theta\,q\,H''''(0)\,s+O\bigl(s^3+qs^2+q^2s\bigr).
\]
The cancellation of the terms in $s^2H''''(0)$ is the Gaussian effect described in
Section~\ref{sec:nc-stability}: at this order the plateau moves the law of $X_q$ only through
its variance. Multiplying by $\Delta_\theta$, which satisfies $|\Delta_\theta|\le C\theta$ on
$[0,q)$, and integrating leaves an error at most $C\theta q^4\le C\theta^5$. This proves
\eqref{eq:D-plateau}. Finally $a(s)=\azero+O(s)$, $U^\theta(s,0)=U^\theta(0,0)+O(s)$ and
$\mathsf m_2(s)=\azero s+O(s^2)$ give $k_\theta(s)=\azero^2U^\theta(0,0)\,s+O(s^2)$, hence
$|\Pi_\theta-\azero^2U^\theta(0,0)I_\theta|\le C\theta\int_0^qs^2\,ds\le C\theta^4$.
\end{proof}

The next lemma is the decisive step: the two contact identities at the common point $q$ fix
the first moment of the plateau.

\begin{lemma}[Closure by the common contact]\label{lem:D-contact-closure}
For all small $\theta>0$,
\[
 |I_\theta|\le C\bigl(\theta^4+\theta L_\theta\bigr),\qquad
 |\Pi_\theta|\le C\bigl(\theta^5+\theta L_\theta\bigr).
\]
Consequently, for every even $H$ with $\|H\|_{10}\le1$,
\[
 \bigl|\nu^\theta_q(H)-\nu^*_q(H)\bigr|\le C\bigl(\theta^5+\theta L_\theta\bigr).
\]
\end{lemma}

\begin{proof}
For small $\theta$ we have $q<v_*$, so $q$ is a contact of both measures:
$\Gamma_\theta(q)=q=\Gamma_*(q)$ by \ref{item:D-F1} and \ref{item:D-F2}.
Test with $H_q=u^*(q,\cdot)^2$. It is
even, $\|H_q\|_{10}\le C_0$ by (U1), $H_q''(0)=2U^*(q,0)^2\in[2(1-\eta)^2,2]$ by
\ref{item:D-K2}, and $|H_q''''(0)|=8|U^*(q,0)W^*(q,0)|\le C_0$. Since
$\Gamma_\zeta(q)=\nu^\zeta_q\bigl(u^\zeta(q,\cdot)^2\bigr)$, the two contact identities give
\[
 \nu^\theta_q(H_q)-\nu^*_q(H_q)
 =\nu^\theta_q\bigl(u^*(q,\cdot)^2-u^\theta(q,\cdot)^2\bigr),
\]
and by \eqref{eq:D-odd-stability}, with $\int_q^1|\Delta_\theta|=L_\theta$, its absolute value is
at most $C_0L_\theta\,\E(X^{\zeta_\theta}_q)^2\le C_0qL_\theta$. Lemma~\ref{lem:D-plateau}
applied to $H_q/C_0$, followed by its second assertion, gives
\[
 \bigl[H_q''(0)\,\azero^2U^\theta(0,0)+\gamma_\theta\,q\,H_q''''(0)\bigr]\,I_\theta
 =O\bigl(\theta^4+\theta L_\theta\bigr).
\]
By \ref{item:D-K2} and \ref{item:D-K3} the bracket is at least $2(1-\eta)^3-C_0q\ge1$ for
small $\theta$, which gives the bound on $I_\theta$. Returning to \eqref{eq:D-plateau} for
$H_q$ gives $H_q''(0)\,|\Pi_\theta|\le C(qL_\theta+\theta^5+\theta^3L_\theta)+Cq|I_\theta|$,
which is the bound on $\Pi_\theta$. The last assertion is \eqref{eq:D-plateau} with these two
bounds.
\end{proof}

The bound at time $q$ propagates to later times, up to a term that is small because the times
involved are below $K$.

\begin{lemma}[Laws above the first contact]\label{lem:D-propagation}
For all small $\theta>0$, every $t\in[q,K]$ and every even $F$ with $\|F\|_{10}\le1$,
\[
 \bigl|\nu^\theta_t(F)-\nu^*_t(F)\bigr|\le C_0\,t\,L_\theta+C\bigl(\theta^5+\theta L_\theta\bigr).
\]
\end{lemma}

\begin{proof}
Apply \eqref{eq:D-duality} with $\zeta=\zetas$, $\tilde\zeta=\zeta_\theta$ and $t_1=t$, and
split the integral at $q$. By the Markov property, the part over $[0,q]$ is the right side of
\eqref{eq:D-duality} for $t_1=q$ and the test function $H=\psi(q,\cdot)$, so it equals
$\nu^\theta_q(H)-\nu^*_q(H)$. This $H$ is even with $\|H\|_{10}\le C_0$, and
Lemma~\ref{lem:D-contact-closure} applies to $H/C_0$. For $s\ge q$, the drifts differ by
$a[\Delta_\theta u^\theta+\alphas(u^\theta-u^*)]$, and \eqref{eq:D-odd-stability} with
$\int_s^1|\Delta_\theta|\le L_\theta$ gives
$|(\omega^{\zeta_\theta}-\omega^{\zetas})(s,x)|\le C_0|x|\,(|\Delta_\theta(s)|+L_\theta)$. With
$|\psi_x(s,x)|\le C_0|x|$ and \eqref{eq:D-moments}, the part over $[q,t]$ is at most
$C_0\int_q^t(|\Delta_\theta(s)|+L_\theta)\,s\,ds\le C_0\,t\,L_\theta$.
\end{proof}

\subsection{Proof of \texorpdfstring{Proposition~\ref*{prop:nc-above-contact}}%
{Proposition \getrefnumber{prop:nc-above-contact}}}\label{sec:D-proof}

\paragraph{Parts (a) and (b).}
\emph{The interior of the common support.} Let $t\in(q,v_-)$. By
Lemma~\ref{lem:D-support-relation}, $\Delta_\theta(t)=R_{\zeta_\theta}(t)-R_{\zetas}(t)$.
Evaluating the jets of both measures at time $t$,
\begin{align*}
 N_{\zeta_\theta}-N_{\zetas}
 ={}&a'\Bigl[\nu^\theta_t\bigl((U^\theta)^2-(U^*)^2\bigr)
   +\nu^\theta_t\bigl((U^*)^2\bigr)-\nu^*_t\bigl((U^*)^2\bigr)\Bigr]\\
 &+a^2\Bigl[\nu^\theta_t\bigl((V^\theta)^2-(V^*)^2\bigr)
   +\nu^\theta_t\bigl((V^*)^2\bigr)-\nu^*_t\bigl((V^*)^2\bigr)\Bigr],
\end{align*}
where $a$ and $a'$ are evaluated at $t$. By \eqref{eq:D-local-jet-stability},
$|(U^\theta)^2-(U^*)^2|\le C_0L_\theta$; by \eqref{eq:D-odd-stability} and (U1),
$|(V^\theta)^2-(V^*)^2|(t,x)\le C_0L_\theta x^2$. Moreover $a'(t)\le t$ by \ref{item:D-K3},
and the even functions $U^*(t,\cdot)^2$, $V^*(t,\cdot)^2$ and $U^*(t,\cdot)^3$ have
$\|\cdot\|_{10}\le C_0$. Lemma~\ref{lem:D-propagation} and \eqref{eq:D-moments} therefore give
\[
 |N_{\zeta_\theta}(t)-N_{\zetas}(t)|\le C_0tL_\theta+C(\theta^5+\theta L_\theta),
 \qquad
 |\sigma_{\zeta_\theta}(t)-\sigma_{\zetas}(t)|\le C_0L_\theta+C(\theta^5+\theta L_\theta).
\]
With \eqref{eq:D-rho-sigma}, in particular $N_{\zetas}(t)\le C_0t$ and
$\sigma_\zeta\ge1$,
\begin{equation}\label{eq:D-pointwise}
 |\Delta_\theta(t)|
 \le\frac{|N_{\zeta_\theta}-N_{\zetas}|}{\sigma_{\zeta_\theta}}
 +\frac{N_{\zetas}\,|\sigma_{\zeta_\theta}-\sigma_{\zetas}|}
 {\sigma_{\zeta_\theta}\,\sigma_{\zetas}}
 \le C_0tL_\theta+C\bigl(\theta^5+\theta L_\theta\bigr)\qquad(q<t<v_-).
\end{equation}

\emph{The upper endpoints.} Suppose $v_-<v_+$, let $\zeta_+$ be the measure among
$\zeta_\theta,\zetas$ whose support ends at $v_+$, and let $\zeta_-$ be the other one. On
$[v_+,1]$ both distribution functions equal one, so the two Parisi equations coincide there,
with the same terminal value, and $U^\theta(v_+,\cdot)=U^*(v_+,\cdot)$. The function
$\Gamma_{\zeta_+}$ is the identity on an interval ending at $v_+$ (on $[q,v_\theta]$ or
$[0,v_*]$), and $\Gamma'_{\zeta_+}$ is continuous, so $\Gamma'_{\zeta_+}(v_+)=1$; likewise
$\Gamma'_{\zeta_-}(v_-)=1$. On $(v_-,v_+)\subset(0,K)$ the measure $\zeta_-$ has a plateau of
height one, so \eqref{eq:D-top-concave} gives $\Gamma'_{\zeta_-}(v_+)\le1-(v_+-v_-)$. Hence,
with $F=U^*(v_+,\cdot)^2$,
\[
 v_+-v_-\le\bigl|\Gamma_\theta'(v_+)-\Gamma_*'(v_+)\bigr|
 =a(v_+)\,\bigl|\nu^\theta_{v_+}(F)-\nu^*_{v_+}(F)\bigr|,
\]
and Lemma~\ref{lem:D-propagation} with $t=v_+\le K$ gives
\begin{equation}\label{eq:D-endpoint}
 |v_\theta-v_*|\le C_0KL_\theta+C\bigl(\theta^5+\theta L_\theta\bigr).
\end{equation}

\emph{The contraction.} On $[v_+,1]$ we have $\Delta_\theta=0$, and $|\Delta_\theta|\le1$
everywhere. Integrating \eqref{eq:D-pointwise} over $(q,v_-)$ and adding \eqref{eq:D-endpoint},
\[
 L_\theta\le\int_q^{v_-}|\Delta_\theta|+(v_+-v_-)
 \le2C_0K\,L_\theta+2C\bigl(\theta^5+\theta L_\theta\bigr).
\]
Since $2C_0K\le\frac14$, for $\theta\le1/(8C)$ this gives $L_\theta\le4C\theta^5$, which is
(a). Inserting (a) into \eqref{eq:D-pointwise} and \eqref{eq:D-endpoint} gives (b) on
$(q,v_-)$ and on $[v_+,1]$. At $t=q$ both distribution functions are right continuous, so
$\Delta_\theta(q)=\lim_{t\downarrow q}\Delta_\theta(t)=O(\theta^5)$; in particular
$x_\theta=\alpha_\theta(q)=\alphas(q)+O(\theta^5)$.

\paragraph{A slope bound and part (c).}
Let $\zeta\in\{\zeta_\theta,\zetas\}$, and let $J$ be its open support interval,
$(q_\theta,v_\theta)$ or $(0,v_*)$. On $J$ the function $\alpha_\zeta=R_\zeta$ is continuous,
the jets are $C^1$ in time by (U1), and It\^o's formula gives, with $U,V,W$ taken along
$X^\zeta$ and $\alpha=\alpha_\zeta(t)$,
\[
 \frac d{dt}\E U^2=a\E\bigl(V^2-2\alpha U^3\bigr),\qquad
 \frac d{dt}\E V^2=a\E\bigl(W^2-6\alpha UV^2\bigr),\qquad
 \frac d{dt}\E U^3=3a\E\bigl(UV^2-\alpha U^4\bigr),
\]
as in the proof of Lemma~\ref{lem:plat-ito}(c) with the plateau height replaced by
$\alpha_\zeta(t)$. Hence
\[
 N_\zeta'=a''\E U^2+aa'\E\bigl(V^2-2\alpha U^3\bigr)+2aa'\E V^2
 +a^3\E\bigl(W^2-6\alpha UV^2\bigr),
\]
and $|\sigma_\zeta'|\le C_0$. By \eqref{eq:D-K4},
$\E W^2\ge(\E|W|)^2\ge(2-2\eta)^2\ge3$, and by \ref{item:D-K3}, (U1) and \eqref{eq:D-moments} the
remaining terms of $N_\zeta'$ are at least $-C_0K$. With \eqref{eq:D-rho-sigma},
\begin{equation}\label{eq:D-slope}
 R_\zeta'\ge\frac{N_\zeta'}{\sigma_\zeta}
 -\frac{N_\zeta\,|\sigma_\zeta'|}{\sigma_\zeta^2}
 \ge\frac{3-C_0K}{3}-C_0K\ge\frac12\qquad\text{on }J,
\end{equation}
because $C_0K\le\eta$ by the choice of $K$ in Appendix~\ref{sec:dat-constants}. Thus
$\alpha_\theta$ and $\alphas$ increase with slope at least $\frac12$ on $[q_\theta,v_\theta)$
and on $[0,v_*)$: for $t_1<t_2$ in either interval, $\alpha(t_2)-\alpha(t_1)\ge\frac12(t_2-t_1)$.

Let $e_\theta=\sup_{[q,v_-)}|\Delta_\theta|+(v_+-v_-)$, which is $O(\theta^5)$ by
(b), and let $x\in(x_\theta,1)$. Then $t:=Q_\theta(x)\in[q,v_\theta]$ and
$\alpha_\theta(t)\ge x$. If $t\ge v_-$, then $Q_*(x)\le v_*\le v_+\le t+e_\theta$. If
$t<v_-$, then $\alphas(t)\ge x-e_\theta$, and the slope bound gives
$Q_*(x)\le t+2e_\theta$ (if $t+2e_\theta\ge v_*$, use $Q_*(x)\le v_*$).
Conversely, let $t'=Q_*(x)\in(0,v_*]$. If $t'\ge v_-$, then
$Q_\theta(x)\le v_\theta\le v_+\le t'+e_\theta$. If $q\le t'<v_-$, the same argument
with the roles of the two measures exchanged gives $Q_\theta(x)\le t'+2e_\theta$. If
$t'<q$, then $x\le\alphas(t')\le\alphas(q)\le x_\theta+e_\theta$; the slope of $\alphas$
on $[0,v_*)$ gives $q-t'\le2(\alphas(q)-\alphas(t'))<2e_\theta$, and the slope of
$\alpha_\theta$ on $[q,v_\theta)$ gives $Q_\theta(x)\le q+2e_\theta<t'+4e_\theta$,
while $Q_\theta(x)\ge q>t'$. In all cases $|Q_\theta(x)-Q_*(x)|\le4e_\theta$, which is
(c).

\paragraph{Part (d).}
By Lemma~\ref{lem:D-contact-closure} and (a), $I_\theta=O(\theta^4)$. On the other hand, by
\eqref{eq:D-alpha-star-linear},
\[
 I_\theta=\int_0^q\bigl(\theta-\alphas(s)\bigr)\,s\,ds
 =q^2\Bigl(\frac\theta2-\frac{\rhostar q}3\Bigr)+O(q^4).
\]
Since $\theta/C\le q\le C\theta$, this gives $\rhostar q=\frac32\theta\,(1+O(\theta))$, the
first claim, because $3/(2\rhostar)=3\sqrt\azero/\Dstar$. Next,
$\zeta_\theta(\{q\})=x_\theta-\theta$, and by (b) and \eqref{eq:D-alpha-star-linear},
$x_\theta=\alphas(q)+O(\theta^5)=\rhostar q+O(\theta^2)=\frac32\theta\,(1+O(\theta))$, so
$\zeta_\theta(\{q\})=\frac\theta2(1+O(\theta))$. Finally
$\Wone(\zeta_\theta,\zetas)=\int_0^1|\Delta_\theta|=\int_0^q|\theta-\alphas(s)|\,ds+L_\theta$.
With $\lambda_\theta=\rhostar q/\theta=\frac32+O(\theta)$ and \eqref{eq:D-alpha-star-linear},
\[
 \int_0^q|\theta-\alphas(s)|\,ds=q\theta\int_0^1|1-\lambda_\theta\varsigma|\,d\varsigma+O(q^3)
 =\tfrac5{12}\,q\theta\,(1+O(\theta)),
\]
because $\int_0^1|1-\lambda\varsigma|\,d\varsigma=\frac\lambda2+\frac1\lambda-1$ for
$\lambda\ge1$, which equals $\frac5{12}$ at $\lambda=\frac32$ and is smooth in $\lambda$.
Hence, with (a),
\[
 \Wone(\zeta_\theta,\zetas)=\frac5{12}\cdot\frac{3}{2\rhostar}\,\theta^2(1+O(\theta))+O(\theta^5)
 =\frac{5\sqrt\azero}{4\Dstar}\,\theta^2(1+O(\theta)).
\]
This completes the proof of Proposition~\ref{prop:nc-above-contact}.
\qed

\section{Proofs for Section~\ref*{sec:skg}}\label{app:skg}

Throughout this appendix $\xi(t)=\beta^2t^2/2$ with $\beta>0$.

\subsection{Transfer of the finite-step inequality}\label{sec:H-transfer}

We prove Proposition~\ref{prop:skg-plateau}. The finite-step inequality,
Proposition~\ref{prop:skg-oai}, is stated for step functions on a uniform grid and at grid
times; we approximate $\zeta$ by such step functions at a fixed rational time, keeping
$M=\beta$, and then pass to all times by continuity.

\paragraph{The change of variables.}
Let $\zeta\in\Pspace([0,1])$, with distribution function $\alpha=\alpha_\zeta$ and diffusion $X$,
driven by $B$. Put
\[
 \Phi^O(t,x)=\beta^{-1}\bigl\{\Phi_\zeta(t,\beta x)+\log2\bigr\},\qquad
 \gamma=\beta\alpha,\qquad M=\beta,\qquad Y_t=\beta^{-1}X_t .
\]
Since $\Phi^O_x(t,x)=\Phi_{\zeta,x}(t,\beta x)$ and $\Phi^O_{xx}(t,x)=\beta\Phi_{\zeta,xx}(t,\beta x)$,
the Parisi equation \eqref{eq:intro-pde} with $\xi''\equiv\beta^2$ becomes
$\partial_t\Phi^O=-\frac12\bigl(\Phi^O_{xx}+\gamma(\Phi^O_x)^2\bigr)$, and
$\Phi^O(1,x)=M^{-1}\log(2\cosh Mx)$. The function $\gamma$ is nondecreasing and right continuous,
with $0\le\gamma\le M$. Since $dX_t=\beta^2\alpha(t)\Phi_{\zeta,x}(t,X_t)\,dt+\beta\,dB_t$, the
process $Y$ solves $dY_t=\gamma(t)\Phi^O_x(t,Y_t)\,dt+dB_t$ with $Y_0=0$. When $\gamma$ is a step
function of the kind in Proposition~\ref{prop:skg-oai}, these are therefore the equation
\eqref{eq:skg-oai-pde} and the diffusion of that proposition, by uniqueness of both. On a plateau
of $\zeta$ of height $m$, with $c=\beta m$, the functions $r=c\Phi^O_x$, $v=r_x$, $w=r_{xx}$ and
$z=r_{xxx}$ take at $x=Y_t$ the values $v=\beta^2mU_t$, $w=\beta^3mV_t$ and $z=\beta^4mW_t$, so
that
\begin{equation}\label{eq:H-scaling}
 z^2-12vw^2+6v^4=\beta^8m^2\bigl(W_t^2-12mU_tV_t^2+6m^2U_t^4\bigr),\qquad
 v^4=\beta^8m^4U_t^4 .
\end{equation}

\begin{proof}[Proof of Proposition~\ref{prop:skg-plateau}]
Fix $\zeta$ and the plateau $(u,v)$ of height $m>0$; then $\alpha=m$ on $[u,v)$.

\emph{Step 1: approximants.} Fix a rational $t\in(u,v)$ and a rational $\delta>0$ with
$[t-\delta,t+\delta]\subset(u,v)$. For integers $n>1/m$ put $\ell_k=k(t-\delta)/n$ and
$\rho_k=t+\delta+k(1-t-\delta)/n$ for $k=0,\dots,n$, and define
\[
 \alpha_n(s)=
 \begin{cases}
  \max\{\alpha(\ell_k),1/n\}, & s\in[\ell_k,\ell_{k+1}),\ k=0,\dots,n-1,\\
  m, & s\in[t-\delta,t+\delta),\\
  \alpha(\rho_k), & s\in[\rho_k,\rho_{k+1}),\ k=0,\dots,n-1,
 \end{cases}
 \qquad \alpha_n(1)=1 .
\]
On the left block $\alpha(\ell_k)\le\alpha(t-\delta)=m$ and $1/n<m$; on the right block
$\alpha(\rho_k)\ge\alpha(t+\delta)=m$. So $\alpha_n$ is nondecreasing, right continuous, with
finitely many jumps and $1/n\le\alpha_n\le1$: it is the distribution function of a finitely atomic
probability measure $\zeta_n$. As $n\to\infty$, $\alpha_n\to\alpha$ at every continuity point of
$\alpha$ off the grid points, because the grid point to the left of such a point $s$ tends to $s$
from below; so $\alpha_n\to\alpha$ almost everywhere, and $\Wone(\zeta_n,\zeta)=\int_0^1|\alpha_n-\alpha|\to0$
by dominated convergence. All the points $\ell_k$, $\rho_k$ and $t$ are rational. Let $D=D_n$ be
a common denominator. Every jump of $\alpha_n$ in $[0,1)$ is at some $\ell_k$ or $\rho_k$, so
$\alpha_n$ is constant on each cell $[j/D,(j+1)/D)$, $j=0,\dots,D-1$, and $t=i/D$ for an integer
$0<i<D$.

\emph{Step 2: the finite-step inequality.} Put $c_j=\beta\alpha_n(j/D)$ for $j=0,\dots,D-1$ and
$c_j=c_{D-1}$ for $j\ge D$. Then $\beta\alpha_n=c_j$ on $[j/D,(j+1)/D)$, the sequence $(c_j)$ is
nondecreasing, and $\beta/n\le c_j\le\beta=M$; equal neighbours are allowed. The cell
$[i/D,(i+1)/D)$ lies in $[t-\delta,t+\delta)$, because $t+\delta$ is a multiple of $1/D$ larger
than $t$, so $c_i=\beta m$. Proposition~\ref{prop:skg-oai} applies with this $D$, with $M=\beta$
and at the grid time $t=i/D$. By the change of variables applied to $\zeta_n$
and~\eqref{eq:H-scaling}, divided by $\beta^8m^2$,
\begin{equation}\label{eq:H-step}
 \E\bigl[(W^n_t)^2-12mU^n_t(V^n_t)^2+6m^2(U^n_t)^4\bigr]\ge\frac1{1000}\,m^2\,\E(U^n_t)^4,
\end{equation}
where $U^n,V^n,W^n$ are the derivatives of $\Phi_{\zeta_n}$ along the diffusion of $\zeta_n$.

\emph{Step 3: the limit at fixed $t$.} Both sides of \eqref{eq:H-step} are expectations, along the
diffusion of $\zeta_n$, of fixed polynomials in the second to fourth spatial derivatives of
$\Phi_{\zeta_n}$. By \cite[Proposition~1(ii)]{AC15a}, read along the diffusion as explained in
the analytic facts at the beginning of Section~\ref{sec:plat}, they converge to the same
expressions for $\zeta$, since $\zeta_n\to\zeta$ weakly; the grid $D_n$ may change with $n$, but
$t$ is fixed. Hence \eqref{eq:skg-plateau} holds at every rational $t\in(u,v)$.

\emph{Step 4: every $t$.} Both sides of \eqref{eq:skg-plateau} are expectations along $X$ of fixed
polynomials in finitely many spatial derivatives of $\Phi_\zeta$, hence continuous in $t$ (same
reference). The rationals are dense in $(u,v)$.
\end{proof}

At a grid time the right-continuous value $c_i$ of the step function enters $r$; in Step~2 the
cells on both sides of $t$ carry the value $\beta m$. The proof uses that $M=\beta$ is fixed, so
the derivatives of $\Phi_{\zeta_n}$ are bounded uniformly in $n$ by \eqref{eq:plat-bounds}, and
neither compact time strips nor terminal smoothings are needed.

\subsection{Sign patterns}\label{sec:H-signs}

\begin{lemma}[Sign patterns]\label{lem:H-signs}
Let $a<b$, and let $h\in C^1[a,b]$ have a convex derivative $h'$ on $[a,b]$.
\begin{thmlist}
\item If $h(a)=h(b)=0$ and $h'(b)<0$, then $h>0$ on $(a,b)$.
\item If $h(a)=0$, $h'(a)<0$ and $h(b)<0$, then $h<0$ on $(a,b]$.
\end{thmlist}
\end{lemma}

\begin{proof}
A convex function $g$ admits no points $s_1<s_2<s_3$ with $g(s_1)\le0<g(s_2)$ and $g(s_3)\le0$,
because $g(s_2)\le\max\{g(s_1),g(s_3)\}$.

(a) Suppose $h(s_0)\le0$ for some $s_0\in(a,b)$. Since $h(b)=0$ and $h'(b)<0$, some
$s_4\in(s_0,b)$ has $h(s_4)>0$. The mean value theorem gives $s_1\in(a,s_0)$ with
$h'(s_1)=h(s_0)/(s_0-a)\le0$ and $s_2\in(s_0,s_4)$ with $h'(s_2)>0$. With $s_3=b$ this is the
excluded pattern for $g=h'$.

(b) Suppose $h(s_0)\ge0$ for some $s_0\in(a,b]$. Then $s_0<b$, and since $h(a)=0$ and $h'(a)<0$,
some $s_4\in(a,s_0)$ has $h(s_4)<0$. The mean value theorem gives $s_2\in(s_4,s_0)$ with
$h'(s_2)>0$ and $s_3\in(s_0,b)$ with $h'(s_3)=(h(b)-h(s_0))/(b-s_0)<0$. With $s_1=a$ this is the
excluded pattern.
\end{proof}

\subsection{The tilted Gaussian}\label{sec:H-gauss}

We prove Lemma~\ref{lem:skg-branch}. Fix $\theta\in[0,1]$ and $w>0$, and write $\E$, $\mathrm{Var}$
and $\mathrm{Cov}$ for the probability measure $\mu$ on $\R$ with density proportional to
$\cosh^\theta(x)\,e^{-x^2/(2w)}$, so that $\E F=\E_{w,\theta}F$ in the notation of
Section~\ref{sec:nc-dat}. Abbreviate
\[
 \eta=\tanh^2x,\qquad s=\sech^2x=1-\eta,\qquad e_\theta=\cosh^\theta x ,
\]
and note $e_\theta'/e_\theta=\theta\tanh x$ and $e_\theta''/e_\theta=\theta s+\theta^2\eta=\theta-\theta(1-\theta)\eta$.

\paragraph{Derivative formulas.}
Let $f$ be even and smooth with $f$ and its derivatives bounded. Then
\begin{gather}
 \partial_w\E f=\tfrac12\E\bigl[f''+2\theta\tanh x\,f'\bigr]-\tfrac12\theta(1-\theta)\,\mathrm{Cov}(f,\eta),
 \label{eq:H-dw}\\
 \partial_\theta\E f=\mathrm{Cov}(f,\log\cosh x),\qquad
 \partial_w\E f=\frac{\mathrm{Cov}(f,x^2)}{2w^2}.\label{eq:H-dtheta}
\end{gather}
For \eqref{eq:H-dw}, write $\E f=M_{fe_\theta}/M_{e_\theta}$ with $M_g(w)=\E g(\sqrt wZ)$, use
$\partial_wM_g=\frac12M_{g''}$ (proof of Lemma~\ref{lem:C-moments}) and
$(fe_\theta)''/e_\theta=f''+2\theta\tanh x\,f'+f\{\theta-\theta(1-\theta)\eta\}$, and apply the
quotient rule. The formulas \eqref{eq:H-dtheta} differentiate the logarithm
$\theta\log\cosh x-x^2/(2w)$ of the unnormalized density under the integral, which is dominated.
For $f=\eta$, with $f'=2s\tanh x$ and $f''=2s^2-4\eta s$, \eqref{eq:H-dw} gives
\begin{equation}\label{eq:H-dwT}
 \partial_w\Tfrak=\E\bigl[s^2-2(1-\theta)\eta s\bigr]-\tfrac12\theta(1-\theta)\,\mathrm{Var}\,\eta ;
 \qquad\text{in particular}\qquad \partial_w\Tfrak(w,1)=\Sfrak(w,1).
\end{equation}

\begin{proof}[Proof of Lemma~\ref{lem:skg-branch}\textup{(a)}]
Put $F_\theta=s^2-2(1-\theta)\eta s$. By \eqref{eq:H-dwT},
\begin{equation}\label{eq:H-drop}
 \Tfrak-w\,\partial_w\Tfrak=\bigl(\E\eta-w\,\E F_\theta\bigr)+\tfrac12w\theta(1-\theta)\,\mathrm{Var}\,\eta
 \ge\E\eta-w\,\E F_\theta .
\end{equation}

\emph{Step 1: removing $w$.} Let $g(x)=e_\theta(x)^{-1}\int_0^xe_\theta F_\theta\,dy$. Then $g$ is
odd, $|g(x)|\le3|x|$ because $|F_\theta|\le3$ and $e_\theta$ increases in $|x|$, and
$(e_\theta g)'=e_\theta F_\theta$, that is, $g'+\theta\tanh x\,g=F_\theta$. The density $\varrho$ of
$\mu$ satisfies $\varrho'=\varrho\,(\theta\tanh x-x/w)$, so integration by parts, with no boundary
terms because of the Gaussian tails, gives $\E g'+\theta\E[\tanh x\,g]=\E[xg]/w$, that is,
$w\,\E F_\theta=\E[xg]$. Hence $\E\eta-w\,\E F_\theta=\E f_\theta$ with $f_\theta(x)=\eta-xg(x)$.

\emph{Step 2: rewriting $f_\theta$.} Since $(e_\theta s\tanh x)'=e_\theta(\theta\eta s+s^2-2\eta s)$,
we have $e_\theta F_\theta=(e_\theta s\tanh x)'+\theta e_\theta\eta s$, and therefore
\[
 f_\theta(x)=k(x)-\theta x\,e_\theta(x)^{-1}\int_0^xe_\theta\eta s\,dy,\qquad
 k(x)=\tanh^2x-x\tanh x\sech^2x .
\]

\emph{Step 3: $f_\theta>0$ for $x\ne0$.} The function $f_\theta$ is even; let $x>0$ and put
$\Psi(x)=e_\theta(x)f_\theta(x)=e_\theta k-\theta x\int_0^xe_\theta\eta s\,dy$. Then $\Psi(0)=0$ and
\[
 \Psi'=e_\theta\bigl(\theta\tanh x\,k+k'\bigr)-\theta\int_0^xe_\theta\eta s\,dy-\theta xe_\theta\eta s .
\]
Since $e_\theta$ is nondecreasing on $[0,\infty)$ and $(\tanh^3)'=3\eta s$, the integral is at most
$e_\theta(x)\tanh^3x/3$. With $\tanh x\,k=\tanh^3x-x\eta s$ this gives
\[
 \Psi'\ge e_\theta\bigl[(1-\theta)k'+\theta k_1\bigr],\qquad
 k'=s\bigl(\tanh x-xs+2x\eta\bigr),\qquad
 k_1=\tfrac23\tanh^3x+s\bigl(\tanh x-xs\bigr).
\]
The function $\tanh x-x\sech^2x$ vanishes at zero and has derivative $2x\sech^2x\tanh x>0$, so it
is positive for $x>0$. Hence $k'>0$ and $k_1>0$ on $(0,\infty)$, so $\Psi'>0$, $\Psi>0$ and
$f_\theta>0$ there. Near zero, $f_\theta(x)=\frac13(2-\theta)x^4+O(x^6)$.

Since $\mu$ has a positive density, $\E f_\theta>0$, and \eqref{eq:H-drop} gives
$\Tfrak>w\,\partial_w\Tfrak$. Finally $\partial_w\mathfrak R=(w\,\partial_w\Tfrak-\Tfrak)/w^2<0$.
\end{proof}

At $\theta=0$, $f_0=k$; at $\theta=1$ the variance term in \eqref{eq:H-drop} vanishes and, by
\eqref{eq:H-dwT}, part (a) reads $\Sfrak(w,1)<\mathfrak R(w,1)$.

\begin{proof}[Proof of Lemma~\ref{lem:skg-branch}\textup{(b)} and \textup{(c)}]
(b) By \eqref{eq:H-dtheta}, $\partial_\theta\Tfrak=\mathrm{Cov}(\eta,\log\cosh x)$,
$\partial_\theta\Sfrak=\mathrm{Cov}(s^2,\log\cosh x)$ and
$\partial_w\Sfrak=\mathrm{Cov}(s^2,x^2)/(2w^2)$. Each function involved depends on $|x|$ only:
$\eta$, $\log\cosh x$ and $x^2$ are strictly increasing in $|x|$, and $s^2$ is strictly
decreasing. Under $\mu$ the variable $|x|$ is not almost surely constant, so Chebyshev's
covariance inequality, $\mathrm{Cov}(\phi(|x|),\psi(|x|))=\frac12\E[(\phi(|x|)-\phi(|x'|))(\psi(|x|)-\psi(|x'|))]$
with $x'$ an independent copy, gives the three signs.

(c) At $\theta=0$, $\mu=\mathcal N(0,w)$ and $\E\tanh x=0$, so the Gaussian Poincar\'e inequality,
with equality only for affine functions, gives
$\Tfrak(w,0)=\mathrm{Var}\tanh(\sqrt wZ)<w\,\E\sech^4(\sqrt wZ)=w\,\Sfrak(w,0)$. At $\theta=1$,
\eqref{eq:H-dwT} and part (a) give $\Tfrak(w,1)>w\,\Sfrak(w,1)$.
\end{proof}

\end{document}